\documentclass{amsart}

\usepackage{array}
\usepackage[all]{xy}
\usepackage[mathscr]{eucal}
\usepackage{amssymb}
\usepackage[centertags]{amsmath}
\usepackage{graphics,graphicx}
\usepackage{bm}
\usepackage{subfigure}
\usepackage{amsfonts}
\usepackage{amscd}
\usepackage{dsfont}
\usepackage{subfigure}

\newtheorem{theorem}{Theorem}[section]
\newtheorem{proposition}[theorem]{Proposition}
\newtheorem{lemma}[theorem]{Lemma}
\newtheorem{corollary}[theorem]{Corollary}
\newtheorem{conjecture}[theorem]{Conjecture}

\theoremstyle{definition}
\newtheorem{definition}{Definition}[section]
\newtheorem{example}{Example}[section]
\newtheorem{remark}[example]{Remark}

\numberwithin{equation}{subsection}

\newcommand{\R}{\mathbb{R}}
\newcommand{\Z}{\mathbb{Z}}
\newcommand{\C}{\mathbb{C}}

\newcommand{\CP}[1]{\mathbb{CP}^{#1}}
\newcommand{\RP}[1]{\mathbb{RP}^{#1}}
\newcommand{\To}{\rightarrow}
\newcommand{\del}[2]{\frac{\partial #1}{\partial #2}}

\newcommand{\wedgestar}{\Lambda \! ^*}
\newcommand{\fmod}[2]{#1\!\operatorname{-}\!#2}
\newcommand{\fmodf}[2]{#1\!\operatorname{-mod-}\!#2}
\newcommand{\Cl}{C \! \ell}
\newcommand{\llbracket}{[\![}
\newcommand{\rrbracket}{]\!]}
\newcommand{\cA}{\mathcal{A}}
\newcommand{\CO}{\mathcal{CO}}
\newcommand{\OC}{\mathcal{OC}}
\newcommand{\cG}{\mathcal{G}}
\newcommand{\cF}{\mathcal{F}}
\newcommand{\cM}{\mathcal{M}}
\newcommand{\hcM}{\widehat{\mathcal{M}}}
\newcommand{\scrA}{\mathscr{A}}

\begin{document}
\title{On the Fukaya category of a Fano hypersurface in projective space}
\author{Nick Sheridan}
\address{Institute for Advanced Study, Einstein Drive, Princeton NJ 08540}
\email{nicks@math.ias.edu}
\thanks{This work was partially supported by the National Science Foundation through Grant number DMS-1310604, and under agreement number DMS-1128155.
Any opinions, findings and conclusions or recommendations expressed in this material are those of the author and do not necessarily reflect the views of the National Science Foundation.}
\begin{abstract}
This paper is about the Fukaya category of a Fano hypersurface $X \subset \CP{n}$.
Because these symplectic manifolds are monotone, both the analysis and the algebra involved in the definition of the Fukaya category simplify considerably.
The first part of the paper is devoted to establishing the main structures of the Fukaya category in the monotone case: the closed--open string maps, weak proper Calabi--Yau structure, Abouzaid's split-generation criterion, and their analogues when weak bounding cochains are included.
We then turn to computations of the Fukaya category of the hypersurface $X$: we construct a configuration of monotone Lagrangian spheres in $X$, and compute the associated disc potential.
The result coincides with the Hori--Vafa superpotential for the mirror of $X$ (up to a constant shift in the Fano index $1$ case).
As a consequence, we give a proof of Kontsevich's homological mirror symmetry conjecture for $X$.
We also explain how to extract non-trivial information about Gromov--Witten invariants of $X$ from its Fukaya category.
\end{abstract}
\maketitle

\tableofcontents

\section{Introduction}

This paper computes the Fukaya category of a Fano hypersurface in projective space, by applying techniques which were developed in \cite{Sheridan2015} to compute the Fukaya category of a Calabi--Yau hypersurface in projective space. 
This enables us to prove a version of Kontsevich's homological mirror symmetry conjecture \cite{Kontsevich1994,Kontsevich1998}.
One main challenge of this work was to understand the `right' way to turn our computations into a proof of homological mirror symmetry: i.e., the way that is most likely to generalize to other monotone symplectic manifolds.

To start with, one has to understand the nature of the Fukaya category (and Gromov--Witten invariants) of a monotone symplectic manifold.
We draw on the results of many authors to give a reasonably complete survey of the construction of these invariants, and the main results about them.
The main results are summarized in \S \ref{subsec:imonfuk} of this introduction.

Then, one has to understand what homological mirror symmetry means for Fano hypersurfaces.
We give our version in \S \ref{subsec:ihms} of the introduction, where we state our main results precisely, then discuss the implications of our results for the interpretation of homological mirror symmetry for more general monotone symplectic manifolds.
In this introduction, the results stated as `theorems' and so on are given precisely, but some of the explanatory text glosses over technical details to explain the important ideas.

\subsection{The monotone Fukaya category}
\label{subsec:imonfuk}

When $X$ is a monotone symplectic manifold, one can use `classical' pseudoholomorphic curve theory to define Gromov--Witten invariants (see \cite{Ruan1995,mcduffsalamon}) and Lagrangian Floer theory (see \cite{Oh1993,Oh1995}), without appealing to the heavy analytic machinery required for the fully general definitions \cite{fooo}.
Furthermore, the algebra involved in studying the Fukaya category simplifies considerably.
We give a survey of the construction of the Fukaya category of a monotone symplectic manifold in \S \ref{sec:monfuk}, together with the associated algebraic structures.
Much of the material is not original, nor is it as general as possible: a more general treatment will appear in \cite{Abouzaid2012}.
Nevertheless, we feel it is useful to collect these results in a unified way in this simple case: it displays many of the crucial features of the general theory, with many fewer technical details, and furthermore some of the features (e.g., the decomposition into eigenvalues of $c_1 \star$) are specific to the monotone case.
We also note that our proof that the closed--open and open--closed string maps are dual (Proposition \ref{proposition:phico}) is original: it avoids the need to construct a strictly cyclic structure on the Fukaya category, and uses instead the notion of a weak proper Calabi--Yau structure.

In this section, we summarize the main results of \S \ref{sec:monfuk}.

\begin{remark}
\label{remark:coeff}
First, a remark about coefficients.
Our algebraic structures (algebras, categories and so on) will be $\C$-linear.
In particular, we do not use a Novikov ring.
This is possible by monotonicity: the infinite sums which the Novikov ring is supposed to deal with are in fact finite in this setting.
We could also have chosen to work over a Novikov polynomial ring:
\begin{equation}
\left\{ \sum_{j=0}^N c_j r^{\lambda_j}: c_j \in \C, \lambda_j \in \R_{\ge 0}\right\}
\end{equation}
(by weighting each count of holomorphic maps by $r^{symp. \, area}$), or over its completion with respect to the energy filtration (the Novikov ring $\Lambda_0$), or over the field of fractions thereof (the Novikov field $\Lambda$).
All of the results in this paper have variations which hold over these various coefficient rings.
We have chosen to work over $\C$ to make things as conceptually simple as possible.
\end{remark}

We recall the \emph{quantum cohomology ring} \cite{Ruan1995,mcduffsalamon}.
We define $QH^*(X):=H^*(X;\C)$, and equip it with the quantum cup product, which we denote by $\star$.
It is a graded, supercommutative, unital $\C$-algebra, and furthermore a Frobenius algebra with respect to the intersection pairing:
\begin{equation} \langle \alpha \star \beta, \gamma \rangle = \langle \alpha, \beta \star \gamma \rangle\end{equation}
(we caution that the grading group is not $\Z$, but we won't go into details about the grading in this introduction).
We denote the unit by $e \in QH^0(X)$, and recall it coincides with the unit $e \in H^0(X)$.

In \S \ref{subsec:monfuk}, we define the monotone Fukaya category $\cF(X)$, following \cite{Seidel2008}, with minor modifications (compare \cite{Biran2013}).
The objects of $\cF(X)$ are monotone Lagrangian submanifolds $L \subset X$ equipped with a brane structure (grading, spin structure, $\C^*$-local system), such that the image of $\pi_1(L)$ is trivial in $\pi_1(X)$.
Such Lagrangians are always orientable, so their minimal Maslov number is $\ge 2$.
For each such $L$, the signed count of Maslov index $2$ discs passing through a generic point on $L$, weighted by the monodromy of the local system around the boundary, defines a number $w(L) \in \C$.

For any two objects $L_0,L_1$, one defines the morphism space $CF^*(L_0,L_1)$ to be the graded $\C$-vector space generated by intersection points between $L_0$ and $L_1$ (perturbing by a Hamiltonian flow to make the intersections transverse).
One defines $A_{\infty}$ structure maps
\begin{equation}\mu^s: CF^*(L_{s-1},L_s) \otimes \ldots \otimes CF^*(L_0,L_1) \To CF^*(L_0,L_s)\end{equation}
for $s \ge 1$ by counting pseudoholomorphic discs with boundary conditions on the $L_j$, weighted by the holonomy of the local systems around the boundary.

These structure maps $\mu^s$ satisfy the $A_{\infty}$ relations, with the sole exception that the differential
\begin{equation} \mu^1: CF^*(L_0,L_1) \To CF^*(L_0,L_1)\end{equation}
does not square to zero:
\begin{equation}
\label{eqn:mu1no0}
\mu^1(\mu^1(x)) = (w(L_0) - w(L_1))x.\end{equation}

\begin{definition}
For each $w \in \C$, we define $\cF(X)_w$ to be the full subcategory whose objects are those $L$ with $w(L) = w$.
Then each $\cF(X)_w$ is individually a graded $A_{\infty}$ category.
\end{definition}

In particular, for any two objects $K,L$ of $\cF(X)_w$, the Floer cohomology group is well defined:
\begin{equation}
 HF^*(K,L) := H^*(CF^*(K,L),\mu^1).
\end{equation}
Furthermore, each $\cF(X)_w$ is cohomologically unital: we denote the cohomological unit by
\begin{equation}
e_L \in HF^*(L,L)
\end{equation}

\begin{proposition}
\label{proposition:co0} (see \cite{Auroux2007}, and our \S \ref{subsec:co})
For any object $L$ of $\cF(X)_w$, there is a unital $\C$-algebra homomorphism
\begin{equation} \CO^0: QH^*(X) \To HF^*(L,L).\end{equation}
\end{proposition}

\begin{proposition}
\label{proposition:coc1} (see \cite{Auroux2007}, and our Lemma \ref{lemma:c1u0})
The map $\CO^0$ satisfies
\begin{equation} \CO^0(c_1) = w \cdot e_L,\end{equation}
where $c_1$ is the first Chern class.
\end{proposition}

Now, let
\begin{equation}
\label{eqn:qhdecomp}
 QH^*(X) \cong \bigoplus_w QH^*(X)_w,
\end{equation}
where $QH^*(X)_w$ is the generalized eigenspace of the endomorphism
\begin{equation}
c_1 \star: QH^*(X) \To QH^*(X)
\end{equation}
corresponding to the eigenvalue $w \in \C$.
If $e \in QH^*(X)$ denotes the identity element, then we denote by $e_w$ the projection of $e$ to $QH^*(X)_w$.
Because $QH^*(X)$ is a Frobenius algebra, \eqref{eqn:qhdecomp} is a decomposition as algebras, and $QH^*(X)_w$ is a unital subalgebra with unit $e_w$.

\begin{proposition}
\label{proposition:co}
(see \cite[\S 13]{fooo}, \cite{Seidel2002}, and our \S \ref{subsec:co}) For each $w \in \C$, there is a unital algebra homomorphism
\begin{equation} \CO: QH^*(X) \To HH^*(\cF(X)_w),\end{equation}
extending $\CO^0$, called the \emph{closed--open string map}.
Here `$HH^*$' denotes Hochschild cohomology: it carries an associative product called the Yoneda product.
\end{proposition}

It is an observation going back to \cite{Auroux2007} that Proposition \ref{proposition:coc1} implies that the Fukaya category, as well as the various closed--open and open--closed maps, split up into components indexed by the eigenvalues of $c_1 \star$. 
The following results make this idea precise: our formulations and proofs of these results are heavily based on work of Ritter and Smith \cite{Ritter2012}.

\begin{proposition}
\label{proposition:covan}
(see \cite[Theorem 9.6]{Ritter2012} and our Proposition \ref{proposition:coeigsplit})
The restriction of the closed--open string map
\begin{equation} \CO: QH^*(X)_{w'} \To HH^*(\cF(X)_w)\end{equation}
vanishes if $w' \neq w$, and is a unital algebra homomorphism if $w' = w$.
\end{proposition}

\begin{corollary}
\label{corollary:eigfuksplit}
$\cF(X)_w$ is trivial unless $w$ is an eigenvalue of $c_1 \star$.
\end{corollary}

\begin{proposition}
\label{proposition:oc}
(see \cite[\S 13]{fooo}, \cite{Abouzaid2010a,Ganatra2012,Ritter2012}, and our \S \ref{subsec:oc}) For each $w \in \C$, there is a homomorphism of $QH^*(X)$-modules
\begin{equation} \OC: HH_*(\cF(X)_w) \To QH^{*+n}(X),\end{equation}
called the \emph{open--closed string map}.
Here `$HH_*$' denotes Hochschild homology: it acquires its $QH^*(X)$-module structure via $\CO$ (recalling that Hochschild homology is naturally a module over Hochschild cohomology).
\end{proposition}

\begin{proposition}
\label{proposition:ocim}
(see \cite[Theorem 9.4]{Ritter2012} and our Corollary \ref{corollary:oceigsplit}) 
The open--closed string map
\begin{equation}
\OC: HH_*(\cF(X)_w) \to QH^{*+n}(X)
\end{equation}
lands in the $w$-generalized eigenspace $QH^{*+n}(X)_w \subset QH^{*+n}(X)$.
\end{proposition}

In particular, both of the maps $\CO$ and $\OC$ decompose into components indexed by eigenvalues of $c_1\star$: 
\begin{align}
\CO_w: QH^*(X)_w &\to HH^*(\cF(X)_w)\\
\OC_w: HH_*(\cF(X)_w) &\to QH^{*+n}(X)_w,
\end{align}
and all other components of the maps vanish.

The next piece of structure that we consider on the monotone Fukaya category is a weak version of a Calabi--Yau structure.

\begin{proposition}
\label{proposition:cy}
(see \S \ref{subsec:infin})
For each $w \in \C$, the element
\begin{equation}
[\phi] \in HH_n(\cF(X)_w)^\vee,
\end{equation}
 defined by
\begin{equation} [\phi](b) := \langle \OC(b), e \rangle,\end{equation}
is an $n$-dimensional weak proper Calabi--Yau structure in the sense of Definition \ref{definition:weakcy}.
\end{proposition}

Proposition \ref{proposition:cy} is a reflection of the Poincar\'{e} duality isomorphisms
\begin{equation}
HF^*(K,L) \cong HF^{n-*}(L,K)^\vee
\end{equation}
in the Donaldson--Fukaya category.
Combining Propositions \ref{proposition:oc} and \ref{proposition:cy}, we obtain

\begin{proposition}
\label{proposition:coocdual}
(see Corollary \ref{corollary:cowdual})
For each $w \in \C$, there is a commutative diagram
\begin{equation}
\xymatrixcolsep{5pc}\xymatrix{QH^*(X)_w \ar[r]^-{\alpha \mapsto \langle \alpha, - \rangle}_-{\cong} \ar[d]^-{\CO_w} & QH^*(X)_w^{\vee}[-2n]  \ar[d]^-{\OC_w^\vee} \\
HH^*(\cF(X)_w) \ar[r]^-{-\cap [\phi]}_-{\cong} & HH_*(\cF(X)_w)^\vee[-n]}
 \end{equation}
(for the bottom isomorphism, see Lemma \ref{lemma:hhdual}).
In particular, the maps $\CO_w$ and $\OC_w$ are dual, up to natural identifications of their respective domains and targets.
\end{proposition}

\begin{remark}
In \cite{fooo,Abouzaid2012}, the duality of $\CO$ and $\OC$ holds because the corresponding moduli spaces are identified, so the maps are dual on the chain level.
This is impossible to arrange with our analytic setup, but Proposition \ref{proposition:coocdual} gives a cohomology-level analogue which suffices for our purposes.
\end{remark}

The next result is a version of Abouzaid's split-generation criterion \cite{Abouzaid2010a}, adapted to the monotone setup:

\begin{proposition}
\label{proposition:gencrit}
(see \cite{Abouzaid2010a,Abouzaid2012}, and our Corollary \ref{corollary:gen2})
If $\cG_w \subset \cF(X)_w$ is a full subcategory such that $e_w$ is contained in the image of the map
\begin{equation} \OC_w: HH_*(\cG_w) \To QH^*(X)_w,\end{equation}
then $\cG_w$ split-generates $\cF(X)_w$.
\end{proposition}

\begin{remark}
The hypothesis of Abouzaid's split-generation criterion for the wrapped Fukaya category of an exact symplectic manifold \cite[Theorem 1.1]{Abouzaid2010a} is that the unit $e$ lies in the image of $\OC$.
However, in order to prove that $\cG_w$ split-generates an object $K$, it suffices to prove that $\CO^0 \circ \OC$ contains the unit $e_K \in HF^*(K,K)$ in its image.
Thus, in view of Proposition \ref{proposition:covan}, it suffices to check that $e_w$ lies in the image of $\OC$: and more importantly, in light of Proposition \ref{proposition:ocim}, $\OC$ could not possibly contain the unit $e$ in its image unless $c_1 \star$ had only a single eigenvalue.
So Proposition \ref{proposition:gencrit} is the `right' split-generation criterion.
\end{remark}

Combining Propositions \ref{proposition:coocdual} and \ref{proposition:gencrit}, we obtain:

\begin{corollary}
\label{corollary:gendual}
(see \cite{Abouzaid2012}, and our Corollary \ref{corollary:splitgen})
If $\cG_w \subset \cF(X)_w$ is a full subcategory such that
\begin{equation} \CO_w: QH^*(X)_w \To HH^*(\cG_w)\end{equation}
is injective, then $\cG_w$ split-generates $\cF(X)_w$.
\end{corollary}

We draw attention to one special case of this result:

\begin{corollary}
\label{corollary:isemisimpgen}
(see \cite{Abouzaid2012}, and our Corollary \ref{corollary:semisimpgen})
If $QH^*(X)_w$ has rank $1$, then any non-trivial object of $\cF(X)_w$ split-generates it.
\end{corollary}

\subsection{The relative Fukaya category}
\label{subsec:irelfuk}

In \S \ref{sec:relfuk}, we recall the construction of the Fukaya category of $X$ relative to a divisor $D$ which is Poincar\'{e} dual to a multiple of the symplectic form, denoted $\cF(X,D)$ \cite{Sheridan2015}.
We recall that it is defined over the ring $R:=\C[r_1,\ldots,r_n]$ (no formal power series ring is necessary by monotonicity).
Its objects are exact Lagrangians in $X \setminus D$, and the $A_\infty$ structure maps count pseudoholomorphic discs $u$, weighted by $r_1^{u \cdot D_1} \ldots r_n^{u \cdot D_n} \in R$, where $D_i$ are the irreducible components of the smooth normal-crossings divisor $D$.

We expect there to exist an embedding
\begin{equation} \cF(X,D) \otimes_R \C \hookrightarrow \cF(X),\end{equation}
obtained by setting all $r_j = 1$ in $R$.
Unfortunately, the relationship between $\cF(X,D)$ and $\cF(X)$ is not quite so straightforward for technical reasons: in particular, because $J$-holomorphic discs are treated differently in the definitions of the two categories.
Nevertheless, we show that the analogy between the two is close enough that computations in the relative Fukaya category can be transferred to the monotone Fukaya category.

We also expect that
\begin{equation}
\label{eqn:codjdif}
 \CO(D_j) = \left[ r_j\del{\mu^*}{r_j} \right] \otimes_R 1,
\end{equation}
as proposed in \cite[Lemma 8.5]{Sheridan2015}, where the proof was only sketched.
Again, the incompatibility of the conventions in the definition of the relative and monotone Fukaya categories make the true relationship slightly more involved, but we prove a version that is sufficient for our purposes.

\subsection{Weak bounding cochains and the disc potential}
\label{subsec:iwbc}

In \S \ref{sec:wbcalg} (algebra) and \S \ref{sec:wbcgeom} (geometry), we explain how to formally enlarge the monotone Fukaya category by including weak bounding cochains, closely following \cite{fooo}.
The first technical issue to confront is that, in order for weak bounding cochains to make sense, we need \emph{strict} units; but the Fukaya category need only have cohomological units as we have defined it.
Following \cite{fooo} (although our technical setup is closer to \cite{Ganatra2012}), we circumvent the issue by constructing a \emph{homotopy unit} structure on the Fukaya category.
For the purposes of the rest of this introduction, we will brush the issue under the rug and pretend the Fukaya category has strict units $e_L \in CF^*(L,L)$.

We construct a \emph{curved} $A_\infty$ category $\cF(X)$ by putting all of the categories $\cF(X)_w$ together and setting $\mu^0 = w(L) \cdot e_L \in CF^*(L,L)$ for all $L$.
It follows from \eqref{eqn:mu1no0} and strict unitality that the $A_\infty$ relations are satisfied.
We then enlarge this category by allowing formal direct sums of objects: the resulting curved, strictly unital $A_\infty$ category is called $\cF(X)^\oplus$.
We then enlarge the category again, by allowing objects $(\bm{L},\alpha)$, where $\bm{L}$ is an object of $\cF(X)^\oplus$, and $\alpha$ is a solution of the \emph{Maurer--Cartan equation},
\begin{equation}
\label{eqn:mcint}
\mu^0+ \mu^1(\alpha)+\mu^2(\alpha,\alpha)+ \ldots = \mathfrak{P}(\alpha) \cdot e_{\bm{L}}.
\end{equation}
Such a solution $\alpha$ is called a \emph{weak bounding cochain}, and the space of all weak bounding cochains for a given $\bm{L}$ is denoted $\hcM_{weak}(\bm{L})$ (its quotient by gauge equivalence is called the \emph{Maurer--Cartan moduli space} $\cM_{weak}(\bm{L})$, but we will not use this notion), and $\mathfrak{P}$ defines a function
\begin{equation} \mathfrak{P}: \hcM_{weak}(\bm{L}) \To \C,\end{equation}
called the \emph{disc potential}.
The $A_\infty$ structure maps of this category are defined by `inserting the weak bounding cochains in all possible ways' into the previous $A_\infty$ structure maps:
\begin{multline}
\label{eqn:mcstruct}
\mu^s_{\cF^{wbc}}(p_s,\ldots,p_1) := \\ \sum_{i_s,\ldots,i_0} \mu^*_{\cF^\oplus}(\underbrace{\alpha_s,\ldots,\alpha_s}_{i_s},p_s,\underbrace{\alpha_{s-1},\ldots,\alpha_{s-1}}_{i_{s-1}},p_{s-1},\ldots,p_1,\underbrace{\alpha_0,\ldots,\alpha_0}_{i_0}).
\end{multline}
This defines a new curved, strictly unital $A_\infty$ category, which we call $\cF^{wbc}(X)$.
The curvature of the object $(\bm{L},\alpha)$ is $\mathfrak{P}(\alpha) \cdot e_{\bm{L}}$.

\begin{remark}
\label{remark:conv}
The Maurer--Cartan equation \eqref{eqn:mcint}  and the definition of the $A_\infty$ structure maps \eqref{eqn:mcstruct} do not make sense as written, because they are infinite sums which may not, a priori, converge.
To make sense of them, extra conditions must be imposed on $\alpha$:
 \begin{itemize}
\item The category of twisted complexes \cite[\S 3l]{Seidel2008} is formally analogous to the construction of $\cF^{wbc}(X)$.
There, $\alpha$ is required to be strictly lower-triangular with respect to some filtration (the curvature $\mu^0_{\bm{L}}$ and disc potential $\mathfrak{P}(\alpha)$ are required to vanish also).
This ensures convergence of \eqref{eqn:mcint}  and \eqref{eqn:mcstruct}.
However, we will want to consider weak bounding cochains which are not lower-triangular with respect to any filtration.

\item In \cite{fooo},  weak bounding cochains are required to have positive energy:
\begin{equation} \alpha \in C^*(L;\Lambda_+)\end{equation}
where $\Lambda_+$ is the maximal ideal in the Novikov ring $\Lambda_0$.
This ensures convergence in \eqref{eqn:mcint}  and \eqref{eqn:mcstruct}, because the Novikov ring is complete with respect to the energy filtration.
However, we have chosen to define the Fukaya category over $\C$, rather than with Novikov coefficients, so there is no analogue of the energy filtration on our coefficient ring (compare Remark \ref{remark:coeff}).

\item We give a different reason for convergence: we place a geometric restriction on our weak bounding cochains called \emph{monotonicity} (see \S \ref{subsec:conv}), which is specific to the geometric context of the monotone Fukaya category, and which ensures that \eqref{eqn:mcint}  and \eqref{eqn:mcstruct} are \emph{finite} sums, for degree reasons.
To motivate this terminology, observe that if we allowed non-monotone Lagrangians as objects of our Fukaya category, we would be forced to use a Novikov coefficient ring to achieve convergence of the disc counts defining the $A_\infty$ structure maps.
Informally, a weak bounding cochain $\alpha$ on $\bm{L}$ corresponds to a deformation of the object $\bm{L}$, which may in principle correspond to a geometric deformation of the Lagrangians $\bm{L}$ (e.g., by Lagrange surgery, compare \cite[\S 3.3.2]{Auroux2013} and \cite{Fukaya2009a}).
It makes sense that only certain special types of weak bounding cochains can correspond to monotone deformations, so that their Floer theory can be defined over $\C$: these are the monotone weak bounding cochains.
\end{itemize}
\end{remark}

We go on to establish some useful results for working with weak bounding cochains.
In Lemma \ref{lemma:prediscdisc}, we establish sufficient conditions under which an entire subspace $V \subset CF^*(\bm{L},\bm{L})$ is contained in the space of weak bounding cochains: $V \subset \hcM_{weak}(\bm{L})$.
This is an analogue of \cite[Proposition 4.3]{Fukaya2010d}, which says that if $L$ is a torus fibre in a symplectic toric manifold, then there is an embedding
\begin{equation}
\label{eqn:toricwbc}
H^1(L) \subset \hcM_{weak}(L).\end{equation}
One of the main differences between the two results is that the weak bounding cochains of \cite[Proposition 4.3]{Fukaya2010d} can \emph{never} be monotone.
The associated disc potential always contains \emph{infinitely} many terms, because constraining a disc by a codimension-$1$ cycle on the boundary does not change its index, as with the divisor axiom in Gromov--Witten theory.
Indeed, the associated disc potential often turns out to be polynomial in the exponentials of the generators, which are infinite power series.
Thus, in \cite{Fukaya2010d}, it is crucial that the coefficient ring be a Novikov ring, to deal with these infinite sums.
In contrast, in the setting of our Lemma \ref{lemma:prediscdisc}, the disc potential
\begin{equation} \mathfrak{P}: V \To \C\end{equation}
will always be a \emph{polynomial}.

We also prove a version of the well-known result that critical points $v \in V \subset CF^*(\bm{L},\bm{L})$ of the disc potential $\mathfrak{P}$ correspond to weak bounding cochains with non-vanishing Floer cohomology (compare \cite{Cho2006}), and that under some additional assumptions, the endomorphism algebra of the object $(\bm{L},v)$ is the Clifford algebra associated to the Hessian of $\mathfrak{P}$ at $v$ (compare \cite{Cho2005,Hori2001,Kapustin2004}).

Finally, we establish analogues of the basic structures of the monotone Fukaya category (i.e., those presented in \S \ref{sec:monfuk}: closed--open string maps, the split-generation criterion, etc.), when weak bounding cochains are included.
In order to obtain an analogue of the split-generation criterion (Proposition \ref{proposition:gencrit} and its dual Corollary \ref{corollary:gendual}), we are forced to impose an additional condition on our weak bounding cochains $\alpha$, which is not satisfied in general: namely, we restrict to those $\alpha$ such that the algebra homomorphism
\begin{equation} \CO^0: QH^*(X)_w \To HF^*((\bm{L},\alpha),(\bm{L},\alpha))\end{equation}
is unital when $\mathfrak{P}(\alpha) = w$, and vanishes when $\mathfrak{P}(\alpha) \neq w$.
We call the resulting subcategories $\cF^{wbc,u}(X)_w$.
Because monotone Lagrangians with vanishing weak bounding cochain have this property by Proposition \ref{proposition:covan}, we have
\begin{equation} \cF(X)_w \subset \cF^{wbc,u}(X)_w \subset \cF^{wbc}(X)_w.\end{equation}

\subsection{Homological mirror symmetry}
\label{subsec:ihms}

Following \cite{Kontsevich1998}, one expects the mirror to the monotone symplectic manifold $X$ to be a Landau-Ginzburg model $(Y,W)$, where $Y$ is a complex algebraic variety and $W: Y \To \C$ a regular function.
The $A$-model on $X$ (i.e., the monotone Fukaya category) should be mirror to the $B$-model on $(Y,W)$, which is Orlov's triangulated category of singularities of the fibres of $W$, $D^bSing(W^{-1}(w))$ (see \cite{Orlov2004}).
The triangulated category of singularities is trivial if the fibre $W^{-1}(w)$ is non-singular.
If $Y = Spec(S)$ is affine, with $W \in S$, then we can also introduce the $\Z/2\Z$-graded DG category of matrix factorizations $MF(S,W-w)$.
There is an equivalence of triangulated categories
\begin{equation}
\label{eqn:orlov}
 D^bSing(W^{-1}(w)) \cong H^*(MF(S,W-w))\end{equation}
for all $w$ (see \cite{Orlov2004}).

The eigenvalues of $c_1 \star$, which index non-trivial components of the Fukaya category, should correspond to singular values of the superpotential $W$, which index fibres with non-trivial triangulated category of singularities (compare \cite[Theorem 6.1]{Auroux2007}).
Homological mirror symmetry then predicts quasi-equivalences of $\C$-linear, $\Z/2\Z$-graded, split-closed, triangulated $A_{\infty}$ categories
\begin{equation} D^{\pi}(\cF(X)_w) \cong D^{\pi}Sing(W^{-1}(w))\end{equation}
for all $w \in \C$ (the superscript `$\pi$' denotes the idempotent or Karoubi closure).
A proof of this version of homological mirror symmetry for Fano toric varieties has been announced by Abouzaid, Fukaya, Oh, Ohta and Ono (in fact, their results go well beyond the Fano case).

Let $X^n_a$ be a smooth degree-$a$ hypersurface in $\CP{n-1}$ (we apologize for the awkward notation, which makes $X^n_a$ an $(n-2)$-dimensional manifold, but it really makes the formulae less complicated).
It is monotone if $a \le n-1$.
Let $P \in QH^*(X^n_a)$ be the class Poincar\'{e} dual to a hyperplane.
The relation satisfied by $P$ in $QH^*(X^n_a)$ is computed by Givental \cite{Givental1996} (extending the results of Beauville \cite{Beauville1995} and Jinzenji \cite{Jinzenji1997} in lower degrees):

\begin{proposition} \cite[Corollaries 9.3 and 10.9]{Givental1996}
\label{proposition:givent}
Define
\begin{equation}\label{eqn:qna}
 q^n_a(x) := x^{n-1} - a^ax^{a-1},\end{equation}
and
\begin{equation}
\label{eqn:bwna} \bm{w}^n_a := \left\{ \begin{array}{rl}
					0 & \mbox{if $a \le n-2$}\\
					-a! & \mbox{if $a=n-1$.}
				\end{array} \right. \end{equation}
Then the subalgebra of $QH^*(X^n_a)$ generated by $P$ is isomorphic to
\begin{equation} \C[P]/q^n_a(P + \bm{w}^n_a).\end{equation}
\end{proposition}

We will re-prove Proposition \ref{proposition:givent}, via our computations in the Fukaya category, in Proposition \ref{proposition:qhfromhh} (see \S \ref{subsec:qhfromhh}).

Now, because $c_1(TX^n_a) = (n-a) c_1(\mathcal{O}(1)) = (n-a)P$, we have

\begin{corollary}
\label{corollary:c1eval}
The eigenvalues of $c_1 \star: QH^*(X^n_a) \To QH^*(X^n_a)$ are equal to the roots of $q^n_a$, shifted by $\bm{w}^n_a$, and multiplied by $(n-a)$.
Explicitly, the eigenvalues are:
\begin{itemize}
\item $0$ (which we call the \emph{big} eigenvalue) and the $n-a$ numbers
\begin{equation} (n-a) \xi\end{equation}
where $\xi^{n-a} = a^a$ (which we call the \emph{small} eigenvalues), if $a \le n-2$;
\item $-a!$  (which we call the \emph{big} eigenvalue) and $a^a - a!$  (which we call the \emph{small} eigenvalue), if $a = n-1$.
\end{itemize}
\end{corollary}

This tells us where the non-trivial Fukaya categories are.
We correspondingly call the component of the Fukaya category corresponding to the big eigenvalue, the \emph{big} component of the Fukaya category, and the other components the \emph{small} components.
We prove (using our Fukaya category computations) that if $w$ is a small eigenvalue, then $QH^*(X)_w$ has rank $1$.
$QH^*(X)_w$ is expected to be isomorphic to $HH^*(\cF(X)_w)$, so it gives a measure of how `complicated' the corresponding Fukaya category is: so we expect the small components of the Fukaya category to be rather simple, and the big components to be more complicated.
For example, Corollary \ref{corollary:isemisimpgen} shows that any non-trivial object in a small component of the Fukaya category necessarily split-generates that component, and as a corollary any two non-trivial objects intersect, and so on.

Now we consider the mirror.

\begin{definition}
\label{definition:mir}
We define the polynomials
\begin{equation} \label{eqn:zna}
Z^n_a := -u_1 \ldots u_n + \sum_{j=1}^n u_j^a \in \C[u_1,\ldots, u_n]\end{equation}
and
\begin{equation} W^n_a := Z^n_a + \bm{w}^n_a\end{equation}
(where the constant $\bm{w}^n_a $ is as in Proposition \ref{proposition:givent}).
We define
\begin{equation} \label{eqn:gammana}
\Gamma^n_a := (\Z/a\Z)^n/(\Z/a\Z),\end{equation}
the quotient of $(\Z/a\Z)^n$ by the diagonal subgroup.
Its character group is
\begin{equation} (\Gamma^n_a)^* \cong \{ (\zeta_1, \ldots,\zeta_n) \in (\Z/a\Z)^n: \zeta_1+ \ldots+\zeta_n = 0 \}.\end{equation}
There is an obvious action of $(\Gamma^n_a)^*$ on $\C^n$ by multiplying coordinates by $a$th roots of unity, and $Z^n_a$ and $W^n_a$ are $(\Gamma^n_a)^*$-invariant.
\end{definition}

For the rest of this section, we will fix $n$ and $a$, and write `$X$' instead of `$X^n_a$', and so on, to avoid notational clutter.

\begin{definition}
The mirror to $X$ is the Landau-Ginzburg model
\begin{equation} (Y,W) := (\C^n/\Gamma^*, W).\end{equation}
The $B$-model category associated to $w \in \C$ is the split-closure of the $\Gamma^*$-equivariant triangulated category of singularities of the corresponding fibre of $W$:
\begin{equation} D^{\pi} \mathrm{Sing}^{\Gamma^*}(W^{-1}(w)).\end{equation}
It is a $\Z/2\Z$-graded, $\C$-linear, triangulated category.
It is equivalent to the cohomology of the corresponding $\Gamma^*$-equivariant category of matrix factorizations
\begin{equation} H^*(D^\pi MF^{\Gamma^*}(W - w)).\end{equation}
\end{definition}

\begin{remark}
\label{remark:dgen}
There is a natural DG enhancement of Orlov's triangulated category of singularities, but the author does not know if the equivalence with the category of matrix factorizations \eqref{eqn:orlov} lifts to a quasi-equivalence of the underlying DG categories (\cite{Orlov2004} only proves equivalence of triangulated categories, on the level of cohomology).
Our proof of homological mirror symmetry (Theorems \ref{theorem:big} and \ref{theorem:small}) is formulated as a quasi-equivalence of $A_\infty$ categories; we always take the DG enhancement of $D^\pi Sing(W^{-1}(w))$ given by the DG category of matrix factorizations $MF^{\Gamma^*}(W - w)$, rather than the natural one (of course, one hopes they are the same).
\end{remark}

\begin{lemma}
\label{lemma:wnacrit}
There are two types of critical points of the mirror superpotential $W$.
First, we have the origin, which we call the \emph{big} critical point.
It is a fixed point of the action of $\Gamma^*$.
The corresponding critical value is the big eigenvalue of $c_1 \star$ on $X$.

The remaining critical points are called \emph{small} critical points.
The critical values associated to the small critical points are equal to the small eigenvalues of $c_1 \star$ on $X$ (in particular, there are $n-a$ of them).
For each of the small critical values, the action of $\Gamma^*$ on the critical points with that critical value is free and transitive.
Furthermore, the Hessian of $W$ at the small critical points is non-degenerate.
\end{lemma}
\begin{proof}
At a critical point of $Z^n_a$, we have
\begin{equation} \frac{u_1 \ldots  u_n}{u_j} = a u_j ^{a-1}\end{equation}
for each $j$.
Taking the product of these relations, we obtain the relation
\begin{equation} q^n_a \left(U \right)=0,\end{equation}
where
\begin{equation} U := \frac{u_1 \ldots u_n}{a} = u_j^a \mbox{ for all $j$.}\end{equation}
Therefore, either $U = 0$, in which case $u_1 = \ldots = u_n = 0$ (the big critical point), or $U = \xi$ is an $(n-a)$th root of $a^a$, in which case we get a small critical point.
In the latter case, one easily verifies that $(\Gamma^n_a)^*$ acts freely and transitively on the critical points corresponding to a fixed $\xi$.
For each such critical point, the corresponding critical value is
\begin{equation} - u_1 \ldots u_n + \sum_j u_j^a = (n-a) \xi.\end{equation}

Finally, it is straightforward to check that the Hessian of $Z^n_a$ at a small critical point has the form
\begin{equation} \omega a^{\frac{n-2}{n-a}} \left[ \begin{array}{cccc}
					a-1 & -1 & \ldots & -1 \\
					-1 & a-1 & \ldots & -1 \\
					\vdots & \vdots & \ddots & \vdots \\
					-1 & -1 & \ldots & a-1
					\end{array} \right],\end{equation}
where $\omega$ is some root of unity.
In particular, as $a \le n-1$, it is invertible.
 \end{proof}

Therefore, the critical values of $W$ match up with the eigenvalues of $c_1 \star$ on $X$, so the non-trivial categories lie in the same places on the two sides of homological mirror symmetry.
Our first main result is that the categories over the big eigenvalue/critical value are quasi-equivalent:

\begin{theorem}
\label{theorem:big}
If $2 \le a \le n-1$, then there is a quasi-equivalence of $\C$-linear, $\Z/2\Z$-graded, triangulated, split-closed $A_{\infty}$ categories
\begin{equation} D^{\pi}\cF(X)_w \cong D^{\pi}Sing^{\Gamma^*}(W^{-1}(w)),\end{equation}
where $w$ is the big eigenvalue of $c_1 \star$ ($=$ the big critical value of $W$).
\end{theorem}

\begin{remark}
The case $a=2$ of Theorem \ref{theorem:big} was proven by Smith in \cite{Smith2012}.
The quadric hypersurface $X^n_2$ can be regarded as a compactification of $T^*S^{n-2}$; Smith proved that the zero-section $S^{n-2}$ split-generates $\cF(X^n_2)_0$, and that its endomorphism algebra is a Clifford algebra.
\end{remark}

To prove Theorem \ref{theorem:big}, we consider the action of
\begin{equation} \Gamma = (\Z/a\Z)^n / (\Z/a\Z)\end{equation}
on the Fermat hypersurface
\begin{equation} X = \left\{ \sum_j z_j^a = 0\right\} \end{equation}
by multiplication of the coordinates by $a$th roots of unity.
In \S \ref{sec:fermats}, we consider the immersed Lagrangian sphere $L$ in $X/\Gamma$, constructed in \cite{Sheridan2011}, and let $\{L_\gamma\}_{\gamma \in \Gamma}$ be its  lifts to $X$.
Although $L$ is immersed, its lifts $L_\gamma$ are embedded (Lemma \ref{lem:Lliftemb}), and all satisfy $w(L_\gamma) = \bm{w}^n_a$ (the big eigenvalue).
We consider their direct sum $\bm{L} = \oplus_\gamma  L_\gamma$, and prove that $\bm{L}$ split-generates the big component of the Fukaya category, and that its endomorphism algebra is $A_\infty$ quasi-isomorphic to the endomorphism algebra of the sum of equivariant twists of the skyscraper sheaf at the origin in the category of singularities.
As the latter split-generates the category of singularities, this completes the proof.

Our computations in the Fukaya category are drawn straight from the computations in the relative Fukaya category of \cite{Sheridan2015}, transferred to the monotone Fukaya category via the results of \S \ref{sec:relfuk}.
Also as in \cite{Sheridan2015}, our proof of split-generation relies on the dual version of the split-generation criterion, Corollary \ref{corollary:gendual}.
Namely, we consider the map
\begin{equation} \CO: QH^*(X)_w \To HH^*(CF^*(\bm{L},\bm{L})).\end{equation}
We know $\CO(P)$ from \eqref{eqn:codjdif}, and that $\CO$ is an algebra homomorphism; by explicitly computing the Hochschild cohomology, this allows us to prove that $\CO$ is injective on the subalgebra generated by $P$.
There is more to $QH^*(X)_w$ than the subalgebra generated by $P$; but we observe that there is a natural action of $\Gamma$ on $X$ and on the objects constituting $\bm{L}$, and that $P$ \emph{does} generate the $\Gamma$-invariant part $QH^*(X)^\Gamma_w$, so we can apply a $\Gamma$-equivariant version of  Corollary \ref{corollary:gendual} to prove that $\bm{L}$ split-generates.

Note that the Fukaya category appearing in Theorem \ref{theorem:big} does not involve weak bounding cochains: the Lagrangians $L_\gamma$ satisfy the hypothesis of the split-generation criterion, with no weak bounding cochains required.
In particular, none of the machinery developed in \S \ref{sec:wbcalg} and \S \ref{sec:wbcgeom} is necessary to study the big component of the Fukaya category.
Furthermore, because these Lagrangians satisfy the hypothesis of the split-generation criterion, they also split-generate $\cF^{wbc,u}(X)_w$.
It follows that Theorem \ref{theorem:big} also holds if we replaced $D^\pi \cF(X)_w$  by $D^\pi \cF^{wbc,u}(X)_w $: adding weak bounding cochains would not add any new information.

However, in order to address the remaining small components of the Fukaya category, we do need weak bounding cochains.
We prove:

\begin{theorem}
\label{theorem:small}
If $2 \le a \le n-1$, then there is a quasi-equivalence of $\C$-linear, $\Z/2\Z$-graded, triangulated, split-closed $A_{\infty}$ categories
\begin{equation} D^{\pi}\cF^{wbc,u}(X)_w \cong D^{\pi}Sing^{\Gamma^*}(W^{-1}(w)),\end{equation}
where $w$ is any of the small eigenvalues of $c_1 \star$.

In fact, both categories are quasi-equivalent to $D^b(\C)$ (if $n$ is even), and $D^\pi(\Cl_1)$ (if $n$ is odd), where $\Cl_1$ denotes the $\Z/2\Z$-graded Clifford algebra $\C[\theta]/\theta^2=1$, with $\theta$ in odd degree.
\end{theorem}

\begin{remark}
The ambiguity of DG enhancements mentioned in Remark \ref{remark:dgen} is irrelevant for the small categories, as they are intrinsically formal.
\end{remark}

Let us start by explaining the category of singularities.
For each of the small critical values of $W$, the corresponding critical points are non-degenerate, and $\Gamma^*$ acts freely and transitively on them: so when we quotient by $\Gamma^*$, they all get identified to a single non-degenerate critical point.
We recall that the triangulated category of singularities associated to a non-degenerate singular point is particularly simple: it is equivalent to a category of finitely-generated modules over a Clifford algebra (see \cite[\S 4.4]{Dyckerhoff2009}, \cite{Buchweitz1986} and \cite[Chapter 14]{Yoshino1990}).

We recall that Clifford algebras are intrinsically formal.
Hence, to prove Theorem \ref{theorem:small}, it suffices to find a single object of the Fukaya category whose endomorphism algebra is a Clifford algebra: such an object will automatically split-generate by (the weak bounding cochain analogue of) Corollary \ref{corollary:isemisimpgen}.

To find this object, we consider the \emph{same} direct sum of Lagrangian spheres $\bm{L}$ as in the proof of Theorem \ref{theorem:big}.
We prove that there is an embedding of $\C^n$ into the Maurer--Cartan moduli space $\hcM_{weak}(L)$ of $L$ in $\cF(X/\Gamma)$, and that the resulting disc potential
\begin{equation} \mathfrak{P}: \C^n \To \C \end{equation}
is given precisely by the superpotential $W$ of the mirror.

We then observe that there is an embedding of $\hcM_{weak}(L)$ (which lives in $\cF(X/\Gamma)$) into $\hcM_{weak}(\bm{L})$ (which lives in $\cF(X)$), which respects the disc potentials: essentially, we take $\alpha \in CF^*(L,L)$ to the sum of all its lifts to $CF^*(\bm{L},\bm{L})$.
Therefore, we have an embedding of $\C^n$ into $\hcM_{weak}(\bm{L})$, such that the resulting disc potential is given by $W$.
This is a partial analogue of the result proven in the case of toric varieties by Cho and Oh \cite{Cho2006} (see also \cite{Fukaya2010d,Fukaya2011,Fukaya2010c}): the mirror is equal to the Maurer--Cartan moduli space, with the superpotential given by the disc potential.
However, there is a proviso: the mirror is actually equal to the \emph{quotient} of $\C^n$ by the action of $\Gamma^*$ (see Definition \ref{definition:mir}).

We show that, for any $\chi \in \Gamma^*$ and $v \in \C^n$, the objects $(\bm{L},v)$ and $(\bm{L},\chi \cdot v)$ are quasi-isomorphic in the Fukaya category of $X$ (we remark that $(L,v)$ and $(L,\chi \cdot v)$ are \emph{not} always quasi-isomorphic objects in the Fukaya category of $X/\Gamma$).
The reader may imagine that this quotient is simply the quotient by gauge equivalence, but that is \emph{not} the case: although we have not given a definition of gauge equivalence in our setup, any sensible definition should have the property that gauge equivalent weak bounding cochains are connected by a continuous family of such, with the same value of the disc potential $\mathfrak{P}$.
But the disc potential $\mathfrak{P}$ that we have computed does not admit any continuous automorphisms, so the gauge equivalence relation on $\hcM_{weak}(\bm{L})$ (whatever the abstract definition) must be trivial.
The moral is that, in order to obtain the true mirror, we need to further quotient our `moduli space of objects' in the Fukaya category by an equivalence relation which identifies quasi-isomorphic objects, not only gauge equivalent ones.
However, this has a further proviso: the majority of the points in our moduli space represent objects quasi-isomorphic to the zero object, so the equivalence relation we want is not simply quasi-isomorphism of objects (probably it is related to quasi-isomorphism of objects in some different version of the Fukaya category involving the Novikov ring).

\begin{remark}
The preprint \cite{Cho2013} also considered the situation that $\C^n$ embeds into  $\hcM_{weak}(L)$, so that the disc potential $\mathfrak{P}|_{\C^n}$ is given by a polynomial $W$.
The authors explained how this already implies the existence of an $A_\infty$ functor from $\cF(X/\Gamma)_w$ to $MF(W-w)$, using a version of the Yoneda embedding.
They also considered the case of a finite group action: the results of \cite[\S 5]{Cho2013} show furtermore that there exists an $A_\infty$ functor from $\cF(X)_w$ to $MF^{\Gamma^*}(W-w)$.
In the present setting, we expect that these functors are quasi-equivalences, and give natural realizations of the functors appearing in Theorems \ref{theorem:big} and \ref{theorem:small}: this ought to follow from our result that the objects $(\bm{L},v)$ split-generate the corresponding components of the Fukaya category.
\end{remark}

Aside from these philosophical remarks, these computations provide us with the non-trivial objects of the associated component of the Fukaya category which are required to complete the proof of Theorem \ref{theorem:small}.
Namely, the weak bounding cochain corresponding to a small critical point of $W$ is a non-trivial object of the Fukaya category, and its endomorphism algebra is the Clifford algebra associated to the Hessian of $W$, which is non-degenerate by Lemma \ref{lemma:wnacrit}.

\begin{remark}
\label{remark:othereig}
It would be interesting to find monotone Lagrangians living in the small components of the Fukaya category, and having non-trivial Floer cohomology, so that we could remove the need for weak bounding cochains in Theorem \ref{theorem:small}.
We construct such a monotone Lagrangian in Appendix \ref{sec:cubsurf}, in the Fano index $1$ case $a = n-1$, and give a heuristic argument that it has the right $w(L)$ and non-trivial Floer cohomology in the case of the cubic surface ($a= 3$, $n=4$), using tropical curve counting.
The exact Lagrangian tori in the affine cubic surface constructed in \cite{Keating2014}, with appropriate local systems, also appear to be natural candidates.
\end{remark}

Notwithstanding Remark \ref{remark:othereig}, the fact that we can split-generate the small components of the Fukaya category by putting weak bounding cochains on our collection of Lagrangian spheres $\bm{L}$ has interesting consequences:

\begin{corollary}
\label{corollary:lagemb}
If $2 \le a \le n-1$ and $K \subset X^n_a$ is a monotone Lagrangian submanifold equipped with a grading, spin structure and $\C^*$-local system such that $HF^*(K,K) \neq 0$, then $K$ intersects at least one of the Lagrangian spheres $L_\gamma$ described above.
\end{corollary}
\begin{proof}
As $HF^*(K,K) \neq 0$, $K$ must lie in one of the components of the Fukaya category, and hence be split-generated by $\bm{L}$, or $\bm{L}$ with a weak bounding cochain.
In particular, $CF^*(K,\bm{L}) \neq 0$, so $K$ and $\bm{L}$ must intersect.
 \end{proof}

Let us briefly remark on possible extensions of Theorems \ref{theorem:big} and \ref{theorem:small} that could be proved using the results in this paper.

\begin{remark}
Firstly, we could use different coefficient rings.
We need to work over a field for the homological algebra associated with the split-generation arguments to go through.
We could, for example, work over the Novikov field $\Lambda$, in which case the mirror would be $\Lambda^n$ with superpotential
\begin{equation} W^n_a = -u_1 \ldots u_n + r \sum_{j=1}^n u_j^a + \bm{w}^n_a r^n.\end{equation}
This would allow us to include non-monotone weak bounding cochains in our definition of the Fukaya category, as convergence of various infinite sums would hold by completeness of the energy filtration.
It would require some other changes of perspective, see in particular Remark \ref{remark:energyfilt}.

Alternatively, we could keep the monotone weak bounding cochains, and work over the ring of Novikov \emph{polynomials}, without completing with respect to the energy filtration; then, in order to obtain split-generation results, we would need to base change to a field.
\end{remark}

\begin{remark}
The Fukaya category $\cF(X^n_a)$ can be equipped with a $\Z/2(n-a)$-grading (if all Lagrangians are assumed graded), because $n-a$ is the minimal Chern number of spheres in $X^n_a$.
In Theorem \ref{theorem:big}, we simply work with the $\Z/2\Z$-graded version of the Fukaya category, as the triangulated category of singularities is naturally $\Z/2\Z$-graded.
One could also define a category of matrix factorizations (again with certain restrictions on the objects, c.f. \cite[\S 7.3]{Sheridan2015}) with a $\Z/2(n-a)$-grading, and prove a $\Z/2(n-a)$-graded version of Theorem \ref{theorem:big}.
However, it appears that our methods do not extend to prove a $\Z/2(n-a)$-graded analogue of Theorem \ref{theorem:small}: the weak bounding cochains we use can not be chosen to have degree $1$ modulo $2(n-a)$ (compare Remark \ref{remark:wbcgrad}), unless of course $a = n-1$.
In light of this, we might add a followup to Remark \ref{remark:othereig}: if indeed there do exist monotone Lagrangians sitting in the small components of the Fukaya category, it would be interesting to know if they are graded.
\end{remark}

\subsection{Quantum cohomology of Fano hypersurfaces from the Fukaya category}
\label{subsec:qhfromhh}

An interesting feature of our calculations in the Fukaya category is that they allow us to compute non-trivial information about Gromov--Witten invariants.
Namely, we prove (in Proposition \ref{proposition:qhfromhh}) that Proposition \ref{proposition:givent} is implied by our computation of the endomorphism algebra of $\bm{L}$ in the relative Fukaya category $\cF(X,D)$, with the exception that in the Fano index $1$ case (i.e., when $a=n-1$), we are only able to prove that the subalgebra generated by the hyperplane class $P$ is $\C[P]/q^n_a(P+w)$, where $w = w(L_\gamma)$ is equal to the count of $J$-holomorphic discs with boundary on one of the monotone Lagrangian spheres $L_\gamma \subset X^n_a$ constructed in the proof of Theorem \ref{theorem:big}.
Note that this number is the same for all $\gamma \in \Gamma$, by symmetry.

In fact, in order to determine the value $w = \bm{w}^n_a$, we are forced to use Givental's proof of Proposition \ref{proposition:givent} (see Corollary \ref{corollary:w}).
We remark that, in Givental's work on closed string mirror symmetry for Fano hypersurfaces in toric varieties \cite[\S 10]{Givental1996}, the extra factor $\bm{w}^n_a$ that appears in Proposition \ref{proposition:givent} in the case $a=n-1$ corresponds to an additional term $e^{-a!Q}$ which has to multiply the correlators in this case, which did not appear for $a \le n-2$.
So in the course of our alternative proof of Proposition \ref{proposition:givent}, we equate this additional term with the open Gromov--Witten invariant $w$ mentioned above (even though we are not able to re-compute its value).

Our re-proof of Proposition \ref{proposition:givent} uses a version of the closed--open string map from quantum cohomology to Hochschild cohomology (in fact, the $\bm{G}$-graded Hochschild cohomology of the relative Fukaya category).
These relations are well-known and have been computed multiple times by other methods, but nevertheless it is interesting that this information can be extracted from the Fukaya category.

\begin{remark}
 We confess that there are two undetermined signs in our computations; these are irrelevant for the proofs of Theorems \ref{theorem:big} and \ref{theorem:small}, but become necessary when computing relations in quantum cohomology, so our re-proof of Proposition \ref{proposition:givent} is not complete.
 Nevertheless, we feel it is interesting as a proof of the concept that non-trivial information about Gromov--Witten invariants can be extracted from the Fukaya category.
\end{remark}

It is well-known, but interesting to remark, that information about genus-zero Gromov--Witten invariants can be extracted from these relations in quantum cohomology.
For example, one can show that the number of genus-zero, degree-one curves on a cubic hypersurface $X^n_3 \subset \CP{n-1}$ (with $n \ge 5$), which send $n+1$ fixed marked points to $n$ hyperplanes and one $2$-dimensional linear subspace, is equal to $81$.
When $n=4$, $X^4_3$ is the cubic surface, and one can show that the number of lines on the cubic surface is equal to $27$ if and only if the open Gromov--Witten invariant $w$ mentioned above is equal to $-6$.
We collect some other interesting facts about the case of the cubic surface in Appendix \ref{sec:cubsurf}.

\begin{remark}
Our re-proof of Proposition \ref{proposition:givent} relies on computations in the relative Fukaya category, $\cF(X,D)$; it does not work for the monotone Fukaya category $\cF(X)$.
The reason is that the closed--open map for the relative Fukaya category (see \S \ref{subsec:relco})
\begin{equation} \CO_{X,D}: QH^*(X,D) \To HH^*(CF^*(\bm{L},\bm{L}))\end{equation}
is injective, hence relations in $QH^*(X,D)$ can be deduced from those in $HH^*$, but
\begin{equation} \CO: QH^*(X) \To HH^*(CF^*(\bm{L},\bm{L}))\end{equation}
is only injective when restricted to $QH^*(X)_{\bm{w}}$; it vanishes when restricted to $QH^*(X)_w$ if $w$ is a small eigenvalue.

This can be understood via mirror symmetry as follows: the big singular fibre $W^{-1}(\bm{w})$ is disjoint from the small singular fibres $W^{-1}(w)$.
Since $\bm{L}$ is mirror to a sheaf supported at the big critical point, one can not expect it to `know' anything about sheaves supported on the small singular fibres.
On the other hand, when one incorporates a Novikov parameter $r$ as in the relative Fukaya category, the situation changes: the mirror superpotential now depends on the parameter $r$, and as $r \To 0$, the small critical points converge to the big critical point.
So as long as one does not invert $r$ in one's coefficient ring, we \emph{can} expect the mirror to $\bm{L}$ to `know' about objects supported in the small singular fibres.
This observation was the inspiration for the idea that $\bm{L}$ could be `pushed' into the small eigenvalues by putting a weak bounding cochain on it, which is how Theorem \ref{theorem:small} is proved.
\end{remark}

\subsection{Organization of the paper}

In \S \ref{sec:monfuk}, we prove the basic results about the monotone Fukaya category, as already discussed in \S \ref{subsec:imonfuk}.
In \S \ref{sec:relfuk}, we establish the relationship between the relative Fukaya category and the monotone Fukaya category, as already discussed in \S \ref{subsec:irelfuk}.
In \S \ref{sec:wbcalg} and \S \ref{sec:wbcgeom}, we give the framework for including monotone weak bounding cochains in the monotone Fukaya category, as already discussed in \S \ref{subsec:iwbc}.
In \S \ref{sec:deftheory}, we collect the algebraic computations that are necessary for the proof of homological mirror symmetry: classifications of $A_{\infty}$ structures and computations of various versions of Hochschild cohomology.
In \S \ref{sec:fermats} (Fukaya category) and \S \ref{sec:bmodel} (matrix factorizations), we give the proof of homological mirror symmetry for $X^n_a$, as summarized in \S \ref{subsec:ihms}.

We also include two appendices. 
Appendix \ref{sec:ainf} summarizes the basic algebraic results about $A_\infty$ categories and bimodules (and in particular, weak proper Calabi--Yau structures) that we use.
Finally, Appendix \ref{sec:cubsurf} collects a few results which are specific to the interesting special case of the cubic surface in $\CP{3}$.

\emph{Acknowledgements:} I would like to thank the following people: Paul Seidel, who was my thesis advisor while most of this work was carried out and whose influence is all-pervasive; Mohammed Abouzaid, for many useful discussions relating to this work, and for showing me a preliminary version of \cite{Abouzaid2012}; Grisha Mikhalkin, for helping me find the construction of the immersed Lagrangian sphere in the pair of pants on which this work is based; Kenji Fukaya, for helpful comments and for showing me the construction of the monotone Lagrangian torus in the cubic surface which appears in Appendix \ref{sec:cubsurf}; Octav Cornea, for drawing my attention to the account of the construction of the monotone Fukaya category in \cite{Biran2013}, and pointing out a subtlety in the construction that I had missed; Cedric Membrez, for helpful conversations concerning Corollary \ref{corollary:ocLe} and Lemma \ref{lemma:ocLf}; Ivan Smith, for suggesting Lemma \ref{lemma:cowbcc2} to me, and for other useful comments on a draft version of the paper; and Alex Ritter and Ivan Smith, for explanations of their work \cite{Ritter2012}, in particular their Theorems 9.4 and 9.6, on which our Proposition \ref{proposition:coeigsplit} and Corollary \ref{corollary:oceigsplit} are based (the proof of Corollary \ref{corollary:oceigsplit} was pointed out by Ritter).

\section{The monotone Fukaya category}
\label{sec:monfuk}

This section gives a survey of the construction of the (small) quantum cohomology ring (following \cite{Ruan1995,mcduffsalamon}), the monotone Fukaya category (following \cite{Oh1993,Oh1995} and \cite{Seidel2008}), and the split-generation result of Abouzaid, Fukaya, Oh, Ohta and Ono \cite{Abouzaid2012} in the monotone setting.
By restricting ourselves to the monotone setting, we are able to give a reasonably complete account using only `classical' pseudoholomorphic curve techniques.
Most of the results in this section are not essentially new: they have appeared in the literature in slightly different settings (e.g., under the assumption of exactness rather than monotonicity), using an identical analytic setup for pseudoholomorphic curve theory.
For that reason, we do not repeat the foundational material, but rather focus on the new issues that arise when extending results to the monotone case (for example, we do not discuss signs).

We should also note that many of the results in this section are known to hold in vastly more general settings (see \cite{fooo} and \cite{Abouzaid2012}), where one must use Kuranishi structures to deal with the analytic issues.
However, there are some structures on the Fukaya category of a monotone symplectic manifold (for example, the decomposition according to eigenvalues of $c_1$) which do not hold in more general settings, so we feel that it is useful to give a unified account in the monotone case, with an emphasis on the aspects that are peculiar to it.
Furthermore, despite the limitations of classical pseudoholomorphic curve theory, it is useful to give the simplest construction possible, because it makes it easier to build new results on top of it.
For example, we will prove several new structural results about the monotone Fukaya category in \S \ref{sec:relfuk} and \S \ref{sec:wbcgeom}.

\subsection{Monotonicity}
\label{subsec:monot}

Let us start by making our notion of monotonicity precise.

\begin{definition}
\label{definition:strmon}
We say that $X$ is a (positively) monotone symplectic manifold if
\begin{equation} [\omega] =  2\tau c_1 \end{equation}
for some $\tau > 0$ (the factor of $2$ allows us to avoid a factor of $1/2$ in Definition \ref{definition:lmon}).
We will always assume that $X$ is compact.
\end{definition}

\begin{definition}
\label{definition:lmon}
We say that a closed Lagrangian submanifold $L \subset X$ is \emph{monotone} if
\begin{itemize}
\item The image of $\pi_1(L)$ in $\pi_1(X)$ is trivial;
\item The homomorphisms given by symplectic area and Maslov class
\begin{equation} [\omega]: H_2(X,L) \To \R, \,\,\, \mu: H_2(X,L) \To \Z\end{equation}
are proportional: $[\omega] = \tau \mu$.
\end{itemize}
\end{definition}

We remark that Definitions \ref{definition:strmon} and \ref{definition:lmon} are stronger than the usual ones; we hope this does not cause confusion.

Let $\cG X$ be the Grassmannian of Lagrangian subspaces of $TX$, and let $\widetilde{\cG}X$ be its universal abelian cover, with covering group $H_1(\cG X)$.
There is an associated grading datum $\bm{G}(X)$, given by the map
\begin{equation} \Z \cong H_1(\cG_xX) \To H_1(\cG X)\end{equation}
(compare \cite[\S 3]{Sheridan2015}).
For the purposes of this section, we will abbreviate $\bm{G}:= \bm{G}(X)$.

\subsection{Quantum cohomology, $QH^*(X)$}

Let $X$ be a compact real $2n$-dimensional monotone symplectic manifold. 
One can define the (small) quantum cohomology ring $QH^*(X)$ as in \cite[Chapter 11]{mcduffsalamon} (to which we refer for all technical details).
We very briefly summarize the definition.

We set
\begin{equation} QH^*(X) := H^*(X;\C)\end{equation}
as a $\bm{G}$-graded complex vector space (classes are equipped with the image of their cohomological degree in $\bm{G}$).
We equip it with the intersection pairing
\begin{equation} \langle -,- \rangle : QH^*(X) \otimes QH^{2n-*}(X) \To \C.\end{equation}

To define the quantum cup product, one first defines the three-point Gromov--Witten invariant
\begin{equation} GW^X_3: QH^*(X) ^{\otimes 3} \To \C,\end{equation}
as follows.
Consider the moduli space of $J_z$-holomorphic spheres in homology class $\beta \in H_2(X;\Z)$, with three marked points, where $J_z$ is a domain-dependent $\omega$-tame almost-complex structure.
For generic $\omega$-compatible $J_z$, it is a smooth manifold of dimension $d(\beta) = 2n + 2c_1(\beta)$.

Let $\alpha_i \in H^{d_i}(X)$ be cohomology classes, for $i=1,2,3$, and suppose that they are Poincar\'{e} dual to pseudocycles $f_i: A_i \To X$.
For generic choice of $J_z$ and the pseudocycles $f_i$, the moduli space of $J_z$-holomorphic spheres with marked points constrained to lie on the pseudocycles $f_1,f_2,f_3$ respectively, is an oriented smooth manifold of dimension
\begin{equation} d:=d(\beta) - d_1 - d_2 - d_3.\end{equation}
If $d=0$, then for generic $J_z$ and $f_i$, this manifold is also compact, and we define $GW^X_{3,\beta}(\alpha_1,\alpha_2,\alpha_3)$ to be the signed count of its points (if $d \neq 0$ we define it to be $0$).
It is independent of $J_z$ and the choice of pseudocycles representing the $\alpha_i$, by a cobordism argument.
We then define
\begin{equation} GW^X_3 := \sum_\beta GW^X_{3,\beta};\end{equation}
the sum converges by the monotonicity assumption.

We now define the (small) quantum cup product
\begin{align}
\star:QH^*(X) \otimes QH^*(X) &\To QH^*(X).\\
 \alpha \otimes \beta & \mapsto  \alpha \star \beta,
\end{align}
by the formula
\begin{equation} \langle \alpha \star \beta , \gamma \rangle := GW^X_3(\alpha,\beta,\gamma).\end{equation}

\begin{remark}
\label{remark:qcupeval}
Another way to define the quantum cup product is to consider the moduli space of holomorphic spheres with three marked points, where two of the marked points are constrained to lie on pseudocycles $f_1,f_2$, representing $PD(\alpha), PD(\beta)$ respectively.
There is an evaluation map at the third marked point; and it can be arranged (in the monotone case) that this evaluation map is a pseudocycle, which represents $PD(\alpha \star \beta)$.
\end{remark}

We refer to \cite[Proposition 11.1.9]{mcduffsalamon} for the proof that the quantum cup product is associative, supercommutative, the unit $e \in H^0(X;\C)$ is also a unit for the quantum cup product, and the quantum cohomology algebra is also a Frobenius algebra:
\begin{equation} \langle \alpha \star \beta, \gamma \rangle = \langle \alpha, \beta \star \gamma \rangle.\end{equation}
Standard index theory of Cauchy--Riemann operators shows that $QH^*(X)$ is a $\bm{G}$-graded $\C$-algebra (compare \cite[\S 4.4]{Sheridan2015}).

\subsection{The monotone Fukaya category}
\label{subsec:monfuk}

It has been known since the work of Oh \cite{Oh1993,Oh1995} that the definition of Lagrangian Floer cohomology for monotone Lagrangian submanifolds with minimal Maslov number $\ge 2$ is significantly simpler than the fully general version.
In this section we outline the construction of the monotone Fukaya category, following \cite{Seidel2008} (see \cite{Biran2013} for a treatment very similar to that given in this section, as well as \cite{Ritter2012}).

To each $w \in \C$, we will associate a $\C$-linear, $\bm{G}$-graded (non-curved) $A_{\infty}$ category $\cF(X)_w$.
Objects of the categories $\cF(X)_w$ are oriented monotone Lagrangian submanifolds $L \subset X$ together with a spin structure, a lift of $L$ to $\widetilde{\cG}(X)$ (called a \emph{grading}, see \cite{Seidel1999}) and a flat $\C^*$-local system.
Because our Lagrangians are orientable, they have minimal Maslov number $\ge 2$.
For simplicity, we will choose a finite set $\mathcal{L}$ of such Lagrangians, and define the subcategory of the monotone Fukaya category with those objects.

\begin{remark}
For any abelian cover $\widetilde{\cG}'(X)$ of $\cG X$, we can define a $\bm{G}'(X)$-graded monotone Fukaya category, where $\bm{G}'X$ is given by the composition
\begin{equation} \Z \cong H_1(\cG_xX) \To H_1(\cG X) \To Y,\end{equation}
where $Y$ is the covering group.
Objects of this category are Lagrangian submanifolds of $X$ equipped with a lift to $\widetilde{\cG}'X$.
For the purposes of this paper, we will always consider the universal abelian cover.
\end{remark}

We define $\mathscr{H} := C^{\infty}(X,\R)$, the space of Hamiltonian functions on $X$.
For each pair of objects $L_0, L_1 \in \mathcal{L}$, we choose a one-parameter family of Hamiltonians:
\begin{equation} H \in C^{\infty}([0;1],\mathscr{H})\end{equation}
such that the time-$1$ flow of the Hamiltonian vector field associated to $H$ makes $L_0$ transverse to $L_1$ (this is one half of a Floer datum, in the terminology of \cite[\S 8e]{Seidel2008}).
We then define the morphism space $CF^*(L_0,L_1)$ to be the $\C$-vector space generated by length-$1$ Hamiltonian chords from $L_0$ to $L_1$ (or, if the local system is non-trivial, the direct sum of hom-spaces between the fibres of the local systems at the start- and end-points of the chords).

Now let $\mathscr{J}$ denote the space of almost-complex structures on $X$ compatible with $\omega$.
For every object $L$, we choose an almost-complex structure $J_L \in \mathscr{J}$, and consider the moduli space $\cM(L)$ of Maslov index $2$ $J_L$-holomorphic discs with boundary on $L$, with a single marked boundary point.
For generic $J_L$, the moduli space of somewhere-injective $J_L$-holomorphic discs of Maslov index $2$ is regular, by standard transversality results \`{a} la \cite[\S 3]{mcduffsalamon}.
It follows from \cite{Oh2000,Lazzarini2011} that any $J_L$-holomorphic disc $u$ with boundary on $L$ contains a somewhere-injective $J_L$-holomorphic disc $v$ with boundary on $L$ in its image.
In particular, if $u$ has Maslov index $2$, then $v$ has Maslov index $\le 2$ by monotonicity.
Because $L$ has minimal Maslov number $\ge 2$, this means we must have $u=v$, and $u$ is somewhere-injective.
Therefore, $\cM(L)$ is regular for generic $J_L$.

Standard index theory of Cauchy--Riemann operators shows that the moduli space $\cM(L)$ is an $n$-dimensional manifold (recall $n$ is half the dimension of $X$), and it is compact by Gromov compactness (because the homology class of a Maslov index $2$ disc cannot be expressed as a sum of two homology classes in $H_2(X,L)$ with positive energy).
There is an evaluation map at the boundary marked point:
\begin{equation} \mathrm{ev}: \cM(L) \To L.\end{equation}
If the $\C^*$-local system on $L$ is trivial, then we define $w(L) \in \Z$ by
\begin{equation} \mathrm{ev}_* [\cM(L)] := w(L) [L] \in H_n(L;\Z).\end{equation}
If it is non-trivial, we weight $\mathrm{ev}_*$ by the monodromy of the local system around the boundary of the disc, so in general $w(L) \in \C$.
The complex number $w(L)$ is independent of the choice of $J_L$.
We furthermore require that the evaluation map is transverse to the finite set of start-points and end-points of time-$1$ Hamiltonian chords between $L$ and the other Lagrangians $L' \in \mathcal{L}$.

Now for each pair of objects $L_0, L_1 \in \mathcal{L}$, we choose a one-parameter family of almost-complex structures:
\begin{equation} J \in C^{\infty}([0;1],\mathscr{J})\end{equation}
such that $J(i) = J_{L_i}$ for $i = 0,1$.
We consider the moduli space $\cM_0(L_0,L_1)$ of spheres of Chern number $1$ which are $J_t$-holomorphic for some $t \in [0;1]$, and are equipped with a marked point.
Any $J_t$-holomorphic sphere is necessarily a branched cover of a somewhere-injective one by \cite[\S 2.5]{mcduffsalamon}, and hence any $J_t$-holomorphic sphere of Chern number $1$ must be somewhere-injective, so this moduli space is regular for generic choice of $J_t$.
Standard index theory of Cauchy--Riemann operators shows it has dimension $2d-1$.

Note that the parameter $t$ can be thought of as a map
\begin{equation}t: \cM_0(L_0,L_1) \To [0;1].\end{equation}
We also have the evaluation map at the marked point,
\begin{equation} \mathrm{ev}: \cM_0(L_0,L_1) \To X.\end{equation}
For any time-$1$ Hamiltonian chord $\gamma: [0;1] \To X$ which starts on $L_0$ and ends on $L_1$, we consider the map
\begin{equation} (\gamma \circ t, \mathrm{ev}): \cM_0(L_0,L_1) \To X \times X.\end{equation}
We require that, for any such $\gamma$ (there are finitely many, by assumption), the image of this map avoids the diagonal in $X \times X$.
This is true for generic $J_t$, because the domain has dimension $2d-1$ and the diagonal has codimension $2d$.

Now, the pairs $(H,J)$ constitute Floer data for the pairs of objects in our category.
For any generators $x,y$ of $CF^*(L_0,L_1)$, we can define the corresponding moduli space $\cM(x,y)$ of pseudoholomorphic strips (following \cite[\S 8f]{Seidel2008}), with translation-invariant perturbation data given by the Floer data.
For a generic choice of $J_t$, these moduli spaces are regular by \cite{Floer1995}.
In particular, all moduli spaces of negative virtual dimension are empty.
Furthermore, the moduli spaces admit an $\R$-action, by translation of the domain.
It follows that the $\R$-action on a moduli space of virtual dimension $0$ must be trivial.
That means the only moduli spaces of virtual dimension $0$ which are non-empty, are those in $\cM(x,x)$ which are constant along their length.

We now define the differential $\mu^1: CF^*(L_0,L_1) \To CF^*(L_0,L_1)$, by counting elements of the moduli spaces $\cM(x,y)/\R$, where $\cM(x,y)$ has dimension $1$ (we weight the count by the monodromy of the local systems around the boundary).
Using monotonicity of the $L_i$, one can show that these moduli spaces have fixed energy, since they have fixed index, and hence are compact.
We omit the argument here, and refer to the more sophisticated version of the argument which is required in \S \ref{subsec:conv}.
By considering the Gromov compactification of the moduli spaces $\cM(x,z)/\R$, where $\cM(x,z)$ has dimension $2$, the argument of \cite{Oh1995} shows that
\begin{equation}
\label{eqn:mu1sq}
 \mu^1(\mu^1(x)) = (w(L_0) - w(L_1))x.
\end{equation}
To see why, suppose we have a nodal strip $u_1 \# u_2$ in the compactification of this moduli space.
If $u_1$ and $u_2$ are Maslov-index $1$ strips, we have a broken strip contributing to the left-hand side.
If $u_1$ is a holomorphic disc of Maslov index $>2$ or a holomorphic sphere of Chern number $>1$, then $u_2$ is a strip of virtual dimension $<0$, hence can't exist by regularity.
If $u_1$ is a holomorphic sphere of Chern number $1$, then $u_2$ is a strip of virtual dimension $0$, hence must be a strip which is constant along its length.
Our regularity assumptions for $J_t$-holomorphic spheres of Chern number $1$ ensure that they can not bubble off a constant holomorphic strip.
If $u_1$ is a holomorphic disc of Maslov index $2$, then $u_2$ is a strip of virtual dimension $0$, hence is constant along its length.
The signed count of such configurations with $u_1$ a disc on $L_0$ is $w(L_0)$, and the signed count with $u_1$ a disc on $L_1$ is $- w(L_1)$.
These are regular boundary points, by our regularity assumptions on the almost-complex structures $J_L$ and a gluing theorem.
The signed count of boundary points of this compact $1$-manifold is $0$, which gives the result.

In particular, if $w(L_0) = w(L_1)$, then $(\mu^1)^2 = 0$ on $CF^*(L_0,L_1)$, so we can define $HF^*(L_0,L_1)$ to be the cohomology of the differential $\mu^1$.

Now we define the $A_\infty$ structure on $\cF$. 
We define $hom^*_\cF(L_0,L_1) := CF^*(L_0,L_1)$, and we use the convenient abbreviation
\begin{equation} \cF(L_s, \ldots, L_0) := hom_\cF^*(L_{s-1},L_s)[1] \otimes \ldots \otimes hom_\cF^*(L_0,L_1)[1].\end{equation}
from \S \ref{subsec:ainfcat} (note that $\cF(L_s,L_0) := CF^*(L_0,L_s)[1]$).
Now we make a consistent choice of strip-like ends and a consistent choice of perturbation data for all moduli spaces of holomorphic discs with $s \ge 2$ incoming and $1$ outgoing boundary punctures, precisely as in  \cite[\S 9i]{Seidel2008}.
We recall that this entails a choice of $(K,J)$ for each disc $S$, where $K \in \Omega^1(S,\mathscr{H})$ is a domain-dependent Hamiltonian function and $J \in C^\infty(S,\mathscr{J})$ is a domain-dependent almost-complex structure.
We require that $J=J_L$ over boundary components labeled by $L$, and that $(K,J)$ coincides with the already-chosen Floer data $(H_t \otimes dt, J_t)$ over the strip-like ends.

We then consider the moduli spaces of smooth maps $u: S \to X$, with Lagrangian boundary conditions, satisfying the pseudoholomorphic curve equation $(du-Y_K)^{0,1}_J = 0$ given by our choice of perturbation data $(K,J)$.
For a generic choice of perturbation data these moduli spaces are regular: the argument in \cite[\S 9k]{Seidel2008} goes through unaltered.
In particular, moduli spaces of negative virtual dimension are empty.
As before, these moduli spaces have fixed energy, hence are compact, and we can define the $A_{\infty}$ structure maps $\mu^* \in CC^2(\cF)$ by signed counts of the zero-dimensional components (weighted by monodromy of the local systems around the boundary).
More explicitly, they are maps
\begin{equation} \mu^s: \cF(L_s,\ldots,L_0) \To \cF(L_s,L_0)[1].\end{equation}

By considering the boundary of one-dimensional moduli spaces as usual, we find that the $A_{\infty}$ relations $(\mu^* \circ \mu^*)^s = 0$ are satisfied for all $s \ge 2$.
In particular, if a nodal pseudoholomorphic disc $u_1 \# u_2$ appears, and $u_1$ is a disc of Maslov index $\ge 2$ or a sphere of Chern number $\ge 1$, then $u_2$ is a disc of negative virtual dimension, hence does not exist for generic choice of perturbation data.

Therefore, the $A_\infty$ relations are satisfied, with the sole exception that the differential does not square to zero, but rather satisfies \eqref{eqn:mu1sq}.
The result is that, for each $w \in \C$, we have a $\bm{G}$-graded, $\C$-linear, non-curved $A_{\infty}$ category $\cF(X)_w$, whose objects are exactly those $L$ such that $w(L) = w \in \C$.

\subsection{The cohomological unit}

We recall the Lagrangian version \cite{Albers2007} of the Piunikhin--Salamon--Schwarz morphism \cite{Piunikhin1996}.
It is a morphism of chain complexes between $C^*(L)$ and $CF^*(L,L)$, defined in degrees up to and including the minimal Maslov number of $L$ minus two.
We will only be interested in the degree-zero part of this map, and in particular the element $e_L \in CF^0(L,L)$ which is the image of the identity.
To define it, we consider the moduli space of pseudoholomorphic discs illustrated in Figure \ref{subfig:Fig1a}.
There is a single outgoing boundary puncture, a single internal marked point which is unconstrained, and a single boundary marked point which is unconstrained (unconstrained marked points serve to stabilize the moduli space).
Counting the zero-dimensional component of this moduli space defines $e_L$.
Counting the boundary points of the one-dimensional component shows that $\mu^1(e_L) = 0$.

\begin{figure}
\centering
\subfigure[The unit.]{
\includegraphics[width=0.48\textwidth]{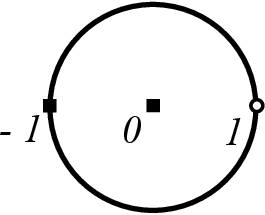}
\label{subfig:Fig1a}}
\subfigure[The element $e_L \in CF^*(L,L)$ is a unit.]{
\includegraphics[width=0.4\textwidth]{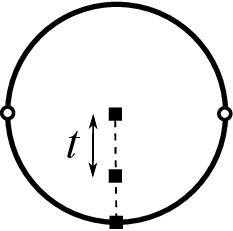}
\label{subfig:unit}}
\caption{Our conventions for diagrams of moduli spaces of pseudoholomorphic discs are as follows:
We will use a solid round dot to denote a marked point (interior or boundary) with some constraint, a solid square dot to denote a marked point (interior or boundary) without constraint, and an open round dot to denote a boundary puncture.
\label{fig:Fig1a}}
\end{figure}

To see that $e_L$ is a unit, consider the one-parameter family of holomorphic discs in Figure \ref{subfig:unit}, parametrized by $t \in [0,1]$.
We make a consistent choice of perturbation data on this family, so that at $t=0$ the perturbation datum is invariant under $\R$-translation, just given by the Floer data on the strip, and as $t \To 1$ the perturbation datum converges to that defining $\mu^2$, glued to that defining the unit.
Counting the zero-dimensional component of the corresponding moduli space of pseudoholomorphic discs defines a map
\begin{equation} H: CF^*(L,L) \To CF^{*-1}(L,L).\end{equation}
Counting the boundary points of the one-dimensional component of the moduli space shows that
\begin{equation} \mu^2( x,e_L) = x + \mu^1(H(x)) + H(\mu^1(x)).\end{equation}
The boundary points at $t=1$ contribute the left-hand side, the boundary points at $t=0$ contribute the first term on the right-hand side (because the perturbation data are invariant under translation, the moduli space admits an $\R$-action, but the moduli space is zero-dimensional so the $\R$-action must be trivial, hence the strip must be constant along its length), and the remaining terms correspond to strip breaking.
Unstable disc and sphere bubbling are ruled out exactly as in the definition of the Fukaya category.

It follows that right-multiplication with $e_L$ is homotopic to the identity.
It follows similarly that left multiplication with $e_L$ is homotopic to the identity, and therefore that $e_L$ is a cohomological unit in $\cF(X)_w$.

\subsection{The closed--open string map, $\CO$}
\label{subsec:co}

In this section, we consider the closed--open string map, which relates quantum cohomology of $X$ to Hochschild cohomology of the Fukaya category of $X$ (compare, e.g., \cite[\S 3.8.4]{fooo}, \cite[\S 6]{Fukaya2010d}, or in the case of open manifolds, \cite[\S 5.4]{Ganatra2012}).

Let us fix $w \in \C$, and denote $\cF:=\cF(X)_w$ to avoid notational clutter.
The \emph{closed--open string map} is a $\bm{G}$-graded $\C$-algebra homomorphism
\begin{equation} \CO: QH^*(X) \To HH^*(\cF).\end{equation}

To define $\CO$, we consider moduli spaces of holomorphic discs with $s \ge 0$ incoming boundary punctures, $1$ outgoing boundary puncture, and a single internal marked point.
We choose strip-like ends and perturbation data for these moduli spaces, and require them to be consistent with the Deligne--Mumford compactification.
Boundary conditions correspond to generators of the Hochschild cochain complex of $\cF$; if $\varphi$ is such a generator, and $\beta$ a homology class of discs, we denote the resulting moduli space by $\cM(\varphi,\beta)$.
It has a Gromov compactification, which we denote by $\overline{\cM}(\varphi,\beta)$, and a continuous evaluation map at the internal marked point:
\begin{equation} ev: \overline{\cM}(\varphi,\beta) \To X.\end{equation}

We define $\cM^0 := \cM$, and $\cM^1$ to be the stratum of the Gromov compactification of virtual codimension $1$.
Its elements are pairs of discs breaking along a strip-like end, one of which contains the interior marked point.
For a generic choice of perturbation data, the Gromov compactification admits a decomposition
\begin{equation} \overline{\cM} = \bigsqcup_{i = 0}^\infty \cM^i, \end{equation}
where (if we denote $ev^i := ev|_{\cM^i}$)
\begin{itemize}
\item $\cM^0$ is regular, hence a smooth, oriented manifold of dimension $d(\varphi,\beta)$, and $ev^0$ is smooth;
\item $\cM^1$ is regular, hence a smooth, oriented manifold of dimension $d(\varphi,\beta) - 1$, and $ev^1$ is smooth;
\item For all $i \ge 2$, the evaluation map $ev^i$ factors through a smooth map from a smooth manifold of dimension $d(\varphi,\beta)-i$,
\begin{equation} \tilde{ev}^i: \widetilde{\cM}^i \To X.\end{equation}
\end{itemize}
Explicitly, $\widetilde{\cM}^i$ is the union of all moduli spaces of nodal discs which have no disc or sphere components with a single special point (a special marked point is a node or a marked point), and such that any sphere with only two marked points is simple.
The map $\cM^i \To \widetilde{\cM}^i$ is defined by forgetting trees of unstable discs and spheres, and replacing multiply-covered spheres with two marked points (which may appear in a chain connecting the disc component to a sphere containing the internal marked point) by the sphere they cover.
This process can only decrease the virtual dimension, by monotonicity.
Furthermore, $\widetilde{\cM}^i$ is regular for generic choice of perturbation data, because all of its components are simple spheres or discs.

Now let $f: A \To X$ be a pseudocycle, representing a homology class which is Poincar\'{e} dual to $\alpha \in H^d(X)$.
We denote the moduli space of discs, with the marked point constrained to lie on the pseudocycle $f$, by $\cM^0(\varphi,\beta,f)$.
Similarly, we define $\cM^1(\varphi,\beta,f)$ and $\widetilde{\cM}^i(\varphi,\beta,f)$.
For generic choice of perturbation data, these moduli spaces are regular, of dimension $d(\varphi,\beta)-d-i$.

If $d(\varphi,\beta) = d$, then for generic choice of perturbation data, the moduli space $\cM^0(\varphi,\beta,f)$ is regular, hence an oriented $0$-manifold; and furthermore, the images of $ev^1,\tilde{ev}^2,\ldots$ are disjoint from the closure of the image of $f$, and the images of $ev^0, ev^1,\tilde{ev}^2,\ldots$ are disjoint from $\Omega_f$, the limit set of $f$ (see \cite[Definition 6.5.1]{mcduffsalamon}).
It follows that
\begin{equation} \Omega_{ev^0} \cap \overline{im(f)} = \Omega_f \cap \overline{im(ev^0)} = \emptyset,\end{equation}
and hence that $\cM^0(\varphi,\beta,f)$ is compact (compare \cite[p. 161]{mcduffsalamon}).
We define the coefficient of $\varphi$ in $\CO(\alpha;f)$ to be the signed count of its points, summed over homology classes $\beta$ such that $d(\varphi,\beta) = d$ (this sum converges by our monotonicity assumptions).

Now consider a moduli space such that $d(\varphi,\beta) = d+1$.
By a similar argument to above, for generic perturbation data, $\cM^0(\varphi,\beta,f)$ is an oriented $1$-manifold, $\cM^1(\varphi,\beta,f)$ is an oriented $0$-manifold, and their union is compact.
By a gluing theorem, their union has the structure of a compact oriented $1$-manifold with boundary points $\cM^1(\varphi,\beta,f)$, so the signed count of points in the latter is $0$; it follows that $\delta(\CO(\alpha;f)) = 0$, where $\delta$ is the Hochschild differential.
Hence, $\CO(\alpha;f)$ defines a class in $HH^*(\cF)$.

If the pseudocycle $f$ is bordant to another pseudocycle $g$, we choose a bordism $h$ between them, and consider the zero-dimensional component of the moduli space $\cM^0(\varphi,\beta,h)$: as before, it is a compact, oriented $0$-manifold, and counting its points defines an element $H(h) \in CC^*(\cF)$.
Next we consider the one-dimensional component of $\cM^0(\varphi,\beta,h)$: as before, it is an oriented, compact $1$-manifold with boundary, and counting its boundary points shows that
\begin{equation} \CO(\alpha;f) - \CO(\alpha;g) = \delta(H(h)).\end{equation}
Therefore, the class of $\CO(\alpha;f)$ in $HH^*(\cF)$ does not depend on the choice of pseudocycle $f$ representing $PD(\alpha)$, so we have a well-defined map
\begin{equation} \CO:QH^*(X) \To HH^*(\cF).\end{equation}
Standard index theory of Cauchy--Riemann operators shows that $\CO$ respects the $\bm{G}$-grading.

To show that $\CO$ is independent of the choice of perturbation data used to define it, one uses a `double category' trick as in \cite[\S 10a]{Seidel2008}; we omit the details.

\begin{proposition}
\label{proposition:coalg} (compare \cite[Proposition 5.3]{Ganatra2012})
$\CO$ is a homomorphism of $\C$-algebras.
\end{proposition}
\begin{proof}
We consider a certain subset of the moduli space of discs with two internal marked points, $s \ge 0$ incoming boundary punctures, and one outgoing boundary puncture.
Namely, parametrizing our disc by the unit disc in $\C$, we require that the outgoing boundary puncture lies at $-i$, and the internal marked points lie at $\pm t$, where $t \in [0,1]$ (see Figure \ref{subfig:Fig2a}).
We choose consistent perturbation data for these moduli spaces.

\begin{figure}
\centering
\subfigure[One part of the moduli space defining {$H:QH^*(X)^{ \otimes 2} \To CC^*(\cF).$}]{
\includegraphics[width=0.4\textwidth]{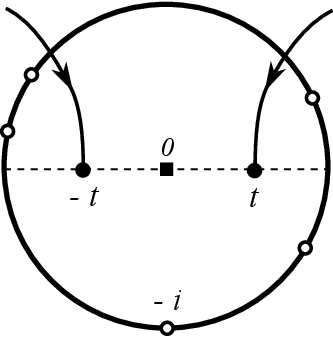}
\label{subfig:Fig2a}}
\subfigure[The part of the moduli space at $t=0$.]{
\includegraphics[width=0.4\textwidth]{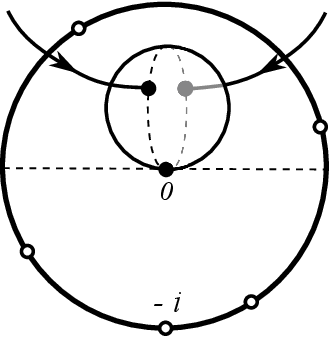}
\label{subfig:Fig2b}}
\subfigure[The boundary component at $t=1$.]{
\includegraphics[width=0.8\textwidth]{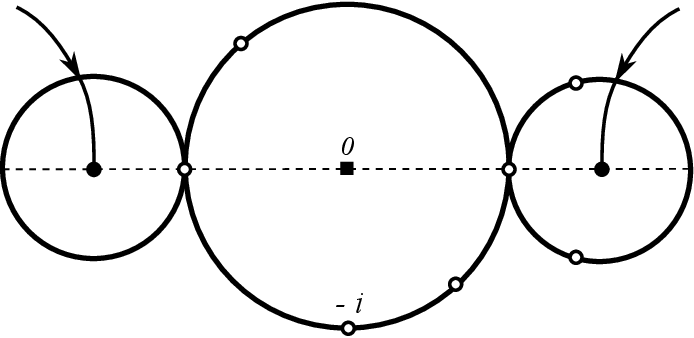}
\label{subfig:Fig2c}}
\caption{Proving the closed--open string map is an algebra homomorphism.
\label{fig:2l1discs}}
\end{figure}

Given cohomology classes $\alpha,\beta \in QH^*(X)$ which are Poincar\'{e} dual to pseudocycles $f,g$, we consider the corresponding moduli space of pseudoholomorphic discs with the internal marked points constrained to lie on the pseudocycles.
Counting the zero-dimensional component of the moduli space defines an element $H(f,g) \in CC^*(\cF)$.
Now we consider the one-dimensional component of the moduli space; the count of its boundary points is $0$.
Boundary points at $t=1$ are illustrated in Figure \ref{subfig:Fig2c}, and correspond to terms in $\CO(\alpha;f) \cup \CO(\beta;g)$, where `$\cup$' denotes the Yoneda product.
Boundary points for $0<t<1$ corresponds to a disc bubbling off one of the the discs defining $H(f,g)$; they correspond to terms in $\delta (H(f,g))$.
Boundary points at $t=0$ are illustrated in Figure \ref{subfig:Fig2b}; they correspond to terms in $\CO(\alpha \star \beta;ev)$, where `$\star$' denotes the quantum cup product and `$ev$' is the evaluation map of Remark \ref{remark:qcupeval}.

Therefore, we have
\begin{equation} \CO(\alpha \star \beta;ev) = \CO(\alpha;f) \cup \CO(\beta;g) + \delta(H(f,g)).\end{equation}
It follows that $\CO$ is an algebra homomorphism on the level of cohomology.
 \end{proof}

\begin{definition}
\label{definition:co0}
Given an object $L$ of some $\cF(X)_w$, we can consider the composition of the closed--open map $\CO$ with the projection to the length-zero part of Hochschild cohomology,
\begin{equation} HH^*(CF^*(L,L)) \To HF^*(L,L)\end{equation}
(see \eqref{eqn:hhproj}).
As this projection and $\CO$ are both algebra homomorphisms, their composition is an algebra homomorphism, which we denote by
\begin{equation} \CO^0: QH^*(X) \To HF^*(L,L).\end{equation}
\end{definition}

\begin{remark}
\label{remark:co0un}
The homomorphism $\CO^0$ is obviously unital, because the moduli spaces defining $\CO^0(e)$ and $e_L$ count the same objects.
\end{remark}

\subsection{The open--closed string map, $\OC$}
\label{subsec:oc}

In this section, we consider the open--closed string map, which relates quantum cohomology to Hochschild homology (see e.g. \cite[\S 3.8.1]{fooo}, \cite[\S 5.3]{Abouzaid2010a}).

The open--closed string map is a map of $\bm{G}$-graded $\C$-vector spaces
\begin{equation} \OC: HH_*(\cF) \To QH^{*+n}(X)\end{equation}
(where $X$ has real dimension $2n$).
It is defined by considering moduli spaces of pseudoholomorphic discs with $s \ge 1$ incoming boundary punctures, and an internal marked point.
We choose consistent strip-like ends and perturbation data for these moduli spaces.

Boundary conditions for these moduli spaces are given by generators of the Hochschild chain complex.
For a generator $\varphi$ and a cohomology class $\alpha$, whose Poincar\'{e} dual is represented by a pseudocycle $f$, we consider the corresponding moduli space of pseudoholomorphic discs, with the internal marked point constrained to lie on $f$.
Counting the zero-dimensional component of this moduli spaces gives a number, which we define to be
\begin{equation} \langle \OC(\varphi),\alpha;f \rangle.\end{equation}
As in the definition of $\CO$, counting the boundary points of the one-dimensional component of the moduli space shows that
\begin{equation} \langle \OC(b(\varphi)),\alpha;f \rangle = 0,\end{equation}
where $b$ denotes the Hochschild differential, so this number depends only on the class of $\varphi$ in $HH_*(\cF)$.
Furthermore, the number is independent of the choice of pseudocycle $f$ representing $PD(\alpha)$, by an argument analogous to the one we gave for $\CO$.
Therefore, we have a well-defined map
\begin{equation} \langle \OC(-),-\rangle: HH_*(\cF) \otimes QH^*(X) \To \C; \end{equation}
dualizing in the $QH^*(X)$ factor gives $\OC$.
Standard index theory of Cauchy--Riemann operators shows it is $\bm{G}$-graded, of degree $n$.

\begin{remark}
\label{remark:dualbas}
Explicitly, if $\{e_i\}$ is a basis for $QH^*(X)$, with dual basis $\{e^j\}$, in the sense that $\langle e_i,e^j \rangle = \delta^j_i$, then
\begin{equation}
\OC(\varphi) = \sum_i \langle \OC(\varphi), e_i \rangle e^i.
\end{equation}
\end{remark}

\begin{remark}
\label{remark:nopseudo}
Note that, in contrast to the quantum cup product (see Remark \ref{remark:qcupeval}), we can not represent $\OC(\varphi)$ as a pseudocycle by the evaluation map from our moduli space, because the moduli space has codimension-$1$ boundary (of course the codimension-$1$ boundary still `cancels' if $\varphi$ is a Hochschild cycle).
\end{remark}

Now we recall that $HH_*(\cF)$ is naturally a $HH^*(\cF)$-module (see \S \ref{subsec:hhalgmod}); hence it is naturally a $QH^*(X)$-module, via the algebra map $\CO$.
The following result is due to \cite{Ritter2012} in the monotone case, and \cite{Ganatra2012} in the exact case.

\begin{proposition}
\label{proposition:ocmod}
$\OC$ is a homomorphism of $QH^*(X)$-modules.
\end{proposition}
\begin{proof}
We consider the same moduli space of discs as in the proof of Lemma \ref{proposition:coalg}, but with some boundary punctures oriented in the opposite direction, and corresponding changes to the perturbation data to achieve consistency.
A virtually identical argument to the proof of Lemma \ref{proposition:coalg} shows that given pseudocycles $f,g$, there exists a map $K(-;f,g): CC_*(\cF) \to \C$ such that
\begin{equation} \langle \OC(\varphi),\alpha \star \beta;ev \rangle = K(b(\varphi);f,g) + \langle \OC(\CO(\alpha;f) \cap \varphi),\beta ;g\rangle\end{equation}
(see the proof of \cite[Proposition 5.4]{Ganatra2012} for more details).
It now follows from the fact that $QH^*(X)$ is a Frobenius algebra that
\begin{equation} \alpha \star \OC(\varphi) = \OC(\CO(\alpha) \cap \varphi)\end{equation}
on the level of cohomology, and hence that $\OC$ is a homomorphism of $QH^*(X)$-modules.
 \end{proof}

\begin{definition}
\label{definition:oc0}
If $L$ is an object of $\cF(X)_w$, we can consider the composition of the inclusion
\begin{equation} HF^*(L,L) \To HH_*(CF^*(L,L))\end{equation}
(see \eqref{eqn:hhinj}) with the open--closed map $\OC$.
We denote the result by
\begin{equation} \OC^0: HF^*(L,L) \To QH^{*+n}(X) .\end{equation}
Because $\OC$ is a $QH^*(X)$-module homomorphism, $\OC^0$ is too.
\end{definition}

\subsection{Two-pointed closed--open and open--closed maps}
\label{subsec:2point}

We now recall (from \cite[\S 5.6]{Ganatra2012}) the construction of the \emph{two-pointed} closed--open and open--closed maps, $_2\CO$ and $_2 \OC$.

To define $_2\CO$, we consider a subset of the moduli space of discs with $k+l+2$ boundary punctures, and an internal marked point.
We label the boundary punctures $p_{out}, q_1, \ldots, q_k, p_{in}, q_{k+1}, \ldots, q_{k+l}$ in order around the boundary, and consider the moduli space of discs such that, if we parametrize the disc as the unit disc in $\C$, then $p_{in}$ lies at $-1$, $p_{out}$ lies at $+1$, and the internal marked point lies on the real axis (see Figure \ref{fig:2CO}).
We define the boundary puncture $p_{out}$ to be outgoing, and all other punctures to be incoming.
We make a consistent choice of strip-like ends and perturbation data for this moduli space, and consider the corresponding moduli space of pseudoholomorphic discs.

Boundary conditions for this moduli space correspond to generators of $_2CC^*(\cA)$ (see \S \ref{subsec:hhi}).
Counting rigid pseudoholomorphic discs in this moduli space, with the marked point constrained to lie on a pseudocycle, defines a map
\begin{equation} _2 \CO: QH^*(X) \To HH^*(\cA).\end{equation}
The by-now-familiar arguments show that it is well-defined and independent of the choices made in its construction.
The argument of \cite[Proposition 5.6]{Ganatra2012}, adapted to the present setting, shows that it coincides with $\CO$.
More precisely, for any $A_\infty$ category $\cA$ there exists an explicit quasi-isomorphism $\Psi: CC^*(\cA) \to {}_2 CC^*(\cA)$  \cite[Equation (2.200)]{Ganatra2012}, and in the case of the Fukaya category, there is an explicit homotopy between $\Psi \circ \CO$ and $_2\CO$, given by counting a moduli space of pseudoholomorphic discs with boundary conditions analogous to those defining $_2\CO$, but the marked point $p_{in}$ is allowed to vary between $-1$ and $+1$ along the lower boundary of the disc (see \cite[Figure 8]{Ganatra2012} for a picture).

One defines $_2 \OC$ using the same moduli space of domains, but with $p_{out}$ now regarded as an incoming boundary puncture, and the perturbation data modified accordingly.
Similar arguments show that it is well-defined, independent of choices made in its construction, and coincides with $\OC$.

\begin{remark}
To show that $\CO$ is an algebra homomorphism, and $\OC$ is a $QH^*(X)$-module homomorphism, we consider the same moduli space, except with two internal marked points, constrained to lie on the real axis in a prescribed order.
\end{remark}

\begin{lemma}
\label{lemma:counital}
$\CO$ is a \emph{unital} algebra homomorphism.
\end{lemma}
\begin{proof}
We make a special choice of perturbation data for the moduli space defining $_2 \CO$: we require the perturbation data to be independent of the position of the internal marked point $q$.
In other words, we require the perturbation data to be independent of the $\R$-action corresponding to moving $q$ along the line connecting $-1$ and $1$.
In particular, when $k=l=0$, we choose translation-invariant perturbation data coming from the corresponding Floer datum: this is compatible with the strip-like ends at $p_{in}$ and $p_{out}$, because one is incoming and one is outgoing.
Consistency with this choice of perturbation data for $k=l=0$ requires us to impose an additional condition on the strip-like ends at $p_{in}$ and $p_{out}$: namely, the dotted line should go down the centre of these strip-like ends.
It is clear that we can always choose strip-like ends and consistent perturbation data in this fashion.
Furthermore, it remains possible to achieve transversality for perturbation data chosen in this special class.

The unit $e \in QH^*(X)$ is Poincar\'{e} dual to the fundamental cycle of the manifold; so in the moduli space defining $_2\CO(e)$, there is no constraint on the internal marked point.
With our choice of $\R$-invariant perturbation data, this means that there is an action of $\R$ on the moduli space of pseudoholomorphic discs.
If one of $k,l$ is non-zero, this action is free, so the moduli space can't be $0$-dimensional, hence can't contribute to $_2 \CO(e)$; when $k=l=0$, the action is free unless the strip is constant along its length.
Therefore, the only contribution to $_2\CO(e)$ is the identity endomorphism of the diagonal bimodule in $\fmodf{\cF}{\cF}$.
 \end{proof}

\begin{figure}
\centering
\includegraphics[width=0.6\textwidth]{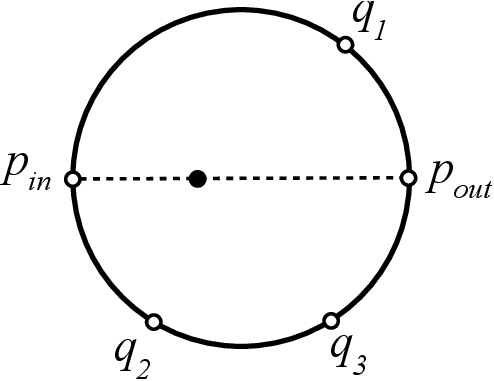}
\caption{The two-pointed closed--open map.
\label{fig:2CO}}
\end{figure}

\subsection{Weak proper Calabi--Yau structures}
\label{subsec:infin}

We now explain how the monotone Fukaya category of $X$ can be equipped with an $n$-dimensional weak proper Calabi--Yau structure, in the sense of Definition \ref{definition:weakcy}.
The idea was outlined in \cite[\S 12j]{Seidel2008} and \cite[Proof of Proposition 5.1]{Seidel2010c}.

\begin{lemma}
\label{lemma:weakcy}
The class $[\phi] \in HH_n(\cF)^\vee$ given by
\begin{equation} [\phi](b) := \langle \OC(b), e\rangle\end{equation}
is an $n$-dimensional weak proper Calabi--Yau structure on $\cF$, in the sense of Definition \ref{definition:weakcy}.
\end{lemma}
\begin{proof}
The class $[\phi]$ is clearly $n$-dimensional, because $\OC$ has degree $n$.
To prove that it is homologically non-degenerate, we must show that the pairing
\begin{align}
HF^*(K,L) \otimes HF^{n-*}(L,K) & \To  \C \\
\label{eqn:comp0} p \otimes q & \mapsto  \langle \OC^0(\mu^2(p,q)),e \rangle
\end{align}
is perfect (see Definition \ref{definition:homnond}).
Equivalently, we must show that the corresponding map
\begin{equation} HF^*(K,L) \To HF^{n-*}(L,K)^\vee\end{equation}
is an isomorphism.

The `reason' this map is an isomorphism is as follows: the pairing is homotopic to the pairing $\langle _2 \OC(- \otimes -), e \rangle$, via the homotopy between $\OC$ and $_2 \OC$.
The corresponding map can be regarded as a continuation map from $CF^*(K,L)$ (defined using Floer datum $(H,J)$) to $CF^*(L,K)^\vee \cong CF^*(K,L)$ (defined using Floer datum $(-H,J)$); so one can apply the standard argument to prove that continuation maps are quasi-isomorphisms.

For the purpose of generalizing this result later (see Lemma \ref{lemma:weakcywbc}), we give a more abstract formulation of the proof.
Firstly, for any $\alpha \in QH^*(X)$, we can consider the map
\begin{align}
 \langle \OC^0(\mu^2(-,-)),\alpha \rangle & \in  (HF^*(K,L) \otimes HF^*(L,K))^\vee \\
\label{eqn:comp1}& \cong  \mathrm{Hom}(HF^*(K,L),HF^*(L,K)^\vee).
\end{align}
Now we define the coproduct
\begin{equation} \Delta: HF^*(L,L) \To HF^*(K,L) \otimes HF^*(L,K)[n]\end{equation}
by counting pseudoholomorphic discs with one incoming and two outgoing boundary punctures.
For any $\beta \in QH^*(X)$, we can consider the element
\begin{align}
\Delta(\CO^0(\beta)) & \in  HF^*(K,L) \otimes HF^*(L,K) \\
\label{eqn:comp2}& \cong  \mathrm{Hom}(HF^*(L,K)^\vee,HF^*(K,L)).
\end{align}
We claim that the composition of the homomorphisms \eqref{eqn:comp1}, \eqref{eqn:comp2} is equal to the map
\begin{align}
HF^*(K,L) & \To  HF^*(K,L) \\
p & \mapsto  \mu^2(\CO^0(\alpha \star \beta),p).
\end{align}
The proof that the two maps are homotopic follows familiar lines, and we omit it.

In particular, if $\alpha = \beta = e$, then the composition is equal to $\mu^2(\CO^0(e),-)$, and hence is the identity, because $\CO^0$ is unital by Remark \ref{remark:co0un}.
A similar argument shows that the composition in the other order is also equal to the identity.
Therefore, the map \eqref{eqn:comp1} is an isomorphism, so the pairing \eqref{eqn:comp0} is perfect, as required.
 \end{proof}

\begin{remark}
The weak proper Calabi--Yau structure introduced in Lemma \ref{lemma:weakcy} coincides with that outlined in  \cite[\S 12j]{Seidel2008} and \cite[proof of Proposition 5.1]{Seidel2010c}, via the identification of $\OC$ with $_2\OC$.
It can be thought of as an expression of Poincar\'{e} duality for the Donaldson--Fukaya category.
\end{remark}

It follows by Lemma \ref{lemma:hhdual} that:

\begin{corollary}
The map
\begin{equation}- \cap [\phi]:HH^*(\cF) \To  HH_*(\cF)^\vee[-n]\end{equation}
is an isomorphism of $HH^*(\cF)$-modules.
Here `$\cap$' denotes the $HH^*(\cF)$-module structure on $HH_*(\cF)^\vee$ which is dual to the cap product on $HH_*(\cF)$, by slight abuse of notation.
\end{corollary}

\begin{proposition}
\label{proposition:phico}
The following diagram commutes:
\begin{equation}
\xymatrixcolsep{5pc}\xymatrix{QH^*(X) \ar[r]^-{\alpha \mapsto \langle \alpha,-\rangle}_-{\cong} \ar[d]^-{\CO}& QH^*(X)^\vee[-2n] \ar[d]^-{\OC^\vee} \\
HH^*(\cF) \ar[r]^-{-\cap [\phi]}_-{\cong} & HH_*(\cF)^\vee [-n].}
\end{equation}
Thus, $\CO$ and $\OC$ are dual, up to natural identifications of the respective domains and targets.
\end{proposition}
\begin{proof}
For any $\alpha \in QH^*(X)$ and $\psi \in HH_*(\cF)$, we have
\begin{align}
\left(\CO(\alpha) \cap [\phi] \right)(\psi) &= \langle \OC(\CO(\alpha) \cap \psi), e \rangle \\
&= \langle \alpha \star \OC(\psi), e \rangle\\
&= \langle \alpha , \OC(\psi) \rangle \\
&= \OC^\vee(\langle \alpha, - \rangle)(\psi).
\end{align}
The first line is the definition of $[\phi]$.
The second line follows because $\OC$ is a $QH^*(X)$-module homomorphism by Proposition \ref{proposition:ocmod}.
The third line follows because $QH^*(X)$ is a Frobenius algebra.
Hence, the diagram commutes.
 \end{proof}

\subsection{Eigenvalues of $c_1 \star$}
\label{subsec:c1egval}

\begin{lemma}
\label{lemma:c1u0}
(due to Auroux, Kontsevich and Seidel, see \cite[\S 6]{Auroux2007})
If $c_1 \in QH^*(X)$ is the first Chern class of $TX$, then we have
\begin{equation} \CO^0(c_1) = w(L) \cdot e_L \in HF^*(L,L).\end{equation}
\end{lemma}
\begin{proof}
Consider a pseudocycle $f: A \To X \setminus L$ which represents a homology class Poincar\'{e} dual to the Maslov class $\mu \in H^2(X,L)$.
Because the diagram
\begin{equation}
\xymatrix{H^2(X,L) \ar[r]^-{\cong} \ar[d] & H_{2n-2}(X \setminus L) \ar[d]\\
H^2(X) \ar[r]^-{\cong} & H_{2n-2}(X)}
\end{equation}
commutes, and the left vertical arrow sends $\mu$ to $2c_1$, the pseudocycle $f$ is Poincar\'{e} dual to $2c_1$ in $X$.

Then $\CO^0(2  c_1)$ is obtained by counting pseudoholomorphic discs as in Figure \ref{subfig:Fig4a}, with the internal marked point constrained to lie on $f$.
We now consider a one-parameter family of holomorphic discs as in Figure \ref{subfig:Fig4b}, parametrized by $t \in [0,1]$.
We choose perturbation data on this family which coincide with those used to define $\CO^0$ at $t=0$, and which coincide with those used to define $e_L$ (with the constant almost-complex structure $J = J_L$ on the disc bubble) at $t=1$.
We consider the corresponding moduli space of pseudoholomorphic discs.

\begin{figure}
\centering
\subfigure[The moduli space defining $\CO^0(2c_1)$.]{
\includegraphics[width=0.4\textwidth]{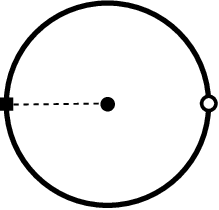}
\label{subfig:Fig4a}}
\subfigure[The moduli space defining $H \in CF^*(L,L)$.]{
\includegraphics[width=0.45\textwidth]{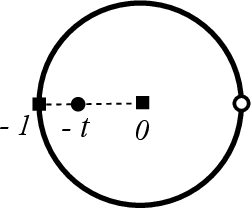}
\label{subfig:Fig4b}}
\subfigure[Boundary points at $t=1$.]{
\includegraphics[width=0.8\textwidth]{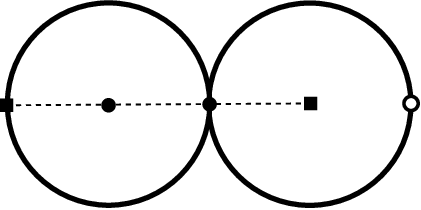}
\label{subfig:Fig4c}}
\caption{Proving Lemma \ref{lemma:c1u0}.
\label{fig:Fig4a}}
\end{figure}

Counting the zero-dimensional component defines an element
\begin{equation} H \in CF^*(L,L).\end{equation}
Counting the boundary points of the one-dimensional component shows that
\begin{equation}
\label{eqn:78}
\CO^0(2c_1) = 2 w(L) \cdot e + \mu^1(H).
\end{equation}
The boundary points at $t=0$ contribute the left-hand side (Figure \ref{subfig:Fig4a}).
The boundary points at $0<t<1$ correspond to breaking off a strip on the strip-like end, and contribute the last term on the right-hand side.
The remaining boundary points at $t=1$ contribute the first term on the right-hand side (Figure \ref{subfig:Fig4c}).
To see why, observe that these boundary points consist of a $J_L$-holomorphic disc bubble $u_1$, together with an internal marked point of $u_1$ constrained to lie on $f$, together with a pseudoholomorphic disc $u_2$ which is an element of the moduli space used to define $e_L$.

Now $u_1$ can not be a constant bubble, because $im(f)$ does not intersect $L$.
Therefore it must have Maslov index $\ge 2$.
If it has Maslov index $>2$ then $u_2$ would generically not exist, so $u_1$ must have Maslov index $2$.
It follows that $u_2$ is rigid, with output $e_L$, and we must count the number of $J_L$-holomorphic Maslov index $2$ discs $u_1$ with an internal marked point lying on $f$, whose boundary marked point coincides with the boundary marked point of the disc $u_2$.
By definition, there are $w(L)$ such discs $u_1$, and for each we have a signed count of $u_1 \cdot f = \mu(u_1) = 2$ choices of internal marked point lying on $f$.
So the contribution of the boundary points at $t=1$ is exactly $2w(L) \cdot e_L$.
Equation \eqref{eqn:78} implies the result.
 \end{proof}

Now let us consider the map
\begin{equation}c_1\star: QH^*(X) \To QH^*(X)\end{equation}
given by quantum cup product with $c_1 := c_1(TX)$.
Denote the set of eigenvalues of $c_1\star$ by $\Lambda$.
Let
\begin{equation}
\label{eqn:eigdec}
QH^*(X) \cong \bigoplus_{w \in \Lambda} QH^*(X)_{w}\end{equation}
be the decomposition of $QH^*(X)$ into generalized eigenspaces of $c_1\star$, and let
\begin{equation} e = \sum_{w \in \Lambda} e_w\end{equation}
be the corresponding decomposition of the identity.
Observe that \eqref{eqn:eigdec} is a direct sum as algebras, i.e., that elements in different components $QH^*(X)_w$ multiply to zero.
Observe also that $e_w \in QH^*(X)_w$ is the identity element, and that
\begin{equation} e_w\star : QH^*(X) \To QH^*(X)\end{equation}
is the projection map onto the generalized eigenspace $QH^*(X)_{w}$.

We now explain that the Fukaya category and the closed--open map split up into components indexed by the eigenvalues $w \in \Lambda$: compare \cite[\S 6]{Auroux2007}).
The following results are heavily based on the work of Alex Ritter and Ivan Smith \cite{Ritter2012}, whom I thank for many explanations on these points.

The first step is an elementary lemma:

\begin{lemma}
\label{lemma:comparison}
Suppose that $\cF$ is an $A_\infty$ category, $\varphi \in HH^*(\cF)$, and $\varphi^0_L = c \cdot e_L \in \mathrm{Hom}^0(L,L)$ for a fixed $c \neq 0$, for all objects $L$ of $\cF$ (here $\varphi^0_L$ is the projection of $\varphi$ to its length-zero component, see equation \eqref{eqn:hhproj}). Then 
\begin{equation} \varphi \cup -: HH^*(\cF) \to HH^*(\cF)\end{equation}
is an isomorphism.
\end{lemma}
\begin{proof}
We equip the Hochschild cochain complex $CC^*(\cF)$ with the length filtration, and consider the map 
\begin{equation}
\varphi \cup-: CC^*(\cF) \to CC^*(\cF),
\end{equation} 
where $\varphi$ is now a cochain-level representative by abuse of notation. 
This clearly preserves the length filtration (as one sees from the cochain-level formula for the Yoneda product, equation \eqref{eqn:yon1pt}). 
The $E_1$ page of the associated spectral sequence is $CC^*(H^*(\cF))$, the Hochschild cochain complex of the cohomological category of $\cF$. 
The endomorphism of $E_1$ induced by $\varphi \cup-$ is simply multiplication by $c$, by the hypothesis: hence it is an isomorphism.
The conclusion now follows by the Eilenberg--Moore comparison theorem \cite[Theorem 5.5.11]{Weibel1994}, as the length filtration is complete and exhaustive.
\end{proof}

\begin{proposition}
\label{proposition:coeigsplit}
The map
\begin{equation}
\label{eqn:coww'} 
\CO: QH^*(X)_{w'} \to HH^*(\cF(X)_{w})
\end{equation}
vanishes if $w' \neq w$, and is a unital homomorphism of $\C$-algebras if $w' = w$.
\end{proposition}
\begin{proof}
Suppose $w' \neq w$. We apply Lemma \ref{lemma:comparison}, with
\begin{equation}
\cF := \cF(X)_{w}\mbox{, and } \varphi := \CO(c_1 - w' \cdot e)^{\cup k}.
\end{equation} 
We have $\varphi^0_L = (w-w')^k \cdot e_L$ for all objects $L$ of $\cF$, by Lemma \ref{lemma:c1u0}: so the hypothesis of Lemma \ref{lemma:comparison} holds, with $c=(w-w')^k$. 
Hence the endomorphism
\begin{equation} \CO(c_1-w' \cdot e)^{\cup k} \cup -: HH^*(\cF(X)_{w}) \to HH^*(\cF(X)_{w})\end{equation}
is an isomorphism.
 
In particular, if  $\alpha \in QH^*(X)_{w'}$, then $(c_1 - w'\cdot e)^{\star k}\star \alpha = 0$ for some $k$; so by Proposition \ref{proposition:coalg}, 
\begin{equation} \CO(c_1 - w' \cdot e)^{\cup k} \cup \CO(\alpha) = 0,\end{equation}
from which it follows by the preceding argument that $\CO(\alpha) = 0$. 
This proves the first part of the statement: the map \eqref{eqn:coww'} vanishes if $w' \neq w$. 

For the second part, we observe that the map $\CO: QH^*(X) \to HH^*(\cF(X)_{w})$ is unital by Lemma \ref{lemma:counital}, and kills all $e_{w'}$ for $w' \neq w$: it follows that the restriction to $QH^*(X)_w$ is unital.
\end{proof}

\begin{corollary}
\label{corollary:eigentriv}
$\cF(X)_w$ is trivial unless $w$ is an eigenvalue of $c_1\star$.
\end{corollary}
\begin{proof}
Suppose $L$ is an object of $\cF(X)_w$. 
It follows from Proposition \ref{proposition:coeigsplit} that $\CO^0: QH^*(X)_w \to HF^*(L,L)$ is a unital algebra homomorphism.
If $w$ is not an eigenvalue of $c_1 \star$, then $QH^*(X)_w \cong 0$, hence $HF^*(L,L) \cong 0$ (by unitality), so $L$ is quasi-isomorphic to the zero object.
\end{proof}

\subsection{Eigenvalues and duality}

We now prove a result that is dual to Proposition \ref{proposition:coeigsplit}. 
The proof of this result was explained to the author by Alex Ritter.

\begin{corollary}
\label{corollary:oceigsplit}
The image of the map
\begin{equation}
\OC: HH_*(\cF_w) \to QH^{*+n}(X)
\end{equation}
lands in $QH^{*+n}(X)_w$.
\end{corollary}
\begin{proof}
Because $QH^*(X)$ is a Frobenius algebra, and $c_1$ an even element in it,
\begin{equation} \langle c_1 \star \alpha, \beta \rangle = \langle \alpha, c_1 \star \beta \rangle,\end{equation}
so $c_1 \star$ is symmetric with respect to $\langle -,- \rangle$.
Therefore, the decomposition into generalized eigenspaces,
\begin{equation} QH^*(X) = \bigoplus_w QH^*(X)_w\end{equation}
is orthogonal with respect to the pairing $\langle -,- \rangle$.
It follows that the top map in the commutative diagram of Proposition \ref{proposition:phico}:
\begin{align}
QH^*(X) & \To  QH^*(X)^\vee \\
\alpha & \mapsto  \langle \alpha, - \rangle
\end{align}
decomposes as a direct sum of maps
\begin{align}
QH^*(X)_w & \To  QH^*(X)_w^\vee \\
\alpha & \mapsto  \langle \alpha, - \rangle.
\end{align}
The result now follows by combining Proposition \ref{proposition:phico} with Proposition \ref{proposition:coeigsplit}.
\end{proof}

According to Proposition \ref{proposition:coeigsplit} and Corollary \ref{corollary:oceigsplit}, the only non-zero components of $\CO$ and $\OC$ are
\begin{align}
\CO_w: QH^*(X)_w & \to  HH^*(\cF(X)_w) \mbox{, and}\\
\OC_w: HH_*(\cF(X)_w) & \to  QH^{*+n}(X)_w.
\end{align}
Furthermore, it is immediately apparent from the proof of Corollary \ref{corollary:oceigsplit} that $\CO_w$ and $\OC_w$ are dual:

\begin{corollary}
\label{corollary:cowdual}
The maps
\begin{equation} \CO_w: QH^*(X)_w \To HH^*(\cF(X)_w)\end{equation}
and
\begin{equation} \OC_w^\vee: QH^*(X)_w^\vee \To HH_*(\cF(X)_w)^\vee\end{equation}
coincide, under the natural identification of their respective domains and targets.
\end{corollary}

The next two results will be crucial to the proof of Proposition \ref{proposition:smallw}.
I thank Cedric Membrez for drawing my attention to \cite[Proposition 2.4.A]{Biran2014}, of which they are a weaker version.

\begin{corollary}
\label{corollary:ocLe}
For any monotone Lagrangian $L$, $\OC^0(e_L)$ is a generalized eigenvector of $c_1 \star$ with eigenvalue $w(L)$.
\end{corollary}
\begin{proof}
This is immediate from Corollary \ref{corollary:oceigsplit}: we remark that one can easily show that $\OC^0(e_L)$ is in fact an eigenvector, but we will not need that. 
\end{proof}

\begin{lemma}
\label{lemma:ocLf}
We have
\begin{equation} 
\label{eqn:oc0el}
\OC^0(e_L) = PD(L) + \mbox{ lower-degree terms}.\end{equation}
\end{lemma}
\begin{proof}
This follows by deforming the perturbation data defining $\OC^0(\CO^0(e))$ to a $J_z$-holomorphic disc with boundary on $L$, and an internal marked point (compare the dual argument in \cite[\S 5a]{Seidel2008b}). 
The discs of Maslov index $0$ are constant, so the evaluation map at the internal marked point sweeps out a copy of $L$: this gives the term $PD(L)$ in \eqref{eqn:oc0el}. 
All other discs have Maslov index $> 0$ by monotonicity, hence contribute lower-degree terms.
\end{proof}

\begin{remark}
\label{remark:pdl}
One ought to be able to apply results similar to those of \cite{Oh2000,Lazzarini2011} to show that the terms of lower degree vanish. I.e., one should have
\begin{equation} \OC^0(e_L) = PD(L)\end{equation}
(see \cite{Biran2012}). 
In combination with Corollary \ref{corollary:ocLe}, this yields the useful result that $PD(L)$ is an eigenvector of $c_1 \star$ with eigenvalue $w(L)$ (compare \cite[Proposition 2.4.A]{Biran2014}).
We do not need this result in the present work.

Indeed, suppose that the $J_z$-holomorphic discs in the proof of Lemma \ref{lemma:ocLf} can be made regular with a \emph{domain-independent} almost-complex structure $J$.
Then the moduli space admits an $S^1$-action by rotation about the interior marked point, and therefore factors through a moduli space of lower dimension (unless the disc is constant), hence its contribution vanishes. 
However, we can only guarantee regularity for moduli spaces of $J$-holomorphic discs of Maslov index $\le$ the minimal Maslov number, using the theorem of Lazzarini \cite{Lazzarini2011} and Kwon-Oh \cite{Oh2000}: otherwise, we need to choose a domain-dependent almost-complex structure $J$, which destroys the $S^1$-action used to prove vanishing.
\end{remark}

\subsection{The split-generation criterion}
\label{subsec:splitgen}

In this section, we give a criterion for split-generating the Fukaya category, which is due to Abouzaid in the setting of the wrapped Fukaya category \cite{Abouzaid2010a}, and Abouzaid, Fukaya, Oh, Ohta and Ono in the general case \cite{Abouzaid2012}.
We follow \cite{Abouzaid2010a} closely.
We remark that the proof is particularly simple in the case of a closed monotone symplectic manifold, because one does not have to deal with weights on the strip-like ends (as in the wrapped Fukaya category), and it is easier to ensure transversality of our moduli spaces in the monotone setting.

For any object $K$ of $\cF$, we define a map of $\fmod{\cF}{\cF}$ bimodules
\begin{equation} \Delta: \cF_{\Delta} \To \mathcal{Y}^l_K \otimes \mathcal{Y}^r_K[n]\end{equation}
of degree $n$.
To define $\Delta$, we consider moduli spaces of holomorphic discs with $k+l+3$ boundary punctures, labelled $p_{in}, q_1, \ldots, q_k, p^l_{out}, p^r_{out}, q_{k+1}, \ldots, q_{k+l}$ as in Figure \ref{subfig:Fig5a}.
We make a consistent choice of perturbation data and consider the corresponding moduli spaces of pseudoholomorphic discs (the boundary component between $p^l_{out}$ and $p^r_{out}$ is labelled $K$).

\begin{figure}
\centering
\subfigure[The coproduct $\Delta:\cF_{\Delta} \To \mathcal{Y}^l_K \otimes \mathcal{Y}^r_K$.]{
\includegraphics[width=0.4\textwidth]{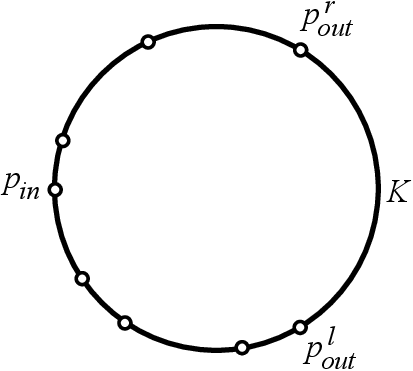}
\label{subfig:Fig5a}}
\subfigure[The moduli space of annuli.]{
\includegraphics[width=0.4\textwidth]{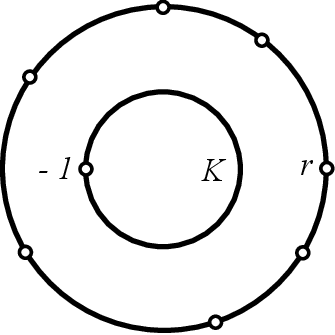}
\label{subfig:Fig5b}}
\subfigure[Boundary components at $r=1$.]{
\includegraphics[width=0.45\textwidth]{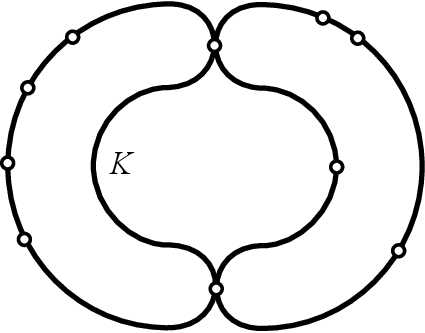}
\label{subfig:Fig5c}}
\subfigure[Points where $r=\infty$.]{
\includegraphics[width=0.45\textwidth]{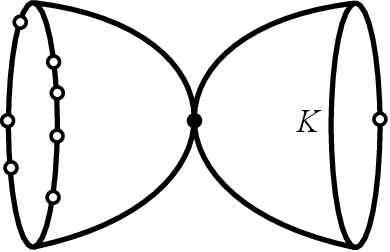}
\label{subfig:Fig5d}}
\caption{The proof of Lemma \ref{lemma:cardy}.
\label{fig:prods}}
\end{figure}

Counting the zero-dimensional component of the moduli space defines the map
\begin{equation} \Delta^{k|1|l}: \cF(K_k, \ldots, K_0) \otimes \cF(K_0,L_0) \otimes \cF(L_0,\ldots, L_l) \To \cF(K,L_l) \otimes \cF(K_k,K).\end{equation}
Counting the boundary points of the one-dimensional component of the moduli space shows that $\Delta$ is an $A_{\infty}$ bimodule homomorphism.

\begin{remark}
The component $\Delta^{0|1|0}$ coincides with the coproduct introduced in the proof of Lemma \ref{lemma:weakcy}.
\end{remark}

By functoriality of Hochschild homology (see \S \ref{subsec:hhi}), $\Delta$ defines a map
\begin{equation} HH_*(\Delta): HH_*(\cF,\cF_{\Delta}) \To HH_*(\cF,\mathcal{Y}^l_K \otimes \mathcal{Y}^r_K)[n].\end{equation}

\begin{lemma}
\label{lemma:cardy}
The following diagram commutes up to a sign $(-1)^{\frac{n(n+1)}{2}}$:
\begin{equation}
\xymatrixcolsep{5pc}\xymatrix{
HH_*(\cF)[-n] \ar[r]^-{\OC} \ar[d]_{HH_*(\Delta)} & QH^*(X) \ar[d]^{\CO^0} \\
HH_*(\cF,\mathcal{Y}^l_K \otimes \mathcal{Y}^r_K) \ar[r]^-{H^*(\mu)} & HF^*(K,K).}
\end{equation}
\end{lemma}
\begin{proof}
We consider a moduli space of holomorphic annuli which can be pa\-ram\-e\-trized as the region between the unit circle and the circle of radius $r \in [1,\infty)$ in $\C$, with $s+1 \ge 1$ incoming boundary punctures $p_0, \ldots, p_s$ on the outer boundary component, with $p_0$ sitting at $-r$, and one outgoing boundary puncture on the inner boundary component, sitting at $1$ (see Figure \ref{subfig:Fig5b}).
We make a consistent choice of perturbation data, and consider the corresponding moduli space of pseudoholomorphic annuli.
As $r \To \infty$, the annulus degenerates to two discs connected at a node (see Figure \ref{subfig:Fig5d})

We would like to say that these configurations correspond to the map $\CO^0 \circ \OC$, but a little care is required: recall (Remark \ref{remark:nopseudo}) that there is no natural choice of pseudocycle representing the image of $\OC$, because the moduli space of pseudoholomorphic discs can have codimension-$1$ boundary.
Instead, we choose dual bases $\{e_i\}$ and $\{e^i\}$ for $QH^*(X)$, as in Remark \ref{remark:dualbas}, and a bordism between $\sum_i e_i \times e^i$ and the diagonal $\Delta$ in $X \times X$.
Then counting configurations of two discs with internal marked points constrained to lie on this bordism defines a homotopy between
\begin{equation} \CO^0 \circ \OC(-) = \sum_i \langle \OC(-) , e_i \rangle \cdot \CO^0(e^i)\end{equation}
and the count of configurations of discs with the nodes connected.
So indeed, codimension-$1$ boundary components of the moduli space at $r = \infty$ contribute a term homotopic to $\CO^0 \circ \OC$.

Now, counting the zero-dimensional component of the moduli space defines a map
\begin{equation} H: CC_*(\cF) \To \cF(K,K).\end{equation}
Counting boundary points of the one-dimensional component of the moduli space shows that
\begin{equation} \CO^0 \circ \OC = (-1)^{\frac{n(n+1)}{2}} \mu \circ CC_*(\Delta) + H \circ \delta + (-1)^n \mu^1 \circ H.\end{equation}
By the preceding argument, the boundary points at $r = \infty$ contribute the left-hand side, up to a homotopy which can be incorporated into $H$ (see Figure \ref{subfig:Fig5d}).
The boundary points at $r=1$ contribute the first term of the right-hand side (see Figure \ref{subfig:Fig5c}).
The remaining terms on the right-hand side correspond to disc bubbling in the moduli space.
This equation shows that the diagram commutes (up to the sign) on the level of cohomology.
 \end{proof}

\begin{corollary}
\label{corollary:gen1}
(\cite{Abouzaid2012})
If $\cG_w \subset \cF(X)_w$ is a full subcategory, $K$ is another object of $\cF(X)_w$, and if the map
\begin{equation} \CO_w^0 \circ \OC_w: HH_*(\cG_w) \To HF^{*+n}(K,K)\end{equation}
contains the identity $e_K$ in its image, then $K$ is split-generated by $\cG_w$.
\end{corollary}
\begin{proof}
Follows immediately from Lemmas \ref{lemma:husplgen} and \ref{lemma:cardy}.
 \end{proof}

\begin{corollary}
\label{corollary:gen2}
(\cite{Abouzaid2012})
If $\cG_w \subset \cF(X)_w$ is a full subcategory, and if the map
\begin{equation} \OC_w: HH_*(\cG_w) \To QH^{*+n}(X)_w\end{equation}
contains the identity $e_w$ in its image, then $\cG_w$ split-generates $\cF(X)_w$.
\end{corollary}
\begin{proof}
Follows from Corollary \ref{corollary:gen1} and Proposition \ref{proposition:coeigsplit}.
 \end{proof}

\begin{corollary}
\label{corollary:splitgen}
(\cite{Abouzaid2012})
If $\cG_w \subset \cF(X)_w$ is a full subcategory, and if the map
\begin{equation}
\label{eqn:cospl}\CO_w: QH^*(X)_w \To HH^*(\cG_w)\end{equation}
is injective, then $\cG_w$ split-generates $\cF(X)_w$.
\end{corollary}
\begin{proof}
By Corollary \ref{corollary:gen2}, it suffices to prove that the map
\begin{equation} \OC_w: HH_*(\cG_w) \To QH^{*+n}(X)_w\end{equation}
is surjective.
This is equivalent to injectivity of the dual map, which is identified with \eqref{eqn:cospl} by Corollary \ref{corollary:cowdual}.
 \end{proof}

\begin{corollary}
\label{corollary:semisimpgen}
Suppose that $QH^*(X)_w$ is one-dimensional.
Then any object $L$ of $\cF(X)_w$ with $HF^*(L,L) \neq 0$ split-generates it.
\end{corollary}
\begin{proof}
The algebra homomorphism
\begin{equation}\CO_w^0: QH^*(X)_w \To HF^*(L,L)\end{equation}
is unital by Proposition \ref{proposition:coeigsplit}, hence $\CO_w^0(e_w) = e_L \neq 0$ by the hypothesis that $HF^*(L,L) \neq 0$.
Let $\cG_w$ be the full subcategory with object $L$; by definition, the map $\CO_w^0$ factors through
\begin{equation}
\label{eqn:cowin}
 \CO_w: QH^*(X)_w \To HH^*(\cG_w)
\end{equation}
via the projection to the length-zero component (see \eqref{eqn:hhproj}), hence $\CO_w(e_w) \neq 0$.

Because $QH^*(X)_w$ is one-dimensional by hypothesis, it is spanned by $e_w$, so the map \eqref{eqn:cowin} is injective.
It follows by Corollary \ref{corollary:splitgen} that $\cG_w$ split-generates $\cF(X)_w$.
 \end{proof}

\section{The relative Fukaya category}
\label{sec:relfuk}

We recall (from \cite[\S 5]{Sheridan2015}) the definition of $\cF(X,D)$, the Fukaya category of a monotone K\"{a}hler manifold $X$ relative to a smooth normal-crossings divisor $D$, each irreducible component of which is Poincar\'{e} dual to some multiple of the symplectic form.
In this section, we explain the relationship between the monotone Fukaya category and the relative Fukaya category.
We use the relationship to prove a result about the closed--open string map, analogous to the divisor axiom for Gromov--Witten invariants.

\subsection{Relating the monotone and relative Fukaya categories}

Let $X$ be monotone, and suppose that $D \subset X$ is a simple normal-crossings divisor that makes $(X,D)$ into a \emph{K\"{a}hler pair}, in the sense of \cite[\S 3.5]{Sheridan2015}. 
We recall that this means there exists a cohomology class $c \in H^2(X)$, so that each component $D_j$ of $D$ is Poincar\'{e} dual to $d_j c$ for some $d_j >0$, there exists a Liouville form $\alpha$ on $X \setminus D$ (i.e., a one-form such that $d \alpha = \omega$) which has `linking number' $\ell_j>0$ with component $D_j$ of $D$. 
This has the consequence that the cohomology class of $\omega$ is
\begin{equation}
\label{eqn:divmon}
 [\omega] = \sum_{j=1}^k d_j \ell_j c.\end{equation}
Recall that we associate a grading datum to the K\"{a}hler pair:  $\bm{G}(X,D) := \{\Z \to H_1(\cG(X \setminus D)) \}$ is the same as the grading datum associated to the non-compact symplectic manifold $X \setminus D$ (see \S \ref{subsec:monot}).

We recall that the coefficient ring $R$ of $\cF(X,D)$ is the completion of a polynomial ring
\begin{equation} \widetilde{R} := \C[ r_1, \ldots, r_k]\end{equation}
in the category of $\bm{G}(X,D)$-graded algebras.

\begin{lemma}
\label{lemma:moncoeff}
If $X$ is monotone, then the completion is unnecessary: $R \cong \widetilde{R}$.
\end{lemma}
\begin{proof}
Taking the completion in the category of graded algebras is equivalent to taking the completion separately in each graded piece, so it suffices to show that each graded piece of $\widetilde{R}$ is finite-dimensional.
In fact we will prove that each graded piece has dimension $\le 1$: i.e., no two monomials $r^{\bm{c}_1}, r^{\bm{c}_2}$ of $\widetilde{R}$ can have the same degree in $\bm{G}(X,D)$.

It is natural to regard the multi-indices $\bm{c}_i$ as living in $H_2(X,X \setminus D)$, via the isomorphism
\begin{equation}
\label{eqn:intnum}
 H_2(X, X \setminus D) \To \Z^k\end{equation}
given by taking intersection numbers with the divisors.
We recall (from \cite[Definition 5.1]{Sheridan2015}) that the grading of the generator $r_j$ corresponding to divisor $D_j$ is defined by choosing a disc
\begin{equation} u:(D^2,\partial D^2) \To (X,X \setminus D)\end{equation}
with intersection number $+1$ with $D_j$ and $0$ with all other divisors, then choosing a lift $\tilde{u}$ of $u$ to $\cG X$; the grading of $r_j$ is then $\tilde{u}|_{\partial D^2} \in H_1(\cG(X \setminus D))$.

Now consider the commutative diagram
\begin{equation}
\xymatrix{
H_2(\cG X) \ar[r] \ar[d] & H_2(\cG X,\cG(X \setminus D)) \ar[r]\ar[d] & H_1(\cG(X \setminus D)) \ar[d]\\
H_2(X) \ar[r] & H_2(X,X \setminus D) \ar[r] & H_1(X \setminus D),
}
\end{equation}
where both top and bottom rows are part of the long exact sequence for a pair in homology.
Suppose now that the degree of $\bm{c}_1 - \bm{c}_2 \in H_2(X,X \setminus D)$ is $0 \in H_1(\cG(X \setminus D))$.
By definition, that means that the class $\bm{c}_1 - \bm{c}_2 \in H_2(X,X \setminus D)$ lifts to a class $\tilde{\bm{c}} \in H_2(\cG X,\cG(X \setminus D))$, whose image under the boundary map to $H_1(\cG(X \setminus D))$ vanishes.
By exactness of the top row, that means $\tilde{\bm{c}}$ lies in the image of $H_2(\cG X)$, and hence that $\bm{c}_1 - \bm{c}_2$ lies in the image of the composition
\begin{equation} H_2(\cG X) \To H_2(X) \To H_2(X,X \setminus D).\end{equation}

Now we observe that the composition
\begin{equation}
\label{eqn:c1map}
H_2(\cG X) \To H_2(X) \xrightarrow{c_1} \C\end{equation}
vanishes: if $\pi: \cG X \To X$ denotes the projection, then the map \eqref{eqn:c1map} is given by $\pi^* c_1(TX) = c_1(\pi^*TX)$, but $\pi^*TX$ has an obvious canonical totally real (in fact, Lagrangian) subbundle, hence is the complexification of a real bundle, so its first Chern class vanishes.
By monotonicity, if we replace $c_1$ by the symplectic class $[\omega]$ in \eqref{eqn:c1map}, the composition also vanishes; hence, since the components of the divisor $D$ are Poincar\'{e} dual to a multiple of the symplectic form (by the definition of a K\"{a}hler pair), the intersection number of the class $\bm{c}_1 - \bm{c}_2$ with each component of the divisor $D$ vanishes.
But the intersection numbers with components of $D$ are precisely the components of the isomorphism \eqref{eqn:intnum}, so it follows that $\bm{c}_1 - \bm{c}_2 = 0$ and the proof is complete.
 \end{proof}

Now we examine the relationship between the gradings of the relative Fukaya category and the monotone Fukaya category.
The relative Fukaya category $\cF(X,D)$ is $\bm{G}(X,D)$-graded, where we recall that $\bm{G}(X,D)$ is the grading datum
\begin{equation} \Z \cong H_1(\cG_x(X \setminus D)) \To H_1(\cG(X \setminus D)).\end{equation}
The monotone Fukaya category is $\bm{G}(X)$-graded, where $\bm{G}(X)$ is the grading datum
\begin{equation} \Z \cong H_1(\cG_xX) \To H_1(\cG X).\end{equation}

\begin{definition}
\label{definition:qxd}
Let $(X,D)$ be a K\"{a}hler pair.
There is an obvious inclusion of fibrations
\begin{equation} \cG(X \setminus D) \hookrightarrow \cG X.\end{equation}
We denote the resulting morphism of grading data by
\begin{equation} \bm{q}_{X,D}: \bm{G}(X,D) \To \bm{G}(X).\end{equation}
We will write $\bm{q}$ when no confusion is possible.
\end{definition}

\begin{lemma}
\label{lemma:relgradcoeff}
If $R$ is the coefficient ring of the relative Fukaya category, then $\bm{q}_* R$ is concentrated in degree $0$.
\end{lemma}
\begin{proof}
As we saw in the proof of Lemma \ref{lemma:moncoeff}, the grading of class $r_j$ is a boundary in $\cG X$, by definition.
 \end{proof}

We would now like to relate the relative Fukaya category, $\cF(X,D)$, to the monotone Fukaya category $\cF(X)_w$.
First we relate the objects.

\begin{lemma}
\label{lemma:c1disc}
(See \cite[Lemma 3.1]{Auroux2007} and comments immediately after).
Let $(X,D)$ be a monotone K\"{a}hler pair, with $[\omega] = 2 \tau c_1$.
If $L \subset X \setminus D$ is an exact, anchored Lagrangian brane, then for any $[u] \in H_2(X,L)$,
\begin{align} \tau \mu(u) &=  \omega(u) \\
&=  \frac{\sum_i d_i \ell_i}{d_j} u \cdot D_j,
\end{align}
where $\mu$ denotes the Maslov class.
\end{lemma}
\begin{proof}
Recall that an anchored Lagrangian brane $L$ is equipped with a lift to the universal abelian cover of the Lagrangian Grassmannian $\cG(X \setminus D)$.
Hence, the boundary of $u$ admits a lift to the universal abelian cover: so there is a surface $u' \subset X \setminus D$ with the same boundary as $u$, which lifts to $\cG(X \setminus D)$.
Then the closed surface $\tilde{u} = u \cup u'$ in $X$ satisfies
\begin{align}
\label{eqn:c1muI}2 c_1(\tilde{u}) &= \mu(u), \\
\label{eqn:omegatild} \omega(\tilde{u}) &= \omega(u), \mbox{ and} \\
\label{eqn:tildu}\tilde{u} \cdot D_j &=u \cdot D_j.
\end{align}
\eqref{eqn:c1muI} follows because $u'$ lifts to the Lagrangian Grassmannian, hence has vanishing Maslov class.
\eqref{eqn:omegatild} follows by exactness of $L$ in $X \setminus D$, which implies that $\omega(u') = 0$ by Stokes' theorem.
\eqref{eqn:tildu} follows because $u'$ lies in the complement of $D$.
The result follows, because
\begin{equation} 2 \tau c_1(\tilde{u}) = \omega(\tilde{u}) =  \frac{\sum_i d_i \ell_i}{d_j} \tilde{u} \cdot D_j\end{equation}
by monotonicity and \eqref{eqn:divmon}.
 \end{proof}

It follows, in particular, that objects of $\cF(X,D)$ are monotone Lagrangians in $X$.
Let $\cF(X,D)_w$ denote the full subcategory of $\cF(X,D)$ whose objects are Lagrangians $L$ with $w(L) = w$.

\begin{corollary}
\label{corollary:tauobj}
If $w \neq 0$, then $\cF(X,D)_w$ has an empty set of objects unless $2\tau d_j$ is an integer multiple of $\sum_i d_i \ell_i$, for all $j$.
\end{corollary}
\begin{proof}
If $L$ is an object of $\cF(X,D)_w$ and $w \neq 0$, then $L$ must bound a Maslov index $2$ disc $u$.
It follows by Lemma \ref{lemma:c1disc} that
\begin{equation} 2 \tau d_j = \left(\sum_i d_i \ell_i \right) ( u \cdot D_j)\end{equation}
for all $j$.
 \end{proof}

\begin{corollary}
\label{corollary:objmap}
Assume that $X$ is simply-connected.
Then a choice of map $\tilde{i}$ that makes the diagram
\begin{equation} \xymatrix{
\widetilde{\cG}(X \setminus D) \ar@{^{(}->}[r]^-{\tilde{i}} \ar[d]& \widetilde{\cG}X \ar[d] \\
\cG(X \setminus D) \ar@{^{(}->}[r]^-{i} & \cG X
} \end{equation}
commute induces a natural map from objects of $\cF(X,D)_w$ to objects of $\cF(X)_w$.
\end{corollary}
\begin{proof}
The assumption that $X$ is simply-connected ensures that the image of $\pi_1(L)$ in $\pi_1(X)$ is always trivial; the condition on $\omega$ and $\mu$ then follows from Lemma \ref{lemma:c1disc}.
 \end{proof}

We would like to extend this to an $A_{\infty}$ functor.
We do this by first introducing a new category.

\begin{definition}
For each $w \in \C$, we define the \emph{monotone relative Fukaya category} $\cF_m(X,D)_w$.
It is a $\bm{G}(X,D)$-graded, $R$-linear $A_{\infty}$ category.
Its objects are the anchored Lagrangian branes $L \subset X \setminus D$ such that $w(L) = w$.
The morphism spaces and $A_{\infty}$ structure maps are defined as for $\cF(X)_w$, except that:
\begin{itemize}
\item All Floer and perturbation data are chosen so that the almost-complex structure makes each divisor $D_j$ into an almost-complex submanifold, and the Hamiltonian part vanishes with its first derivative along each divisor; then intersection numbers of pseudoholomorphic discs with divisors are non-negative.
\item Each pseudoholomorphic disc is counted with a coefficient $r^{u \cdot D} \in R$.
\end{itemize}
This category is $\bm{G}(X,D)$-graded, for the same reason that the relative Fukaya category is.
\end{definition}

\begin{remark}
It is still possible to achieve transversality with this restricted set of Floer and perturbation data: moduli spaces of discs and spheres transverse to the divisors are still generically regular, but moduli spaces of holomorphic spheres inside the divisor may not be.
However, the only place where we needed regularity of a moduli space of holomorphic spheres in the construction of the monotone Fukaya category was when we ruled out spheres bubbling off a pseudoholomorphic strip that is constant along its length, its image coinciding with a Hamiltonian chord between two Lagrangians.
The Hamiltonian chords lie in the complement of the divisors, so this type of sphere bubbling is still ruled out, without the need for regularity of moduli spaces of holomorphic spheres inside the divisors.
\end{remark}

\begin{lemma}
\label{lemma:reltomon}
The map on objects defined in Corollary \ref{corollary:objmap} extends to a strict full embedding of $\bm{G}(X)$-graded, $\C$-linear $A_{\infty}$ categories
\begin{equation} \bm{q}_* \cF_m(X,D)_w \otimes_R \C  \hookrightarrow \cF(X)_w,\end{equation}
where $\bm{q}$ is the morphism of grading data of Definition \ref{definition:qxd}, and we regard $\C$ as an $R$-algebra via the map $R \To \C$ sending each $r_j \mapsto 1$: this map is well-defined by Lemma \ref{lemma:moncoeff} and respects the $\bm{G}(X)$-grading by Lemma \ref{lemma:relgradcoeff}.
\end{lemma}
\begin{proof}
The embedding is tautologous: the $A_{\infty}$ structure maps on both sides count the same moduli spaces of pseudoholomorphic discs.
The coefficient $r^{u \cdot D}$ with which discs contribute to an $A_{\infty}$ structure map in $\cF_m(X,D)_w$ reduces to $1$ after tensoring with $\C$.
 \end{proof}

Now we would like to compare $\cF_m(X,D)_w$ with $\cF(X,D)_w$.
The fact that $\cF(X,D)$ may be curved makes it a bit complicated to prove the categories are quasi-equivalent, because we would first need to construct a non-curved model for $\cF(X,D)_w$.
Anyway this turns out to be unnecessary for our purposes.
Namely, it turns out to be sufficient to prove this result to first order, so we need only a first-order quasi-equivalence (in the sense of \cite[Definition 2.110]{Sheridan2015}) of  $R/\mathfrak{m}^2$-linear $A_\infty$ categories, where $\mathfrak{m} \subset R$ denotes the maximal ideal generated by $r_1, \ldots, r_n$.
For this purpose, we need only consider moduli spaces of pseudoholomorphic discs with $\le 1$ intersection points with the divisors.
As long as $D$ has $\ge 2$ irreducible components, the problem of curvature does not arise, because any non-constant pseudoholomorphic disc with boundary on a single Lagrangian $L$ must necessarily intersect all of the divisors $D_j$ (by Lemma \ref{lemma:c1disc}).

\begin{proposition}
\label{proposition:monrel}
Suppose that $D$ has $\ge 2$ irreducible components.
Then there is a first-order quasi-equivalence of $\bm{G}(X,D)$-graded, $R/\mathfrak{m}^2$-linear $A_{\infty}$ categories
\begin{equation}
\label{eqn:monreldoub} \cF_m(X,D)_w/\mathfrak{m}^2 \cong \cF(X,D)_w /\mathfrak{m}^2.\end{equation}
\end{proposition}
\begin{proof}
The two categories clearly have the same set of objects.
We use the usual trick of `doubling' the category \cite[\S 10a]{Seidel2008}: consider a category $\cF^{tot}$ which contains two copies of each object, and so that there are strict $\bm{G}(X,D)$-graded, $R/\mathfrak{m}^2$-linear embeddings
\begin{equation} \cF_m(X,D)_w/\mathfrak{m}^2 \hookrightarrow \cF^{tot} \hookleftarrow \cF(X,D)_w/\mathfrak{m}^2.\end{equation}
The $A_{\infty}$ structure maps on the left are defined by choosing perturbation data on moduli spaces of discs with $\le 1$ marked point which are pulled back via the forgetful map forgetting the marked point.
The $A_{\infty}$ structure maps on the right are defined by choosing perturbation data which are consistent with respect to the Deligne--Mumford compactification, including when the internal marked point bubbles off at the boundary.
We extend these to choices of perturbation data for all other moduli spaces of pseudoholomorphic discs with $\le 1$ marked points and boundary conditions on the objects of our category.
We require consistency of these perturbation data with respect to the Deligne--Mumford compactification, except when the marked point bubbles off in a disc on its own, with boundary labelled by a Lagrangian coming from the left-hand side of \eqref{eqn:monreldoub}.
This type of degeneration can not happen in the corresponding moduli spaces of pseudoholomorphic discs, because any pseudoholomorphic disc with boundary on a single Lagrangian has to intersect each divisor $D_j$ at least once (by Lemma \ref{lemma:c1disc}), and in particular has $\ge 2$ internal marked points.

We use the corresponding moduli spaces of pseudoholomorphic discs to define an $A_{\infty}$ structure on $\cF^{tot}$.
It follows as in \cite[\S 10a]{Seidel2008} that the order-zero components of these categories are quasi-equivalent, and by \cite[Lemma 2.111]{Sheridan2015} that the $R/\mathfrak{m}^2$-linear categories are first-order quasi-equivalent.
 \end{proof}

\subsection{The relative Fukaya category and the closed--open string map}
\label{subsec:relco}

If $(X,D)$ is a K\"{a}hler pair, then we can define a $\bm{G}(X,D)$-graded $R$-algebra $QH^*(X,D)$, by analogy with $QH^*(X)$.
The underlying $\bm{G}(X,D)$-graded $R$-module is the cohomology $H^*(X;\C) \otimes R$.
Each holomorphic sphere $u$ contributes to the quantum cup product with a coefficient $r^{u \cdot D} \in R$.

Furthermore, there is a $\bm{G}(X,D)$-graded homomorphism of $R$-algebras,
\begin{equation}\CO_{X,D}: QH^*(X,D) \To HH^*_{\bm{G}(X,D)}(\cF_m(X,D)_w),\end{equation}
for every $w \in \C$.
It is defined by counting the same pseudoholomorphic discs as are used to define $\CO$, but every pseudoholomorphic disc $u$ is counted with a coefficient $r^{u \cdot D} \in R$.
Thus, we have
\begin{equation}\CO_{X,D} \otimes_R \C = \CO.\end{equation}
$\CO_{X,D}$ is a $\bm{G}(X,D)$-graded map by standard index theory of Cauchy--Riemann operators.

\begin{proposition}
\label{proposition:cod}
Let $(X,D)$ be a K\"{a}hler pair.
Denote by $[D_j]$ the cohomology class Poincar\'{e} dual to $D_j$.
For a given $w \in \C$, we have
\begin{equation} \CO_{X,D}([D_j]) =  \left[ r_j \del{\mu^*}{r_j}\right] +r_j \del{(w \cdot T)}{r_j} \cdot e  \in HH^*_{\bm{G}}(\cF_m(X,D)_w)\end{equation}
where $e \in HH^*_{\bm{G}}(\cF_m(X,D)_w)$ is the unit of Hochschild cohomology, and
\begin{align}
T &:= r_1^{ad_1} \ldots r_k^{a d_k} \in R \mbox{, where }\\
a & :=  \frac{2\tau}{\sum_i d_i \ell_i}.
\end{align}
We remark that the exponents $2\tau d_j/\sum_i d_i \ell_i$ are not always integral, so $T$ does not always lie in $R$; however, when that happens, $\cF(X,D)_w$ has an empty set of objects unless $w=0$ by Corollary \ref{corollary:tauobj}, in which case the term involving $T$ vanishes.
\end{proposition}

\begin{remark}
The length-$0$ component of this equation was proven in the course of the proof of Lemma \ref{lemma:c1u0}.
\end{remark}

\begin{remark}
If we had instead followed \cite{fooo} and defined the Fukaya category as a curved $A_{\infty}$ category, then the term involving $e$ could be absorbed into the first term, where it would correspond to the curvature term $\mu^0$.
However we have chosen to use a different definition of the Fukaya category (taking advantage of the monotonicity of our manifolds), in which $\mu^0$ is set equal to zero.
\end{remark}

\begin{proof}
The left-hand side is given by the count of pseudoholomorphic discs with a single internal marked point constrained to lie on $D_j$ (taking $D_j$ as the pseudocycle representing its homology class).
On the other hand, the first term of the right-hand side is given by the count of pseudoholomorphic discs defining the $A_{\infty}$ structure map, multiplied by their intersection number with $D_j$.
For generic choice of perturbation data, all such pseudoholomorphic discs intersect $D_j$ transversely.
Then this term is given by the signed count of pseudoholomorphic discs together with a choice of internal marked point which lies on divisor $D_j$.

So we have two moduli spaces of pseudoholomorphic discs with an internal marked point: on the first, the perturbation data are chosen to be consistent with the Deligne--Mumford compactification, whereas on the second, the perturbation data are pulled back from the perturbation data used to define the $A_{\infty}$ structure maps, via the map forgetting the internal marked point.
We observe that the latter choice is \emph{not} consistent with the Deligne--Mumford compactification: consistency would require that, when the marked point approaches the boundary, a `thin' region modeled on a strip would develop separating the marked point from the rest of the disc, and the perturbation data along the thin region should coincide with the Floer data.
In the pulled-back perturbation data, when the marked point approaches the boundary, a holomorphic disc bubbles off on which the perturbation data has vanishing Hamiltonian part and constant almost-complex structure part $J = J_L$.

To compare the two choices, we consider the same moduli space of holomorphic discs used in the proof that $\CO$ is a homomorphism of $\C$-algebras (see \S \ref{subsec:co}, in particular Figure \ref{fig:2l1discs}).
However, we define a new choice of perturbation data on this moduli space: the perturbation data are pulled back from the moduli space used to define $\CO$, via the map which forgets the left-hand marked point.

Now we consider the moduli space of pseudoholomorphic discs with these perturbation data, where the left-hand marked point is constrained to lie on $D_j$ and the right-hand marked point is unconstrained.
For generic choice of perturbation data defining $\CO$, this moduli space is regular.
Counting the boundary points of the one-dimensional component shows that
\begin{equation}
\label{eqn:coxd}
\CO_{X,D}([D_j]) = \CO_{X,D}(e) \cup \left(r_j \del{\mu^*}{r_j} \right) + r_j \del{(w \cdot T)}{r_j} \cdot \CO_{X,D}(e) + \delta(H),
\end{equation}
where $H$ is defined by counting the zero-dimensional components of the moduli space.

\begin{figure}
\centering
\subfigure[The first term on the right-hand side of \eqref{eqn:coxd}. The perturbation data for the disc on the left are pulled back from the perturbation data for $\mu^s$, via the map which forgets the internal marked point. In particular they are not equal to the perturbation data for $\CO_{X,D}({[}D_j{]})$.]{
\includegraphics[width=0.8\textwidth]{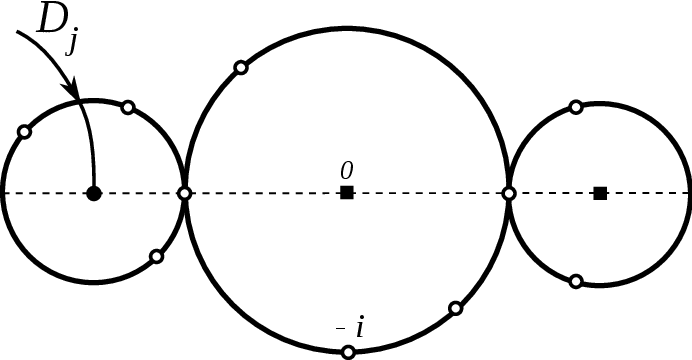}
\label{subfig:Fig6a}}
\subfigure[The second term on the right-hand side of \eqref{eqn:coxd}. The disc on the left is $J_L$-holomorphic, where $J_L$ is the fixed almost-complex structure associated to the Lagrangian $L$. In particular it is attached at a boundary node, rather than along a strip-like end as in the previous picture.]{
\includegraphics[width=0.8\textwidth]{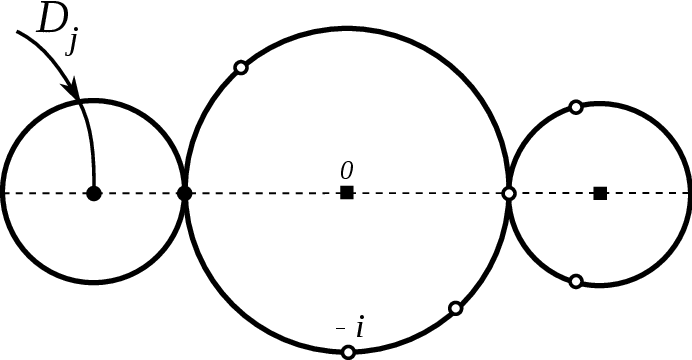}
\label{subfig:Fig6b}}
\caption{The first and second terms on the right-hand side of \eqref{eqn:coxd}.
\label{fig:Fig6}}
\end{figure}

The left-hand side of \eqref{eqn:coxd}  corresponds to the boundary component at $t=0$: compare Figure \ref{subfig:Fig2b} (the holomorphic sphere that bubbles off is necessarily constant, so it is equivalent to constraining the centre marked point to lie on $D_j$).

The first term on the right-hand side of \eqref{eqn:coxd}  corresponds to the boundary component at $t=1$, when the disc that bubbles off on the left has $\ge 1$ incoming marked boundary points (see Figure \ref{subfig:Fig6a}).
By the argument given at the start of the proof, the count of such discs contributes exactly $r_j \partial \mu^*/\partial r_j$.
The right-hand disc corresponds to $\CO_{X,D}(e)$, and the middle disc gives the Yoneda product.

The second term on the right-hand side of \eqref{eqn:coxd} corresponds to the boundary component at $t=1$, when the disc that bubbles off on the left has no incoming marked boundary points (see Figure \ref{subfig:Fig6b}).
Because the perturbation data on this disc are pulled back via the forgetful map, the disc must have constant perturbation data given by the almost-complex structure $J = J_L$. The count of such discs is $w$ by definition.
We must also choose an internal marked point lying on $D_j$; the disc has  intersection number $a d_j$ with divisor $D_j$ by Lemma \ref{lemma:c1disc}, and hence there are $a d_j$ possible choices for the internal marked point lying on divisor $D_j$.
Each such disc contributes with a coefficient $r^{u \cdot D} = T$, so the total count of these discs is $r_j \partial (w \cdot T)/\partial r_j$.

The central disc in such a configuration has a boundary marked point where it meets the $J_L$-holomorphic disc; the perturbation data are independent of the position of this boundary marked point, so it can be varied freely.
It follows that the moduli space factors through a moduli space of lower (hence negative) dimension, and is therefore empty, unless the central disc is actually a constant strip meeting the $J_L$-holomorphic disc on one boundary component (i.e., there are no open dots on the central disc in Figure \ref{subfig:Fig6b} except for those at $-i$ and $+1$, and the strip is constant along its length).
The disc on the right is precisely an element of the moduli space used to define $\CO_{X,D}(e)$.
It follows that this boundary component contributes the second term on the right-hand side.

Finally, disc bubbling for $t \in (0,1)$ contributes the final term of \eqref{eqn:coxd} (compare Figure \ref{subfig:Fig2a}).
The result now follows from \eqref{eqn:coxd}, as $\CO_{X,D}$ is a unital algebra homomorphism by Lemma \ref{lemma:counital}.
 \end{proof}

\section{Weak bounding cochains: algebra}
\label{sec:wbcalg}

\subsection{Weak bounding cochains}
\label{subsec:wbc}

Let $\cA$ be a \emph{curved, strictly unital}, $\bm{G}$-graded $A_\infty$ category, defined over some coefficient ring $R$.
In \cite{fooo} a procedure for formally enlarging $\cA$ is described, by introducing weak bounding cochains.
We will denote the result by $\cA^{wbc}$.
The construction of $\cA^{wbc}$ is formally analogous to the construction of the category of twisted complexes of an $A_\infty$ category \cite[\S 3l]{Seidel2008}.

Let $\cA^{\oplus}$ denote the additive enlargement of $\cA$ by introducing arbitrary finite direct sums of objects (this was denoted $\Sigma \cA$ in \cite[\S 3k]{Seidel2008}).
$\cA^\oplus$ is also a curved, strictly unital $A_\infty$ category.

An object of $\cA^{wbc}$ is a pair $(\bm{L},\alpha)$, where $\bm{L}$ is an object of $\cA^{\oplus}$, and
\begin{equation} \alpha \in hom^*_{\cA^{\oplus}}(\bm{L},\bm{L})\end{equation}
is an element of odd degree (not necessarily degree $1$) which is a \emph{weak bounding cochain}, i.e., a solution of the \emph{Maurer--Cartan equation:}
\begin{equation}
\mu^0_{\bm{L}} + \mu^1(\alpha) + \mu^2(\alpha,\alpha) + \ldots = \mathfrak{P}(\alpha) \cdot e_{\bm{L}} \label{eqn:mc}
\end{equation}
for some $\mathfrak{P}(\alpha) \in R$ (the `disc potential').

We remark that the Maurer--Cartan equation \eqref{eqn:mc}  does not make sense as written, because it is an infinite sum which has no reason to converge.
There are various ways of dealing with this, which we have summarized in Remark \ref{remark:conv}.
In the next section we will show that, for a specific class of weak bounding cochains on the monotone Fukaya category (which we call \emph{monotone} weak bounding cochains), the Maurer--Cartan equation \eqref{eqn:mc} has only finitely many terms for degree reasons.

Let us continue recalling the construction of $\cA^{wbc}$, with the understanding that certain restrictions have been imposed on our weak bounding cochains $\alpha$ to ensure convergence of \eqref{eqn:mc}, but without specifying their nature.

The hom-spaces of $\cA^{wbc}$ coincide with those in $\cA^{\oplus}$:
\begin{equation} hom^*_{\cA^{wbc}}((\bm{L}_0,\alpha_0),(\bm{L}_1,\alpha_1)) := hom^*_{\cA^{\oplus}}(\bm{L}_0,\bm{L}_1).\end{equation}
The $A_\infty$ structure maps of $\cA^{wbc}$ are defined by inserting the weak bounding cochains in all possible ways into the $A_\infty$ structure maps of $\cA^{\oplus}$: if
\begin{equation} a_i \in hom^*((\bm{L}_{i-1},\alpha_{i-1}),(\bm{L}_i,\alpha_i)) \mbox{  for $i=1,\ldots,s$},\end{equation}
then we define

\begin{multline} \mu^s_{\cA^{wbc}}(a_s,\ldots,a_1) :=  \\
\sum_{i_0,\ldots,i_s} \mu^*_{\cA^{\oplus}}(\underbrace{\alpha_s,\ldots}_{i_s},a_s, \underbrace{\alpha_{s-1},\ldots}_{i_{s-1}}, a_{s-1}, \ldots,a_1,\underbrace{\alpha_0,\ldots}_{i_0}). \label{eqn:defainf}
\end{multline}
Once again, convergence of \eqref{eqn:defainf}  is an issue to be dealt with, which we set aside for the purposes of this section.

These structure maps satisfy the $A_\infty$ associativity relations, making $\cA^{wbc}$ a curved, strictly unital $A_\infty$ category, where $(\bm{L},\alpha)$ has curvature
\begin{equation} \mu^0 = \mathfrak{P}(\alpha) \cdot e_{\bm{L}}.\end{equation}

It follows that for all $w \in R$, there is a non-curved, strictly unital $A_\infty$ category
\begin{equation} \cA^{wbc}_w \subset \cA^{wbc},\end{equation}
whose objects are those $(\bm{L},\alpha)$ with
\begin{equation} \mathfrak{P}(\alpha) = w\end{equation}
(compare \S \ref{subsec:curved}).

\begin{remark}
\label{remark:wbcgrad}
If $\cA$ is $\bm{G}$-graded, and all weak bounding cochains are chosen to have degree $1$, then $\cA^{wbc}$ is also $\bm{G}$-graded.
On the other hand, if the weak bounding cochains have degree $1+y$, for $y$ in some subgroup $U$ of the grading datum $\bm{G}$, then $\cA^{wbc}$ need only be $\bm{G}/U$-graded.
\end{remark}

\subsection{Hochschild invariants}
\label{subsec:hhwbc}

See Appendix \ref{sec:ainf} for our conventions on Hochschild homology and cohomology of $A_\infty$ categories.
Because $\cA$ sits inside $\cA^{wbc}$ as a full subcategory (the subcategory whose objects are those with weak bounding cochains equal to $0$), there are obvious morphisms of chain complexes
\begin{equation} \phi: CC_*(\cA) \hookrightarrow CC_*(\cA^{wbc})\end{equation}
(given by inclusion) and
\begin{equation} \psi: CC^*(\cA^{wbc}) \To CC^*(\cA)\end{equation}
(given by restriction).

There are also morphisms of chain complexes in the respective opposite directions, but once again they are only defined if there is a reason for them to converge.
The formula for Hochschild cohomology is:
\begin{equation}
\label{eqn:Psi} \Psi: CC^*(\cA)  \To  CC^*(\cA^{wbc})
\end{equation}
\begin{multline}
\Psi(\eta)(a_s, \ldots , a_1) := \\
\sum_{i_0,\ldots,i_s} \eta(\underbrace{\alpha_s,\ldots,\alpha_s}_{i_s},a_s, \underbrace{\alpha_{s-1},\ldots,\alpha_{s-1}}_{i_{s-1}}, a_{s-1}, \ldots,a_1,\underbrace{\alpha_0,\ldots,\alpha_0}_{i_0}),
\end{multline}
which  is easily verified to be a chain map, and to respect the Yoneda product \eqref{eqn:yon1pt} on the cochain level.
It obviously satisfies $\psi \circ \Psi = \mathrm{Id}$.

The formula for Hochschild homology is:
\begin{equation}
\Phi: CC_*(\cA^{wbc}) \To  CC_*(\cA)
\end{equation}
\begin{multline}
\Phi(a_s \otimes \ldots \otimes a_0) := \\
 \sum_{i_0,\ldots,i_s} a_s \otimes  \underbrace{\alpha_s \otimes \ldots \otimes \alpha_s}_{i_s} \otimes  a_{s-1} \otimes  \ldots \otimes a_0 \otimes \underbrace{\alpha_0 \otimes \ldots \otimes \alpha_0}_{i_0},
\end{multline}
which is easily verified to be a chain map, and to respect the cap product \eqref{eqn:cap1pt} on the cochain level.
It obviously satisfies $\Phi \circ \phi = \mathrm{Id}$.

When $\alpha$ is lower-triangular, as in the definition of twisted complexes, convergence holds, and these morphisms can be used to show the Morita invariance of Hochschild homology and cohomology (taking twisted complexes does not change the Hochschild invariants).

When $\alpha$ is required to have positive energy, a modification in the definition of the Hochschild homology and cohomology chain complexes is required to make convergence hold: namely, one has to take `completed' tensor products in the definition of Hochschild chains, and `filtered' Hochschild cochains (see \cite[\S 3.8]{fooo}).

In our intended application (which is to define the closed--open and open--closed string maps in the presence of weak bounding cochains, see \S \ref{subsec:cowbc}), we will not actually prove that these maps are well-defined: we will cook up the closed--open and open--closed maps using the formulae for $\Phi$ and $\Psi$ in a formal way, then show that the resulting maps are well-defined, i.e., that convergence holds for geometric reasons.

\subsection{Homotopy units}
\label{subsec:homun}

We recall that there are several notions of unitality for an $A_\infty$ category (see, e.g., \cite[\S 2a]{Seidel2008}).
We will need to work with homotopy unital $A_\infty$ categories, so we recall the definition here for reference, and to clarify conventions.

Let $\cA$ be a cohomologically unital, $R$-linear $A_\infty$ category, with cochain-level (cohomological) units
\begin{equation} e_L \in \cA(L,L).\end{equation}
A \emph{homotopy unit structure} (see \cite[\S 3.3]{fooo}) on $\cA$ is an $A_\infty$ structure on $\cA^+$, where
\begin{align}
\cA^+(L,L) &:= \cA(L,L) \oplus R \cdot f_L[1] \oplus R \cdot e^+_L,\\
\cA^+(K,L) &:= \cA(K,L) \mbox{ for $K \neq L$}
\end{align}
so that the $A_\infty$ structure coincides with the $A_\infty$ structure on $\cA$, and furthermore satisfies:
\begin{align}
\mu^1(f_L) &= e^+_L - e_L;\\
\mu^1(e^+_L) &= 0; \\
(-1)^{\sigma(a)}\mu^2(e^+_L,a) &= \mu^2(a,e^+_L) = a;\\
\mu^s(\ldots,e^+_L,\ldots) &= 0 \mbox{ for $s \ge 3$.}
\end{align}

In particular, $\cA^+$ is a strictly unital $A_\infty$ category, with cochain level units $e^+_L$.

If $\cA$ is equipped with homotopy units, then we will denote
\begin{equation} \cA^{wbc} := \left(\cA^+\right)^{wbc}\end{equation}
for brevity.

\subsection{The disc potential}
\label{subsec:dp}

Following \cite[\S 3.6.3]{fooo}, we make the following:

\begin{definition}
\label{definition:mweak}
Suppose that $\bm{L}$ is an object of $\cA^\oplus$.
We define $\hcM_{weak}(\bm{L})$ to be the set of solutions $\alpha$ to \eqref{eqn:mc}.
\end{definition}

\begin{remark}
There is a natural equivalence relation on $\hcM_{weak}(\bm{L})$, called `gauge equivalence'.
Gauge equivalent weak bounding cochains give rise to quasi-isomorphic objects of the Fukaya category.
The quotient of $\hcM_{weak}(\bm{L})$ by gauge equivalence is called the `Maurer--Cartan moduli space' or `Maurer--Cartan scheme' $\cM_{weak}(\bm{L})$ \cite[\S 4.3.1]{fooo}.
In this paper we will be concerned with a certain subspace of a certain $\hcM_{weak}(\bm{L})$, on which the gauge equivalence relation turns out to be trivial.
For that reason, and because the technical machinery is a bit involved, we will bypass the notion of gauge equivalence and work directly with $\hcM_{weak}(\bm{L})$.
\end{remark}

The `disc potential' defines a function
\begin{equation} \mathfrak{P}: \hcM_{weak}(\bm{L}) \To R\end{equation}
(compare \eqref{eqn:mc}).

\begin{remark}
\label{remark:dpnat}
It is natural with respect to change of coefficients, in the following sense: if
\begin{equation} \Phi: R \To S\end{equation}
is an algebra homomorphism, and $\cA$ an $R$-linear $A_\infty$ category, we can form the $S$-linear $A_\infty$ category $\cA \otimes_R S$.
Then there is a commutative diagram
\begin{equation}
\xymatrixcolsep{5pc} \xymatrix{
\hcM_{weak}(\bm{L}; \cA) \ar[r]^-{v \mapsto v \otimes_R 1} \ar[d]^{\mathfrak{P}}& \hcM_{weak}(\bm{L}; \cA\otimes_R S) \ar[d]^{\mathfrak{P}} \\
R \ar[r]^-{\Phi} & S.
} \end{equation}
\end{remark}

In \cite[Proposition 4.3]{Fukaya2010d}, it is proven that for any Lagrangian torus fibre $L$ of a symplectic toric manifold, there is an embedding
\begin{equation} H^1(L;\Lambda_+) \hookrightarrow \hcM_{weak}(L).\end{equation}
We will now prove an analogue of that result.
For the purposes of the following results, the reader should have in mind the case that $A = CF^*(L,L)$ is the exterior algebra on the vector space $V$, e.g., $L$ is a Lagrangian torus and $V = H^1(L) \subset A$.

\begin{lemma}
\label{lemma:prediscdisc}
Let $\cA$ be a cohomologically unital $A_\infty$ category over $R$, with cochain-level cohomological units $e_L \in hom^0_\cA(L,L)$.
Let $L$ be an object of $\cA$, and denote
\begin{equation} A := hom^*_\cA(L,L).\end{equation}
Suppose that
\begin{equation} V \subset A\end{equation}
is a subspace, concentrated in odd degree, such that
\begin{enumerate}
\item For any $v \in V$, $\mu^s(v,v,\ldots,v)$ is a multiple of $e_L$;
\item The sum
\begin{equation} \sum_{s=1}^\infty \mu^s(v,\ldots,v) =: \mathfrak{P}'(v) \cdot e_{L}\end{equation}
converges, for all $v \in V$  (c.f. Remark \ref{remark:conv} ; observe that each term in this sum is a multiple of $e_L$, by the preceding assumption);
\item \label{itm:3} $CC^{\le 0}(V,A) \cong \C \cdot e_L$ (as a $\C$-vector space).
\end{enumerate}
We call the resulting map
\begin{equation} \mathfrak{P}': V \To R\end{equation}
the \emph{pre-disc potential}.
Now suppose $\cA$ is equipped with homotopy units, so that $\cA^+$ is a strictly unital $A_\infty$ category with cochain-level strict units $e_L^+$, and define
\begin{equation} \hcM_{weak}(L),\end{equation}
the space of weak bounding cochains for L, computed in $\cA^+$.
Then there is an embedding
\begin{equation} \iota: V \hookrightarrow \hcM_{weak}(L),\end{equation}
which makes the following diagram commute:
\begin{equation} \xymatrix{
V \ar@{^{(}->}[rr]^-{\iota} \ar[rd]_{\mathfrak{P}'} & & \hcM_{weak}(L) \ar[ld]^{\mathfrak{P}}  \\
& R &}
\end{equation}
\end{lemma}
\begin{proof}
The embedding  $\iota$ is defined by
\begin{equation} \iota(v) := v + \mathfrak{P}'(v) f_L.\end{equation}
To prove that $\iota(v)$ does lie in $\hcM_{weak}(L)$, and does make the diagram commute, we first observe that whenever
\begin{equation} s = \sum_{j=0}^k i_j,\end{equation}
the expression
\begin{equation}
\mu^*(\underbrace{-, \ldots,-}_{i_k},f,\underbrace{-, \ldots,-}_{i_{k-1}},f,\ldots,f,\underbrace{-,\ldots,-}_{i_0}): V^{\otimes s} \To A^+ \label{eqn:muccf}
\end{equation}
defines an element of $CC^{2-2k}(V,A^+)$.
By hypothesis \eqref{itm:3}, $CC^{2-2k}(V,A)$ vanishes for $k \ge 2$, and is generated by the single element $e_L$ if $k=1$.
It follows that
\begin{itemize}
\item If $k \ge 2$, then the expression \eqref{eqn:muccf}  vanishes;
\item If $k=1$, it vanishes except for the map
\begin{equation} \mu^1(f_L) =  e_L^+ - e_L;\end{equation}
\item If $k = 0$, it coincides with $\mu^s$.
\end{itemize}
Hence,
\begin{align}
\mathfrak{P}(\iota(v)) &= \sum_{s=1}^\infty \mu^s\left(v + \mathfrak{P}'(v) \cdot f_L, \ldots, v + \mathfrak{P}'(v) \cdot f_L\right)\\
 &=  \sum_{s=1}^\infty \mu^s(v,\ldots,v) + \mathfrak{P}'(v) \mu^1(f_L) \\
&= \mathfrak{P}'(v) e_L + \mathfrak{P}'(v) (e_L^+-e_L) \\
&= \mathfrak{P}'(v) e_L^+,
\end{align}
from which the result follows.
 \end{proof}

\begin{remark}
In \cite[\S 3.6.3]{fooo}, the authors explicitly caution that one should work with $\mathfrak{P}$, rather than with $\mathfrak{P}'$.
However, Lemma \ref{lemma:prediscdisc} shows that, under certain hypotheses, the two can be related.
\end{remark}

We now prove an analogue, in the setting of Lemma \ref{lemma:prediscdisc}, of the well-known theorem that critical points of $\mathfrak{P}$ correspond to objects with vanishing differential, and therefore non-trivial cohomology (compare \cite{Cho2006}).
Suppose that $V$ is a free $R$-module with basis $\{\theta_1,\ldots,\theta_k\}$.
So if
\begin{equation} v = \sum_{j=1}^k v_j \theta_j \in V,\end{equation}
then
\begin{equation} \mathfrak{P} \circ \iota(v) = \mathfrak{P}'(v) = \sum_{s=1}^\infty \mu^s(v_1\theta_1 + \ldots + v_k \theta_k,\ldots,v_1\theta_1 + \ldots + v_k \theta_k).\end{equation}
We observe that $\mathfrak{P} \circ \iota$ is a power series in the $v_i$ (in fact, in our application, it will be a polynomial).

\begin{proposition}
\label{proposition:critpd}
Consider the setup of Lemma \ref{lemma:prediscdisc}.
Suppose furthermore that there is a decomposition as free $R$-modules
\begin{equation} A \cong \bigoplus_{k=0}^N A_k,\end{equation}
and that
\begin{equation} \mu^* = \mu^*_0 + \mu^*_1,\end{equation}
such that:
\begin{enumerate}
\item $A_0 = R \cdot e_L $;
\item $A_1 = V$;
\item $\mu^*_0$ has length $2$ and sends
\begin{equation} \mu^2_0 : A_k \otimes A_l \To A_{k+l}; \end{equation}
\item $\mu^*_1$ sends
\begin{equation} \mu^*_1: V^{\otimes b} \otimes A_k \otimes V^{\otimes c} \otimes A_l \otimes V^{\otimes d} \To \bigoplus_{m < k+l} A_m.\end{equation}
\item \label{itm:gens} It follows that $\mu^2_0$ defines an associative algebra structure on $A$; we suppose that $V$ generates $A$ as an associative algebra, with respect to $\mu^2_0$, and that $e_L$ is an identity element for this algebra.
\end{enumerate}
Then, if $v \in V$ is a critical point of $\mathfrak{P}'$, and $\alpha = \iota(v)$ is the corresponding Maurer--Cartan element, then the differential $\mu^1_\alpha$ vanishes on the endomorphism algebra of $(L,\alpha)$ in $\cA^{wbc}$, with the exception that $\mu^1_\alpha(f_L) = e^+_L - e_L$.
In particular, the cohomology is
\begin{equation} \mathrm{Hom}^*((L,\alpha),(L,\alpha)) \cong A,\end{equation}
as an $R$-module.
\end{proposition}
\begin{proof}
The proof is a modification of that of \cite[Lemma 13.1]{Fukaya2010d}.
We extend the decomposition of $A$ to be defined on
\begin{equation} A^+ = A \oplus R \cdot f_L \oplus R \cdot e_L^+,\end{equation}
by putting $f_L$ in $A^+_{-1}$ and $e^+_L$ in $A^+_0$.
It follows as in the proof of Lemma \ref{lemma:prediscdisc} that
\begin{equation} \mu^1_\alpha(f_L) = e_L^+ - e_L.\end{equation}
Because $\mu^1_\alpha(\mu^1_\alpha(f_L)) = 0$, it follows that
\begin{equation} \mu^1_\alpha(e_L) = \mu^1_\alpha(e_L^+) = 0,\end{equation}
by strict unitality.
This proves that $\mu^1_\alpha$ has the stated form when applied to $f_L,e^+_L, e_L$.

Next, we observe that, for any generator $\theta_i$ of $V$, we have
\begin{align}
\label{eqn:mu11}\mu^1_\alpha(\theta_i) &:= \sum_{i_0,i_1} \mu^*(\underbrace{\alpha,\ldots,\alpha}_{i_0}, \theta_i, \underbrace{\alpha,\ldots,\alpha}_{i_1}) \\
\label{eqn:mu12}&= \sum_{i_0,i_1} \mu^*(\underbrace{v,\ldots,v}_{i_0}, \theta_i, \underbrace{v,\ldots,v}_{i_1})  \\
&= \del{}{v_i} \mathfrak{P}'(v) \\
&= 0,
\end{align}
by assumption (in the passage from \eqref{eqn:mu11} to \eqref{eqn:mu12}, we must check that other terms vanish, c.f. proof of Lemma \ref{lemma:prediscdisc}).
Hence, $\mu^1_\alpha(V) = 0$.

We now prove that $\mu^1_\alpha$ vanishes on $A^+_k$ for all $k \ge 0$, by induction on $k$.
Above, we have proven the result for all $k \le 1$.
Now suppose that we have proven the result for all $k \le l-1$.
Let $x$ be a generator of $A_l$.
By assumption \eqref{itm:gens}, we can choose $y \in V$ and $z \in  A_{l-1}$, so that
\begin{equation} \mu^2_0(y,z) = x.\end{equation}
We then apply the second $A_\infty$ relation for $\mu^*_\alpha$:
\begin{align}
\mu^1_\alpha(\mu^2_\alpha(y,z)) &=  -\mu^2_\alpha(y,\mu^1_\alpha(z)) + (-1)^{\sigma(z)} \mu^2_\alpha(\mu^1_\alpha(y),z) \\
&= 0,
\end{align}
by the inductive assumption.
We now have
\begin{equation} \mu^2_\alpha(y,z) = \mu^2_0(y,z) + \sum_{s} \mu^s_1(\alpha,\ldots,\alpha,y, \alpha, \ldots,\alpha,z,\alpha,\ldots,\alpha)\end{equation}
By hypothesis, the second term is a sum of elements of $A_k$ for $k<l$.
Hence, $\mu^1_\alpha$ vanishes on it, by the inductive assumption.
It follows that
\begin{equation} \mu^1_\alpha(x) = \mu^1_\alpha(\mu^2_0(y,z)) = 0,\end{equation}
which completes the inductive step.
 \end{proof}

\begin{remark}
It may be that the function $\mathfrak{P}$ has no critical points over $R$, but that if we change coefficients from $R$ to $S$, then the function $\mathfrak{P} \otimes_R S$ does have critical points.
Proposition \ref{proposition:critpd} applies equally in that case.
Namely, to each $v \in V \otimes_R S$, there corresponds a weak bounding cochain $\alpha$ in the category $\cA \otimes_R S$ (compare Remark \ref{remark:dpnat}).
Moreover, if $v$ is a critical point of $\mathfrak{P} \otimes_R S$, then the differential $\mu^1_\alpha$ on the endomorphism algebra of $(L,\alpha)$ in $(\cA \otimes_R S)^{wbc}$ vanishes, except for $\mu^1_\alpha(f_L) = e^+_L - e_L$: the argument goes through unchanged.
\end{remark}

The following is an analogue of Cho's result \cite[Theorem 5.6]{Cho2005} (see also \cite{Hori2001,Kapustin2004}).

\begin{proposition}
\label{proposition:hess}
In the setting of Proposition \ref{proposition:critpd}, suppose that $v \in V$ is a critical point of $\mathfrak{P}'$, and $\alpha = \iota(v)$ is the corresponding Maurer--Cartan element, so that $(L,\alpha)$ is a non-zero object of $\cA^{wbc}$.
Then there is a surjective map of $R$-algebras
\begin{equation} \Cl \left(-Hess_v(\mathfrak{P}') \right) \twoheadrightarrow \mathrm{Hom}^*((L,\alpha),(L,\alpha)) \end{equation}
sending $V$ to $V$.
Here, the left-hand side denotes the Clifford algebra on the free $R$-module $V$, with the quadratic form $-Hess_v(\mathfrak{P}')$, which has matrix
\begin{equation} -(Hess_v(\mathfrak{P}'))_{ij} := -\del{^2\mathfrak{P}'}{v_i \partial v_j} (v)\end{equation}
with respect to the basis $\{\theta_1,\ldots,\theta_k\}$ (see \S \ref{subsec:cliff} for reminders on Clifford algebras).
\end{proposition}
\begin{proof}
By Proposition \ref{proposition:critpd}, we have
\begin{equation}\mathrm{Hom}^*((L,\alpha),(L,\alpha)) \cong A\end{equation}
as an $R$-module; it remains to identify the algebra structure $\mu^2_\alpha$.
By the proof of Lemma \ref{lemma:prediscdisc}, we have
\begin{align}
\mu^2_\alpha(\theta_i,\theta_j) + \mu^2_\alpha(\theta_j,\theta_i) &= \sum \mu^*(\alpha,\ldots,\theta_i,\alpha,\ldots,\theta_j, \alpha,\ldots) \\
&\phantom{{}=\sum} + \sum \mu^*(\alpha,\ldots,\theta_j,\alpha,\ldots,\theta_i, \alpha,\ldots) \\
&= \sum \mu^*(v,\ldots,\theta_i,v,\ldots,\theta_j, v,\ldots) \\
&\phantom{{}=\sum}  + \sum \mu^*(v,\ldots,\theta_j,v,\ldots,\theta_i, v,\ldots) \\
&= \del{^2}{v_i \partial v_j} \sum \mu^*(v,\ldots,v) \\
&= \del{^2\mathfrak{P}'}{v_i \partial v_j} (v) \cdot e_L.
\end{align}
It follows that, for any $u \in V$,
\begin{equation} \mu^2_\alpha(u,u) = Hess_v(\mathfrak{P}')(u,u) \cdot e_L.\end{equation}
Hence, in the cohomology category we have (see  \eqref{eqn:assoc})
\begin{equation} u \cdot u = -Hess_v(\mathfrak{P}')(u,u) \cdot e_L,\end{equation}
as all $u \in V$ are odd by definition of a weak bounding cochain.
Now $V$ generates $A$ as an algebra with respect to the product $\mu^2_0$, by assumption (c.f. Proposition \ref{proposition:critpd}); because $\mu^2_0$ is the leading term of $\mu^2_\alpha$, it follows that $V$ also generates $A$ as an algebra with respect to the product $\mu^2_\alpha$.
This completes the proof.
 \end{proof}

Now let us consider the setting of Proposition \ref{proposition:critpd} further.
Observe that $\mu^2_0(V,V) \subset A_2$, but $\mu^2_0(v,v)$ must be a multiple of $e_L \in A_0$; hence $\mu^2_0(v,v) = 0$ for any $v \in V$.
It follows that, for any $v_1,v_2 \in V$, there is a differential
\begin{align}
d_{v_1,v_2}: A & \To  A;\\
d_{v_1,v_2}(x) &:= \mu^2_0(v_2,x) + \mu^2_0(x,v_1).
\end{align}

\begin{proposition}
\label{proposition:diffcrit}
In the setting of Proposition \ref{proposition:critpd}, suppose that $v_1, v_2 \in V$ are distinct critical points of $\mathfrak{P}'$, but with the same critical value $w := \mathfrak{P}'(v_1) = \mathfrak{P}'(v_2)$; and suppose that the differential $d_{v_1,v_2}$ admits a contracting homotopy, i.e., there exists a map
\begin{equation} h: A_* \To A_{*-1}\end{equation}
so that
\begin{equation} [d_{v_1,v_2},h] = \mathrm{Id}.\end{equation}

Now let $\alpha_i = \iota(v_i)$ for $i = 1,2$ be the corresponding weak bounding cochains, giving objects $(L,\alpha_1)$ and $(L,\alpha_2)$ of $\cA^{wbc}_w$.
Consider the chain complex
\begin{equation} hom^*((L,\alpha_1),(L,\alpha_2))\end{equation}
in $\cA^{wbc}_w$; denote the differential by
\begin{align}
\mu^1_{\alpha_1,\alpha_2}: A^+ & \To  A^+,\\
\mu^1_{\alpha_1,\alpha_2}(x) & :=  \sum \mu^*(\alpha_2,\ldots,\alpha_2,x,\alpha_1,\ldots,\alpha_1).
\end{align}
Then the differential $\mu^1_{\alpha_1,\alpha_2}$ also admits a contracting homotopy: i.e., there exists a map
\begin{equation} H: A^+ \To A^+\end{equation}
(of odd degree, but no longer necessarily $\Z$-graded), such that
\begin{equation} [\mu^1_{\alpha_1,\alpha_2},H] = \mathrm{Id}.\end{equation}
In particular, the differential $\mu^1_{\alpha_1,\alpha_2}$ is acyclic:
\begin{equation} \mathrm{Hom}^*((L,\alpha_1),(L,\alpha_2)) \cong 0.\end{equation}
\end{proposition}
\begin{proof}
Consider the $\Z$-grading on $\mathrm{End}(A)$ induced by the $\Z$-grading on $A$.
Write
\begin{equation} \mu^1_{\alpha_1,\alpha_2} =: d= \sum_{i \ge 0} d_{1-2i},\end{equation}
where $d_j$ has degree $j$.
We can do this because $d = \mu^1_{\alpha_1,\alpha_2}$ has odd degree, and furthermore has degree $\le 1$ by the hypotheses of Proposition \ref{proposition:critpd}.
Furthermore, we can identify $d_1$ as the part of $\mu^1_{\alpha_1,\alpha_2}$ coming from $\mu^2_0$:
\begin{align}
d_1(x) &= d_{v_1,v_2}(x) \mbox{ for $x \in A$} \\
d_1(f_L) &= e_L^+ - e_L \\
d_1(e_L^+) &= v_2 - v_1.
\end{align}
We construct $H$ order-by-order, as
\begin{equation} H = \sum_{i \ge 0} H_{-1-2i}.\end{equation}

The first step is to define $H_{-1}$: we set
\begin{align}
H_{-1}(x) &= h(x) \mbox{ for $x \in A$} \\
H_{-1}(f_L) &= 0 \\
H_{-1}(e_L^+) &= f_L.
\end{align}
One easily checks that $[d_1,H_{-1}] = \mathrm{Id}$ (one must use the fact that $e_L$ is a unit for the associative product $\mu^2_0$, which was part of hypothesis \eqref{itm:gens} of Proposition \ref{proposition:critpd}).

Now define
\begin{equation} F:= [d,H] - \mathrm{Id} \in \mathrm{End}(A);\end{equation}
we prove inductively that it is possible to choose $H_{-1}, \ldots,H_{-1-2i}$ so that $F_{\ge -2i} = 0$.
The only terms which can contribute to $F_{\ge 0}$ are $[d_1,H_{-1}] -\mathrm{Id}$, which we have shown to be zero; so the base case $i=0$ of the induction is established.

Now suppose that the hypothesis has been established to order $i-1$.
To establish it to order $i$, we must show it is possible to choose $H_{-1-2i}$ so that $F_{-2i}$ vanishes.
Note that
\begin{equation}F_{-2i} =  \left([d_{\ge 1-2i},H_{\ge 1-2i}]\right)_{-2i} + [d_1,H_{-1-2i}].\end{equation}
Now we have
\begin{equation} [d,[d,H_{\ge 1-2i}]] = 0;\end{equation}
combining this with $[d,H_{\ge 1-2i}]_{\ge -2(i-1)} = \mathrm{Id}$ (the inductive hypothesis), we have
\begin{equation} [d_1,[d,H_{\ge 1-2i}]_{-2i}] = 0,\end{equation}
so $[d,H_{\ge 1-2i}]_{-2i}$ is a chain map from $(A^+,d_1)$ to itself.
But the latter chain complex admits a contracting homotopy, by hypothesis, so the endomorphism $[d,H_{\ge 1-2i}]_{-2i}$ is nullhomotopic.
It now suffices to set $H_{-1-2i}$ equal to the nullhomotopy.
This completes the inductive step and hence the proof.
 \end{proof}

\subsection{Finite abelian group actions}
\label{subsec:fabg}

In this section, we describe how the disc potential behaves under finite covers  (compare \cite[\S 5]{Cho2013}).
We use the terminology of \cite[\S 4]{Seidel2003} and \cite[\S 2]{Sheridan2015}.

Let $\bm{G}_i = \{ \Z \To Y_i\}$ be grading data, for $i=1,2$, and suppose that
\begin{equation} \bm{p}: \bm{G}_1 \hookrightarrow \bm{G}_2\end{equation}
is an injective morphism of grading data, with finite cokernel $\Gamma \cong Y_2/Y_1$.
Suppose that $R$ is a $\bm{G}_1$-graded $\C$-algebra, and that $\cA$ is a $\bm{G}_2$-graded, $\bm{p}_* R$-linear $A_\infty$ category.
Recall that we can then form the $\bm{G}_1$-graded, $R$-linear $A_\infty$ category $\bm{p}^* \cA$.

There is an action of the character group of $\Gamma$,
\begin{equation} \Gamma^* := \mathrm{Hom}(\Gamma,\C^*),\end{equation}
on the morphism spaces of $\cA$: namely, if $a \in hom^*(K,L)$ has pure degree $y \in Y_2$, then $\chi$ acts on $a$ by
\begin{equation} \chi \cdot a := \chi([y]) a.\end{equation}
Because $\cA$ is $\bm{G}_2$-graded, this action strictly commutes with the $A_\infty$ maps $\mu^*$.

Let $L$ be an object of $\cA$, and $A := hom^*(L,L)$ its endomorphism algebra.
Then $\Gamma^*$ acts (strictly) on $A$, and we can form the semidirect product
\begin{equation} A \rtimes \Gamma^* := A \otimes \C\left[\Gamma^*\right],\end{equation}
with the $A_\infty$ structure maps
\begin{multline}
\label{eqn:semidir}
\mu^s(a_s \otimes \chi_s,\ldots, a_1 \otimes \chi_1) := \\
\mu^s(a_s,\chi_s \cdot a_{s-1}, \chi_s \chi_{s-1} \cdot  a_{s-2},\ldots,\chi_s \ldots\chi_2 \cdot a_1 ) \otimes \chi_s \ldots \chi_1
\end{multline}
(see \cite[Formula (4.8)]{Seidel2003}).
The semidirect product is also sometimes denoted `$\#$' rather than `$\rtimes$'.

The semidirect product is related to $\bm{p}^* \cA$ by a Fourier transform.
To see this, we choose a set-theoretic splitting $\theta$ of the following short exact sequence of abelian groups:
\begin{equation} 0 \To Y_1 \To Y_2 \overset{\overset{\theta}{\dashleftarrow}}{\To} \Gamma \To 0.\end{equation}
We suppose, furthermore, that $\sigma \circ \theta = 0$, where $\sigma: Y_2 \To \Z/2\Z$ is the sign morphism of $\bm{G}_2$ (one can always choose such a splitting).
We now consider the object
\begin{equation} L^\theta := \bigoplus_{\gamma \in \Gamma} L[\theta(\gamma)]\end{equation}
(it can be regarded as an object of $\cA$ or $\bm{p}^*\cA$),
and we denote
\begin{equation} A^\theta:= hom^*_{\bm{p}^* \cA}(L^\theta,L^\theta).\end{equation}

The objects making up $L^\theta$ are indexed by elements of $\Gamma$, and there is an action of $\Gamma$ on $hom^*_{\cA}(L^\theta,L^\theta)$ by permuting the objects: explicitly, $\gamma \in \Gamma$ takes
\begin{equation} \gamma: hom^*(L[\theta(\gamma_1)],L[\theta(\gamma_2)]) \To hom^*(L[\theta(\gamma_1+\gamma)],L[\theta(\gamma_2+\gamma)])\end{equation}
by a shift isomorphism.
This action lifts to an action of $\Gamma$ on $A^\theta$ (using the fact that $\theta$ is a splitting).

\begin{lemma}
\label{lemma:semidirpullback}
There is a strict isomorphism of $\Z/2\Z$-graded $A_\infty$ algebras,
\begin{equation} A \rtimes \Gamma^* \cong A^\theta.\end{equation}
\end{lemma}
\begin{proof}
We first define a map of vector spaces
\begin{equation} f: A \To A^\theta\end{equation}
by sending an element $a$ of pure degree $y \in Y_2$ to the unique element
\begin{equation} f(a) \in hom^*(L[\theta(0)],L[\theta(y)])\end{equation}
which corresponds to $a$ under a shift isomorphism (this map respects the $\Z/2\Z$-grading, by the condition that $\sigma \circ \theta = 0$).
We then define the isomorphism
\begin{equation}F: A \rtimes \Gamma^* \To A^\theta\end{equation}
by
\begin{equation} F(a \otimes \chi) := \sum_{\gamma \in \Gamma} \chi(\gamma) \gamma \cdot f(a)\end{equation}
(it is easy to check that $F$ is an isomorphism of vector spaces).

We now prove that $F$ is a strict isomorphism of $A_\infty$ algebras: suppose that $a_s, \ldots, a_1 \in A$ are elements of pure degree $y_s,\ldots, y_1$ respectively, then
\begin{multline}
\label{eqn:semipull}
\mu^s(F(a_s \otimes \chi_s),\ldots,F(a_1 \otimes \chi_1)) =\\
\sum_{\gamma_s,\ldots,\gamma_1} \mu^s( \chi_s(\gamma_s) \gamma_s \cdot f(a_s), \ldots, \chi_1(\gamma_1) \gamma_1 \cdot f(a_1)).
\end{multline}
Now the $A_\infty$ product on the right-hand side vanishes unless the morphisms are composable in $\bm{p}^* \cA$, which means we have
\begin{align}
\gamma_{2} &= \gamma_1 + y_1 \\
\vdots&  \phantom{{}= =} \vdots \\
\gamma_s &= \gamma_{s-1} + y_{s-1}.
\end{align}
Using these to write the $\gamma_i$'s in terms of $\gamma := \gamma_1$ and the $y_i$'s, the right-hand side of \eqref{eqn:semipull}  becomes:
\begin{equation}
\sum_{\gamma} \mu^s\left(\chi_s\left(\gamma+\sum_{i=1}^{s-1}y_i\right) \left(\gamma+\sum_{i=1}^{s-1}y_i\right) \cdot f(a_s), \ldots, \chi_1(\gamma) \gamma \cdot f(a_1)\right)
\end{equation}
\begin{align}
&= \sum_{\gamma} \gamma \cdot f(\mu^s(a_s,\ldots,a_1)) \prod_{j=1}^s \chi_j\left(\gamma + \sum_{i=1}^{j-1}y_{i} \right)\\
&= \sum_\gamma \gamma \cdot f(\mu^s(a_s,\ldots,a_1)) \prod_{j=1}^s \chi_j(\gamma) \prod_{i=1}^{s-1} \chi_s \ldots \chi_{i+1}(y_i) \\
&= F\left(\mu^s(a_s,\chi_s \cdot a_{s-1}  ,\ldots,\chi_s \ldots\chi_2 \cdot a_1) \otimes \chi_s \ldots \chi_1\right) \\
&= F\left( \mu^s(a_s \otimes \chi_s, \ldots, a_1 \otimes \chi_1) \right).
\end{align}
This completes the proof that $F$ is a strict isomorphism of $A_\infty$ algebras.
 \end{proof}

\begin{remark}
Lemma \ref{lemma:semidirpullback} only holds if $\Gamma$ is finite.
If $\Gamma$ is infinite, then one should regard the $\Gamma$-grading of $A$ as a $\Gamma$-\emph{coaction}, and replace $A \rtimes \Gamma^*$ by the smash product $A \# \C[\Gamma]^*$, in the notation of \cite{Cohen1984}.
\end{remark}

Another way of saying things is that the action of $\Gamma$ on $A^\theta$ gives rise to a Fourier decomposition:
\begin{equation}
\label{eqn:fourier}
 A^\theta \cong \bigoplus_{\chi \in \Gamma^*} A^\theta_\chi,
\end{equation}
where
\begin{equation} A^\theta_\chi := \left\{a \in A^\theta: \gamma \cdot a = \chi(\gamma) a\right\}.\end{equation}
The isomorphism of Lemma \ref{lemma:semidirpullback} sends
\begin{equation}
\label{eqn:fouriest} A \otimes \chi \overset{\cong}{\To} A^\theta_{\chi^{-1}}.
\end{equation}

In particular, we have an isomorphism of $A$ with the $\Gamma$-invariant part of $A^\theta$:
\begin{equation} A \cong A^\theta_1 \cong(A^\theta)^\Gamma \subset A^\theta.\end{equation}
We denote the resulting inclusion by
\begin{align}
j: A &\hookrightarrow  A^\theta,\\
j(a) & :=  a \otimes 1.
\end{align}

If $\cA$ is strictly unital, then $j$ induces an inclusion
\begin{equation}
\label{eqn:jinc}
 j: \hcM_{weak}(L;\cA) \hookrightarrow \hcM_{weak}(L^\theta;\bm{p}^*\cA),\end{equation}
such that $\mathfrak{P} \circ j = \mathfrak{P}$.
Furthermore, we have an action of $\Gamma^*$ on $\hcM_{weak}(L)$, coming from the action of $\Gamma^*$ on $A$, and $\mathfrak{P}(\chi \cdot \alpha) = \mathfrak{P}(\alpha)$.

\begin{proposition}
\label{proposition:characts}
The objects $(L^\theta,j(\alpha))$ and $(L^\theta,j(\chi \cdot \alpha))$ of $(\bm{p}^* \cA)^{wbc}_{\mathfrak{P}(\alpha)}$ are quasi-isomorphic.
So $\Gamma^*$ acts by quasi-isomorphisms on the image of $j$ in $\hcM_{weak}(L^\theta)$.
\end{proposition}
\begin{proof}
Let $e \in A$ be the strict unit, and $e \otimes \chi$ be the corresponding endomorphism of $L^\theta$ under Lemma \ref{lemma:semidirpullback}.
Consider the morphism
\begin{equation} e \otimes \chi \in hom^*((L^\theta,j(\alpha)),(L^\theta,j(\chi  \cdot \alpha))).\end{equation}
Then $e \otimes \chi$ is closed: its differential is equal to
\begin{align}
&  \sum \mu^*(\chi \cdot \alpha \otimes 1,\ldots, \chi \cdot \alpha \otimes 1, e \otimes \chi, \alpha \otimes 1, \ldots, \alpha \otimes 1) \\
&= \sum \mu^*(\chi \cdot \alpha,\ldots,\chi \cdot \alpha,e, \chi \cdot \alpha,\ldots,\chi \cdot \alpha) \otimes \chi \\
&= (\chi \cdot \alpha - \chi \cdot \alpha) \otimes \chi \\
&= 0,
\end{align}
using \eqref{eqn:semidir} and strict unitality.
Similarly, $e \otimes \chi^{-1}$ is a closed morphism in the opposite direction.

Furthermore, their composition $\mu^2(e\otimes \chi,e \otimes \chi^{-1})$ is equal to
\begin{align}
 & \sum \mu^*(\chi \cdot \alpha \otimes 1, \ldots, e \otimes \chi, \alpha \otimes 1, \ldots, e \otimes \chi^{-1}, \chi \cdot \alpha \otimes 1,\ldots) \\
&= \sum \mu^*( \chi \cdot \alpha,\ldots,e, \chi \cdot \alpha,\ldots,e,\chi \chi^{-1} \chi \cdot \alpha,\ldots) \otimes \chi \chi^{-1} \\
&= e \otimes 1,
\end{align}
which is the (strict) unit.
Similarly, their composition in the other direction is the unit: therefore, the objects are quasi-isomorphic.
 \end{proof}

\begin{proposition}
\label{proposition:pullwbc}
Suppose that $\cA$ is equipped with homotopy units to give the strictly unital $\bm{G}_2$-graded $A_\infty$ category $\cA^+$.
Suppose that $L$ is an object of $\cA$, and
\begin{equation} V \subset A := hom^*_\cA(L,L)\end{equation}
satisfies the hypotheses of Proposition \ref{proposition:critpd}, and is invariant under the action of $\Gamma^*$.
Now suppose that $v \in V$ is a critical point of $\mathfrak{P}'$, $\alpha = \iota(v)$ is the corresponding weak bounding cochain in $\cA^+$, and $j(\alpha)$ the corresponding weak bounding cochain in $\bm{p}^* \cA^+$.
Suppose furthermore that the differential $d_{v,\chi \cdot v}$ admits a contracting homotopy (as in Proposition \ref{proposition:diffcrit}), for all $\chi \in \Gamma^* \setminus \{1\}$.
Then there is a quasi-isomorphism of $A_\infty$ algebras:
\begin{equation} hom^*_{\bm{p}^* \cA^{wbc}}\left((L^\theta,j(\alpha)),(L^\theta,j(\alpha))\right) \cong hom^*_{\cA^{wbc}}((L,\alpha),(L,\alpha)).\end{equation}
\end{proposition}
\begin{proof}
The inclusion $j: A \hookrightarrow A^\theta$ is a strict homomorphism of $A_\infty$ algebras, and sends $\alpha$ to $j(\alpha)$; it follows that it induces a strict embedding of $A_\infty$ algebras,
\begin{equation} j: hom^*((L,\alpha),(L,\alpha)) \hookrightarrow hom^*((L^\theta,j(\alpha)),(L^\theta,j(\alpha))).\end{equation}
We will show that this embedding is in fact a quasi-isomorphism.

Because $j(\alpha)$ is $\Gamma$-invariant, and the $A_\infty$ structure is strictly $\Gamma$-equivariant, the differential $\mu^1_{j(\alpha)}$ is $\Gamma$-equivariant.
Hence, it preserves the Fourier decomposition \eqref{eqn:fourier}.
Furthermore, for any element $a \otimes \chi \in A^\theta_\chi$ (under the isomorphism \eqref{eqn:fouriest}), we have
\begin{align}
\mu^1_{j(\alpha)}(a \otimes \chi) &= \sum \mu^*(\alpha \otimes 1,\ldots,\alpha \otimes 1, a \otimes \chi, \alpha \otimes 1,\ldots,\alpha \otimes 1) \\
&= \mu^*( \alpha, \ldots, \alpha, a, \chi \cdot \alpha, \ldots, \chi \cdot \alpha) \otimes \chi  \label{eqn:foury}\\
&= \mu^1_{\alpha,\chi \cdot \alpha}(a) \otimes \chi,
\end{align}
where \eqref{eqn:foury} follows from \eqref{eqn:semidir}.

Now, by Proposition \ref{proposition:diffcrit}, using the hypothesis that $d_{v,\chi \cdot v}$ admits a contracting homotopy, the differential $\mu^1_{\alpha,\chi \cdot \alpha}$ admits a contracting homotopy, for all $\chi \neq 1$.
Therefore, the cohomology of the direct summand cochain complex $A^\theta_\chi$ is zero, for all $\chi \neq 1$.
On the other hand, for $\chi = 1$, the differential on the remaining direct summand $A^\theta_1$ is given by $\mu^1_{\alpha}$, whose comology is $A$, by Proposition \ref{proposition:critpd}.
It follows that $j$ is a quasi-isomorphism of $A_\infty$ algebras, as required.
 \end{proof}

\section{Weak bounding cochains in the monotone Fukaya category}
\label{sec:wbcgeom}

In this section, we sketch the procedure for including weak bounding cochains in our definition of the monotone and monotone relative Fukaya categories.
Our formulation is slightly more general than that which has appeared in the literature before: we consider weak bounding cochains on multiple Lagrangians, and we continue to work over $\C$ (for the monotone Fukaya category), and over $R$ (for the monotone relative Fukaya category), rather than over a Novikov ring.
In particular, we do \emph{not} require our coefficient ring to be complete with respect to the energy filtration; in the various places where convergence of some Maurer--Cartan equation is required for the theory to make sense, the convergence will occur for geometric reasons to do with monotonicity.

\subsection{Homotopy units}
\label{subsec:homungeom}

Following \cite[\S 10]{Ganatra2012} (which is based on \cite[\S 7.3]{fooo}, but whose formalism is more closely aligned with our own), we briefly recall how to define a homotopy unit structure on the monotone Fukaya category $\cF(X)$.

The $A_\infty$ structure maps which have $e_L^+$ as an entry are completely prescribed (see \S \ref{subsec:homun}).
Thus, to give a homotopy unit structure on $\cF(X)$, we need to define $\mu^*(\ldots,f,\ldots,f,\ldots)$ which satisfy the $A_\infty$ relations (when combined with $\mu^1(f) = e^+ - e$).
These structure maps are defined by counting pseudoholomorphic discs, but where some boundary marked points are now labelled by $f$, rather than a Hamiltonian chord.
The boundary marked points labelled by $f$ are treated in a different way from the others.
Namely, they are not punctured, and rather than a strip-like end, they come equipped with an embedding of the upper half-disc, sending boundary components to boundary components, and $0$ to the corresponding boundary marked point.
They furthermore come equipped with an additional parameter $\rho \in (0,1]$, corresponding to the position of an unconstrained marked point lying on the line connecting $0$ to $i$ in the upper half-disc.
Perturbation data are chosen so that:
\begin{itemize}
\item As $\rho \To 0$, the perturbation data converge to a strip-like end with the cohomological unit moduli space glued on.
\item At $\rho = 1$, the perturbation data are pulled back via the forgetful map which forgets the unconstrained marked point.
\end{itemize}

\begin{remark}
There is one technical point to be aware of: perturbation data can only be defined in this way if, after one forgets all points labelled $f$, the resulting moduli space is semistable (otherwise, the perturbation data can not be pulled back via the forgetful map after one forgets all of the $f$'s).
These are called `\textbf{f}-semistable' domains, in the language of \cite[\S 10]{Ganatra2012}.
As a result, we are forced to define $\mu^s$ to equal zero if all inputs are $e^+$ or $f$, with the exceptions
\begin{equation}
\label{eqn:mu1f}
\mu^1(f) = e^+ - e
\end{equation}
and
\begin{equation}
\label{eqn:mu2eplus}
 \mu^2(e^+,a) = (-1)^{\sigma(a)}\mu^2(a,e^+) = a \mbox{ for $a = e^+, f$,}
\end{equation}
then prove that this is compatible with the proof of the $A_\infty$ associativity equations for $\cF^+(X)$ (see \cite[Remark 10.3]{Ganatra2012}).
I.e., we must prove that no disc can bubble off a one-dimensional moduli space, carrying only forgotten points and homotopy units and nothing else with it, except for those corresponding to Equations \eqref{eqn:mu1f}, \eqref{eqn:mu2eplus}.
This is true by the monotonicity assumption: indeed, since $f$ has degree $-1$ and $e^+$ has degree $0$, the output of $\mu^s$ with $k$ inputs equal to $f$ and $s-k$ inputs equal to $e^+$, must lie in degree $2-s - k - \mu \le 2-s-k$, where $\mu$ is the Maslov index of the corresponding disc (which is $\ge 0$ by monotonicity).
Because $CF^*(L,L) \cong C^*(L)$ is concentrated in degrees $\ge 0$ (for appropriately chosen Floer data), there is no degree for the output to live in, unless we have one of the situations enumerated in \eqref{eqn:mu1f}, \eqref{eqn:mu2eplus}.
It follows that for generic Floer and perturbation data, such a disc can not bubble off in a one-dimensional family.
\end{remark}

Counting the resulting moduli spaces of pseudoholomorphic discs defines the $A_\infty$ structure maps.
Note that $e^+$ and $f$ never appear as the output, by definition.
Considering the boundary points of the one-dimensional moduli spaces proves that the resulting structure maps satisfy the $A_\infty$ relations (c.f. \cite[\S 10]{Ganatra2012}).
In particular, we have new kinds of boundary components, where $\rho \To 0$ or $\rho \To 1$ for the parameter $\rho$ associated to a marked point labelled by $f$.
When $\rho \To 0$, the disc bubbles off a cohomological unit $e$.
The boundary component at $\rho = 1$ is generically empty, because the moduli space factors through a moduli space of lower dimension (by forgetting the boundary marked point labelled $f$), unless forgetting $f$ yields a strip, in which case the strip is necessarily constant and corresponds to the identity map $\mu^2(e^+,-)$ or $\mu^2(-,e^+)$.

We denote the resulting strictly unital $A_\infty$ category by $\cF^+(X)$.
One can similarly define a homotopy unit structure on the monotone relative Fukaya category, and we denote the resulting strictly unital category by $\cF^+_m(X,D)$.

We recall that the $A_\infty$ structure maps of the monotone Fukaya category do not quite satisfy the $A_\infty$ associativity equations:
\begin{equation} \mu^1: CF^*(L_0,L_1) \To CF^*(L_0,L_1)\end{equation}
satisfies
\begin{equation} \mu^1(\mu^1(x)) = (w(L_0) - w(L_1))x.\end{equation}
Introducing `curvature' terms
\begin{equation} \mu^0_L := w(L) \cdot e^+_L\end{equation}
into the category $\cF^+(X)$ by hand, one obtains a curved, strictly-unital $A_\infty$ category $\cF^{c}(X)$ (respectively $\cF^c_m(X,D)$).
Here $c$ stands for `curved'.

For each $w \in \C$, there is an embedding
\begin{equation} \cF(X)_w \hookrightarrow \cF^{c}(X),\end{equation}
respectively
\begin{equation} \cF_m(X,D)_w \hookrightarrow \cF^{c}_m(X,D),\end{equation}
whose image consists of all objects $L$ with curvature $w \cdot e^+_L$.

\subsection{Convergence}
\label{subsec:conv}

Because $\cF^c(X)$ (respectively $\cF^c_m(X,D)$) is a curved, strictly unital $A_\infty$ category, we can define the formal enlargement $\cF^{wbc}(X)$ (respectively $\cF^{wbc}_m(X,D)$), in accordance with the procedure outlined in \S \ref{subsec:wbc}.
However, we must explain the geometric restrictions we place on the weak bounding cochains in the monotone Fukaya category, and how they ensure convergence in the equations \eqref{eqn:mc}, \eqref{eqn:defainf}.
Recall that our symplectic manifold $X$ is monotone, and the objects of the monotone Fukaya category are required to be monotone Lagrangian submanifolds $L \subset X$, in the sense of Definition \ref{definition:lmon}.

We recall the Lagrangian Grassmannian of $X$, $\pi: \cG X \To X$, and that any Lagrangian immersion $\iota : L \To X$ comes equipped with a canonical lift
\begin{equation} \iota_*: L \To \cG X.\end{equation}
We also recall that, for any map of a surface with boundary into $X$,
\begin{equation} u: \Sigma \To X,\end{equation}
together with a lift of $u|_{\partial \Sigma}$ to $\cG X$,
\begin{equation} \tilde{u}: \partial \Sigma \To \cG X,\end{equation}
we can define the boundary Maslov index $\mu(u,\tilde{u}) \in \Z$ (see \cite[Appendix C.3]{mcduffsalamon}).

\begin{figure}
\centering
\includegraphics[width=0.6\textwidth]{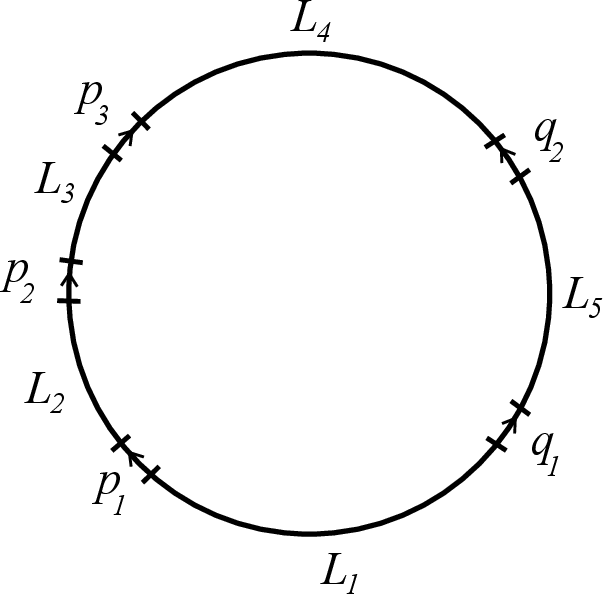}
\caption{Monotonicity. In this case, $k=3, l=2$. The path associated to $p_1$ goes from the lift of $L_1$ to the lift of $L_2$; the path associated to $q_1$ goes from the lift of $L_1$ to the lift of $L_5$.
\label{fig:Fig7}}
\end{figure}

\label{definition:strmonpaths}
Let $\mathcal{L}$ be a set of smooth manifolds $L$, together with Lagrangian immersions $\iota_L: L \To X$, having lifts $\iota_{L,*}$.
Let $\mathcal{P}$ be a set of paths
\begin{equation} \gamma: [0,1] \To \cG X,\end{equation}
connecting the lifts of the $L \in \mathcal{L}$, in the sense that to each $\gamma \in \mathcal{P}$, there are associated points $p_0 \in L_0$, $p_1 \in L_1$, for some $L_0,L_1 \in \mathcal{L}$, such that
\begin{equation} \gamma(i) = \iota_{L_i,*}(p_i)\end{equation}
for $i = 0,1$.

Let $\mathbb{D}$ be a disc, with disjoint intervals labelled $p_1,\ldots,p_k, q_l, \ldots,q_1$ in \emph{clockwise} order on the boundary, and the boundary components between the intervals labelled by elements $L \in \mathcal{L}$.
Now label the intervals by elements $\gamma \in \mathcal{P}$, where $\gamma$ is a path between the Lagrangians labelling the boundary components on either side of the interval, oriented clockwise for boundary components corresponding to $p_i$, and anti-clockwise for boundary components corresponding to $q_i$ (see Figure \ref{fig:Fig7}).
Now let
\begin{equation} \tilde{u}: \partial \mathbb{D} \To \cG X\end{equation}
be a map, coinciding with the paths $\gamma$ on boundary components corresponding to boundary marked points, and mapping to the lift of $L$ on boundary components labelled by $L$ (in the case that $L$ is immersed, we do not allow $\tilde{u}$ to change `sheets' of $L$, unless along a path $\gamma$).

\begin{definition}
We say that the collection $(\mathcal{L},\mathcal{P})$ is \emph{monotone} if
\begin{itemize}
\item Each $L \in \mathcal{L}$ is monotone;
\item For any such $\tilde{u}$, the map $\pi \circ \tilde{u} : \partial \mathbb{D} \To X$ extends to a continuous map $u: \mathbb{D} \To X$;
\item If we fix $k$, then there exists $\tau_k < \tau$ such that, for sufficiently large $l$, any such extension $u$ satisfies
\begin{equation}
\label{eqn:Fkl}
 \omega(u) \le \tau \mu(u,\tilde{u}) + \tau_k l.
\end{equation}
The numbers $\tau_k$ may depend on $k$ and on $(\mathcal{L},\mathcal{P})$, but should not depend on $\tilde{u}$ or $u$.
\end{itemize}
\end{definition}

We can make a parallel definition, which applies to the monotone relative Fukaya category.

\begin{definition}
If $(X,D)$ is a K\"{a}hler pair, we say that the collection $(\mathcal{L},\mathcal{P})$ is \emph{relatively monotone} if
\begin{itemize}
\item Each $L \in \mathcal{L}$ is exact in $X \setminus D$;
\item Each path in $\mathcal{P}$ is contained inside $\cG(X \setminus D) \subset \cG X$;
\item For any $\tilde{u}: \partial \mathbb{D} \To \cG(X \setminus D)$ as in Definition \ref{definition:strmonpaths}, the map $\pi \circ \tilde{u} : \partial \mathbb{D} \To X \setminus D$ extends to a continuous map $u: \mathbb{D} \To X$;
\item  If we fix $k$, then there exists $\tau_k < \tau$ such that, for sufficiently large $l$, any such extension $u$ satisfies
\begin{equation}
 \sum_j \ell_j (u \cdot D_j) \le \tau \mu(u,\tilde{u}) + \tau_k l,\label{eqn:Fklrel}
\end{equation}
where $\ell_j$ is the linking number of $\alpha$ with component $D_j$ of $D$ (see \cite[Definition 3.11]{Sheridan2015}).
\end{itemize}
\end{definition}

Now observe that, for any objects $L_0,L_1$ of $\cF(X)$, a generator $p$ of $CF^*(L_0,L_1)$ is a Hamiltonian chord from $L_0$ to $L_1$.
Furthermore, this chord comes with a canonical homotopy class of lifts to $\cG X$, with endpoints on the respective lifts of $L_0$ and $L_1$, and such that the corresponding orientation operator has Maslov index $0$ (see \cite[\S 11l]{Seidel2008}).
We choose a lift in this homotopy class and denote it by $\gamma_p$.

\begin{definition}
\label{definition:strmonwbc}
Let $\mathcal{L}$ be a set of objects of $\cF^c(X)$, and consider the object
\begin{equation} \bm{L} := \bigoplus_{L \in \mathcal{L}} L \end{equation}
of $\cF^c(X)^\oplus$.
Let
\begin{equation} \alpha \in hom^*(\bm{L},\bm{L})\end{equation}
be an element of odd degree, and let $\mathcal{P}$ be the set of classes $\gamma_p$, for all generators $p$ which appear in $\alpha$ with a non-zero coefficient.
We say that $\alpha$ is \emph{monotone} if the pair $(\mathcal{L},\mathcal{P})$ is monotone.
We similarly define the notion of a \emph{relatively monotone} element $\alpha \in hom^*(\bm{L},\bm{L})$, where $\bm{L}$ is an object of $\cF^c_m(X,D)^\oplus$.
\end{definition}

\begin{lemma}
\label{lemma:conv}
Let us abbreviate $\cA := \cF^c(X)^\oplus$ (respectively, $\cF^c_m(X,D)^\oplus$).
If $\bm{L}_i$ are objects of $\cA$ for $i=0,\ldots,s$, and
\begin{equation} \alpha_i \in hom^*(\bm{L}_{i},\bm{L}_i)\end{equation}
are monotone in the sense of Definition \ref{definition:strmonwbc} (respectively, relatively monotone), and
\begin{equation} a_i \in hom^*(\bm{L}_{i-1}, \bm{L}_i),\end{equation}
then
\begin{equation}
\mu^*(\underbrace{\alpha_s,\ldots,\alpha_s}_{i_s},a_s, \underbrace{\alpha_{s-1},\ldots,\alpha_{s-1}}_{i_{s-1}}, a_{s-1}, \ldots,a_1,\underbrace{\alpha_0,\ldots,\alpha_0}_{i_0}) = 0 \label{eqn:term}
\end{equation}
for $i_0,\ldots,i_s$ sufficiently large; in particular, Equations \eqref{eqn:mc}  and \eqref{eqn:defainf}  converge.
\end{lemma}
\begin{proof}
We give the proof in the monotone case; the relatively monotone case is analogous.
Suppose that $v_1$ and $v_2$ are two pseudoholomorphic discs which contribute to the coefficient of $a_0$ in \eqref{eqn:term}, with values of $i_j$ given by $k_j$ for disc $v_1$, and $l_j$ for disc $v_2$.
Define lifts $\tilde{v}_1,\tilde{v}_2$ of the boundaries to $\cG X$, by gluing the index-$0$ orientation operators onto the strip-like ends of $v_1, v_2$.
Because these $v_1$ and $v_2$ contribute to $A_\infty$ products, they are rigid, which implies that
\begin{equation}
 \mu(v_1,\tilde{v}_1) = 2-s- \sum_{j=0}^s k_j\label{eqn:v1rig}
\end{equation}
and
\begin{equation}
 \mu(v_2,\tilde{v}_2) = 2-s-\sum_{j=0}^s l_j. \label{eqn:v2rig}
\end{equation}

Now the disc $\overline{v}_2$ can be glued to $v_1$ along the Hamiltonian chords $a_0,\ldots,a_s$; the result is a genus-0 surface in $X$, with $s+1$ holes.
The boundary of the $j$th hole has the form $\tilde{u}_j$ described in Definition \ref{definition:strmonpaths}, with $k=k_j$ and $l=l_j$.
Therefore, by the monotonicity condition, the $j$th hole can be filled in by a disc $u_j$.
Filling in the holes gives us a closed surface
\begin{equation} v = v_1 \cup \overline{v}_2 \cup u_1 \cup \ldots \cup u_j.\end{equation}
We now have
\begin{align}
\omega(v) &=  2\tau c_1(v)\\
\Rightarrow \omega(v_1) - \omega(v_2) + \sum_{j=0}^s \omega(u_j) &= 
 \tau \left( \mu(v_1,\tilde{v}_1) - \mu(v_2,\tilde{v}_2) + \sum_{j=0}^s \mu(u_j) \right),
\end{align}
where the right-hand side follows by the `Composition' property of the Maslov index \cite[Theorem C.3.5]{mcduffsalamon}.
Substituting in \eqref{eqn:v1rig}  and \eqref{eqn:v2rig}  and rearranging, we obtain
\begin{align}
\omega(v_2) &= \omega(v_1) + \sum_{j =0}^s \omega(u_j) - \tau\left(\mu(u_j) + l_j - k_j\right) \\
\label{eqn:omv2} &\le  \omega(v_1) + \tau \left(\sum_{j=0}^s k_j \right) + \sum_{j=0}^s (\tau_{k_j} - \tau) l_j
\end{align}
by the monotonicity property \eqref{eqn:Fkl}.

In particular, if we fix disc $v_1$, then the first two terms of \eqref{eqn:omv2} are fixed, and $\tau_{k_j} - \tau <0$ by definition, so for sufficiently large $l_j$, $\omega(v_2)$ is arbitrarily negative.
This contradicts the fact that pseudoholomorphic discs have positive energy (the necessary estimate relating symplectic area to energy is proven in \cite[Lemma 3.3.3]{Biran2013} -- note that the curvature term, which accounts for the difference between symplectic area and energy, is compactly supported and hence bounded).
It follows that for each $a_0$, the coefficient of $a_0$ in \eqref{eqn:term}  vanishes for sufficiently large $i_0,\ldots,i_s$; since the $hom$-spaces have a finite number of generators, the result follows.
 \end{proof}

\begin{definition}
We can now define the generalization of the monotone Fukaya category, $\cF^{wbc}(X)$, which includes weak bounding cochains.
The objects are pairs $(\bm{L},\alpha)$, where $\bm{L}$ is an object of $\cF^c(X)^\oplus$ and
\begin{equation} \alpha \in CF^*(\bm{L},\bm{L})\end{equation}
is monotone, and is a weak bounding cochain, i.e., has odd degree and satisfies the Maurer--Cartan equation \eqref{eqn:mc}  (which converges by Lemma \ref{lemma:conv}, because $\alpha$ is monotone).
$\cF^{wbc}(X)$ becomes a strictly unital, curved $A_\infty$ category, with $A_\infty$ structure maps given by equation \eqref{eqn:defainf}  (which again converges by Lemma \ref{lemma:conv}).
In particular, for each $w \in \C$, $\cF^{wbc}(X)_w$ is a strictly unital, non-curved $A_\infty$ category, whose objects are those $(\bm{L},\alpha)$ with $\mathfrak{P}(\alpha) = w$.
If $(\bm{L}_0,\alpha_0)$ and $(\bm{L}_1,\alpha_1)$ are objects of $\cF^{wbc}(X)_w$, then we denote by
\begin{equation} HF^*((\bm{L}_0,\alpha_0),(\bm{L}_1,\alpha_1)) := \mathrm{Hom}^*((\bm{L}_0,\alpha_0),(\bm{L}_1,\alpha_1))\end{equation}
the corresponding Floer cohomology group.

We similarly define $\cF^{wbc}_m(X,D)$, requiring our weak bounding cochains to be relatively monotone.
\end{definition}

\begin{remark}
\label{remark:wbcrel}
We would like to relate $\cF^{wbc}(X)$ to $\cF^{wbc}_m(X,D)$.
Firstly, following Lemma \ref{lemma:reltomon} (enhanced to include homotopy units), there is a strict embedding
\begin{equation} \cF^c_m(X,D) \otimes _R \C \hookrightarrow \cF^c(X).\end{equation}
This would induce a map from weak bounding cochains on $\cF^c_m(X,D)$ to weak bounding cochains on $\cF^c(X)$, given by
\begin{equation} \alpha \mapsto \alpha \otimes_R 1,\end{equation}
except for the fact that
\begin{equation} \alpha \mbox{ relatively monotone} \nRightarrow \alpha \otimes_R 1 \mbox{ monotone.}\end{equation}
\end{remark}

Nevertheless, we have the following:

\begin{lemma}
\label{lemma:actioncontraction}
Let $(X,D)$ be a K\"{a}hler pair, and suppose that $X$ is simply-connected.
Let $\varphi_t: X \setminus D \To X \setminus D$ denote the time-$t$ reverse Liouville flow.
It is defined for all $t \geq 0$.
Now suppose that $\mathcal{L}$ is a collection of exact Lagrangians in $X \setminus D$, and $\mathcal{P}$ a finite set of paths in $\cG(X \setminus D)$, as in Definition \ref{definition:strmonpaths}, so that $(\mathcal{L},\mathcal{P})$ is relatively monotone.
Then, for sufficiently large $t$, $(\varphi_t(\mathcal{L}),\varphi_t(\mathcal{P}))$ is monotone.
Furthermore, $\varphi_t(L)$ is Hamiltonian isotopic to $L$, for all $L \in \mathcal{L}$.
\end{lemma}
\begin{proof}
First, observe that all $L \in \mathcal{L}$ are monotone, by Corollary \ref{corollary:objmap}.

Let $\alpha = d^ch$ be the Liouville form on $X \setminus D$, following the notation of \cite[Definition 3.11]{Sheridan2015}.
For each $L \in \mathcal{L}$, fix $f_L \in C^\infty(L;\R)$ so that
\begin{equation} \alpha|_L = d f_L.\end{equation}
We define the \emph{symplectic action} of a path $\gamma \in \mathcal{P}$ to be
\begin{equation} A(\gamma) := f_{L_1}(\pi \circ \gamma(1)) - f_{L_0}(\pi \circ \gamma(0))  -\int_{\pi \circ \gamma} \alpha \end{equation}
where $\pi: \cG(X \setminus D) \To X \setminus D$ is the obvious projection.

The reverse Liouville flow $\varphi_t$ is, by definition, the time-$t$ flow of the negative Liouville vector field $-Z$, where
\begin{equation} \iota_Z \omega = \alpha.\end{equation}
By Cartan's magic formula,
\begin{equation} \mathcal{L}_Z \alpha = d \iota_Z \alpha + \iota_Z d \alpha = \alpha\end{equation}
(using $\omega = d\alpha$ and $\omega(Z,Z) = 0$).
It follows that
\begin{equation} \varphi_t^*\alpha = e^{-t}\alpha.\end{equation}
Hence, for each $L \in \mathcal{L}$, we have
\begin{equation} \alpha|_{\varphi_t(L)} = d(e^{-t}f_L),\end{equation}
so $\varphi_t(L)$ is exact Lagrangian isotopic to $L$, and hence Hamiltonian isotopic to $L$.
It also follows that
\begin{equation}A(\varphi_t(\gamma)) = e^{-t} A(\gamma).\end{equation}
In particular, as $\mathcal{P}$ is finite, for any $\epsilon > 0$ we may choose $t \gg 0$ sufficiently large that
\begin{equation}
\label{eqn:actionphi}
|A(\varphi_t(\gamma))| < \epsilon \mbox{  for all $\gamma \in \mathcal{P}$.}
\end{equation}

Now let $\tilde{u}: \partial \mathbb{D} \To \cG (X \setminus D)$ be a map of the type considered in Definition \ref{definition:strmonpaths}, changing between elements of $\varphi_t(\mathcal{L})$ along generators $\varphi_t(\gamma_{p_1})$,\ldots, $\varphi_t(\gamma_{p_k})$, $\varphi_t(\gamma_{q_l})$, \ldots, $\varphi_t(\gamma_{q_1})$ as in Figure \ref{fig:Fig7}, with extension $u: \mathbb{D} \To X$ as in Definition \ref{definition:strmonpaths}.
Applying Stokes' theorem to the surface $u$, with small balls surrounding the points $u^{-1}(D)$ removed, yields the following formula for the symplectic area of $u$:
\begin{equation} \omega(u) = \sum_{j=1}^k A(\varphi_t(\gamma_{p_j})) - \sum_{j=1}^l A(\varphi_t(\gamma_{q_j})) + \sum_j \ell_j (u \cdot D_j),\end{equation}
where $\ell_j$ is the linking number of $\alpha$ with component $D_j$ of $D$ (compare \cite[Lemma 3.12]{Sheridan2015}).
Applying \eqref{eqn:actionphi}, then the definition of relative monotonicity, we obtain
\begin{align}
\omega(u) & \le  \epsilon(k+l) + \sum_j \ell_j u \cdot D_j  \\
& \le  \epsilon k + \tau \mu(\tilde{u}) + (\tau_k + \epsilon) l,
\end{align}
where $\tau_k < \tau$.
It now suffices to choose $\epsilon>0$ small enough that $\tau_k + \epsilon < \tau$, then choose $t$ large enough that $|A(\varphi_t(\gamma))| < \epsilon$ for all $\gamma \in \mathcal{P}$ (the term $\epsilon k$ is dominated by $ \epsilon' l$ as $l \To \infty$, for any $\epsilon' > 0$).
 \end{proof}

\begin{remark}
\label{remark:nonquasi}
Because $L$ and $\varphi_t(L)$ are Hamiltonian isotopic, their moduli spaces of weak bounding cochains ought to be isomorphic by \cite{fooo}; however, for our purposes in this paper, it suffices simply to know that $L$ and $\varphi_t(L)$ are quasi-isomorphic objects in the exact Fukaya category $\cF(X \setminus D)$ (which is true by \cite{Seidel2008}), essentially by \cite[Lemma 5.3]{Sheridan2015}.
\end{remark}

\begin{remark}
\label{remark:energyfilt}
Lemma \ref{lemma:actioncontraction} would also be important if we wanted to relate the relative Fukaya category to the Fukaya category, and define the latter over a Novikov ring, with weak bounding cochains required to have positive energy (as we must do if our symplectic manifold is not monotone).
That is because the natural embedding
\begin{equation} \cF(X,D) \otimes_R \Lambda \hookrightarrow \cF(X)\end{equation}
is defined by mapping
\begin{align}
 CF^*_{\cF(X,D)} (L_0,L_1) \otimes_R \Lambda &\To CF^*_{\cF(X,D)}(L_0,L_1) \\
p & \mapsto  r^{A(p)} p
\end{align}
(see \cite[\S 8.1]{Sheridan2015}).
This map does not send positive-energy weak bounding cochains to positive-energy weak bounding cochains, because $A(p)$ may be negative.
Nevertheless, Lemma \ref{lemma:actioncontraction} shows that the reverse Liouville flow makes $A(p)$ arbitrarily small; so if we flow a weak bounding cochain in $\cF(X,D)$ sufficiently far, then its image in $\cF(X)$ will have positive energy (of course  the objections of Remark \ref{remark:nonquasi} must also be dealt with).
\end{remark}

\subsection{The disc potential}
\label{subsec:fdp}

\begin{definition}
If $\bm{L}$ is an object of $\cF^c(X)^\oplus$, we will write $\hcM_{weak}(\bm{L})$ for the space defined in Definition \ref{definition:mweak}, and
\begin{equation} \mathfrak{P}: \hcM_{weak}(\bm{L}) \To \C\end{equation}
for the disc potential.
If $\bm{L}$ is instead an object of $\cF^c_m(X,D)^\oplus$, we will write $\hcM_{weak,rel}(\bm{L})$ for this space, and
\begin{equation} \mathfrak{P}_{rel}: \hcM_{weak,rel}(\bm{L}) \To R\end{equation}
for the disc potential.
\end{definition}

\begin{lemma}
\label{lemma:relpredisc}
Let $(X,D)$ be a K\"{a}hler pair, $L$ an object of $\cF_m(X,D)_w$, and $\mathcal{P}$ be a set of generators of $CF^*(L,L)$, such that $(\{L\},\mathcal{P})$ is relatively monotone.
Denote
\begin{equation} A := CF^*_{\cF_m(X,D)_w}(L,L),\end{equation}
and $V \subset A$ the subspace spanned by the generators in $\mathcal{P}$.
Suppose that the conditions of Lemma \ref{lemma:prediscdisc} are satisfied, for $V \subset A$, so that we can define the pre-disc potential
\begin{equation} \mathfrak{P}'_{rel}: V \To R\end{equation}
in $\cF_m(X,D)_w$.

Then there is an embedding $\iota_{rel}: V \hookrightarrow \hcM_{weak,rel}(L)$, so that the diagram
\begin{equation} \xymatrix{
V \ar@{^{(}->}[rr]^-{\iota_{rel}} \ar[rd]_{\mathfrak{P}'_{rel} + w \cdot T}& & \hcM_{weak,rel}(L) \ar[ld]^{\mathfrak{P}_{rel}}\\
& R &}
\end{equation}
commutes, where $T$ is as in Proposition \ref{proposition:cod}.
\end{lemma}
\begin{proof}
This is immediate from Lemma \ref{lemma:prediscdisc}; the only remark to make is that, when we pass from $\cF_m(X,D)_w$ to $\cF^{wbc}_m(X,D)$, we introduce a curvature term $\mu^0 = w \cdot T \cdot e^+$, so
\begin{equation} \mathfrak{P}_{rel} = \mathfrak{P}'_{rel} + w \cdot T.\end{equation}
 \end{proof}

Now let us relate this result about $\cF_m^{wbc}(X,D)$ to $\cF^{wbc}(X)$, following Remark \ref{remark:wbcrel}.
Recall that monotonicity is not a consequence of relative monotonicity, so we impose it as an additional assumption.

\begin{corollary}
\label{corollary:relpredisc2}
Let $(X,D)$ be a K\"{a}hler pair, $L$ an object of $\cF_m(X,D)_w$, and $\mathcal{P}$ be a set of generators of $CF^*(L,L)$, such that $(\{L\},\mathcal{P})$ is both relatively monotone \emph{and} monotone.
Denote
\begin{equation} A := CF^*_{\cF_m(X,D)_w}(L,L),\end{equation}
and $V \subset A$ the subspace spanned by the generators in $\mathcal{P}$.
Suppose that the conditions of Lemma \ref{lemma:prediscdisc} are satisfied, for $V \subset A$.

Now consider $L$ as an object of $\cF^c(X)$: this is possible because $\mathcal{P}$ is monotone (c.f. Remark \ref{remark:wbcrel}).
Denote $A_\C := A \otimes_R \C$ and $V_\C := V \otimes_R \C$.
Then there is an embedding $\iota: V_\C \hookrightarrow \hcM_{weak}(L)$, so that the diagram
\begin{equation} \xymatrix{
V_\C  \ar@{^{(}->}[rr]^-{\iota} \ar[rd]_{\mathfrak{P}' + w} & & \hcM_{weak}(L) \ar[ld]^{\mathfrak{P}}  \\
& \C &.
}
\end{equation}
commutes (and $\mathfrak{P}'$ and $\mathfrak{P}$ converge, i.e., the diagram makes sense).
\end{corollary}
\begin{proof}
Follows from Lemma \ref{lemma:relpredisc}, by setting $\iota:=\iota_{rel} \otimes_R \C$.
Note that
\begin{align}
 \mathfrak{P}'&= \mathfrak{P}'_{rel} \otimes_R 1 \mbox{ and} \\
 \mathfrak{P}&= \mathfrak{P}_{rel} \otimes_R 1.
\end{align}
 \end{proof}

\begin{remark}
The only reason that Corollary \ref{corollary:relpredisc2} makes any reference to the relative Fukaya category is that, in our intended application, the grading hypothesis of Lemma \ref{lemma:prediscdisc} is not satisfied in the monotone Fukaya category, which is only $\Z/2N \Z$-graded.
So we need to use the stronger grading on the relative Fukaya category to prove the existence of $\iota_{rel}$, then tensor with $\C$ to obtain $\iota$.
\end{remark}

\begin{remark}
Let us remark on the relationship between Lemma \ref{lemma:relpredisc} and \cite[Proposition 4.3]{Fukaya2010d}.
One could apply Lemma \ref{lemma:relpredisc} to $V = C^1(L) \subset C^*(L) = A$, where $L$ is a Lagrangian torus fibre inside a symplectic toric manifold, which we consider relative to the boundary toric divisor.
Then the output of $\mu^s: V^{\otimes s} \To A$ has degree $2-\mu$, where $\mu$ is the Maslov index of the homotopy class of the holomorphic disc.
The holomorphic discs with $\mu=0$ correspond to the standard Morse $A_\infty$ structure on $C^*(L)$ \cite{Fukaya1997}; $\mu^2(v,v)$ vanishes by antisymmetry, and the higher products can be made to vanish by applying a formal diffeomorphism, because the cohomology of the torus is formal.
The holomorphic discs with $\mu=2$ contribute an output in degree $2-2=0$, which must be a multiple of $e_L$ (because $CF^0(L,L)$ is generated by $e_L$).
One easily checks that $CC^{\le 0}(V,A) = \C \cdot e_L$.
The only hypothesis of Lemma \ref{lemma:relpredisc} which is not satisfied is that $V \subset A$ is not monotone: in fact, if we tried to prove monotonicity, we would be forced to choose $\tau_k > \tau$ for \eqref{eqn:Fkl}  to be satisfied.
In some sense, $V \subset A$ only `just' fails to be monotone.
This forces us to work over a Novikov ring, so that convergence of the Maurer--Cartan equation follows from Gromov compactness.
If we do that, then all of the hypotheses of Lemma \ref{lemma:relpredisc} are satisfied, and we obtain the analogue of \cite[Proposition 4.3]{Fukaya2010d} (albeit under significantly more restrictive technical hypotheses).
\end{remark}

\subsection{Closed--open and open--closed string maps}
\label{subsec:cowbc}

The closed--open map extends to the homotopy-unital version of the monotone Fukaya category: namely, we have a homomorphism of $\C$-algebras
\begin{equation} \CO: QH^*(X) \To HH^*(\cF^c(X)).\end{equation}
It is defined by counting the same kind of pseudoholomorphic discs as before, but modifying to allow $f$ and $e^+$ to be inputs as in \S \ref{subsec:homungeom}.
Here, $HH^*(\cF^c(X))$ is defined via the one-pointed Hochschild cochain complex (see \S \ref{subsec:curved}).

Furthermore, the usual argument using forgetful maps (as in \S \ref{subsec:homungeom}) shows that
\begin{equation} \CO(\beta)(\ldots, e_L^+, \ldots) = 0,\end{equation}
so in fact $\CO$ factors through the normalized Hochschild cochain complex $\overline{HH}^*(\cF^c(X))$ (see \S \ref{subsec:curved}).

It follows that there is a homomorphism
\begin{equation} \CO^{wbc}:=\Psi \circ \CO: QH^*(X) \To \overline{HH}^*(\cF^{wbc}(X)),\end{equation}
where $\Psi$ is defined as in \S \ref{subsec:hhwbc}.
This composition converges by monotonicity of the weak bounding cochain, by the same argument as in Lemma \ref{lemma:conv} (even though the map $\Psi$ need not converge).
$\CO^{wbc}$ is also a homomorphism of $\C$-algebras, as both $\CO$ and $\Psi$ are (the argument for $\CO$ in the presence of homotopy units is a minor extension of that given in the proof of Proposition \ref{proposition:coalg}).

It follows that, for every $w \in \C$, we obtain a homomorphism of $\C$-algebras
\begin{equation} \CO^{wbc,w}: QH^*(X)  \To HH^*(\cF^{wbc}(X)_w),\end{equation}
by composing $\CO^{wbc}$ with the restriction map
\begin{equation} \overline{HH}^*(\cF^{wbc}(X)) \To \overline{HH}^*(\cF^{wbc}(X)_w) \cong HH^*(\cF^{wbc}(X)_w)\end{equation}
(see Lemma \ref{lemma:curved}).
We will abbreviate $\CO^{wbc,w}$ by $\CO$ where we feel no confusion is possible, and we introduce the notation
\begin{equation} \CO^0: QH^*(X) \To HF^*((\bm{L},\alpha),(\bm{L},\alpha))\end{equation}
for the algebra homomorphism given by the length-zero component of $\CO^{wbc,w}$ as in Definition \ref{definition:co0}, for any object $(\bm{L},\alpha)$ of $\cF^{wbc}(X)_w$.

\begin{remark}
\label{remark:nounits}
We have not checked that $\CO^{wbc,w}$ is a \emph{unital} algebra homomorphism.
The argument of Lemma \ref{lemma:counital} does not go through immediately in the presence of homotopy units, because it requires making a choice of translation-invariant perturbation data, which makes no sense when the input at $p_{in}$ is $f_L$.
\end{remark}

The fact that $\CO^{wbc,w}$ need not be unital leads to some inconveniences.
We circumvent them by restricting to a certain class of objects:

\begin{definition}
\label{definition:cow0u}
An object $(\bm{L},\alpha)$ of $\cF^{wbc}(X)_w$ will be called \emph{$\CO^0$-unital} if the map
\begin{equation} \CO^0: QH^*(X)_{w'} \To HF^*((\bm{L},\alpha),(\bm{L},\alpha))\end{equation}
is a \emph{unital} algebra homomorphism for $w' = w$, and vanishes for $w' \neq w$.
We denote by $\cF^{wbc,u}(X)_w$ the full subcategory of $\CO^0$-unital objects of $\cF^{wbc}(X)_w$.
Note that if $L$ is a monotone Lagrangian, then the object $(L,0)$ is $\CO^0$-unital by Proposition \ref{proposition:coeigsplit}: thus we have
\begin{equation} \cF(X)_w \subset \cF^{wbc,u}(X)_w \subset \cF^{wbc}(X)_w.\end{equation}
\end{definition}

By similar arguments to those just given for the closed--open map, the open--closed map extends to the homotopy-unital version of the monotone Fukaya category, and satisfies
\begin{equation} \OC(a_s \otimes \ldots \otimes e_L^+ \otimes \ldots) = 0,\end{equation}
hence factors through the normalized Hochschild chain complex $\overline{HH}_*(\cF^c(X))$ (see \S \ref{subsec:curved}).

\begin{remark}
The incoming marked point corresponding to the first `slot' in the Hochschild chain is treated slightly differently from the others: when that input is labelled $f$, the perturbation data at $\rho = 1$ need not be pulled back via a forgetful map, and their count need not be zero.
Indeed, counting the pseudoholomorphic discs at $\rho = 1$ gives the coefficient of $\OC(e_L^+ \otimes \ldots)$.
\end{remark}

It follows that we can define a homomorphism
\begin{equation} \OC^{wbc}:=\OC \circ \Phi: \overline{HH}_*(\cF^{wbc}(X)) \To QH^{*+n}(X).\end{equation}
This composition converges by the same argument as in Lemma \ref{lemma:conv} (even though the map $\Phi$ need not converge).
$\OC^{wbc}$ is also a homomorphism of $QH^*(X)$-modules, using the fact that $\OC$ is a homomorphism of $QH^*(X)$-modules (the argument in the presence of homotopy units is a minor extension of that given in the proof of Proposition \ref{proposition:ocmod}), and that $\Phi$ is a homomorphism of $\overline{HH}^*(\cF^c(X))$-modules.

It follows that, for every $w \in \C$, we obtain a homomorphism of $QH^*(X)$-modules,
\begin{equation} \OC^{wbc,w}: HH_*(\cF^{wbc}(X)_w) \To QH^{*+n}(X),\end{equation}
by composing $\OC^{wbc}$ with the inclusion map
\begin{equation} HH_*(\cF^{wbc}(X)_w) \cong \overline{HH}_*(\cF^{wbc}(X)_w) \To \overline{HH}_*(\cF^{wbc}(X)) \end{equation}
(see Lemma \ref{lemma:curved}).
We will abbreviate $\OC^{wbc,w}$ by $\OC$ where we feel no confusion is possible.

\subsection{The split-generation criterion}

For the purposes of this section, we fix $w \in \C$, and abbreviate
\begin{equation} \cF_w := \cF^{wbc}(X)_w\end{equation}
and
\begin{equation} \cF^u_w := \cF^{wbc,u}(X)_w\end{equation}
 (see Definition \ref{definition:cow0u}).
We will also abbreviate the notation for objects of $\cF_w$, writing $K := (\bm{K},\alpha)$ and $L := (\bm{L},\beta)$.

In this section, we prove analogues of the results in \S \ref{subsec:splitgen}, incorporating weak bounding cochains into the picture.
The proofs are largely the same: first one incorporates homotopy units into all operations, then one incorporates weak bounding cochains by inserting arbitrarily many copies of the appropriate weak bounding cochain along every boundary component in every diagram in \S \ref{sec:monfuk}.
We will draw attention to the places where the proofs differ: in particular, we are only able to prove some of the results for $\cF^u_w$, and do not know how to prove them for $\cF_w$.

First we give an analogue of Lemma \ref{lemma:cardy}:

\begin{lemma}
\label{lemma:cardywbc}
For any object $K$ of $\cF_w$, there exists a (strictly unital) bimodule homomorphism
\begin{equation} \Delta: (\cF_w)_\Delta \To \mathcal{Y}^l_K \otimes \mathcal{Y}^r_K[n],\end{equation}
such that the following diagram commutes up to a sign $(-1)^{\frac{n(n+1)}{2}}$:
\begin{equation}
\xymatrixcolsep{5pc} \xymatrix{
HH_*(\cF_w)[-n] \ar[r]^-{\OC} \ar[d]^{HH_*(\Delta)} & QH^*(X) \ar[d]^{\CO^0} \\
HH_*(\cF_w,\mathcal{Y}^l_K \otimes \mathcal{Y}^r_K) \ar[r]^-{H^*(\mu)} & HF^*(K,K).
} \end{equation}
\end{lemma}
\begin{proof}
The proof is analogous to the proof of Lemma \ref{lemma:cardy}, modified to include homotopy units and weak bounding cochains.
 \end{proof}

Now we give an analogue of Corollary \ref{corollary:gen2}:

\begin{corollary}
\label{corollary:gen2wbc}
Suppose that $\cG_w \subset \cF_w$ is a full subcategory such that the composition
\begin{equation} 
\label{eqn:gprojw}
HH_*(\cG_w) \xrightarrow{\OC} QH^*(X) \xrightarrow{proj_w} QH^*(X)_w \end{equation}
contains $e_w$ in its image.
Then any $\CO^0$-unital object $K$ of $\cF^u_w$ is split-generated by $\cG_w$.
\end{corollary}

Now we give an analogue of Lemma \ref{lemma:weakcy}:

\begin{lemma}
\label{lemma:weakcywbc}
The class $[\phi] \in HH_n(\cF^u_w)^\vee$ given by
\begin{equation} [\phi](\psi) = \langle \OC(\psi),e \rangle \end{equation}
is an $n$-dimensional weak proper Calabi--Yau structure on $\cF^u_w$.
\end{lemma}
\begin{proof}
We must prove that $[\phi]$ is homologically non-degenerate.
The proof follows that of Lemma \ref{lemma:weakcy} with slight modifications.
In particular, the quasi-inverse to the map
\begin{align}
HF^*(K,L) & \To  HF^{n-*}(K,L)^\vee \\
p & \mapsto  \langle \OC(\mu^2(p, -)), e\rangle
\end{align}
is given by contraction with the element
\begin{equation}
\Delta^{wbc}(\CO^0(e_w))  \in  HF^*(K,L) \otimes HF^*(L,K).
\end{equation}
The argument given in the proof of Lemma \ref{lemma:weakcy} shows that the composition of these two maps is equal to the map
\begin{align}
HF^*(K,L) & \To HF^*(K,L) \\
p & \mapsto  \mu^2(\CO^0(e \star e_w),p).
\end{align}
In particular, as $e \star e_w = e_w$, and $L$ is $\CO^0$-unital by the definition of $\cF^u_w$, the composition of the two maps is the identity.
Similarly, the composition in the opposite order is the identity, so the map is an isomorphism, and $[\phi]$ is homologically non-degenerate.
 \end{proof}

Now we give an analogue of Proposition \ref{proposition:phico} and Corollary \ref{corollary:cowdual}:

\begin{lemma}
\label{lemma:phicowbc}
The following diagrams commute:
\begin{equation}
\xymatrixcolsep{5pc}\xymatrix{
QH^*(X) \ar[r]^-{\alpha \mapsto \langle \alpha,-\rangle}_-{\cong} \ar[d]^{\CO} & QH^*(X)^\vee[-2n] \ar[d]^{\OC^\vee} \\
HH^*(\cF^u_w) \ar[r]^-{-\cap [\phi]}_-{\cong} & HH_*(\cF^u_w)^\vee [-n].
} \end{equation}
and
\begin{equation}
\xymatrixcolsep{5pc}\xymatrix{
QH^*(X)_w \ar[r]^-{\alpha \mapsto \langle \alpha,-\rangle}_-{\cong} \ar[d]^{\CO}& QH^*(X)_w^\vee[-2n] \ar[d]^{(proj_w \circ \OC)^\vee} \\
HH^*(\cF^u_w) \ar[r]^-{-\cap [\phi]}_-{\cong} & HH_*(\cF^u_w)^\vee [-n].
} \end{equation}
\end{lemma}
\begin{proof}
The proof of the first commutative diagram follows that of Proposition \ref{proposition:phico}: in particular, the only input to that proof was the fact that $\CO$ is an algebra homomorphism, $\OC$ is a $QH^*(X)$-module homomorphism, $QH^*(X)$ is a Frobenius algebra, and that $[\phi]$ is homologically non-degenerate, all of which remain true in the present setting. 
The second commutative diagram is immediate from the first, by restricting to the subspace $QH^*(X)_w \subset QH^*(X)$ of the top left-hand element of the square.
 \end{proof}

Now we give an analogue of Corollary \ref{corollary:splitgen}:

\begin{corollary}
\label{corollary:splitgenwbc}
If $\cG_w \subset \cF^u_w$ is a full subcategory, and the map
\begin{equation} \CO: QH^*(X)_w \To HH^*(\cG_w) \end{equation}
is injective, then $\cG_w$ split-generates $\cF^u_w$.
\end{corollary}
\begin{proof}
Follows from Lemma \ref{lemma:phicowbc} and Corollary \ref{corollary:gen2wbc}, by the proof of Corollary \ref{corollary:splitgen}.
 \end{proof}

Finally we give an analogue of Corollary \ref{corollary:semisimpgen}:

\begin{corollary}
\label{corollary:semisimpgenwbc}
Suppose that $QH^*(X)_w$ is one-dimensional, that $L$ is an object of $\cF^u_w$, and that $HF^*(L,L) \neq 0$.
Then $L$ split-generates $\cF^u_w$.
\end{corollary}
\begin{proof}
Follows from Corollary \ref{corollary:splitgenwbc}, by the proof of Corollary \ref{corollary:semisimpgen}.
 \end{proof}

\subsection{Proving $\CO^0$-unitality}

In this section we establish the results needed to prove that an object $(L,\alpha)$ of $\cF^{wbc}(X)_w$ is $\CO^0$-unital, in the sense of Definition \ref{definition:cow0u}.

\begin{lemma}
\label{lemma:cowbc0un}
Suppose that we are in the situation of Corollary \ref{corollary:relpredisc2}, with $v \in V_\C$ and $\alpha = \iota(v) \in \hcM_{weak}(L)$, such that $\mathfrak{P}(\alpha) = w$.
Then the length-$0$ part of the closed--open string map
\begin{equation} \CO^0: QH^*(X) \To HF^*((L,\alpha),(L,\alpha)) \end{equation}
is a \emph{unital} algebra homomorphism.
\end{lemma}
\begin{proof}
We first prove the corresponding result for the relative Fukaya category.
Suppose $v = u \otimes_R 1$, with $u \in V$, and set
\begin{equation} \alpha_{rel} := \iota_{rel}(u) := u + \mathfrak{P}'_{rel}(u) \cdot f_L\end{equation}
(compare the proof of Lemma \ref{lemma:prediscdisc}), so that
\begin{equation} \alpha = \alpha_{rel} \otimes_R 1.\end{equation}
Then by definition,
\begin{equation}
\label{eqn:coxde}
\CO^0_{X,D}(e) := \sum_{i \ge 0} \CO_{X,D}(e)(\underbrace{\alpha_{rel},\ldots,\alpha_{rel}}_i).
\end{equation}
Now because $\CO_{X,D}$ is $\bm{G}(X,D)$-graded, we have
\begin{equation}\CO_{X,D}(e) \in CC^0(A) \end{equation}
in the notation of Corollary \ref{corollary:relpredisc2}.
Now by the condition
\begin{equation} CC^{\le 0}(V,A) \cong \C \cdot e_L\end{equation}
of Lemma \ref{lemma:prediscdisc}, all of the terms in \eqref{eqn:coxde} vanish except for the $i=0$ term, which is $e_L$ by Remark \ref{remark:co0un}.
We also have
\begin{equation} \mu^1_\alpha(f_L) = e_L^+ - e_L,\end{equation}
again using the condition $CC^{\le 0}(V,A) = \C \cdot e_L$ (see the proof of Lemma \ref{lemma:prediscdisc}).
Therefore, $\CO^0_{X,D}(e)$ is cohomologous to $e_L^+$, so the map is unital.
It follows that
\begin{equation} \CO^0 = \CO^0_{X,D} \otimes_R 1\end{equation}
is too.
 \end{proof}

\begin{remark}
Note that Lemma \ref{lemma:cowbc0un} does not establish $\CO^0$-unitality of $(L,\alpha)$: it establishes unitality of the map
\begin{equation} \CO^0: QH^*(X) \To HF^*((L,\alpha),(L,\alpha)),\end{equation}
but not of the map from $QH^*(X)_w$, nor does it prove vanishing of the map when restricted to the other eigenspaces.
\end{remark}

Next we will prove an analogue of Lemma \ref{lemma:c1u0}.
First we need a preliminary result, which was explained to the author by Ivan Smith:

\begin{lemma}
\label{lemma:cowbcc2}
In the situation of Corollary \ref{corollary:relpredisc2}, suppose furthermore that $2c_1(X \setminus D) = 0$ in $H^2(X \setminus D;\Z)$.
Then if  $v \in V_\C$ and $\alpha = \iota(v) \in \hcM_{weak}(L)$ is the corresponding weak bounding cochain, the map
\begin{equation} \CO^0: QH^*(X) \To HF^*((L,\alpha),(L,\alpha))\end{equation}
satisfies
\begin{equation}
\label{eqn:co0c1wbc} \CO^0(2c_1) = \left( \sum_{s \ge 0} (2-s) \cdot \mathfrak{P}^s(\alpha) \right) \cdot e_L^+,
\end{equation}
where $\mathfrak{P}^s$ denotes the length-$s$ component of $\mathfrak{P}$ (i.e., the part corresponding to $\mu^s$).
\end{lemma}
\begin{proof}
The hypothesis that $2c_1(X \setminus D) = 0$ implies that $X\setminus D$ admits a quadratic complex volume form $\eta$ (i.e., a nowhere-vanishing section of $\lambda^{top}_\C(T(X \setminus D))^{\otimes 2}$, see \cite[\S 12a]{Seidel2008}).
Such an $\eta$ induces a map
\begin{align}
\cG(X \setminus D) & \To  S^1 \\
L & \mapsto  \eta(e_1 \wedge \ldots \wedge e_n \otimes e_1 \wedge \ldots \wedge e_n),
\end{align}
where $\{ e_1,\ldots,e_n \}$ is an $\R$-basis for $L$.
The induced map on $H_1$ defines a morphism of grading data
\begin{equation} \bm{q}_\eta: \bm{G}(X,D) \To \Z\end{equation}
(see \cite[Definition 3.6]{Sheridan2015}).
Hence, we can define the $\Z$-graded, curved, $\bm{q}_{\eta,*} R$-linear $A_\infty$ category $\bm{q}_{\eta,*}\cF^c_m(X,D)$.

We define $p_j \in \Z$ to be the degree of the generator $r_j$ in $\bm{q}_{\eta,*}R$ (it can be thought of as the `order of pole of $\eta$ about divisor $D_j$').
It follows from \cite[Lemma 3.19]{Sheridan2015} that
\begin{equation} 2c_1 = \sum_j p_j [D_j].
\end{equation}
As a consequence,  $\CO^0_{X,D}(2c_1) $ is equal to
\begin{align}
\label{eqn:2lcow} \CO^0_{X,D}\left( \sum_{j} p_j [D_j] \right)  &= \sum_{j,i} p_j \cdot \CO_{X,D}([D_j])(\underbrace{\alpha_{rel},\ldots,\alpha_{rel}}_i) \\
\label{eqn:3lcow}&= \sum_{j,i} p_j \cdot \left(\CO_{X,D}(e) \cup r_j \del{\mu^*}{r_j} \right.+ \\ 
& \phantom{{}=\sum} \left. r_j \del{ (w \cdot T)}{r_j} \cdot \CO_{X,D}(e)  + \delta(H) \right)(\underbrace{\alpha_{rel},\ldots}_i).
\end{align}
Here, \eqref{eqn:2lcow} is the definition of $\CO^0$ in the presence of weak bounding cochains, and \eqref{eqn:3lcow} follows as in the proof of Proposition \ref{proposition:cod}, modified to take into account the homotopy units (compare Remark \ref{remark:nounits}).

We observe that the final term of \eqref{eqn:3lcow} (i.e., the term $\delta(H)(\alpha_{rel},\ldots)$) is exact, so we may ignore it.
This follows from the fact that the map $\Psi$ defined in \eqref{eqn:Psi} is (formally) a map of chain complexes, and that the moduli spaces in the proof of Proposition \ref{proposition:cod} can be arranged so that $H$ lies in the normalized Hochschild cochain complex $\overline{CC}^*(\cF^u)$.

Now, consider a single pseudoholomorphic disc $u$, contributing a map
\begin{align}
 \mu_u: A^{\otimes s} & \To  A \,\,\,\,\,\,\,\mbox{ which sends} \\
y_s \otimes \ldots \otimes y_1 & \mapsto  \pm r^{u \cdot D} y_0,
\end{align}
so that the $A_\infty$ structure maps are given by a sum over all discs:
\begin{equation}
\mu^* = \sum_u \mu_u.
\end{equation}
By the definition of a $\Z$-graded $A_\infty$ category, we now have
\begin{equation}
\label{eqn:zgrads} d\left(r^{u \cdot D}\right) + d(y_0) = 2-s + \sum_{i=1}^s d(y_i),\end{equation}
where $d$ denotes the $\Z$-grading induced by $\bm{q}_\eta$.

It follows that the contribution of $u$ to the first and second terms in \eqref{eqn:3lcow} (before inputting $\alpha_{rel}$) is
\begin{align}
\label{eqn:4lcow} \sum_j p_j (u \cdot D_j) \mu_u &= d(r^{u \cdot D}) \mu_u \\
\label{eqn:5lcow}&=  \left( 2-s -d(y_0) + \sum_{i=1}^s d(y_i) \right) \mu_u
\end{align}
by \eqref{eqn:zgrads}.

Now we consider the Euler element $\tau \in \overline{CC}^*(A)$, which is the element of length $1$ such that
\begin{equation} \tau(a) = d(a) a\end{equation}
for $a \in A$ pure of degree $d(a)$ (note that $e^+_L$ has degree $0$, so this is a normalized Hochschild cochain).
We observe that the contribution of $\mu_u$ to $\delta(\tau) = [\mu^*,\tau]$ is precisely
\begin{equation} \left( -d(y_0) + \sum_{i=1}^s d(y_i) \right) \mu_u\end{equation}
(compare \cite[Equation (3.14)]{Seidel2003}).
It follows by the same argument as we gave for $\delta(H)$ that
\begin{equation}
\label{eqn:deleuler} \delta(\tau)(\alpha_{rel},\ldots,\alpha_{rel})\end{equation}
is exact in $CF^*((L,\alpha_{rel}),(L,\alpha_{rel}))$.

Now we substitute \eqref{eqn:5lcow} into \eqref{eqn:3lcow}; cancelling the exact term $\delta(H)(\ldots)$, and the exact term $\delta(\tau)(\ldots)$ in \eqref{eqn:5lcow}, and substituting in $e_L^+$ for $\CO_{X,D}(e)(\alpha_{rel},\ldots)$ by Lemma \ref{lemma:cowbc0un}, we obtain \eqref{eqn:co0c1wbc}.
 \end{proof}

\begin{corollary}
\label{corollary:cowbcc3}
In the situation of Lemma \ref{lemma:cowbcc2}, suppose furthermore that $v \in V_\C$ is a critical point of the disc potential $\mathfrak{P} \circ \iota: V_\C \To \C$, and $\alpha = \iota(v) \in \hcM_{weak}(L)$ is the corresponding weak bounding cochain, then the map
\begin{equation} \CO^0: QH^*(X) \To HF^*((L,\alpha),(L,\alpha))\end{equation}
satisfies
\begin{equation}
\label{eqn:co0c1crit} \CO^0(c_1) = \mathfrak{P}(\alpha) \cdot e_L^+,
\end{equation}
and $(L,\alpha)$ is $\CO^0$-unital.
\end{corollary}
\begin{proof}
When $v$ is a critical point of $\mathfrak{P} \circ \iota$, we have
\begin{align}
\sum_{s} s \mathfrak{P}^s \circ \iota (v) &= \sum_i v_i \del{(\mathfrak{P} \circ \iota)}{v_i} \\
&= 0,
\end{align}
so \eqref{eqn:co0c1wbc} follows immediately from Lemma \ref{lemma:cowbcc2}.

Combining with Lemma \ref{lemma:cowbc0un}, we have that $\CO^0(c_1-w' \cdot e) = (w-w') \cdot e_{(L,\alpha)}$ for any $w'$. A simplified version of the proof of Proposition \ref{proposition:coeigsplit} now applies: if $\beta \in QH^*(X)_{w'}$, then $(c_1 - w' \cdot e)^{\star k} \star \beta = 0$ for some $k$, from which it follows that $(w-w')^k \cdot \CO^0(\beta) = 0$ in $HF^*((L,\alpha),(L,\alpha))$, so $\CO^0(\beta) = 0$. 
It follows that $\CO^0$ vanishes when restricted to $QH^*(X)_{w'}$ for all $w' \neq w$, and (combining with Lemma \ref{lemma:cowbc0un}) that it is unital when restricted to $QH^*(X)_w$: so $(L,\alpha)$ is $\CO^0$-unital.
 \end{proof}

\begin{remark}
Note that we need both the relative and non-relative Fukaya categories to get the full strength of Corollary \ref{corollary:cowbcc3}: $\mathfrak{P} \circ \iota$ may have critical points when $\mathfrak{P}_{rel} \circ \iota_{rel}$ has none, whereas the grading arguments used to prove Lemma \ref{lemma:cowbcc2} only work for the relative Fukaya category.
\end{remark}

\section{Deformation theory}
\label{sec:deftheory}

In this section, we establish some of the algebraic results needed to prove homological mirror symmetry.
In \S \ref{subsec:cliff}, we recall some basic results about Clifford algebras and modules over them; these are used to prove homological mirror symmetry over the small eigenvalues (Theorem \ref{theorem:small}).
\S \ref{subsec:class} and \S \ref{subsec:ana} establish a `recognition theorem' for the $A_\infty$ algebra that appears on both sides of homological mirror symmetry over the big eigenvalue (Theorem \ref{theorem:big}).
\S \ref{subsec:comphh} computes various versions of Hochschild cohomology for the type of $A_\infty$ algebra that is `recognized' by the results of the preceding two sections: this is used to establish the generation criterion for the big component of the Fukaya category, and also to extract relations in quantum cohomology from the Hochschild cohomology of the relative Fukaya category.

\subsection{Clifford algebras}
\label{subsec:cliff}

The following material is standard, but we wish to make some points clear.

\begin{definition}
\label{definition:cliff}
Let $R$ be a ring, $V$ a free $R$-module, $Q: V^{\otimes 2} \To R$ an $R$-linear quadratic form on $V$.
Define the corresponding \emph{Clifford algebra} $\Cl(Q)$ to be the quotient of the $R$-linear tensor algebra $T(V)$ by the two-sided ideal generated by the elements $v \otimes v - Q(v) \cdot 1$, for $v \in V$.
We will regard $\Cl(Q)$ as a $\Z/2\Z$-graded $R$-algebra, where $V$ has odd degree.
When $R = \C$, $V$ is $n$-dimensional, and the quadratic form $Q$ is non-degenerate, we will write $\Cl_n$ for the isomorphism class of $\Cl(Q)$.
\end{definition}

\begin{lemma}
\label{lemma:cliff} (see, e.g., \cite[Theorem 4.3]{Lawson1989})
We introduce the $\Z/2\Z$-graded vector space $U := \C[0] \oplus \C[1]$.
There are isomorphisms of $\Z/2\Z$-graded  algebras
\begin{align}
\Cl_{2k} &\cong \mathrm{End}\left( U^{\otimes k}\right),\\
\Cl_{2k+1} &\cong \mathrm{End} \left(U^{\otimes k} \right) \otimes \Cl_1
\end{align}
(there is also an isomorphism $\Cl_1 \cong \C \oplus \C$, but it is not $\Z/2\Z$-graded).
\end{lemma}

\begin{corollary}
\label{corollary:morita}
$\Cl_{2k}$ is Morita-equivalent to $\C$ as a $\Z/2\Z$-graded algebra, and $\Cl_{2k+1}$ is Morita-equivalent to $\Cl_1$ as a $\Z/2\Z$-graded algebra.
\end{corollary}

\begin{corollary}
\label{corollary:hhvans}
The $\Z/2\Z$-graded Hochschild cohomology of $\Cl_{n}$ is
\begin{equation} HH^{s+t}(\Cl_n)^s \cong \left\{ \begin{array}{ll}
					\C & \mbox{ if $s=0$ and $n$ is even (in degree $t=0$)}\\
					\Cl_1^t & \mbox{ if $s=0$ and $n$ is odd} \\
					0 & \mbox{ if $s>0$.}
				\end{array}\right.
\end{equation}
\end{corollary}
\begin{proof}
By Corollary \ref{corollary:morita} and Morita-invariance of Hochschild cohomology for $\Z/2\Z$-graded algebras (compare \cite[1.5.6]{Loday1998}; the only difference is the signs), it suffices to check the cases $n=0$ and $n=1$.
$n=0$ is trivial, as $\Cl_0 \cong \C$.
When $n=1$, we observe that $\Cl_1$ is projective as a $\Z/2\Z$-graded $\fmod{\Cl_1}{\Cl_1}$ bimodule.
This follows as the map
\begin{align}
\Cl_1 & \To  \Cl_1 \otimes \Cl_1^{op} \\
1 & \mapsto 1 \otimes 1 + \theta \otimes \theta
\end{align}
is a $\Z/2\Z$-graded inclusion of $\Cl_1$ as a direct summand of $\Cl_1 \otimes \Cl_1^{op}$.
Therefore,
\begin{align}
HH^{s+t}(\Cl_1)^s &\cong \mathrm{Ext}^s_{\fmod{\Cl_1}{\Cl_1}}\left(\Cl_1,\Cl_1[t]\right)\\
& \cong \left\{ \begin{array}{ll}
														\mathrm{Hom}^t_{\fmod{\Cl_1}{\Cl_1}}(\Cl_1,\Cl_1) & \mbox{ if $s=0$,}\\
														0 & \mbox{ if $s>0$} \end{array} \right.\\
&\cong \left\{ \begin{array}{ll}
														Z(\Cl_1)^t & \mbox{ if $s=0$,}\\
														0 & \mbox{ if $s>0$.}\\
													\end{array} \right.
\end{align}
Because $\Cl_1$ is commutative, the result follows.
 \end{proof}

We now recall some terminology (compare \cite{Thomas2001}):

\begin{definition}
A $\Z/2\Z$-graded $A_\infty$ algebra $\cA$ is called \emph{formal} if it is $A_\infty$ quasi-isomorphic to its cohomology algebra.
A $\Z/2\Z$-graded associative algebra $A$ is called \emph{intrinsically formal} if any $\Z/2\Z$-graded $A_\infty$ algebra with cohomology algebra $A$ is formal.
\end{definition}

\begin{corollary}
\label{corollary:intform}
$\Cl_n$ is intrinsically formal.
\end{corollary}
\begin{proof}
By \cite[Lemma 1.9]{Seidel2008}, $A$ is intrinsically formal if $HH^2(A)^{s} = 0$ for all $s\ge 3$ (note that our grading conventions differ from \cite{Seidel2008}: our $HH^{s+t}(A)^s$ is isomorphic to Seidel's $HH^{s+t}(A)^t$).
Therefore, $\Cl_n$ is intrinsically formal by Corollary \ref{corollary:hhvans}.
 \end{proof}

\begin{corollary}
\label{corollary:cliffcat}
If $\cA$ is a $\C$-linear, $\Z/2\Z$-graded $A_\infty$ category which is split-generated by an object $X$, and there is a $\Z/2\Z$-graded isomorphism on the level of cohomology:
\begin{equation} \mathrm{Hom}^*(X,X) \cong \Cl_n,\end{equation}
then there is an $A_\infty$ quasi-equivalence
\begin{equation}
\label{eqn:dclmod}
 D^\pi \cA \cong \left\{ \begin{array}{ll}
						D^b(\C) & \mbox{ if $n$ is even} \\
						D^\pi(\Cl_1) & \mbox{ if $n$ is odd}.
					\end{array} \right.
\end{equation}
\end{corollary}
\begin{proof}
By Lemma \ref{lemma:cliff}, $X$ admits a ($\Z/2\Z$-graded) formal direct summand $Y$, whose endomorphism algebra is
\begin{equation} \mathrm{Hom}^*(Y,Y) \cong \left\{ \begin{array}{ll}
						\C & \mbox{ if $n$ is even} \\
						\Cl_1 & \mbox{ if $n$ is odd}.
						\end{array} \right. \end{equation}
Furthermore, $X$ is quasi-isomorphic to $Y^{\oplus n}$: so $Y$ also split-generates $\cA$.
The result follows, by \cite[Corollary 4.9]{Seidel2008}.
 \end{proof}

To clarify the notation: on the right-hand side of \eqref{eqn:dclmod}, `$\C$' and `$\Cl_1$' denote the $A_\infty$ categories with a single object, with the indicated endomorphism algebra (which has $\mu^s = 0$ for $s \neq 2$).
`$D^b(\cA)$' and `$D^\pi(\cA)$' denote the derived and split-closed derived $A_\infty$ categories of a $\Z/2\Z$-graded $A_\infty$ category $\cA$, in the sense of \cite{Seidel2008} (in whose notation they would be written `$Tw(\cA)$' and `$\Pi(Tw(\cA))$' respectively).
In particular, we note that $D^b(\C)$ is obviously split-closed (it is quasi-equivalent to the category of $\Z/2\Z$-graded finite-dimensional complex vector spaces), so $D^\pi(\C) \cong D^b(\C)$.

\subsection{A classification result for $A_{\infty}$ algebras}
\label{subsec:class}

Let $\bm{G}$ be a grading datum, let
\begin{equation} \widetilde{R} \cong \C [ r_1, \ldots , r_n ]\end{equation}
be a polynomial ring equipped with a $\bm{G}$-grading, and let $R$ be its $\bm{G}$-graded completion with respect to the order filtration.
We denote by $R^j$ the part of $R$ of order $j \in \Z_{\ge 0}$.

Let $\cA = (A,\mu_0)$ be a $\bm{G}$-graded minimal $A_{\infty}$ algebra over $\C$.
A \emph{$\bm{G}$-graded minimal deformation of $\cA$ over $R$} is a minimal $\bm{G}$-graded $R$-linear $A_\infty$ algebra $\scrA = (A \otimes R,\mu)$ that specialises to $\cA$ when we set all $r_j=0$ (see \cite[Definition 2.74]{Sheridan2015}).
In other words, it is an element $\mu \in CC^2(A,A \otimes R)$ satisfying the $A_\infty$ relations $\mu \circ \mu = 0$; minimality means that its length-$0$ and length-$1$ components $\mu^0$ (the curvature) and $\mu^1$ (the differential) vanish; and if we expand order-by-order in $R$, $\mu = \mu_0+\mu_1+\ldots$ where $\mu_j \in CC^2(A,A \otimes R^j)$, then $\mu_0$ coincides with the already-prescribed $\C$-linear $A_\infty$ structure on $A$.

The $0$th-order part of the $A_\infty$ relations says simply that $\mu_0$ satisfies the $A_\infty$ relations.
The first-order part says that $\mu_1 \in CC^2(A,A \otimes R^1)$ is a Hochschild cocycle: the corresponding class $[\mu_1] \in HH^2(\cA,\cA \otimes R^1)$ is called the \emph{first-order deformation class} of the deformation (see \cite[Definition 2.80]{Sheridan2015}).
In fact, because $\mu^0=\mu^1=0$, the Hochschild cochain $\mu$ lies in the \emph{truncated} Hochschild cochain complex $TrCC^*(A,A \otimes R)$: this is a $\Z$-graded complex, with $TrCC^p(A) \subset CC^p(A)$ equal to the subspace of Hochschild cochains of length $s \ge p$ (see \cite[Definition 2.31]{Sheridan2015}).
Thus, the first-order deformation class defines a class
\begin{equation} [\mu_1] \in TrHH^2_{\bm{G}}(\cA,\cA \otimes R^1).\end{equation}

Now suppose that $\scrA_1$ and $\scrA_2$ are $\bm{G}$-graded minimal deformations of $\cA$ over $R$: we recall that a \emph{formal diffeomorphism} between the two is an $A_\infty$ homomorphism $F: \scrA_1 \dashrightarrow \scrA_2$ whose linear term $F^1: \scrA_1 \to \scrA_2$ is an isomorphism.
We can regard $F$ as an element of $CC^1(A,A \otimes R)$: in fact, because the `curvature' term $F^0$ vanishes, we can regard $F$ as an element of $TrCC^1(A,A \otimes R)$.
We recall that formal diffeomorphisms have the following convenient property: given $\scrA_1=(A \otimes R,\mu)$ and $F \in TrCC^1(A,A \otimes R)$ such that $F^1$ is an isomorphism, there is a \emph{unique} $A_\infty$ structure $F_* \mu$ on $A \otimes R$ such that $F$ is an $A_\infty$ homomorphism from $(A \otimes R,\mu)$ to $(A \otimes R,F_* \mu)$ (see \cite[\S 1c]{Seidel2008}).

\begin{proposition}
\label{proposition:verseasy}
Suppose that
\begin{equation} TrHH^2_{\bm{G}}(\cA,\cA \otimes R^{j}) \cong 0\end{equation}
for $j \ge 2$.
Then any two $\bm{G}$-graded minimal deformations of $\cA$ over $R$, whose first-order deformation classes are equal, are related by a formal diffeomorphism.
\end{proposition}
\begin{proof}
Let $(A \otimes R,\mu)$ and $(A\otimes R,\eta)$ be two such deformations.
We construct, order-by-order, a $\bm{G}$-graded formal diffeomorphism $F \in TrCC^1_{\bm{G}}(A,A \otimes R)$ so that $\mu = F_* \eta$ (we call this the $A_{\infty}$ relation for the purposes of the proof).
We set
\begin{equation} F = F_0 + F_1 + \ldots,\end{equation}
where $F_j \in TrCC^1_{\bm{G}}(\cA,\cA \otimes R^j)$.

We start with $F_0 = \mathrm{id}$.
The order-$0$ $A_{\infty}$ relation holds by definition.
Now suppose, inductively, that we have constructed $F_j$ for all $j \le k-1$, so that
\begin{equation} (\mu - F_* \eta)_j = 0 \mbox{ for all $j \le k-1$.}\end{equation}
We show it is possible to construct $F_k$ so that $(\mu-F_* \eta)_k = 0$.

First, note that
\begin{equation} [ \mu - F_* \eta, \mu + F_* \eta] = 0\end{equation}
by expanding brackets: cross-terms vanish by symmetry and the other terms vanish because $\mu$ and $F_* \eta$ are $A_{\infty}$ structures.
Now note that $(\mu + F_* \eta)_0 = 2\mu_0$ by definition, so the order-$k$ part of this equation says that $(\mu - F_* \eta)_k$ is a Hochschild cochain.

Now we observe that
\begin{equation} (\mu - F_* \eta)_k = \delta(F_k) + D_k,\end{equation}
where $D_k$ are the terms not involving $F_k$.
Our previous argument says that $D_k$ is a Hochschild cochain; in fact the corresponding class
\begin{equation} [D_k] \in TrHH^2_{\bm{G}}(\cA,\cA \otimes R^k)\end{equation}
vanishes.
When $k = 1$, this is true by assumption (the first-order deformation classes of $\mu$ and $\eta$ coincide); when $k \ge 2$, it is true because the Hochschild cohomology group vanishes by assumption.

Therefore, $D_k$ is a Hochschild coboundary, so we can choose $F_k$ to make the order-$k$ $A_{\infty}$ relation hold.
This completes the proof, by induction.
 \end{proof}

\subsection{$A_{\infty}$ algebras of type A$^n_a$}
\label{subsec:ana}

For any $a,n \in \Z$, we define a grading datum
\begin{align}
\label{eqn:gna}
\Z & \To  Y := (\Z \oplus \Z \langle y_1, \ldots, y_n \rangle)/(2(a-n), y_1 + \ldots + y_n), \\
j & \mapsto  j \oplus 0
\end{align}
(the sign of $j \oplus y$ is the sign of $j$).
We denote this grading datum by $\bm{G}^n_a$, following \cite[Example 2.14]{Sheridan2015}.
In this section we will abbreviate $\bm{G}:= \bm{G}^n_1$.
We introduce the morphism of grading data
\begin{align}
\label{eqn:morphp}
\bm{p}: \bm{G}^n_a & \To  \bm{G}, \\
p(y_j) &= 2(1-a) \oplus a y_j.
\end{align}

We define a $\bm{G}$-graded vector space
\begin{equation}U \cong \C \langle u_1, \ldots, u_n \rangle,\end{equation}
where
\begin{equation} deg(u_j) = (-1,y_j)\end{equation}
(compare \cite[Example 2.24]{Sheridan2015}).
We consider the $\bm{G}$-graded exterior algebra
\begin{equation}
\label{eqn:Aext}
A := \Lambda^* U
\end{equation} (compare \cite[Definition 2.27]{Sheridan2015}).

We consider the $\bm{G}$-graded vector space
\begin{equation}V_a := \C \langle r_1, \ldots, r_n \rangle,\end{equation}
where
\begin{equation} deg(r_j) = (2-2a, ay_j)\end{equation}
(compare \cite[Example 2.25]{Sheridan2015}).
We define the $\bm{G}$-graded polynomial ring
\begin{equation}
\label{eqn:ra}
R_a := \C [ V_a ],\end{equation}
(compare \cite[Definition 2.28]{Sheridan2015}).
We remark that the filtration of $R_a$ by order is complete in the category of $\bm{G}$-graded algebras, essentially by Lemma \ref{lemma:moncoeff} applied to $(X^n_a,D)$.
Therefore there is no need for the completion referred to at the start of \S \ref{subsec:class}.

We define the polynomial ring
\begin{equation}
\label{eqn:sa}
 S_a := R_a[U],
\end{equation}
and set
\begin{equation}
\label{eqn:ztild}
\tilde{Z}^n_a := - u_1 \ldots u_{n} + \sum_{j=1}^{n} r_j u_j^{a} \in S_a.
\end{equation}

We recall the Hochschild--Kostant--Rosenberg map \cite{Hochschild1962}, which gives an explicit quasi-isomorphism
\begin{equation}
\Phi: CC^*(A \otimes R_a  | R_a) \To S_a \otimes A.
\end{equation}
To define it, let $\{u_i\}$ be a basis for $U \subset A$, and $\{v^i\}$ the dual basis for $U^\vee \subset \C[U]$; then
\begin{align}
\Phi(\alpha)&:=  \sum_{s=0}^\infty \alpha^s(\bm{u},\ldots,\bm{u}),\mbox{ where}\\
\bm{u} & :=  \sum_{i=1}^n v^i u_i,
\end{align}
 (see \cite[Definition 2.89]{Sheridan2015}).

We will denote monomials in $R_a$ by $r^{\bm{c}}$, where $\bm{c} \in \Z_{\ge 0}^n$, monomials in $S_a$ by $u^{\bm{b}}$ where $\bm{b} \in \Z_{\ge 0}^n$, and generators in $A$ by $\theta^K$, where $K \subset \{1,\ldots,n\}$.
Thus, we have a $\C$-basis for $S_a \otimes A$ consisting of elements of the form $r^{\bm{c}} u^{\bm{b}} \theta^K$.
The results in this section require a detailed understanding of the $\bm{G}$-grading on $S_a \otimes A$.
This is afforded by \cite[Lemma 2.93]{Sheridan2015}, which we will not reproduce here.

\begin{definition}
\label{definition:typean}
We say that a (non-curved) $\bm{G}$-graded $A_{\infty}$ algebra $\scrA$ over $R_a$ has \emph{type A$^n_a$} if it satisfies the following properties:
\begin{itemize}
\item Its underlying $R_a$-module and order-$0$ cohomology algebra is
\begin{equation} (\scrA,\mu^2_0) \cong A \otimes R_a;\end{equation}
\item It satisfies
\begin{equation}\label{eqn:Phimu} \Phi(\mu^* ) = \tilde{Z}^n_a + \mathcal{O}(r^2);\end{equation}
\item $\mu^0_0 = 0$.
\end{itemize}
\end{definition}

Now we will show that the differential $\mu^1$ vanishes, for any $A_{\infty}$ algebra of type A$^n_a$, as long as $n-1 \ge a \ge 2$ and $n \ge 4$.

\begin{lemma}
\label{lemma:mu1}
Suppose that $n-1 \ge a \ge 2$ and $n \ge 4$, and $\scrA = (A \otimes R,\mu^*)$ is an $A_{\infty}$ algebra of type A$^n_a$.
Then for any
\begin{align}
\alpha & \in  CC^2_{c,\bm{G}}(A \otimes R)^1 \\
&\cong \mathrm{Hom}^1_R(A \otimes R,A \otimes R),
\end{align}
and any $K \neq \{1,\ldots,n\}$, we have
\begin{equation} \alpha(\theta^K) = 0.\end{equation}
\end{lemma}
\begin{proof}
Suppose that $\alpha$ sends
\begin{equation} \theta^{K_1} \mapsto r^{\bm{c}} \theta^{K_0}.\end{equation}
By \cite[Lemma 2.93]{Sheridan2015}, we have
\begin{equation} 1 = t = (n-2)q + (2-a)j.\end{equation}
If $a = 2$ then we have $1 = (n-2)q$, which is impossible as $n \ge 4$.
So we may assume $a \ge 3$.

Applying \cite[Lemma 2.93]{Sheridan2015} again yields
\begin{equation} a \bm{c} + y_{K_0} = y_{K_1} + q y_{\{1,\ldots,n\}}.\end{equation}
Thus, for each $k \in \{1,\ldots,n\}$, we have
\begin{equation} a c_k = q + (\mbox{ $-1$ or $0$ or $1$}).\end{equation}
Because $a \ge 3$, this implies that all $c_k$ are equal.
Suppose they are all equal to $c$, so we have $\bm{c} = c y_{\{1,\ldots,n\}}$ and $j = nc$.
Therefore, we have
\begin{equation} y_{K_0} - y_{K_1} = (q-ac) y_{\{1,\ldots,n\}}.\end{equation}
Now observe that we have
\begin{align}
1 = t &= (n-2)q + (2-a)cn \\
&= (n-2)(q - ac) + 2c(n-a).
\end{align}
We saw earlier that $q-ac$ is equal to $-1$, $0$ or $1$.
If $q-ac = 1$, the right-hand side is $\ge 2$, so we have a contradiction.
If $q-ac = 0$, the right-hand side is even, so we have a contradiction.

If $q-ac = -1$, then we have
\begin{equation} y_{K_1} = y_{K_0} + y_{\{1,\ldots,n\}},\end{equation}
so we must have $K_0 = \phi$ and $K_1 = \{1,\ldots,n\}$; in particular, the only input generator on which $\alpha$ can be non-zero is $\theta^{\{1,\ldots,n\}}$, as required.
 \end{proof}

\begin{corollary}
\label{corollary:min}
Suppose that $n-1 \ge a \ge 2$ and $n \ge 4$, and $\scrA$ is a non-curved $A_{\infty}$ algebra of type A$^n_a$.
Then the differential $\mu^1$ vanishes.
\end{corollary}
\begin{proof}
Note that
\begin{equation} \mu^1 \in CC_{c,\bm{G}}^2(A \otimes R)^1.\end{equation}
So, by Lemma \ref{lemma:mu1}, $\mu^1(\theta^K) = 0$ for all $K \neq \{1,\ldots,n\}$.
Now choose some $K \subset \{1,\ldots,n\}$ with $K \neq \phi, \{1,\ldots,n\}$, and apply the $A_{\infty}$ equation:
\begin{align}
\mu^1\left(\mu^2\left(\theta^K, \theta^{\bar{K}}\right) \right) &= \mu^2(\mu^1(\theta^K), \theta^{\bar{K}}) + \mu^2(\theta^K, \mu^1(\theta^{\bar{K}})) \\
\Rightarrow \mu^1 \left( \theta^{\{1,\ldots,n\}} + \mathcal{O}(r) \right) &= 0 \\
\Rightarrow \mu^1 \left( (1 + \mathcal{O}(r)) \theta^{\{1,\ldots,n\}} \right) &= 0 \mbox{ (because $\mu^1(\theta^K) = 0$ for $K \neq \{1,\ldots,n\}$)}\\
\Rightarrow \mu^1(\theta^{\{1,\ldots,n\}}) &= 0.
\end{align}
Therefore, $\mu^1$ vanishes.
 \end{proof}

\begin{corollary}
\label{corollary:cohalg}
Suppose that $\scrA$ is an $A_\infty$ algebra of type A$^n_a$.
If $n-1 \ge a \ge 3$, then we have an isomorphism of $\C$-algebras
\begin{equation} H^*(\scrA \otimes_R \C) \cong 	\wedgestar(\C^n). \end{equation}
If $a=2$ and $n \ge 4$, then we have an isomorphism of $\C$-algebras
\begin{equation}H^*(\scrA \otimes_R \C) \cong \Cl_n.\end{equation}
\end{corollary}
\begin{proof}
By Corollary \ref{corollary:min}, the differential $\mu^1$ vanishes, which implies the result on the level of vector spaces.
On the other hand, the length-2 part of \eqref{eqn:Phimu} implies that, for any
\begin{equation} v = \sum_i v_i \theta_i \in \scrA \otimes_R \C,\end{equation}
we have
\begin{equation}
\mu^2(v,v) = \left\{ \begin{array}{ll}
				0 & \mbox{ if $a \ge 3$} \\
				Q(v) & \mbox{ if $a = 2$},
			\end{array} \right.
\end{equation}
where $Q$ is the quadratic form $Q(\theta_i,\theta_j) = \delta_{ij}$.
The result follows.
 \end{proof}

\begin{theorem}
\label{theorem:typean}
Suppose that $\scrA_1 = (A \otimes R, \mu)$ and $\scrA_2 = (A \otimes R,\eta)$ are two $A_{\infty}$ algebras of type A$^n_a$, where $n-1 \ge a \ge 2$ and $n \ge 4$.
Then there exists a $\bm{G}$-graded formal diffeomorphism $F$ such that
\begin{equation} \scrA_1 = F_* \scrA_2.\end{equation}
\end{theorem}

We will prove Theorem \ref{theorem:typean} by applying Proposition \ref{proposition:verseasy}.

\begin{lemma}
\label{lemma:ord0}
The $0$th-order parts $\scrA_j \otimes_R \C$ (where the map $R \To \C$ sends each $r_j$ to $0$) are necessarily minimal and related by a formal diffeomorphism.
\end{lemma}
\begin{proof}
This is identical to \cite[Corollary 2.97]{Sheridan2015}.
 \end{proof}

The next step is to determine the first-order deformation space.

\begin{lemma}
\label{lemma:defclasshh}
If $n \ge 4$, then the vector space
\begin{equation} HH_{\bm{G}}^2(A, A \otimes R^1_a)\end{equation}
is generated by the elements
\begin{equation} r_j u_j^a,\end{equation}
for $j = 1,\ldots,n$.
\end{lemma}
\begin{proof}
This is identical to \cite[Lemma 2.101]{Sheridan2015}.
 \end{proof}

The next step is to show that the higher-order deformation spaces vanish.

\begin{lemma}
\label{lemma:j2}
Suppose that $n-1 \ge a \ge 2$ and $j \ge 2$.
Then we have
\begin{equation} TrHH^2_{\bm{G}}(A,A \otimes R^j) \cong 0.\end{equation}
\end{lemma}
\begin{proof}
If $r^{\bm{c}} u^{\bm{b}} \theta^K$ is a generator of $TrHH^2_{\bm{G}}(A, A \otimes R^j)$, and $j = |\bm{c}| \ge 2$, then we have $\bm{c} \ge y_{\{1,\ldots,n\}}$ (see \cite[Lemma 102]{Sheridan2015}), and hence $j \ge |y_{\{1,\ldots,n\}}| = n$.

Now, applying \cite[Lemma 2.94, Equation (2.8)]{Sheridan2015}, we have $2(1+q-j) = |K| \ge 0$, so $q-j \ge -1$.
We also have $|\bm{b}| = s = 2-t \ge 2$, because $t \le 0$ by definition for a truncated Hochschild cochain.
Applying \cite[Lemma 2.94, Equation (2.6)]{Sheridan2015}, we have
\begin{align}
y_K + a\bm{c} &= q y_{\{1,\ldots,n\}} + \bm{b} \\
\Rightarrow |K| + a j & \ge  nq + 2 \\
\Rightarrow |K| & \ge  (n-a)j + n(q-j) + 2.
\end{align}

We split into two cases: if $q -  j =-1$, then $|K| = 0$, so we obtain
\begin{align}
0 & \ge  (n-a)j - n + 2\\
\Rightarrow n &\ge (n-a)j + 2.
\end{align}
If, on the other hand, $q - j \ge 0$, then because $|K| \le n$, we obtain
\begin{equation}n \ge |K| \ge (n-a)j + n(q-j) + 2 \ge (n-a)j + 2.\end{equation}

In either case, we have
\begin{equation}n \ge (n-a)j  + 2 \ge n + 2,\end{equation}
where the second inequality follows because $n-a \ge 1$ by assumption, and we have shown that $j =|\bm{c}| \ge n$.
This is a contradiction, so the result follows.

We remark that $HH^2_{\bm{G}}(A,A \otimes R^j)$ has an extra generator $r^{y_{\{1,\ldots,n\}}}$ in the case $a = n-1$; but this does not correspond to a generator of truncated Hochschild cohomology because it has degree $t = 2 > 0$.
 \end{proof}

\begin{proof} \textit{(of Theorem \ref{theorem:typean})}. Suppose we are given $\scrA_1$ and $\scrA_2$ of type A$^n_a$, both non-curved.
First, by Lemma \ref{lemma:ord0}, we can apply a formal diffeomorphism $F$ to $\scrA_2$ so that
$F_* \scrA_2 \cong \scrA_1$ to order $0$.
Thus, we may assume without loss of generality that $\scrA_1$ and $\scrA_2$ are both deformations of the same $A_{\infty}$ algebra $\cA := \scrA_1 \otimes_R \C$, over the ring $R$.

By Corollary \ref{corollary:min}, these deformations are minimal.
Furthermore, by Lemma \ref{lemma:j2}, we have
\begin{equation} TrHH^2_{\bm{G}}(A,A \otimes R^j) \cong 0\end{equation}
for $j \ge 2$.
We recall that the spectral sequence induced by the length filtration converges:
\begin{equation} TrHH^2_{\bm{G}}(A, A \otimes R) \Rightarrow TrHH^2_{\bm{G}}(\cA,\cA \otimes R)\end{equation}
(see \cite[Lemma 2.98]{Sheridan2015}).
Therefore,
\begin{equation} TrHH^2_{\bm{G}}(\cA,\cA \otimes R^j) \cong 0\end{equation}
for $j \ge 2$.
It follows by Proposition \ref{proposition:verseasy} that, if the first-order deformation classes of $\scrA_1$ and $\scrA_2$ coincide, then there is a formal diffeomorphism $F$ such that $\scrA_1 \cong F_* \scrA_2$.

By Lemma \ref{lemma:defclasshh},
\begin{equation} TrHH^2_{\bm{G}}(A,A \otimes R^1) \cong \C \langle r_1 u_1^a , \ldots , r_n u_n^a \rangle.\end{equation}
It follows from the convergence of the length spectral sequence as above, that $TrHH^2_{\bm{G}}(\cA,\cA \otimes R^1)$ is generated by elements of the form
\begin{equation} r_j u_j^a + \mbox{ (lower-order in length filtration),}\end{equation}
for $j = 1,\ldots , n$.
Because $\scrA_1$ and $\scrA_2$ are both of type A$^n_a$, by definition their first-order deformation classes are both of the form
\begin{equation} \sum_{j=1}^n r_j u_j^a + \mbox{ (lower-order in length filtration).}\end{equation}
Therefore, their first-order deformation classes coincide.
The claim now follows from Proposition \ref{proposition:verseasy}.
 \end{proof}

\subsection{Computation of Hochschild cohomology}
\label{subsec:comphh}

For the purposes of this section, let $\scrA$ be a non-curved $A_{\infty}$ algebra of type A$^n_a$, where $n-1 \ge a \ge 2$ and $n \ge 4$.
We will compute two versions of the Hochschild cohomology $HH^*(\scrA)$.
Our strategy is suggested by \cite[\S 3.6]{Smith2012} (compare also \cite{Seidel2008a}).
It relies on a slightly modified version of \cite[Proposition 1]{Dolgushev2005}.

Let us start by stating this modified version explicitly.
Let $\bm{G} = \{f: \Z \To Y\}$ be a grading datum.
A $\bm{G}$-graded $L_{\infty}$ algebra is a $\bm{G}$-graded vector space $\mathfrak{g}$ together with $L_{\infty}$ structure maps
\begin{equation} \ell^k: \mathfrak{g}^{\otimes k} \To \mathfrak{g}\end{equation}
for $k \ge 1$, of degree $f(2-k)$, satisfying a system of relations (see \cite{Lada1995}).
In fact, we will only consider differential graded Lie algebras, with $\ell^{\ge 3} = 0$; nevertheless, working with $L_{\infty}$ algebras helps one see why the techniques we employ make sense.

If $\mathfrak{g}$ and $\mathfrak{h}$ are $\bm{G}$-graded $L_{\infty}$ algebras, then a $\bm{G}$-graded $L_{\infty}$ morphism $\Phi$ from $\mathfrak{g}$ to $\mathfrak{h}$ consists of maps
\begin{equation} \Phi^k: \mathfrak{g}^{\otimes k} \To \mathfrak{h}\end{equation}
for $k \ge 1$, of degree $f(1-k) \in Y$, satisfying a system of relations (see \cite[\S 4]{Kontsevich2003}).

Now let $\bm{G} \oplus \Z$ denote the grading datum $\{f \oplus \mathrm{id}: \Z \To Y \oplus \Z\}$.
Let $\mathfrak{g}$ be a $\bm{G} \oplus \Z$-graded $L_{\infty}$ algebra, where the $\Z$-grading is bounded below.
The $\Z$-grading on $\mathfrak{g}$ then induces a bounded-above, decreasing filtration:
\begin{equation} F_r\mathfrak{g} := \bigoplus_{y \oplus s \in Y \oplus \Z_{\ge r}} \mathfrak{g}_{y \oplus s}.\end{equation}
Let $\hat{\mathfrak{g}}$ denote the completion of $\mathfrak{g}$ with respect to this filtration, in the category of $\bm{G}$-graded modules.
It is a $\bm{G}$-graded $L_{\infty}$ algebra.
It has a filtration $F_r \hat{\mathfrak{g}}$, but this is \emph{not} a filtration by $L_{\infty}$ subalgebras: the $L_{\infty}$ products send
\begin{equation} \ell^k: (F_r \hat{\mathfrak{g}})^{\otimes k} \mapsto F_{2+k(r-1)} \hat{\mathfrak{g}}.\end{equation}
This means that the Maurer--Cartan equation for $\alpha \in \hat{\mathfrak{g}}_{f(1)}$:
\begin{equation} \sum_{j \ge 1} \frac{1}{j!} \ell^j\left(\alpha, \ldots, \alpha\right) = 0,\end{equation}
does not make sense because the infinite sum may not converge.
However, if we assume that $\alpha \in F_2 \hat{\mathfrak{g}}_{f(1)}$ then the sum \emph{does} make sense, because the $k$th term in the Maurer--Cartan equation lies in $F_{2+k}\hat{\mathfrak{g}}$, so the terms in the sum are of successively higher and higher orders in the filtration.

If $\alpha \in F_2 \hat{\mathfrak{g}}_{f(1)}$ is a Maurer--Cartan element, we define the $L_{\infty}$ structure on $\hat{\mathfrak{g}}$ \emph{twisted by} $\alpha$:
\begin{align}
\ell_{\alpha}: \hat{\mathfrak{g}}^{\otimes k} & \To  \hat{\mathfrak{g}} \\
\ell_{\alpha}^k(x_1, \ldots, x_k) & :=  \sum_{j \ge 0} \frac{1}{j!} \ell_{j+k}(\underbrace{\alpha,\ldots,\alpha}_j,  x_1 , \ldots, x_k).
\end{align}
This converges, and defines a new $\bm{G}$-graded $L_{\infty}$ structure on $\hat{\mathfrak{g}}$ (see \cite[Proposition 4.4]{Getzler2009}).

Now suppose that $\mathfrak{g}$ and $\mathfrak{h}$ are $\bm{G} \oplus \Z$-graded $L_{\infty}$ algebras, and $\Phi$ is a $\bm{G} \oplus \Z$-graded $L_{\infty}$ morphism from $\mathfrak{g}$ to $\mathfrak{h}$.
Then $\Phi$ induces a $\bm{G}$-graded $L_{\infty}$ morphism $\hat{\Phi}$ from $\hat{\mathfrak{g}}$ to $\hat{\mathfrak{h}}$.
If $\alpha \in F_2 \hat{\mathfrak{g}}_{f(1)}$ is a Maurer--Cartan element, then
\begin{equation} \hat{\Phi}_* \alpha := \sum_{j \ge 1} \frac{1}{j!} \hat{\Phi}^j(\alpha, \ldots, \alpha) \end{equation}
is also a Maurer--Cartan element in $F_2 \hat{\mathfrak{h}}_{f(1)}$.
Furthermore, there is a $\bm{G}$-graded $L_{\infty}$ morphism $\hat{\Phi}_{\alpha}$ from $(\hat{\mathfrak{g}},\ell_{\alpha})$ to $(\hat{\mathfrak{h}},\ell_{\hat{\Phi}_*\alpha})$, defined by
\begin{align}
\hat{\Phi}_{\alpha}^k: \hat{\mathfrak{g}}^{\otimes k} & \To  \hat{\mathfrak{h}} \\
\hat{\Phi}_{\alpha}^k(x_1, \ldots, x_k) & :=  \sum_{j \ge 0} \frac{1}{j!} \hat{\Phi}^{j+k}(\underbrace{\alpha,\ldots,\alpha}_j,  x_1 , \ldots, x_k).
\end{align}
These last two claims are proven by the same argument as \cite[Proposition 1]{Dolgushev2005}.
There are two differences in our case.
Firstly, we are dealing with $L_{\infty}$ algebras rather than dg Lie algebras: however the necessary alterations to the proofs are obvious.
Secondly, we have made different assumptions on the filtrations from those in \cite{Dolgushev2005} (we do not have a filtration by dg Lie subalgebras).
Nevertheless, because we restrict ourselves to Maurer--Cartan elements in $F_2\hat{\mathfrak{g}}$, it is easy to check that all of the infinite sums we have written down converge.

Finally, suppose that $\Phi$ is a quasi-isomorphism, i.e., the chain map
\begin{equation} \Phi^1: (\mathfrak{g},\ell^1) \To (\mathfrak{h},\ell^1)\end{equation}
induces an isomorphism on cohomology.
Then we claim that $\hat{\Phi}_{\alpha}$ is also a quasi-isomorphism, i.e., the chain map
\begin{equation} \hat{\Phi}^1_{\alpha}: (\hat{\mathfrak{g}},\ell^1_{\alpha}) \To (\hat{\mathfrak{h}},\ell^1_{\hat{\Phi}_*\alpha})\end{equation}
induces an isomorphism on cohomology (again, compare \cite[Proposition 1]{Dolgushev2005}).

To prove it, consider the spectral sequences induced by the filtrations on the complexes $(\hat{\mathfrak{g}},\ell^1_{\alpha})$ and $(\hat{\mathfrak{h}},\ell^1_{\hat{\Phi}_*\alpha})$, and the morphism between them induced by $\hat{\Phi}^1_{\alpha}$.
The $E_1$ pages of the induced spectral sequences are $\mathfrak{g}$ and $\mathfrak{h}$ respectively, with the differentials given by $\ell^1$ on both sides, and the chain map $\Phi^1$ between them.
Therefore, the $E_2$ pages are the cohomologies $H^*(\mathfrak{g},\ell^1)$ and $H^*(\mathfrak{h},\ell^1)$, with the map between them induced by $\Phi^1$.
By assumption, this map is an isomorphism.
Because these filtrations are complete and bounded above, hence exhaustive, the Eilenberg--Moore comparison theorem \cite[Theorem 5.5.11]{Weibel1994} shows that $\hat{\Phi}^1_{\alpha}$ is a quasi-isomorphism.

Now let us apply this to compute the Hochschild cohomology of an $A_{\infty}$ algebra $\scrA$ of type A$^n_a$, in the case $a \ge 3$ (the computation for $a=2$ will be slightly different, see the proof of Proposition \ref{proposition:hhana}).
Let $\mathfrak{g}$ be the $\bm{G} \oplus \Z$-graded vector space $CC^*_{c,\bm{G}}(A \otimes R)[1]$, the compactly-supported $\bm{G}$-graded Hochschild cohomology of $A$.
The Gerstenhaber bracket satisfies the graded Jacobi relation, and makes $\mathfrak{g}$ into a $\bm{G} \oplus \Z$-graded Lie algebra.
The completion $\hat{\mathfrak{g}}$ is the $\bm{G}$-graded Hochschild cohomology $CC^*_{\bm{G}}(A \otimes R)[1]$; the filtration $F_r \hat{\mathfrak{g}}$ is the length filtration (shifted by $1$).

The $A_{\infty}$ structure maps define a Maurer--Cartan element $\mu^* \in \hat{\mathfrak{g}}_{f(1)}$.
For a general $A_{\infty}$ structure, we only have $\mu^* \in F_0 \hat{\mathfrak{g}}_{f(1)}$, so this does not fit into the setup outlined above; however, because we are only dealing with a Lie algebra at the moment, there are no higher products and the Maurer--Cartan equation makes sense.

In our case, by Corollary \ref{corollary:min}, $\scrA$ is minimal, so $\mu^2$ defines an associative product.
In particular, $\mu^2$ is itself a Maurer--Cartan element.
Twisting our Lie algebra structure by this element defines a new $L_{\infty}$ (in fact dg Lie) algebra with the differential $\ell^1 = [\mu^2,-]$ and $\ell^2$ given by the Gerstenhaber bracket.
The cohomology of $\ell^1$ is the Hochschild cohomology of the associative algebra $A \otimes R$.
The remainder of the $A_{\infty}$ products define a Maurer--Cartan element
\begin{equation} \alpha:= \mu^{\ge 3} \in F_2 \hat{\mathfrak{g}}_{f(1)},\end{equation}
which fits into the previously-described setup.
Twisting by this Maurer--Cartan element defines a new $L_{\infty}$ (in fact dg Lie) algebra structure, with differential equal to the Hochschild differential:
\begin{align}
 \ell^1_{\alpha}(x) &= \ell^1(x) + \ell^2(\alpha,x) \\
&= [\mu^*,x].
\end{align}
Therefore, the cohomology of $\ell^1_{\alpha}$ is the Hochschild cohomology of the $A_{\infty}$ algebra $\scrA$.

Now observe that, by Lemmas \ref{lemma:defclasshh} and \ref{lemma:j2}, $TrHH^2_{c,\bm{G}}(A \otimes R)^2 \cong 0$ (here we use the assumption $a \ge 2$).
It follows that one can construct an isomorphism of algebras
\begin{equation}F: (A \otimes R, \mu^2) \cong (A \otimes R, \mu^2_0).\end{equation}
Pushing forward $\mu^*$ by $F$ gives us a quasi-isomorphic $A_{\infty}$ algebra of type A$^n_a$, where the product $\mu^2$ is simply the exterior product $\mu^2_0$.
So we may assume without loss of generality that $\mu^2$ is the exterior product.

Now let $\mathfrak{h}$ be the $\bm{G} \oplus \Z$-graded dg Lie algebra $R[u_1, \ldots, u_n] \otimes A$, with vanishing differential and the Schouten-Nijenhuis bracket.
The HKR isomorphism \cite{Hochschild1962} defines a $\bm{G} \oplus \Z$-graded quasi-isomorphism
\begin{equation} \Phi^1: \mathfrak{g} \To \mathfrak{h}.\end{equation}
By the $R$-linear extension of Kontsevich's formality theorem \cite{Kontsevich2003}, this extends to a $\bm{G} \oplus \Z$-graded $L_{\infty}$ quasi-isomorphism from $\mathfrak{g}$ to $\mathfrak{h}$.
Kontsevich's original paper actually considered the case of a polynomial algebra, but the case of an exterior algebra is parallel.
It also considered an $L_{\infty}$ morphism in the other direction, but this implies the existence of such a $\Phi$ because $L_{\infty}$ quasi-isomorphisms can be inverted.

We should explain why Kontsevich's $L_{\infty}$ morphism is $\bm{G} \oplus \Z$-graded.
The morphism is $\Z$-graded by construction, where the grading on $\mathfrak{g}$ is by length (shifted by $1$) and the grading on $\mathfrak{h}$ is by the degree of the polynomial in $R[u_1, \ldots, u_n]$ (shifted by $1$, where $R$ and $A$ are equipped with the zero grading).
So the Taylor coefficient $\Phi^k$ of Kontsevich's $L_{\infty}$ morphism $\Phi$  has degree $1-k$.
Furthermore, the morphism is $GL(\C^n)$-equivariant, and in particular $(\C^*)^n$-equivariant, where $(\C^*)^n$ is the subgroup of invertible diagonal $n \times n$ matrices.
Equivalently, if we equip $\C^n$ with the natural $\Z^n$ grading, then the maps $\Phi^k$ have degree $0 \in \Z^n$.
We define the $\bm{G}$-grading of $A \cong \Lambda(\C^n)$ by pushing forward the grading coming from the obvious $\Z^n$-grading along a morphism from $\Z^n$ to $\bm{G}$.
This defines a $\bm{G}$-grading on $CC_{c,\bm{G}}(A)$ and hence on $\mathfrak{g}:=CC_{c,\bm{G}}(A \otimes R)$ (where we recall that $R$ has its own $\bm{G}$-grading), and similarly on $\mathfrak{h}$.
The formality morphisms $\Phi^k$ have degree $0$ with respect to this grading.

However, we recall from \cite[Definition 2.30]{Sheridan2015} that, if a Hochschild cochain of length $s$ changes $\bm{G}$-degree by $y \in \bm{G}$, then we equip it with the grading $y+f(s) \in \bm{G}$.
It follows that, after the shift by $1$, the Taylor coefficient $\Phi^k$ of the $L_{\infty}$ morphism $\Phi$ has degree $f(1-k) \oplus (1-k) \in \bm{G} \oplus \Z$, with respect to the standard grading.
Therefore, $\Phi$ is a $\bm{G} \oplus \Z$-graded $L_{\infty}$ quasi-isomorphism.

It follows from our preceding discussion that $\hat{\Phi}^1_{\alpha}$ induces an isomorphism
\begin{align}
 HH^*_{\bm{G}}(\scrA) &\cong  H^*(\hat{\mathfrak{g}},\ell^1_{\alpha}) \\
&\cong  H^*(\hat{\mathfrak{h}},\ell^1_{\hat{\Phi}_*\alpha}).
\end{align}
We will now compute $\hat{\mathfrak{h}}$ and the differential $\ell^1_{\hat{\Phi}_* \alpha}$.
The first step is to show that $\hat{\mathfrak{h}} \cong \mathfrak{h}$.

\begin{lemma}
\label{lemma:filtcomp}
The filtration $F_r \mathfrak{h}$ is complete in the category of $\bm{G}$-graded vector spaces; in particular, $\hat{\mathfrak{h}} \cong \mathfrak{h}$.
\end{lemma}
\begin{proof}
The filtration is complete if and only if it is complete on each graded piece $\mathfrak{h}_y$, in the category of vector spaces.
To prove this is true, we prove that each graded piece $\mathfrak{h}_y$, for $y = (y_1 ,y_2) \in Y$, is finite-dimensional.
Recalling the proof of \cite[Lemma 2.93]{Sheridan2015}, if $r^{\bm{c}}u^{\bm{b}}\theta^K$ is a generator of $R[U]\otimes A$ of degree $(y_1,y_2 )\in Y$, then we have
\begin{equation} (0 , y_{K}) -( 0, \bm{b}) + ((2-a)|\bm{c}| , a \bm{c}) + (s, 0) = (y_1 , y_2 )+ q((2-n), y_{\{1,\ldots,n\}})\end{equation}
for some $q \in \Z$.
Hence we have, setting $|\bm{c}| = j$ and $|\bm{b}|=s$,
\begin{align}
(2-a)j + (n-2)q + s = const \mbox{ and}\\
|K| - s + aj -nq = const.
\end{align}
Eliminating $q$, and observing that $|K| \ge 0$, gives an equation of the form
\begin{equation} s + (n-a)j \le const.\end{equation}
Observing that $n > a$ by assumption, and both $s$ and $j$ are non-negative, shows that $\mathfrak{h}_y$ is finite-dimensional.
Hence, the length filtration on each graded piece of $\mathfrak{h}$ is complete, so the length filtration is complete in the category of $\bm{G}$-graded vector spaces.
 \end{proof}

Now we compute the Maurer--Cartan element $\hat{\Phi}_* \alpha$.
Because $\scrA$ is of type A$^n_a$, we know the leading-order term
\begin{equation} \Phi^1(\alpha) = \tilde{Z}^n_a \in R[U] \subset R[U] \otimes A.\end{equation}
We claim that the higher-order terms vanish for grading reasons.
To see this, we introduce the grading
\begin{equation} d:= s+(n-a)j\end{equation}
on $CC_{\bm{G}}(A \otimes R)$ and $R[U] \otimes A$.
We denote by $\mu^*_d$ the part of $\mu^*$ with $s+(n-a)j = d$.
Note that $d \ge s$.
It follows from \cite[Lemma 2.93, Equation (2.3)]{Sheridan2015}, with $s+t = 2$, that if $\mu^*_d \neq 0$, then
\begin{equation} d = s+(n-a)j = 2 + (n-2)(q-j).\end{equation}
In particular, $d$ must be congruent to $2$ modulo $n-2$.
In particular, $\mu^{\ge 3}_d = 0$ for $d<n$.

Now we recall that $\Phi^k$ has degree $1-k$, which after the shifts by $1$ means it has degree $2-2k$ with respect to the length grading $s$, and it clearly preserves $j$, the degree in $R$, because $\Phi^k$ is $R$-multilinear.
It follows that the $d$-grading of the $k$th term in $\hat{\Phi}_* \alpha$,
\begin{equation} \Phi^k(\alpha,\ldots, \alpha),\end{equation}
is $2 + (n-2)k$.
In particular, if $k \ge 2$ then the $d$-grading is $>n$.
But, by Lemmas \ref{lemma:defclasshh} and \ref{lemma:j2}, and \cite[Lemma 2.96]{Sheridan2015}, $HH^2_{\bm{G}}(A \otimes R)$ is generated by monomials $u_1 \ldots u_n$ (with $d = s+(n-a)j =n$), $r_j u_j^a$ (with $d = s+(n-a)j = a+(n-a) = n$), and (if $a=n-1$), $r_1 \ldots r_n$, which has length $0$ and therefore does not lie in $F_2 \hat{\mathfrak{h}}$.
Therefore, the higher-order terms of $\hat{\Phi}_*\alpha$ necessarily vanish for degree reasons, and we have
\begin{equation}
\label{eqn:pushedmc} \hat{\Phi}_* \alpha = \Phi^1(\alpha) = \tilde{Z}^n_a.\end{equation}

The differential $\ell^1_{\hat{\Phi}_*\alpha}$ is therefore $[\tilde{Z}^n_a,-]$, the Schouten-Nijenhuis bracket with $\tilde{Z}^n_a$.
This is exactly the Koszul differential associated with the sequence
\begin{equation}\label{eqn:grobels} \del{\tilde{Z}^n_a}{u_j} = -u_1 \ldots \hat{u}_j \ldots u_n + a r_j u_j^{a-1}\end{equation}
for $j=1, \ldots, n$, in the ring
\begin{equation} \label{eqn:grobring}
R[U] \cong \C[r_1, \ldots, r_n,u_1, \ldots, u_n].\end{equation}
Using an elimination order with respect to $r_1, \ldots, r_n$, one easily verifies that the elements \eqref{eqn:grobels} of the polynomial ring \eqref{eqn:grobring} form a Gr\"{o}bner basis, whose initial terms $ar_j u_j^{a-1}$ are coprime; Schreyer's theorem (see, e.g., \cite[Theorem 15.10]{Eisenbud1995}) then shows that the only syzygies between them are the trivial ones, so they form a regular sequence.
Therefore, the cohomology of the Koszul complex is simply the Jacobian ring
\begin{equation} \frac{R[U]}{\left(\del{\tilde{Z}^n_a}{u_1}, \ldots, \del{\tilde{Z}^n_a}{u_n} \right)}.\end{equation}

\begin{proposition}
\label{proposition:hhana}
Let $\scrA$ be a non-curved $A_{\infty}$ algebra of type A$^n_a$, with $n-1 \ge a \ge 2$ and $n \ge 4$.
We introduce the element
\begin{equation} \beta := \left[ r_j \del{}{r_j}\mu^* \right] \in HH^*_{\bm{G}}(\scrA|R)\end{equation}
(note that it is indeed a Hochschild cochain, as can be seen by applying $r_j \partial/\partial r_j$ to the $A_{\infty}$ equation $\mu^* \circ \mu^* = 0$).
The $R$-subalgebra of $HH^*_{\bm{G}}(\scrA|R)$ generated by $\beta$ is isomorphic to
\begin{equation} R[\beta]/\tilde{q}^n_a(\beta),\end{equation}
where we define
\begin{equation}
\label{eqn:tildeq} \tilde{q}^n_a(\beta) := \beta^{n-1} - a^a T \beta^{a-1},\end{equation}
where we recall
\begin{equation} T := r_1 \ldots r_n.\end{equation}
\end{proposition}
\begin{proof}
We first give the proof in the case $a \ge 3$.
We recall the quasi-isomorphism
\begin{equation} \hat{\Phi}^1_{\alpha}: (CC^*_{\bm{G}}(\scrA),[\mu^*,-]) \To (R[U] \otimes A, [\tilde{Z}^n_a,-]),\end{equation}
and that
\begin{equation} \hat{\Phi}^1_{\alpha}(\mu^{\ge 3}) = \tilde{Z}^n_a.\end{equation}
It follows that
\begin{align}
\hat{\Phi}^1_{\alpha}(\beta) &= \hat{\Phi}^1_{\alpha}\left(r_j \del{\mu^*}{r_j} \right) \\
&= r_j \del{\tilde{Z}^n_a}{r_j} \\
&= r_j u_j^a,
\end{align}
because we arranged that $\mu^2 = \mu^2_0$ is independent of $r_j$.
Because $\hat{\Phi}^1_{\alpha}$ is a quasi-isomorphism, it suffices for us to compute the $R$-subalgebra of the Jacobian ring generated by
\begin{equation} \tilde{\beta}:= r_j u_j^a.\end{equation}

We observe that, in the Jacobian ring, we have relations
\begin{equation} \label{eqn:jacrels}
u_1 \ldots \hat{u}_j \ldots u_n = ar_j u_j^{a-1}.\end{equation}
In particular, we have
\begin{equation} \tilde{\beta} = \frac{1}{a} u_1 \ldots u_n.\end{equation}
Now if we take the product of the relations \eqref{eqn:jacrels}, we obtain
\begin{equation} (u_1 \ldots u_n)^{n-1} = a^n r_1 \ldots r_n (u_1 \ldots u_n)^{a-1}.\end{equation}
Plugging our expressing for $\tilde{\beta}$ into this gives the relation
\begin{align}
 (a \tilde{\beta})^{n-1} &= a^n T (a \tilde{\beta})^{a-1}\\
\label{eqn:jacrel}\Rightarrow \tilde{\beta}^{n-1} &= a^a T \tilde{\beta}^{a-1}.
\end{align}

It remains to check that there are no $R$-linear relations between the elements $1,\tilde{\beta},\ldots,\tilde{\beta}^{n-2}$ in the Jacobian ring.
To do this, we compute a Gr\"{o}bner basis for the Jacobian ideal, with respect to an elimination order with respect to $u_1,\ldots,u_n$, and which is given by the homogeneous lexicographic order in the variables $u_j$, such that $u_i > u_j$ iff $i > j$.
If $K \subset \{1,\ldots,n\}$, we denote
\begin{equation} u_K := \prod_{k \in K} u_k,\end{equation}
and similarly for $r_K$.
It is convenient to make a change of variables so that the Jacobian ideal is generated by elements
\begin{equation} u_{\overline{\{j\}}} - u_j^{a-1}.\end{equation}
Our Gr\"{o}bner basis now consists of three kinds of elements:
\begin{align}
\label{eqn:grob1} u_{\overline{\{1\}}} - r_1 u_1^{a-1} &&\\
r_ju_j^a - r_1 u_1^a && \mbox{ for $j  \neq 1$} \\
\label{eqn:grob3} (r_1u_1^a)^{|\overline{K}| - 1} u_K - r_{\overline{K}} u_{\overline{K}}^{a-1} && \mbox{for all $1 \in K \varsubsetneqq \{1,\ldots,n\}$}
\end{align}
(the initial term is written first).
It is a tedious exercise to check that this does indeed form a Gr\"{o}bner basis, by Buchberger's criterion.

Now suppose that some $R$-linear polynomial in $\tilde{\beta} = r_1u_1^a$ of degree $n-2$ lay in the Jacobian ideal.
Then its initial term would have the form $r^{\bm{c}}(r_1u_1^a)^{n-2}$, and must lie in the initial ideal of the Jacobian ideal; but one easily checks that this monomial is not divisible by any of the initial terms of the Gr\"{o}bner basis elements \eqref{eqn:grob1}-\eqref{eqn:grob3}.
Hence there are no further relations, and the proof is complete in the case $a \ge 3$.

In the case $a = 2$, one might be concerned that the terms $r_j u_j^2$ in $\tilde{Z}^n_2$ correspond to deformations of the product $\mu^2$ on $\cA$, so this doesn't fit into our framework (specifically, the Maurer--Cartan element $\mu^* - \mu^2_0$ lies in $F_1 \hat{\mathfrak{g}}_{f(1)}$, but not $F_2\hat{\mathfrak{g}}_{f(1)}$, so we have convergence issues).
In this case, instead of using the filtration by length $s$, we can use the filtration by $d = s+(n-a)j$.
This filtration is still complete (by the proof of Lemma \ref{lemma:filtcomp}), the leading-order term of $\mu^*$ is the exterior algebra product $\mu^2_0$, and the Maurer--Cartan element $\mu^* - \mu^2_0$ lies in $F_n \hat{\mathfrak{g}}_{f(1)}$ by the computations immediately preceding Proposition \ref{proposition:hhana}, from which it follows that the expressions for $\ell^1_\alpha$ and $\hat{\Phi}^1_\alpha$ converge.
The rest of the argument is identical to the case $a \ge 3$.
 \end{proof}

Now, let $\bm{G}^n_a$ be the grading datum of equation \eqref{eqn:gna}, and
\begin{equation} \bm{p}: \bm{G}^n_a \To \bm{G}^n_1 =:\bm{G}\end{equation}
the morphism of grading data of equation \eqref{eqn:morphp}.

Let $\scrA$ be a non-curved $A_{\infty}$ algebra of type A$^n_a$, where $n-1 \ge a \ge 2$ and $n \ge 4$, and $\underline{\scrA}$ be its extension to a $\bm{G}^n_1$-graded $A_{\infty}$ category.
Consider the $\bm{G}^n_a$-graded $A_{\infty}$ category
\begin{equation}\widetilde{\scrA} := \bm{p}^* \underline{\scrA}.\end{equation}
It is $R$-linear, where $R \cong \bm{p}^* R$ is now considered as a $\bm{G}^n_a$-graded ring.
In light of \cite[Remark 2.68]{Sheridan2015}, we have

\begin{corollary}
\label{corollary:hhanatild}
Let
\begin{equation} \gamma := \left[ r_j \del{\mu^*}{r_j} \right] \in HH^*_{\bm{G}^n_a}\left(\widetilde{\scrA}|R\right).\end{equation}
Then the $R$-subalgebra generated by $\gamma$ is isomorphic to
\begin{equation} R[\gamma]/\tilde{q}^n_a(\gamma).\end{equation}
\end{corollary}

Now we define
\begin{equation} \scrA_{\C} := \left(\bm{\sigma}_*\scrA \right) \otimes_{R} \C,\end{equation}
where $\C$ is an $R$-algebra via the map
\begin{align}
\C [ r_1, \ldots, r_n] & \To  \C,\\
r_j & \mapsto  1 \mbox{ for all $j$,}
\end{align}
and $\bm{\sigma}: \bm{G}^n_1 \To \bm{G}_{\sigma}$ is the sign morphism (recall $\bm{G}_{\sigma} := \{\Z \To \Z/2\Z\}$, so $\bm{G}_{\sigma}$-graded algebra is the same as $\Z/2\Z$-graded algebra).

\begin{lemma}
\label{lemma:hhanov}
When $n-1 \ge a \ge 3$, there is an isomorphism of algebras,
\begin{equation} HH^*_{\bm{G}_{\sigma}}(\scrA_{\C}) \cong \frac{\C \llbracket u_1, \ldots, u_n \rrbracket}{\left(\del{Z^n_a}{u_1}, \ldots, \del{\tilde{W}^n_a}{u_n} \right)},\end{equation}
where
\begin{equation} Z^n_a := -u_1 \ldots u_n + \sum_{j=1}^n u_j^a \in \C \llbracket u_1, \ldots, u_n \rrbracket\end{equation}
as in \eqref{eqn:zna}.
\end{lemma}
\begin{proof}
We use the same strategy as we did to make computations in $HH^*_{\bm{G}}(\bm{p}^*\scrA|R)$.
Let
\begin{equation} \mathfrak{g} := CC^*_{c,\bm{G}_{\sigma}}(A)\end{equation}
be the $\bm{G}_{\sigma} \oplus \Z$-graded dg Lie algebra, where $\ell^1 = [\mu^2_{\C},-]$, where $\mu^2_{\C}$ is the exterior product, and $\ell^2$ is the Gerstenhaber bracket.
We then have, by definition,
\begin{equation} \hat{\mathfrak{g}} := CC^*_{\bm{G}_{\sigma}}(A),\end{equation}
and the Maurer--Cartan element $\alpha := \mu^{\ge 3}_{\C}  \in F_2\hat{\mathfrak{g}}_1$.
The twisted differential is, as before, the Hochschild differential:
\begin{equation} HH^*_{\bm{G}_{\sigma}}(\scrA_{\C}) \cong H^*(\hat{\mathfrak{g}},\ell^1_{\alpha}).\end{equation}

Let $\mathfrak{h} := \C [U] \otimes A$ be the $\bm{G}_{\sigma} \oplus \Z$-graded dg Lie algebra of polyvector fields, with zero differential and the Schouten-Nijenhuis bracket.
Kontsevich's formality theorem gives a $\bm{G}_{\sigma} \oplus \Z$-graded $L_{\infty}$ quasi-isomorphism
\begin{equation} \Phi: \mathfrak{g} \To \mathfrak{h}.\end{equation}

We now have
\begin{equation} \hat{\mathfrak{h}} := \C \llbracket U \rrbracket \otimes A,\end{equation}
and the corresponding $\bm{G}_{\sigma}$-graded $L_{\infty}$ quasi-isomorphism $\hat{\Phi}$ from $\hat{\mathfrak{g}}$ to $\hat{\mathfrak{h}}$.
The pushed-forward Maurer--Cartan element is given by
\begin{equation} \hat{\Phi}_* \alpha = \tilde{Z}^n_a \otimes 1 \in (R[U] \otimes A) \otimes_R \C\end{equation}
by \eqref{eqn:pushedmc}, because the $A_{\infty}$ structure on $\scrA_{\C}$ is given by $\mu^*_{\C} := \mu^* \otimes 1$ by definition.
It follows that
\begin{equation} \hat{\Phi}_* \alpha = Z^n_a \in \C \llbracket U \rrbracket.\end{equation}

It follows as before that there is an isomorphism
\begin{equation} \hat{\Phi}^1_{\alpha}: HH^*_{\bm{G}_{\sigma}}(\scrA_{\C}) \cong H^*(\hat{\mathfrak{h}},\ell^1_{Z^n_a}).\end{equation}
The right-hand side is given by the Koszul complex for the sequence $\partial Z^n_a/\partial u_j$, for $j=1,\ldots, n$.
This sequence is regular, because $\C \llbracket U \rrbracket$ is a local ring, and the ideal generated by the $n$ derivatives of $Z^n_a$ contains powers of $u_j$ for all $j$, which form a regular sequence of length $n$ (see \cite[Corollary 17.7]{Eisenbud1995}).
To see this, we use the same arguments as in the proof of Proposition \ref{proposition:hhana} to show that the element
\begin{equation}
(u_1 \ldots u_n)^{a-1}((u_1 \ldots u_n)^{n-a} - a^a)
\end{equation}
lies in the Jacobian ideal (see in particular \eqref{eqn:jacrel}).
The term $(u_1 \ldots u_n)^{n-a} - a^a$ is invertible in the power series ring; it follows that $(u_1 \ldots u_n)^{a-1}$
lies in the ideal.
Substituting in $u_1 \ldots u_n = au_j^a$, we obtain that
\begin{equation} \label{eqn:upower0} (u_j^a)^{a-1} =0,\end{equation}
and hence $u_j^{a(a-1)}$ lies in the ideal generated by our sequence, for all $j$.

It follows that the sequence is regular, and hence that its cohomology is exactly the Jacobian ring, as required.
 \end{proof}

We observe that the cokernel of the morphism $\bm{p}$ of grading data is the group $\Gamma^n_a$ introduced in Definition \ref{definition:mir}.
The $\bm{G}$-grading on $\scrA$ induces a $\Gamma^n_a$-grading, with respect to which $R$ has degree $0$; therefore, $\scrA_{\C}$ comes with a $\Gamma^n_a$-grading.
This equips $HH^*_{\bm{G}_{\sigma}}(\scrA_{\C})$ with a $\Gamma^n_a$-grading (because $\Gamma^n_a$ is a finite group, the completion in the category of $\Gamma^n_a$-graded modules coincides with the completion in the category of $\bm{G}_{\sigma}$-graded modules).
A $\Gamma^n_a$-grading induces a $(\Gamma^n_a)^*$-action.
We will be particularly interested in computing the $(\Gamma^n_a)^*$-equivariant part of $HH^*_{\bm{G}_{\sigma}}(\scrA_{\C})$, or equivalently the part of degree $0 \in \Gamma^n_a$.

The isomorphism of Lemma \ref{lemma:hhanov} respects the $\Gamma^n_a$-grading, where the degree of the variable $u_j$ in the Jacobian ring is $y_j$, the image of the $j$th generator in $(\Z/a\Z)^n/(\Z/a\Z)$.
In order to make computations in the Jacobian ring of the formal power series ring, we first need to understand the Jacobian ring of the polynomial ring:

\begin{lemma}
\label{lemma:novpoly}
Let
\begin{equation} \overline{\beta} := \beta \otimes 1 \in \C[u_1, \ldots, u_n]\end{equation}
denote the image of $\beta$ under the identification
\begin{equation} R[U] \otimes_R \C \cong \C[U].\end{equation}
Then we have an isomorphism of $\C$-algebras,
\begin{equation} \left(\frac{\C [ u_1, \ldots, u_n ]}{\del{Z^n_a}{u_1}, \ldots, \del{Z^n_a}{u_n}}\right)
^{(\Gamma^n_a)^*} \cong \C[\overline{\beta}]/q^n_a(\overline{\beta}),\end{equation}
where $q^n_a$ is as in \eqref{eqn:qna}.
\end{lemma}
\begin{proof}
It is easy to check that the part of $\C[U]$ of degree $0 \in \Gamma^n_a$ is generated (as an algebra) by the elements $u_1 \ldots u_n$ and $u_j^a$.
Furthermore, these are all identified with a multiple of $\overline{\beta}$ in the Jacobian ring; so the part of the Jacobian ring of degree $0$ has the form $\C[\overline{\beta}]/p(\overline{\beta})$.
By the same arguments as in the proof of Proposition \ref{proposition:hhana}, $q^n_a(\overline{\beta}) = 0$.
By setting all $r_j$ equal to $1$, the Gr\"{o}bner basis computed in Proposition \ref{proposition:hhana} becomes a Gr\"{o}bner basis for the Jacobian ideal here, with respect to the homogeneous lexicographic order; and it follows as in the proof of Proposition \ref{proposition:hhana} that $\overline{\beta}$ satisfies no other polynomial relations.
 \end{proof}

\begin{corollary}
\label{corollary:hhanoveq}
Let us denote by $\hat{\beta}$ the image of $\beta \otimes 1$ in $HH^*_{\bm{G}_{\sigma}}(\scrA_{\C})$, where $\beta \in HH^*_{\bm{G}}(\scrA|R)$ is the element defined in Proposition \ref{proposition:hhana}.
Then for $n-1 \ge a \ge 2$, we have
\begin{equation} HH^*_{\bm{G}_{\sigma}}(\scrA_{\C})^{(\Gamma^n_a)^*} \cong \C[\hat{\beta}]/\hat{\beta}^{a-1}.\end{equation}
\end{corollary}
\begin{proof}
First we give the proof in the case $a=2$.
In this case, $H^*(\scrA_\C) \cong \Cl_n$ by Corollary \ref{corollary:cohalg}, hence there is an $A_\infty$ quasi-isomorphism $\scrA_\C \cong \Cl_n$ by Corollary \ref{corollary:intform}.
It follows that
\begin{equation} HH^*(\scrA_\C)^{(\Gamma^n_a)^*} \cong \C \cong \C[\hat{\beta}]/\hat{\beta}^{a-1}\end{equation}
by Corollary \ref{corollary:hhvans}.

Now we give the proof for $a \ge 3$.
By Lemma \ref{lemma:hhanov}, we can make computations in the equivariant part of the power series Jacobian ring.
As in Lemma \ref{lemma:novpoly}, we can show that $\hat{\beta}$ generates the equivariant part.
We showed in the proof of Lemma \ref{lemma:hhanov} that $\hat{\beta}^{a-1} = 0$ (see \eqref{eqn:upower0}).
Suppose that a smaller power of $\hat{\beta}$ vanished.
Then we would have
\begin{equation} \hat{\beta}^{a-2} = \sum_{j=1}^n f_j \del{Z^n_a}{u_j},\end{equation}
where $f_j \in \C \llbracket u_1, \ldots, u_n \rrbracket$.
By removing all of the terms in each $f_j$ of length $\ge N$, this shows that $\beta^{a-2}$ is equivalent to an element of length $\ge N$ in the polynomial Jacobian ring, for arbitrarily large $N$; this contradicts Lemma \ref{lemma:novpoly}.
Therefore $\hat{\beta}^{a-2} \neq 0$, and the proof is complete.
 \end{proof}

We now define the $\bm{G}_{\sigma}$-graded $A_{\infty}$ category
\begin{equation}\widetilde{\scrA}_{\C} := \bm{\sigma}_* \bm{p}^* \underline{\scrA} \otimes_R \C.\end{equation}

\begin{corollary}
\label{corollary:hh2na}
Let
\begin{equation} \hat{\gamma} := \gamma \otimes 1 \in HH^*_{\bm{G}_{\sigma}}(\widetilde{\scrA}_{\C}),\end{equation}
where $\gamma$ is as in Corollary \ref{corollary:hhanatild}.
When $n-1 \ge a \ge 2$, we have an isomorphism of $\Z/2\Z$-graded $\C$-algebras
\begin{equation} HH^*_{\bm{G}_{\sigma}}\left(\widetilde{\scrA}_{\C}\right)^{\Gamma^n_a} \cong \C[\hat{\gamma}]/\hat{\gamma}^{a-1},\end{equation}
 where $\hat{\gamma}$ has degree $0\in \Z/2\Z$.
\end{corollary}
\begin{proof}
We recall from \cite[Remark 2.66]{Sheridan2015} that there is an action of $\Gamma^n_a$ on $CC^*_{\bm{G}^n_a}(\bm{p}^*\underline{\scrA})$ (by shifts on the objects), and we have
\begin{equation} CC^*_{\bm{G}^n_a}(\bm{p}^*\underline{\scrA})^{\Gamma^n_a} \cong CC^*_{\bm{G}^n_1}(\scrA)^{(\Gamma^n_a)^*}.\end{equation}
Using the fact that taking invariants and taking direct products commute, it follows that
\begin{equation}CC^*_{\bm{G}_{\sigma}}(\bm{\sigma}_* \bm{p}^* \underline{\scrA} \otimes_R \C)^{\Gamma^n_a} \cong  CC^*_{\bm{G}_{\sigma}}(\bm{\sigma}_* \scrA \otimes_R \C)^{(\Gamma^n_a)^*}. \end{equation}
It follows that
\begin{align}
HH^*_{\bm{G}_{\sigma}}(\widetilde{\scrA}_{\C})^{\Gamma^n_a} &\cong  HH^*_{\bm{G}_{\sigma}}(\scrA_{\C})^{(\Gamma^n_a)^*} \\
&\cong \C[\hat{\beta}]/\hat{\beta}^{a-1}
\end{align}
by Corollary \ref{corollary:hhanoveq}; $\hat{\beta}$ corresponds to $\hat{\gamma}$ under this isomorphism, so the proof is complete.
 \end{proof}

We remark that the $\Z/2\Z$-grading in Corollary \ref{corollary:hh2na} can be enhanced to a $\Z/2(n-a)$-grading.
Let $\bm{G}_{2(n-a)}$ denote the grading datum $\{\Z \To \Z/2(n-a)\}$.
Let
\begin{equation} 
\label{eqn:qgna}
\bm{q}: \bm{G}^n_a \To \bm{G}_{2(n-a)}\end{equation}
be a morphism of grading data, sending all $y_j$ to $0$.
Because $2(a-n) \oplus y_{\{1,\ldots,n\}}$ maps to $0$ in $\bm{G}_{2(n-a)}$, this map is well-defined.

Then one can check that $\bm{q}_* \bm{p}^* R$ is concentrated in degree $0$; therefore it makes sense to define the $\bm{G}_{2(n-a)}$-graded $\C$-linear category
\begin{equation} \widetilde{\scrA}'_{\C} := \bm{q}_* \bm{p}^* \underline{\scrA} \otimes_R \C.\end{equation}
Then the result of Corollary \ref{corollary:hh2na} can be upgraded to:

\begin{corollary}
\label{corollary:hhg2}
When $n-1 \ge a \ge 2$, there is an isomorphism of $\Z/2(n-a)$-graded $\C$-algebras\begin{equation} HH^*_{\bm{G}_{2(n-a)}}\left(\widetilde{\scrA}'_{\C}\right)^{\Gamma^n_a} \cong \C[\hat{\gamma}]/\hat{\gamma}^{a-1},\end{equation}
 where $\hat{\gamma}$ has degree $2$.
\end{corollary}

\section{Computations in the Fukaya category of a Fermat hypersurface}
\label{sec:fermats}

\subsection{Fermat hypersurfaces}

We consider the Fermat hypersurfaces
\begin{equation} X^n_a:= \left\{ \sum_{j=1}^n z_j^a = 0 \right\} \subset \CP{n-1},\end{equation}
with the smooth normal-crossings divisor
\begin{align}
D & :=  \bigcup_{j=1}^n D_j, \mbox{ where} \\
D_j & := \{z_j = 0\}.
\end{align}
There is a branched cover
\begin{equation}
\label{eqn:brcov} \phi: (X^n_a,D) \To (X^n_1,D),\end{equation}
given by
\begin{align}
\phi: \left\{ \sum_{j=1}^n z_j^a = 0\right\} &\To \left\{ \sum_{j=1}^n z_j = 0 \right\}, \\
{[} z_1: \ldots : z_n {]} & \mapsto  {[} z_1^a : \ldots : z_n^a {]}.
\end{align}
It has branching of order $a$ about each component $D_j$ of $D$.

A structure of K\"{a}hler pair on $(X^n_a,D)$ was specified in \cite[Example 3.14]{Sheridan2015}: in particular, the linking number of the Liouville one-form with divisor $D_j$ is
\begin{equation}
\label{eqn:link}
\ell_j = a,
\end{equation}
hence the cohomology class of the symplectic form is
\begin{equation} [\omega] = na c_1(\mathcal{O}(1)).\end{equation}
The first Chern class is
\begin{equation} c_1 = (n-a) c_1(\mathcal{O}(1)).\end{equation}
As a result, $X^n_a$ is monotone in the sense of Definition \ref{definition:strmon}, with
\begin{equation}
\label{eqn:xnatau}
\tau = \frac{na}{2(n-a)}.
\end{equation}

\subsection{The immersed Lagrangian sphere}
\label{subsec:immLag}

We recall the construction of an immersed Lagrangian sphere $L: S^{n-2} \To X^n_1 \setminus D$ from \cite[\S 2.2]{Sheridan2011} .
To start, we observe that $X^n_1 \cong \{z_1+\ldots+z_n = 0\} \subset \CP{n-1}$: so in fact $X^n_1 \cong \CP{n-2}$.
The immersion $L$ is constructed by starting with the immersion
\begin{equation} L':S^{n-2} \xrightarrow{2:1} \RP{n-2} \hookrightarrow \CP{n-2} \cong X^n_1\end{equation}
(which intersects $D$), and pushing it off itself using a certain Morse function $f: S^{n-2} \to \R$.

To describe the construction of the Morse function $f$, it is easiest to represent $S^{n-2}$ as
\begin{equation} S^{n-2} = \left\{(x_1,\ldots,x_n) \in \R^n: \sum_{j=1}^n x_i = 0, \sum_{j=1}^n x_i^2=1\right\},\end{equation}
so that the immersion $L'$ sends $(x_1,\ldots,x_n) \mapsto [x_1:\ldots:x_n]$.
We then define
\begin{equation} f(x_1,\ldots,x_{n}) = \sum_{j=1}^{n} g(x_j),\end{equation}
where $g: \R \To \R$ is a monotone, odd function with the property that $g(x) = x$ inside a (small) neighbourhood of $0$, and $g'(x)$ is small for $x$ outside a (slightly larger) neighbourhood of $0$.
The construction ensures that  $\nabla f$ is positively transverse to the hypersurfaces $\{x_j = 0\}$.
A picture of the Morse flow of $f$ in the $2$-dimensional case can be found in \cite[Figure 5]{Sheridan2011}.

We use the Weinstein neighbourhood theorem to extend the immersion $L'$ to an immersion from a sufficiently small cotangent disc bundle of $S^{n-2}$ to $X^n_1$.
The immersion $L$ is then constructed as the graph of the one-form $\epsilon df$, mapped into $X^n_1$ via this immersion (for $\epsilon >0$ sufficiently small).
The fact that $\nabla f$ is transverse to the hypersurfaces $\{x_j = 0\}$ ensures that $L$ avoids the divisors $D_j$: so indeed the image of $L$ lies in $X^n_1 \setminus D$.
The fact that $f$ can be arranged to be a Morse function ensures that $L$ has transverse self-intersections at the critical points of $f$ (where it intersects the other branch of the double cover).

The hypersurfaces $\{x_j = 0\}$ split $S^{n-2}$ into $2^{n}-2$ regions, indexed by the subsets $K \subset \{1,\ldots,n\}$ of coordinates $x_j$ that are positive in the region: all subsets $K$ are realized except for $K=\emptyset, \{1,\ldots,n\}$, because we can't have $\sum_j x_j= 0$ if the coordinates $x_j$ are either all positive or all negative.
Each region contains a unique critical point $p_K$ of $f$.
The Floer endomorphism algebra of $L$ can be defined, despite it being immersed, using the fact that $L$ admits an embedded lift to a cover of $X^n_1 \setminus D$ (see \cite[\S 3.1]{Sheridan2011}).
It is generated by the self-intersection points of $L$ (which are indexed by proper non-empty sets $K \subset \{1,\ldots,n\}$), together with the Morse cohomology of $S^n$ (which we choose to have generators $p_{\phi}$ and $p_{\{1,\ldots,n\}}$, corresponding to the identity and top class respectively).
Thus, $CF^*(L,L)$ has generators $p_K$ indexed by the subsets $K \subset \{1,\ldots,n\}$.

Now recall the branched cover $\phi: (X^n_a,D) \to (X^n_1,D)$ of \eqref{eqn:brcov}.
It restricts to an unbranched cover $\phi: X^n_a \setminus D \to X^n_1 \setminus D$, and $L$ admits a lift to $X^n_a \setminus D$ by the homotopy lifting criterion, because the image $\pi_1(L) \to \pi_1(X^n_1 \setminus D)$ is trivial.

\begin{lemma}
\label{lem:Lliftemb}
If $a \ge 2$, then any lift of $L$ to $X^n_a$ is embedded.
\end{lemma}
\begin{proof}
We have $H_1(X^n_1 \setminus D) \cong \Z\langle e_1,\ldots,e_n \rangle/(e_1+\ldots+e_n) \cong H_1(X^n_a \setminus D)$ (the generators correspond to meridian loops about the components $D_j$ of $D$).
The map
\begin{equation} \phi_*: H_1(X^n_a\setminus D) \to H_1(X^n_1 \setminus D)\end{equation}
can be identified with multiplication by $a$ (because $\phi$ has branching of degree $a$ about each $D_j$, compare \cite[Corollary 3.30]{Sheridan2015}).

To any self-intersection $x$ of $L$, we can associate an element $y \in H_1(X^n_1 \setminus D)$: it is the image of the path in $S^{n-2}$ connecting the two antipodal points that intersect at $x$, under the immersion $L$.
The self-intersection $x$ in $X^n_1 \setminus D$ lifts to a self-intersection in $X^n_a \setminus D$ if and only if $y$ lies in the image of $H_1(X^n_a \setminus D)$: otherwise, $x$ lifts to an intersection between two different lifts of $L$.

These classes $y$ were computed in \cite[Proposition 3.3]{Sheridan2011}: the class associated to $p_K$ is $\sum_{j \in K} e_j$.
It is clear that none of these lie in the image of $H_1(X^n_a \setminus D)$, with the exception of $p_\emptyset$ and $p_{\{1,\ldots,n\}}$, which do not correspond to self-intersections.
\end{proof}

\subsection{Fukaya category computations}

The relative Fukaya category $\cF(X^n_a,D)$ can be computed using the branched cover $\phi: X^n_a \to X^n_1$, via techniques developed in \cite{Sheridan2015}. 
We recall the details.

We start by recalling the relationship between $\cF(X^n_a \setminus D)$ and $\cF(X^n_1 \setminus D)$. 
The grading data for these two categories are denoted $\bm{G}(X^n_a,D)$ and $\bm{G}(X^n_1,D)$ respectively, and the branched cover \eqref{eqn:brcov} determines a morphism of grading data,
\begin{equation} \bm{p}:\bm{G}(X^n_a,D) \To \bm{G}(X^n_1,D),\end{equation}
whose cokernel is the covering group of $\phi$. 
The relationship between $\cF(X^n_a \setminus D)$ and $\cF(X^n_1 \setminus D)$ is particularly simple, because pseudoholomorphic discs in $X^n_1 \setminus D$ lift to $X^n_a \setminus D$ as they are contractible: so the difference between the two categories is essentially one of `bookkeeping'. 
This idea was exploited in \cite[\S 8b]{Seidel2003}; we will use the language of \cite[Proposition 3.10]{Sheridan2015}, which says that a choice of identification between the universal abelian covers of $\cG(X^n_a \setminus D)$ and $\cG(X^n_1 \setminus D)$ induces a fully faithful strict embedding
\begin{equation} \bm{p}^* \cF(X^n_1 \setminus D) \hookrightarrow \cF(X^n_a \setminus D)\end{equation}
(recall that applying the operation `$\bm{p}^*$' to a category does not change the objects, but changes the morphisms: it only keeps the morphisms whose grading lies in the image of $\bm{p}$, see \cite[Definition 2.65]{Sheridan2015}). 

In our situation, the grading data associated to the K\"{a}hler pairs $(X^n_a,D)$ are
\begin{equation}
\bm{G}(X^n_a,D)  \cong  \bm{G}^n_a,
\end{equation}
where the definition of the right-hand side was recalled in equation \eqref{eqn:gna} (see \cite[Lemma 3.24]{Sheridan2015}).
The morphism of grading data $\bm{p}$ is given by equation \eqref{eqn:morphp}, by \cite[Lemma 3.29]{Sheridan2015}.
The cokernel of $\bm{p}$ (which we recall is isomorphic to the covering group of $\phi$) is naturally isomorphic to the group $\Gamma^n_a$ of the introduction (see \eqref{eqn:gammana}).

On the other hand, the relationship between the relative Fukaya categories $\cF(X^n_a,D)$ and $\cF(X^n_1,D)$ is more subtle, because pseudoholomorphic discs that pass through $D$ do not admit lifts under $\phi$. 
Instead, we consider the $\bm{a}$-orbifold relative Fukaya category $\cF(X^n_1,D,\bm{a})$, where $\bm{a} := (a,\ldots,a)$. 
This is the Fukaya category whose structure maps count pseudoholomorphic discs which are tangent to $D$ to order $a-1$ wherever they meet. 
These discs do admit lifts under $\phi$, so we should obtain an embedding 
\begin{equation}
\label{eqn:relFlift}
\bm{p}^* \cF(X^n_1,D,\bm{a}) \hookrightarrow \cF(X^n_a,D)
\end{equation}
as above. 
The coefficient ring of $\cF(X^n_1,D,\bm{a})$ is the ring $R_a$ of \eqref{eqn:ra} (see \cite[Example 5.2]{Sheridan2015}).
It follows that the coefficient ring of $\cF(X^n_a,D)$ is $\bm{p}^* R_a$, by \cite[Remark 5.10]{Sheridan2015}.

\begin{remark}
\label{rmk:subtlety}
Actually there are some technical complications that prevent us from proving \eqref{eqn:relFlift} exactly: instead we prove it to first order in $R$. 
See the proof of Proposition \ref{proposition:atild} for details.   
\end{remark}

\begin{proposition}
\label{proposition:atild}
Let $\widetilde{\scrA}$ denote the full subcategory of the $\bm{G}^n_a$-graded monotone relative Fukaya category $\cF_m(X^n_a,D)_w$ whose objects are lifts of the immersed Lagrangian sphere $L \subset X^n_1$ under the branched cover $\phi$.
Then there exists a non-curved $A_{\infty}$ algebra $\scrA$ of type A$^n_a$ (in the sense of Definition \ref{definition:typean}), such that
\begin{equation} \widetilde{\scrA} \cong \bm{p}^* \underline{\scrA}\end{equation}
(where $\underline{\scrA}$ denotes the unique extension of $\scrA$ to a $\bm{G}^n_1$-graded $A_{\infty}$ category by formally adding all shifts, as in \cite[Definition 2.64]{Sheridan2015}).
\end{proposition}
\begin{proof}
Because $\Gamma^n_a$ acts freely on the set of lifts of $L$, we can choose $\Gamma^n_a$-equivariant perturbation data for $\widetilde{\scrA}$.
It follows that
\begin{equation} \widetilde{\scrA} \cong \bm{p}^* \underline{\scrA}\end{equation}
for some $\bm{G}^n_1$-graded, $R$-linear $A_{\infty}$ algebra $\scrA$.

Now we come to the subtlety alluded to in Remark \ref{rmk:subtlety}: instead of \eqref{eqn:relFlift} holding, we have \cite[Proposition 5.11]{Sheridan2015}. 
This proves the closely-related result that there is a $\bm{G}^n_1$-graded $R_a$-linear $A_{\infty}$ category $\cF(\phi)$, such that there is a strict $R/\mathfrak{m}^2$-linear isomorphism
\begin{equation} \cF(\phi) /\mathfrak{m}^2 \cong \cF(X^n_1,D,\bm{a}) /\mathfrak{m}^2,\end{equation}
where
\begin{equation} \mathfrak{m}:=(r_1, \ldots, r_n)  \subset R_a,\end{equation}
and there is also a fully faithful strict $\bm{p}^* R_a$-linear embedding
\begin{equation} \bm{p}^* \cF(\phi) \hookrightarrow \cF(X^n_a,D).\end{equation}

Modifying the proof of Proposition \ref{proposition:monrel} to take into account the $\Gamma^n_a$-equivariance, we obtain a first-order quasi-equivalence
\begin{equation} \scrA/\mathfrak{m}^2 \cong CF^*_{\cF(\phi)}(L,L)/\mathfrak{m}^2,\end{equation}
and hence a first-order quasi-equivalence
\begin{equation} \scrA/\mathfrak{m}^2 \cong \scrA'/\mathfrak{m}^2,\end{equation}
where
\begin{equation}\scrA' := CF^*_{\cF(X^n_1,D,\bm{a})}(L,L)/\mathfrak{m}^2.\end{equation}
Because being of type A$^n_a$ only depends on the zeroth- and first-order coefficients of the $A_{\infty}$ structure maps, it suffices to prove that $\scrA'$ is of type A$^n_a$.

$\scrA'$ is $\bm{G}^n_1$-graded and $R_a$-linear by definition.
We already saw that it has generators $p_K$ corresponding to the subsets $K \subset \{1,\ldots,n\}$. 
The $\bm{G}^n_1$-grading on these generators is computed in \cite[Proposition 3.3]{Sheridan2011} and \cite[Proposition 3.7]{Sheridan2011}: the result is that there is an isomorphism
\begin{equation} CF^*(L,L) \cong A\end{equation}
as $\bm{G}^n_1$-graded vector spaces, where $A$ is the $\bm{G}^n_1$-graded exterior algebra introduced in \S \ref{subsec:ana} (see equation \eqref{eqn:Aext}).

By \cite[Proposition 6.2]{Sheridan2015}, the order-$0$ cohomology algebra of $\scrA'$ is an exterior algebra and
\begin{equation} \Phi(\mu^* ) = \pm u_1 \ldots u_n + \mathfrak{m}.\end{equation}
\cite[Proposition 6.2]{Sheridan2015} also shows that the first-order deformation class of the endomorphism algebra of $L$ in the non-orbifold relative Fukaya category $\cF(X^n_1,D)$ is
\begin{equation} \pm \sum_{j=1}^n r_j u_j,\end{equation}
i.e., the first-order deformation classes are $u_j$. 
The first-order deformation classes of the orbifold relative Fukaya category $\cF(X^n_1,D,\bm{a})$ are obtained from those of the relative Fukaya category $\cF(X^n_1,D)$ by taking the $a$th power with respect to Yoneda product, by \cite[Theorem 5.12]{Sheridan2015}.
Thus, the first-order deformation class of $\cF(X^n_a,D,\bm{a})$ is
\begin{equation} \pm \sum_{j=1}^nr_ju_j^a.\end{equation}
It follows that (after an appropriate change of variables to fix the signs)
\begin{equation} \Phi(\mu^* ) = \tilde{Z}^n_a + \mathfrak{m}^2,\end{equation}
where $\tilde{Z}^n_a$ is as in \eqref{eqn:ztild}, and hence that $\scrA'$ (and therefore $\scrA$) is of type A$^n_a$, as required.
 \end{proof}

\begin{lemma}
\label{lemma:minchern}
If $X$ is simply-connected, then
\begin{equation} \bm{G}(X) \cong \bm{G}_{2N},\end{equation}
where $N \in \Z$ is the minimal Chern number on spherical classes, i.e., the generator of the image of the map
\begin{equation} c_1(TX): \pi_2(X) \To \Z,\end{equation}
and we recall that $\bm{G}_{2N}$ denotes the grading datum $\{\Z \To \Z/2N\}$.
\end{lemma}
\begin{proof}
Follows from the long exact sequence for the fibration $\cG X$:
\begin{equation} \pi_2(X) \To \pi_1(\cG_xX) \To \pi_1(\cG X) \To \pi_1(X),\end{equation}
together with the observation that $\pi_1(\cG_xX) \cong \Z$ (given by the Maslov class), and the first map in the exact sequence coincides with evaluation of $2c_1(TX)$.
 \end{proof}

The map
\begin{equation} \bm{q}: \bm{G}(X^n_a,D) \To \bm{G}(X^n_a)\end{equation}
coincides with the map
\begin{equation} \bm{G}^n_a \To \bm{G}_{2(n-a)}\end{equation}
of \eqref{eqn:qgna}.

\begin{corollary}
\label{corollary:anaemb}
If $2 \le a \le n-1$, then there is a fully faithful embedding of $\bm{G}_{2(n-a)}$-graded, $\C$-linear, non-curved $A_{\infty}$ categories
\begin{equation} \bm{q}_* \bm{p}^* \underline{\scrA} \otimes_R \C \hookrightarrow \cF(X^n_a)_{\bm{w}},\end{equation}
where $\scrA$ is an $A_{\infty}$ algebra of type A$^n_a$.
If $a \le n-2$ then $\bm{w} = 0$, but if $a=n-1$ then $\bm{w}$ may be non-zero.
\end{corollary}
\begin{proof}
By Lemma \ref{lemma:c1disc}, the Maslov index of a disc with boundary on a lift of $L$ is $2(n-a) u \cdot D$.
Therefore, if $a \le n-2$ then all lifts of $L$ cannot bound Maslov index $2$ discs, and therefore have $w(L) = 0$.
If $a=n-1$, then $L$ may bound Maslov index $2$ discs.
All lifts of $L$ have the same value $\bm{w}$ of $w(L)$, by symmetry.

The result then follows from Proposition \ref{proposition:atild} and Lemma \ref{lemma:reltomon}.
 \end{proof}

\begin{remark}
We will prove (in Corollary \ref{corollary:w}) that, when $a=n-1$, we have $\bm{w} = -a!$.
\end{remark}

\subsection{The quantum cohomology of $X^n_a$}
\label{subsec:qhxna}

In this section, we will study $QH^*(X^n_a)$, using the closed--open and open--closed string maps.
First we would like to understand $H^*(X^n_a)$ better.
We recall the Lefschetz decomposition of $H^*(X^n_a)$:
\begin{equation} H^*(X^n_a) \cong H^*_H(X^n_a) \oplus H^*_P(X^n_a),\end{equation}
where `$H$' stands for `Hodge' and `$P$' stands for `primitive'.
By the Lefschetz hyperplane theorem, the primitive cohomology is concentrated in the middle degree $d = n-2$.
So the Hodge part is generated by the hyperplane class $P$:
\begin{equation} H^*_H(X^n_a) \cong \C[P]/P^{n-1},\end{equation}
and the primitive part is
\begin{equation} H^*_P(X^n_a) \cong ker(\wedge P: H^d(X^n_a) \To H^{d+2}(X^n_a)).\end{equation}

\begin{lemma}
\label{lemma:lef}
The image of the map
\begin{equation} H^*_c(X^n_a \setminus D) \To  H^*(X^n_a) \end{equation}
contains the primitive cohomology $H^*_P(X^n_a)$.
\end{lemma}
\begin{proof}
We observe that
\begin{equation} H^*_c(X^n_a \setminus D) \cong H^*(X^n_a,D),\end{equation}
and apply the long exact sequence in cohomology for the pair $(X^n_a,D)$:
\begin{equation} \ldots \To H^d(X^n_a,D) \To H^d(X^n_a) \To H^d(D) \To \ldots\end{equation}
It suffices to prove that the image of the primitive cohomology in $H^d(D)$  vanishes.
To do this, it is sufficient to show that the map
\begin{equation} \wedge P: H^d(D) \To H^{d+2}(D)\end{equation}
is injective   (where $P$ is the restriction of the hyperplane class to $D$).
In other words, $D$ has no `primitive cohomology' in degree $d$.

To understand the cohomology of $D$, we apply the generalized Mayer-Vietoris principle \cite[Proposition 8.8]{Bott1982} to the open cover by neighbourhoods of the components $D_i$.
This yields a bounded double complex, hence a spectral sequence converging to $H^*(D)$, with $E_1$ page
\begin{equation} E_1^{p,q} \cong \bigoplus_{K \subset [k], |K| = p-1} H^q(D_K),\end{equation}
where $D_K$ denotes the intersection of all divisors indexed by $i \in K$.
The map $\wedge P$ defines a homomorphism from this spectral sequence to itself, of degree $2$.

Now observe that $D_K$ is a hypersurface in a projective space $\CP{d+1-|K|}$, so only has primitive cohomology in degree $d-|K|$.
It follows that the map
\begin{equation} \wedge P: H^q(D_K) \To H^{q+2}(D_K)\end{equation}
is injective unless $q=d-|K|$ or $q=2(d-|K|)$.
Note that neither of these conditions can be satisfied when $d=p+q = |K|+1+q$.
Therefore, the map
\begin{equation} \wedge P: E_1^{p,q} \To E_1^{p,q+2}\end{equation}
is injective for all $p+q = d$.
Since the spectral sequence converges, it follows that
\begin{equation} \wedge P: H^d(D) \To H^{d+2}(D)\end{equation}
is injective.
This completes the proof.
 \end{proof}

Now let $\Gamma^n_a$ denote the covering group of the branched cover
\begin{equation}\phi: (X^n_a,D) \To (X^n_1,D).\end{equation}
It clearly acts on $X^n_a$, and hence on the cohomology $H^*(X^n_a)$.

\begin{lemma}
\label{lemma:eqhodg}
The $\Gamma^n_a$-invariant part of $H^*(X^n_a)$ is exactly the Hodge part of cohomology:
\begin{equation} H^*(X^n_a)^{\Gamma^n_a} \cong H^*_H(X^n_a).\end{equation}
\end{lemma}
\begin{proof}
The action of $\Gamma^n_a$ obviously fixes the Hodge part of the cohomology, and preserves the primitive cohomology, so it suffices to prove that the $\Gamma^n_a$-fixed part of $H^*_P(X^n_a)$ is trivial.
By Lemma \ref{lemma:lef}, the map
\begin{equation} H^*_c(X^n_a \setminus D) \To H^*_P(X^n_a)\end{equation}
is surjective.
This map is clearly $\Gamma^n_a$-equivariant, and it follows that the induced map
\begin{equation} H^*_c(X^n_a \setminus  D)^{\Gamma^n_a} \To H^*_P(X^n_a)^{\Gamma^n_a}\end{equation}
is surjective.

The restriction $\phi: X^n_a \setminus D \To X^n_1 \setminus D$ is an unbranched cover, so we have
\begin{equation} H^*_c(X^n_a \setminus D)^{\Gamma^n_a} \cong H^*_c(X^n_1 \setminus D)\end{equation}
(using de Rham cohomology, this can be realized by averaging differential forms).

Therefore, because the diagram
\begin{equation} \xymatrix{
H^*_c(X^n_a \setminus D)^{\Gamma^n_a} \ar[r] & H^*(X^n_a)^{\Gamma^n_a} \\
H^*_c(X^n_1 \setminus D) \ar[u]^{\phi^*} \ar[r] & H^*(X^n_1) \ar[u]^{\phi^*}
} \end{equation}
commutes, the map
\begin{equation} \phi^*: H^*(X^n_1) \To H^*(X^n_a)^{\Gamma^n_a} \end{equation}
contains the equivariant primitive cohomology classes in its image.
However, $X^n_1 \cong \CP{n-2}$, so its cohomology is generated by the hyperplane class $P$; it follows that the only equivariant primitive cohomology class is $0$.
 \end{proof}

\begin{lemma}
\label{lemma:primspheres}
The homology classes of the lifts of the Lagrangian sphere $L$ to $X^n_a$ span the primitive homology of $X^n_a$.
\end{lemma}
\begin{proof}
We regard $X^n_a \setminus D$ as a submanifold of
\begin{equation} \CP{n-1} \setminus D \cong (\C^*)^{n-1}.\end{equation}
We observe that the argument map
\begin{equation} \mathrm{Arg}: X^n_a \setminus D \To (S^1)^{n-1}\end{equation}
is a homotopy equivalence onto its image (the `coamoeba').
This follows from \cite[Proposition 2.2]{Sheridan2011}.
The coamoeba of $X^n_a \setminus D$ is a $\Gamma^n_a$-cover of the coamoeba of the pair of pants $X^n_1 \setminus D$, which is homotopy equivalent to $(S^1)^{n-1}$ with one point removed: so $X^n_a \setminus D$ is homotopy-equivalent to an $(n-1)$-torus with $a^{n-1}$ points removed.
The homology classes of the lifts of the Lagrangian $L$ correspond to balls around each removed point by construction (see \cite[Proposition 2.6]{Sheridan2011}).

Therefore the lifts of $L$, together with the $n-1$ coordinate $(n-2)$-tori, span the middle-dimensional homology of $X^n_a \setminus D$.
The coordinate $(n-2)$-tori correspond to small tori near the zero-dimensional strata of the boundary divisor $D$; when we compactify by adding the divisor back in, their homology classes disappear because they are bounded by polydiscs near $D$.

It follows that the homology classes of the Lagrangian spheres span the image of the map
\begin{equation} H_{n-2}(X^n_a \setminus D) \To H_{n-2}(X^n_a).\end{equation}
It follows by Lemma \ref{lemma:lef}, using Poincar\'{e} duality, that the homology classes of the Lagrangian spheres span the primitive homology.
 \end{proof}

Now we consider the quantum cohomology, $QH^*(X^n_a,D)$.
It also admits a $\Gamma^n_a$-action, so we may talk about the invariant part.
The closed--open string map intertwines this $\Gamma^n_a$-action with the $\Gamma^n_a$-action on the relative Fukaya category, so we have a map
\begin{equation} \CO: QH^*(X^n_a,D)^{\Gamma^n_a} \To HH^*_{\bm{G}}(\cF_m(X^n_a,D))^{\Gamma^n_a}\end{equation}
on the invariant parts.

\begin{proposition}
\label{proposition:qhfromhh}
Let $\widetilde{\scrA}$ be the full subcategory of $\cF_m(X^n_a,D)$ generated by lifts of $L$.
If $2 \le a \le n-1$, then the map
\begin{equation} \CO: QH^*(X^n_a,D)^{\Gamma^n_a} \To HH^*_{\bm{G}}\left(\widetilde{\scrA} \right)^{\Gamma^n_a}\end{equation}
is an isomorphism and, in the notation of Corollary \ref{corollary:hhanatild}, sends $P \mapsto \gamma + \bm{w} \cdot T$, where $\bm{w} \in \Z$ is the integer appearing in Corollary \ref{corollary:anaemb}.
In particular, the subalgebra of $QH^*(X^n_a,D)$ generated by $P$ is isomorphic to
\begin{equation} \C[P]/\tilde{q}^n_a(P - \bm{w} \cdot T),\end{equation}
where $\tilde{q}^n_a$ is defined in \eqref{eqn:tildeq}.
\end{proposition}
\begin{proof}
By Corollary \ref{corollary:anaemb}, we have
\begin{equation} \widetilde{\scrA} \cong \bm{p}^* \underline{\scrA},\end{equation}
where $\scrA$ is an $A_{\infty}$ algebra of type A$^n_a$.
We therefore have
\begin{equation} HH^*_{\bm{G}}\left(\widetilde{\scrA} \right)^{\Gamma^n_a} \cong  R[\gamma]/q^n_a(\gamma)\end{equation}
by Corollary \ref{corollary:hhanatild}.
We have
\begin{equation} \CO(P) = \left[ r_j \del{\mu^*}{r_j} \right] + \bm{w} \cdot T = \gamma + \bm{w} \cdot T\end{equation}
by Proposition \ref{proposition:cod}.

We know that $\CO$ is a unital algebra homomorphism, and the equivariant part of the Hochschild cohomology is generated by $\gamma$.
It follows that $\CO$ is surjective.
By Lemma \ref{lemma:eqhodg}, the equivariant part of the quantum cohomology is a free $R$-module of rank $n-2$, and by Corollary \ref{corollary:hhanatild}, so is the equivariant part of the Hochschild cohomology.
Since $\CO$ is surjective, it follows that it is an isomorphism.
 \end{proof}

The computations of quantum cohomology from Proposition \ref{proposition:qhfromhh} agree with the results of \cite[Corollaries 9.3 and 10.9]{Givental1996}.

\begin{corollary}
\label{corollary:w}
If $2 \le a \le n-1$, and $L$ is a lift of the immersed Lagrangian sphere to $X^n_a$, then $w(L) = \bm{w}^n_a$ (in the notation of \eqref{eqn:bwna}).
\end{corollary}
\begin{proof}
If $a \le n-2$, the result follows as in Corollary \ref{corollary:anaemb}.
If $a = n-1$, the result follows by comparing the relation established in Proposition \ref{proposition:qhfromhh} with  \cite[Corollary 10.9]{Givental1996}.
 \end{proof}

We recall the eigenvalues of $c_1 \star$, as computed in Corollary \ref{corollary:c1eval}.
We will denote the big eigenvalue by $\bm{w}$, and small eigenvalues by $w$.

\begin{corollary}
\label{corollary:evects}
If $2 \le a \le n-1$, then the Hodge part of the big generalized eigenspace has rank $a-1$:
\begin{equation} \mathrm{dim}(QH^*(X^n_a)_{\bm{w}} \cap H^*_H(X^n_a)) = a-1,\end{equation}
with a basis consisting of the generalized eigenvectors
\begin{equation} (P-\bm{w})^{i} ((P-\bm{w})^{(n-a)} - a^a ), \mbox{ for $i = 0, \ldots, a-2$.}\end{equation}
The Hodge part of the $n-a$ small generalized eigenspaces have rank 1:
\begin{equation} \mathrm{dim}(QH^*(X^n_a)_w \cap H^*_H(X^n_a)) = 1,\end{equation}
each spanned by an eigenvector of the form
\begin{equation}
\label{eqn:smallevect}
(P-\bm{w})^{a-1} \frac{(P-\bm{w})^{n-a} - a^a}{P-w}.\end{equation}
\end{corollary}
\begin{proof}
First, one can check that these are indeed generalized eigenvectors, which follows from the relation $(P-\bm{w})^{a-1}((P-\bm{w})^{n-a} - a^a) = 0$.
Note that any polynomial in $P$ is $\Gamma^n_a$-equivariant, hence lies in $H^*_H(X^n_a)$ by Lemma \ref{lemma:eqhodg}, so the generalized eigenvectors identified do lie in $H^*_H(X^n_a)$.
They also span $H^*_H(X^n_a)$, as it has rank $n-1$.
 \end{proof}

\begin{proposition}
\label{proposition:smallw}
If $2 \le a \le n-1$, and $w$ is one of the small eigenvalues of $c_1\star$ on $QH^*(X^n_a)$, then the generalized eigenspace $QH^*(X^n_a)_w$ has rank $1$.
\end{proposition}
\begin{proof}
By Corollary \ref{corollary:ocLe} and Lemma \ref{lemma:ocLf}, the classes
\begin{equation} \OC^0(e_L) = PD(L) + \mbox{ lower-degree terms}\end{equation}
are eigenvectors of $c_1 \star$ with eigenvalue $\bm{w}$, for each of the Lagrangian spheres $L$.

We now observe that the primitive cohomology $H^*_P(X^n_a)$ is concentrated in the middle degree $n-2$; therefore the terms of degree $<n-2$ lie in $H^*_H(X^n_a)$.
So for each lift $L$, we obtain an eigenvector $\OC^0(e_L)$ of $c_1 \star$ with eigenvalue $\bm{w}$, whose primitive component is $PD(L)$.
Because these classes span the primitive cohomology (by Lemma \ref{lemma:primspheres}), the composition
\begin{equation} QH^*(X^n_a)_{\bm{w}} \hookrightarrow H^*(X^n_a) \To H^*(X^n_a)/H^*_H(X^n_a) \cong H^*_P(X^n_a)\end{equation}
is surjective.
Therefore,
\begin{align}
\mathrm{dim}(QH^*(X^n_a)_{\bm{w}}) &= \mathrm{dim}(QH^*(X^n_a)_{\bm{w}} \cap H^*_H(X^n_a)) + \mathrm{dim}(H^*_P(X^n_a)) \\
&= a-1 + \mathrm{dim} (H^*_P(X^n_a)),
\end{align}
by Corollary \ref{corollary:evects}.

It follows that the $n-a$ remaining small generalized eigenspaces have a combined rank of
\begin{align}
\mathrm{dim}(H^*(X^n_a)) - \mathrm{dim}(QH^*(X^n_a)_{\bm{w}}) &= \mathrm{dim}(H^*_H(X^n_a)) - (a-1) \\
&= n-a.
\end{align}
As each has rank at least $1$, it follows that each has rank exactly $1$.
 \end{proof}

\begin{corollary}
If $L$ is any monotone Lagrangian with $w(L)$ not equal to the big eigenvalue $\bm{w}$, then $L$ is nullhomologous.
\end{corollary}
\begin{proof}
By Corollary \ref{corollary:ocLe} and Lemma \ref{lemma:ocLf}, for any monotone Lagrangian $L$,
\begin{equation}
\label{eqn:oc0small} \OC^0(e_L) = PD(L) + \mbox{ lower-degree terms} \end{equation}
lies in the eigenspace $QH^*(X)_{w(L)}$.
If $w(L)$ is not equal to an eigenvalue of $c_1 \star$, then \eqref{eqn:oc0small} must vanish, so $PD(L)$ vanishes.
If $w(L)$ is equal to a small eigenvalue, it follows by Proposition \ref{proposition:smallw} that \eqref{eqn:oc0small} must be a multiple of \eqref{eqn:smallevect}.
The latter has the form
\begin{equation} PD(pt) + \mbox{ lower-degree terms},\end{equation}
so \eqref{eqn:oc0small} can only be the zero multiple of it, hence $PD(L)$ vanishes.
 \end{proof}

\subsection{The big eigenvalue}

In this section, we examine the component $\cF(X^n_a)_{\bm{w}}$ of the monotone Fukaya category, where $\bm{w}$ is the big eigenvalue.

\begin{proposition}
\label{proposition:comon}
Let $\widetilde{\scrA}_{\C}$ be the full subcategory of $\cF(X^n_a)_{\bm{w}}$ whose objects are lifts of $L$.
If $2 \le a \le n-1$, then the map
\begin{equation} \CO: QH^*(X^n_a)^{\Gamma^n_a} \To HH^*\left(\widetilde{\scrA}_{\C}\right)^{\Gamma^n_a}\end{equation}
restricts to an isomorphism on the big generalized eigenspace $QH^*(X^n_a)_{\bm{w}}$, and vanishes on the small eigenspaces.
\end{proposition}
\begin{proof}
By Corollary \ref{corollary:evects}, the big generalized eigenspace has rank $a-1$, with basis
\begin{equation} (P-\bm{w})^i((P-\bm{w})^{n-a} - a^a), \mbox{ $i=0,\ldots,a-2$,}\end{equation}
and the other generalized eigenspaces have rank $1$, spanned by eigenvectors of the form $(P-\bm{w})^{a-1}f(P)$.

As in the proof of Proposition \ref{proposition:qhfromhh}, we have
\begin{equation} \widetilde{\scrA}_{\C} \cong \bm{q}_* \bm{p}^* \underline{\scrA} \otimes_R \C,\end{equation}
where $\scrA$ is an $A_{\infty}$ algebra of type A$^n_a$.
By Corollary \ref{corollary:hhg2}, we have
\begin{equation} HH^*\left(\widetilde{\scrA}_{\C}\right)^{\Gamma^n_a} \cong \C[\hat{\beta}]/\hat{\beta}^{a-1},\end{equation}
and as in the proof of Proposition \ref{proposition:qhfromhh}, we have
\begin{equation} \CO(P-\bm{w}) = \hat{\beta}.\end{equation}
It follows immediately that $\CO$ vanishes on elements of the form $(P-\bm{w})^{a-1}f(P)$, and hence on all of the generalized eigenspaces other than that associated to $0$.
Furthermore, on the generalized eigenspace associated to $0$, the image is generated by the elements
\begin{equation} \hat{\beta}^i (\hat{\beta}^{n-a} - a^a), \mbox{ for $i=0,\ldots,a-2$,}\end{equation}
which span $\C[\hat{\beta}]/\hat{\beta}^{a-1}$.
So the restriction to this generalized eigenspace is surjective; because the eigenspace has rank $a-1$, this means it is an isomorphism.
 \end{proof}

\begin{remark}
Note that Proposition \ref{proposition:comon} agrees with Proposition \ref{proposition:coeigsplit}: $\CO$ is unital on the big eigenspace, and vanishes on all the other eigenspaces.
\end{remark}

\begin{corollary}
\label{corollary:asplitgens}
If $2 \le a \le n-1$, then $\widetilde{\scrA}_{\C}$ split-generates $\cF(X^n_a)_{\bm{w}}$, where $\bm{w}$ is the big eigenvalue.
\end{corollary}
\begin{proof}
Follows immediately from (a $\Gamma^n_a$-equivariant version of) Corollary \ref{corollary:splitgen}, together with Proposition \ref{proposition:comon}.
 \end{proof}

\subsection{The small eigenvalues}
\label{subsec:smalleig}

In this section, we examine the components $\cF(X^n_a)_w$, where $w$ is a small eigenvalue.
In fact, we will need to make use of weak bounding cochains; so in fact, we work in $\cF^{wbc}(X^n_a)_w$.

We aim to apply Proposition \ref{proposition:hess} to the $A_\infty$ algebra of type A$^n_a$ appearing in Proposition \ref{proposition:atild}.
The first remark to make is that $\scrA$ is not, strictly speaking, an endomorphism algebra of a Lagrangian in a monotone relative Fukaya category, so Proposition \ref{proposition:hess} does not hold exactly as stated.
Instead $\scrA$ is first-order quasi-equivalent to the endomorphism algebra of the immersed Lagrangian sphere $L$ in $\cF(\phi)$ (see \cite[\S 5.1]{Sheridan2015} for the definition).
We recall that $\cF(\phi)$ should be thought of as an `orbifold' relative Fukaya category of $X^n_1 = X^n_a/\Gamma^n_a$, which has orbifolding of degree $a$ about the divisors (although it is actually defined by counting pseudoholomorphic maps into $X^n_a$).
The arguments of section \ref{subsec:fdp} must be modified to take into account this slight change in perspective, but the necessary modifications are trivial.

Recalling that $\scrA$ is an exterior algebra on $n$ generators $\theta_1, \ldots, \theta_n$ (Definition \ref{definition:typean}), we define $V \subset \scrA$ to be the subspace spanned by $\{\theta_1,\ldots,\theta_n\}$.
We need to establish that the various hypotheses of Proposition \ref{proposition:hess} are satisfied for $V \subset \scrA$.

\begin{lemma}
Let $\mathcal{L} = \{L\}$ be the set containing the immersed Lagrangian $L \subset X^n_1$, and $\mathcal{P}$ be the set of generators $\theta_i$ of $\scrA$.
Then $(\mathcal{L},\mathcal{P})$ is relatively monotone.
\end{lemma}
\begin{proof}
Suppose that $(u,\tilde{u})$ are as in Definition \ref{definition:strmonpaths}, with $\tilde{u}$ changing `sheets' of $L$ at generators $\theta_{i_m}$, $m=1,\ldots, l$, in positive direction, then at generators $\theta_{j_n}, n = 1,\ldots,k,$ in negative direction (compare Figure \ref{fig:Fig7}).
A straightforward modification of the grading computation of \cite[Lemma 2.93]{Sheridan2015} shows that
\begin{equation} aj + k - l = nq\end{equation}
(where $j = \sum \ell_i (u \cdot D_i)$), and
\begin{equation} t = (n-2)q + (2-a)j\end{equation}
(where $t = \mu(u,\tilde{u})$).
Eliminating $q$, it follows that
\begin{equation} \sum \ell_i (u \cdot D_i) = \frac{n}{2(n-a)} \mu(u,\tilde{u}) + \frac{n-2}{2(n-a)} (l-k).\end{equation}
We have $\tau = n/2(n-a)$, so if we choose $\tau_k$ so that
\begin{equation} \frac{n-2}{2(n-a)} < \tau_k < \frac{n}{2(n-a)} = \tau,\end{equation}
then
\begin{equation} \sum \ell_i (u \cdot D_i) < \tau \mu(u,\tilde{u}) + \tau_k l\end{equation}
for sufficiently large $l$.
It follows that $(\mathcal{L},\mathcal{P})$ is relatively monotone.
 \end{proof}

Now although $(\mathcal{L},\mathcal{P})$ is relatively monotone, it need not be monotone.
To deal with this, we apply Lemma \ref{lemma:actioncontraction}: if $\varphi_t$ denotes the time-$t$ reverse Liouville flow on $X^n_1 \setminus D$, and
\begin{equation} (\mathcal{L}_t,\mathcal{P}_t) := ( \varphi_t(\mathcal{L}), \varphi_t(\mathcal{P})),\end{equation}
then for sufficiently large $t>0$, $(\mathcal{L},\mathcal{P}_t)$ is both relatively monotone and monotone.

Following Proposition \ref{proposition:atild}, let $\widetilde{\scrA}_t$ denote the full subcategory of the $\bm{G}^n_a$-graded monotone relative Fukaya category $\cF_m(X^n_a,D)_{\bm{w}}$ whose objects are lifts of the immersed Lagrangian sphere $L_t \subset X^n_1$ under the branched cover $\phi$.

\begin{lemma}
For all $t>0$, there exists a non-curved $A_{\infty}$ algebra $\scrA_t$ of type A$^n_a$, such that
\begin{equation} \widetilde{\scrA}_t \cong \bm{p}^* \underline{\scrA}_t.\end{equation}
\end{lemma}
\begin{proof}
By Lemma \ref{lemma:actioncontraction}, $L_t$ is Hamiltonian isotopic to $L$.
It follows that $\widetilde{\scrA}_t$ and $\widetilde{\scrA}$ become quasi-equivalent, when we set all $r_j = 0$ and work in the affine Fukaya category $\cF(X^n_a \setminus D)$.
It now follows by \cite[Lemma 2.111]{Sheridan2015} that $\widetilde{\scrA}_t$ and $\widetilde{\scrA}$ are `first-order quasi-equivalent'.
The result now follows from Proposition \ref{proposition:atild}, because whether an algebra is of type A$^n_a$ or not only depends on first-order data.
 \end{proof}

Henceforth, we will drop the `$t$' and write $\scrA$ for this $A_\infty$ algebra $\scrA_t$: it has all the same properties as before, except the generators $\mathcal{P}$ are now monotone.

\begin{lemma}
\label{lemma:hypsat}
Let $V \subset \scrA$ be the subspace spanned by the generators in $\mathcal{P}$.
Then the conditions of Lemma \ref{lemma:prediscdisc} are satisfied:
\begin{itemize}
\item For any $v \in V$, $\mu^s(v,\ldots,v)$ is a multiple of $e_L$;
\item The sum defining the pre-disc potential $\mathfrak{P}'$ converges;
\item $CC^{\le 0}(V,\scrA) \cong \C \cdot e_L$ (as $\C$-vector space).
\end{itemize}
\end{lemma}
\begin{proof}
Combining the grading computations of \cite[Lemma 2.96]{Sheridan2015} with Lemmas \ref{lemma:defclasshh} and \ref{lemma:j2} shows that any $A_\infty$ product $\mu^s$ with $s \ge 3$, when input are generators $\theta_i$, outputs a multiple of the identity.
We also have $\mu^1 = 0$ by Corollary \ref{corollary:min}, and $\mu^0 = 0$ by definition.
Finally, the product $\mu^2(v,v)$ vanishes by antisymmetry, as $\mu^2$ is the exterior algebra product.
It follows that $\mu^s(v,\ldots,v)$ is a multiple of $e_L$, for any $v \in V$.

The sum defining the pre-disc potential converges, because $(\mathcal{L},\mathcal{P})$ is relatively monotone.

The fact that $CC^{\le 0}(V,\scrA) \cong \C \cdot e_L$ follows from \cite[Lemma 2.94]{Sheridan2015}, in particular Equations (2.7) and (2.8).
In the notation used there, $s+t \le 0$ means
\begin{align}
0 &\ge  -s \mbox{ (length $s$ is always $\ge 0$)}  \\
& \ge  t \mbox{ (as $s+t \le 0$)}\\
&= (n-2)(q-j) + (n-a)j  \mbox{ (by \cite[Equation (2.7)]{Sheridan2015})} \\
& \ge  (n-2)(q-j) \mbox{ (as $n-a > 0$ and $j \ge 0$)} \\
&= \frac{n-2}{2} (|K| - (s+t)) \mbox{ (by \cite[Equation (2.8)]{Sheridan2015})} \\
& \ge  0 \mbox{ (as $|K| \ge 0$ and $s+t \le 0$)}
\end{align}
Hence we must have equality everywhere, and in particular, $s = t = |K| = j = 0$, so the only generator is the identity $e_L$.
 \end{proof}

\begin{corollary}
\label{corollary:weakemb}
Consider the immersed Lagrangian sphere $L$ as an object of the $\C$-linear category $\cF^c(\phi) \otimes_R \C$ (where the `$c$' indicates that we have enhanced with homotopy units and introduced curvature), and let $V_\C$ be the $n$-dimensional complex vector space with basis $\{\theta_1,\ldots,\theta_n\}$.
Then there is an embedding
\begin{equation} \iota: V_\C \hookrightarrow \hcM_{weak}(L),\end{equation}
so that the disc potential is given by
\begin{equation} \mathfrak{P} \circ \iota(v_1,\ldots,v_n) = W^n_a(v_1,\ldots,v_n).\end{equation}
\end{corollary}
\begin{proof}
The existence of $\iota$ follows from Lemma \ref{lemma:hypsat}, which verifies the hypotheses of Corollary \ref{corollary:relpredisc2}.
The pre-disc potential is $\mathfrak{P}' = Z^n_a$, by the definition of an algebra of type A$^n_a$ (together with Lemma \ref{lemma:j2}, which shows there can be no corrections to $Z^n_a$ of quadratic or higher order in the $r_j$).
Therefore, the disc potential is precisely
\begin{equation} \mathfrak{P} \circ \iota(v) = \bm{w} + \mathfrak{P}'(v) = \bm{w} + Z^n_a(v_1,\ldots,v_n) = W^n_a(v_1,\ldots,v_n).\end{equation}
 \end{proof}

Now we recall the setup of \S \ref{subsec:fabg}.
We apply it to the object $L$ of $\underline{\scrA}$, the category constructed in the proof of Proposition \ref{proposition:atild}, which had the property that there was an isomorphism
\begin{equation} \widetilde{\scrA} \cong \bm{p}^* \underline{\scrA}.\end{equation}
Choose a splitting $\theta$, and let $L^\theta$ be the corresponding object of $\cF^c(X^n_a)$.

\begin{corollary}
\label{corollary:weakembups}
There is an embedding
\begin{equation}
j \circ \iota: V_\C \hookrightarrow \hcM_{weak}(L^\theta),
\end{equation}
such that the disc potential is given by
\begin{equation}
\mathfrak{P} \circ j \circ \iota(v) = W^n_a(v).
\end{equation}
Furthermore, the objects $(L^\theta,j \circ \iota(v))$ and $(L^\theta,j \circ \iota(\chi \cdot v))$ are quasi-isomorphic, for any $\chi \in \Gamma^*$.
\end{corollary}
\begin{proof}
The embedding $j \circ \iota$ is the composition of the embedding $\iota$ given in Corollary \ref{corollary:weakemb} with the embedding $j$ of \eqref{eqn:jinc}.
The fact that $\chi$ acts via quasi-isomorphisms follows from Proposition \ref{proposition:characts}.
 \end{proof}

\begin{remark}
\label{rmk:notgauge}
Corollary \ref{corollary:weakembups} is our version of the folklore result that `the mirror is the Maurer--Cartan moduli space, with the superpotential given by the disc potential'.
We recall that the mirror is $(\C^n/\Gamma^*,W^n_a)$ (see Definition \ref{definition:mir}).
It is interesting to note that the quotient by $\Gamma^*$ that appears in the definition of the mirror does not appear in Corollary \ref{corollary:weakembups} as the quotient by gauge equivalence: gauge equivalence is an equivalence relation generated by the flow of a vector field, not by a discrete group action.
We also note that the quotient by $\Gamma^*$ is not simply quotient by quasi-isomorphism: most of the objects $(L^\theta,j \circ \iota(v))$ are quasi-isomorphic to the zero object, but are not identified by the action of $\Gamma^*$.
\end{remark}

Now, because $\scrA$ is an $A_\infty$ algebra of type A$^n_a$, its underlying vector space can be identified with an exterior algebra on the generators $\theta_1,\ldots, \theta_n$.
Let $\scrA_k$ denote the subspace of $\scrA$ spanned by elements $\theta_{i_1} \wedge \ldots \wedge \theta_{i_k}$, so
\begin{equation}
\label{eqn:decompa}
 \scrA \cong \bigoplus_{k=0}^n \scrA_k.
\end{equation}

\begin{lemma}
\label{lemma:hypsat2}
The subspace $V \subset \scrA$, together with the decomposition \eqref{eqn:decompa}, satisfies the hypotheses of Proposition \ref{proposition:critpd}, namely:
\begin{itemize}
\item $A_0 = R \cdot e $;
\item $A_1 = V$;
\item $\mu^2_0$ sends
\begin{equation} \mu^2_0 : A_k \otimes A_l \To A_{k+l}; \end{equation}
\item If $s > 2$ or $j>0$, then $\mu^s_j$ sends
\begin{equation} \mu^s_j: V^{\otimes b} \otimes A_k \otimes V^{\otimes c} \otimes A_l \otimes V^{\otimes d} \To \bigoplus_{m < k+l} A_m,\end{equation}
whenever $b+c+d + 2 = s$.
\item $V$ generates $A$ as an associative algebra, with respect to $\mu^2_0$, and $e_L$ is a unit for this product;
\end{itemize}
\end{lemma}
\begin{proof}
It is clear that $\scrA_0$ is spanned by the unit $e$, and that $\scrA_1 = V$ is the subspace spanned by the $\theta_i$.
$\mu^2_0$ coincides with the exterior product, by definition of an algebra of type A$^n_a$.
Hence it respects the decomposition, $V$ generates $\scrA$ as an associative algebra with respect to the product $\mu^2_0$, and $e_L$ is a unit.
Finally, we apply \cite[Lemma 2.93]{Sheridan2015} to prove the final hypothesis.
Using the notation from there, if the coefficient of $r^{\bm{c}}\theta^{K_0}$ in
\begin{equation} \mu^s(\theta_{j_1}, \ldots, \theta_{j_a},\theta^{K_1},\theta_{j_{b+1}},\ldots,\theta_{j_{b+c}},\theta^{K_2},\theta_{j_{b+c+1+1}},\ldots,\theta_{j_{b+c+d}})\end{equation}
is non-zero, then
\begin{equation}
|K_0| - |K_1| - |K_2| = (s-2) + nq - aj
\end{equation}
(dotting \cite[Equation (2.2)]{Sheridan2015} with $y_{\{1,\ldots,n\}}$), and
\begin{equation}
2-s = (n-2)q + (2-a)j
\end{equation}
(by \cite[Equation (2.3)]{Sheridan2015}, with $s+t = 2$).
Eliminating $q$ from these equations, and setting $|K_1| = k, |K_2| = l, |K_0| = m$ gives
\begin{equation} (n-2)(m - k -l) = 2(2-s) + 2(a-n)j.\end{equation}
In particular, if $s>2$ or $j>0$, then $m<k+l$, as required.
 \end{proof}

\begin{corollary}
\label{corollary:hessL}
Consider the immersed Lagrangian sphere $L$ as an object of $\cF^c(\phi) \otimes_R \C$, as in Corollary \ref{corollary:weakemb}.
Suppose that $v = (v_1,\ldots,v_n)$ is one of the small critical points of $\mathfrak{P} \circ \iota = W^n_a$ (see Lemma \ref{lemma:wnacrit}), and let $\alpha = \iota(v)$ be the corresponding weak bounding cochain.
Then we have an isomorphism of $\Z/2\Z$-graded algebras
\begin{equation} HF^*((L,\alpha),(L,\alpha)) \cong \Cl_n.\end{equation}
The right-hand side denotes the Clifford algebra of a non-degenerate quadratic form on a complex vector space of dimension $n$ (see \S \ref{subsec:cliff}).
\end{corollary}
\begin{proof}
By Lemma \ref{lemma:hypsat}, which verifies the hypotheses of Proposition \ref{proposition:critpd}, we can apply Proposition \ref{proposition:hess}. Corollary \ref{corollary:weakemb} shows that the disc potential is given by $W^n_a$.
Note that the surjection of Proposition \ref{proposition:hess} is in fact an isomorphism, as the domain and target both have rank $2^n$.
Also note that the Hessian at a small critical point is non-degenerate by Lemma \ref{lemma:wnacrit}.
 \end{proof}

\begin{corollary}
\label{corollary:smallLlift}
Let $v \in V_\C$ be one of the small critical points of $W^n_a$, and let $(L^\theta,j \circ \iota(v))$ be the corresponding object of $\cF^{wbc}(X^n_a)_w$, in the notation of Corollary \ref{corollary:weakembups}.
Then there is an isomorphism of $\Z/2\Z$-graded algebras
\begin{equation} HF^*((L^\theta,j(\alpha)),(L^\theta,j(\alpha))) \cong \Cl_n.\end{equation}
\end{corollary}
\begin{proof}
Follows from Corollary \ref{corollary:hessL} and Proposition \ref{proposition:pullwbc}.
We must verify the extra hypothesis of Proposition \ref{proposition:pullwbc}: namely, that for any character $\chi \in (\Gamma^n_a)^*$ not equal to $1$, the differential $d_{v,\chi \cdot v}$ admits a contracting homotopy.
The differential acts on the exterior algebra $A$ by
\begin{equation} d_{v,\chi \cdot v}(a) = \mu^2_0(v,a) + \mu^2_0(a,\chi \cdot v) = ( - v + \chi \cdot v) \wedge a,\end{equation}
hence admits a contracting homotopy $\iota_{\eta}$, where $\eta \in V^\vee$ satisfies $\eta(\chi \cdot v-v) = 1$.
Note that such an $\eta$ exists because $\chi \cdot v \neq v$ for $\chi \neq 1$, because the character group acts freely on the small critical points by Lemma \ref{lemma:wnacrit}.
 \end{proof}

\begin{lemma}
\label{lemma:lthetacou}
The object $(L^\theta,j \circ \iota(v))$ of Corollary \ref{corollary:smallLlift} is $\CO^0$-unital.
\end{lemma}
\begin{proof}
It is easy to see that $X \setminus D$ admits a holomorphic volume form, so Lemma \ref{lemma:cowbcc2} applies.
The result then follows from Corollary \ref{corollary:cowbcc3}.
 \end{proof}

\begin{lemma}
\label{lemma:lthetagen}
The object $(L^\theta,j(\alpha))$ of Corollary \ref{corollary:smallLlift} split-generates $\cF^{wbc,u}(X^n_a)_w$.
\end{lemma}
\begin{proof}
Follows from Corollary \ref{corollary:semisimpgenwbc}, because $QH^*(X)_w$ is one-dimensional by Proposition \ref{proposition:smallw}.
 \end{proof}

\begin{corollary}
\label{corollary:smallcat}
If $w $ is a small eigenvalue of $c_1 \star$, then there is an $A_\infty$ quasi-equivalence
\begin{equation}
D^\pi \cF^{wbc,u}(X^n_a)_w \cong \left\{ \begin{array}{ll}
									D^b(\C) & \mbox{ if $n$ is even} \\
									D^\pi(\Cl_1) & \mbox{ if $n$ is odd}
								\end{array} \right.
\end{equation}
\end{corollary}
\begin{proof}
Follows from Lemma \ref{lemma:lthetagen}, Corollary \ref{corollary:smallLlift}, and Corollary \ref{corollary:cliffcat}.
 \end{proof}

\section{Matrix factorization computations}
\label{sec:bmodel}

We refer to \S \ref{subsec:ana} for the definition of the grading datum $\bm{G}$, the $\bm{G}$-graded ring
\begin{equation} R_a := \C[r_1, \ldots, r_n],\end{equation}
and the $\bm{G}$-graded polynomial ring
\begin{equation} S_a := R_a[u_1, \ldots, u_n],\end{equation}
together with the element
\begin{equation} \tilde{Z}^n_a :=- u_1 \ldots u_n + \sum_{j=1}^n r_j u_j^a \in S_a\end{equation}
of degree $2$.
We also consider the algebra homomorphism
\begin{equation} R_a \To \C\end{equation}
sending all $r_j$ to $1$, so that
\begin{equation} S_a \otimes_{R_a} \C \cong \C[u_1,\ldots,u_n],\end{equation}
and
\begin{equation} \tilde{Z}^n_a \otimes_{R_a} 1 = Z^n_a \in S_\C\end{equation}
(where $Z^n_a$ is as in \eqref{eqn:zna}).

We will sometimes drop the `$n$' or `$a$' from the notation to avoid clutter, when we feel no confusion is possible.

In accordance with \cite[Definition 7.2]{Sheridan2015}, we consider the differential $\bm{G}$-graded category of matrix factorizations of $\tilde{Z}^n_a$, $MF^{\bm{G}}(S_a,\tilde{Z}^n_a)$.
We introduce a matrix factorization $\mathcal{O}_0 := (K,\delta_K)$, where
\begin{equation} K := S_a \otimes \Lambda^*(U^{\vee}) \cong S_a[\theta_1, \ldots, \theta_n],\end{equation}
where the $\theta_i$ anti-commute.
The differential is given by
\begin{equation} \delta_K := \sum_{j} u_j \del{}{\theta_j} + w_j \theta_j,\end{equation}
where
\begin{equation} w_j := -\frac{u_1 \ldots u_n}{n u_j} + r_j u_j^{a-1},\end{equation}
so that $\sum u_j w_j = \tilde{Z}^n_a$.
It is easy to check that $\delta_K^2 = \tilde{Z}^n_a \cdot \mathrm{id}$, so $(K,\delta_K)$ is an object of $MF^{\bm{G}}(S_a,\tilde{Z}^n_a)$.
We define its endomorphism DG algebra,
\begin{equation} \mathscr{B} := \mathrm{Hom}_{MF^{\bm{G}}(S_a,\tilde{Z}^n_a)}(\mathcal{O}_0,\mathcal{O}_0).\end{equation}

\begin{proposition}
\label{proposition:bana}
If $3 \le a \le n-1$, then $\mathscr{B}$ is quasi-isomorphic to an $A_{\infty}$ algebra of type A$^n_a$.
\end{proposition}
\begin{proof}
We use the homological perturbation lemma to construct a quasi-isomorphic minimal $A_{\infty}$ algebra structure on the cohomology of $\mathscr{B}$; it follows from \cite[Proposition 7.1]{Sheridan2015} that this $A_{\infty}$ algebra is of type A$^n_a$.
 \end{proof}

Now we consider the morphisms of grading data
\begin{equation} \bm{p}: \bm{G}^n_a \To \bm{G}^n_1\end{equation}
of equation \eqref{eqn:morphp}, and the sign morphism
\begin{equation} \bm{\sigma}: \bm{G}^n_a \To \bm{G}_{\sigma}.\end{equation}
We obtain a $\C$-linear, differential $\Z/2\Z$-graded category
\begin{equation} \bm{\sigma}_* \bm{p}^* MF^{\bm{G}}(S_a,\tilde{Z}^n_a) \otimes_{R_a} \C.\end{equation}

\begin{proposition}
\label{proposition:btild}
Let $\widetilde{\mathscr{B}}_{\C}$ be the full subcategory of $MF^{\Gamma^*}(S_\C,Z)$ whose objects correspond to equivariant twists of the skyscraper sheaf at the origin (under Orlov's equivalence $H^*(MF^{\Gamma^*}) \cong D^bSing^{\Gamma^*}$).
Then $\widetilde{\mathscr{B}}_{\C}$ split-generates, and for $2 \le a \le n-1$ there is a quasi-isomorphism
\begin{equation} \widetilde{\mathscr{B}}_{\C} \cong \bm{\sigma}_* \bm{p}^* \underline{\mathscr{B}} \otimes_R \C\end{equation}
where $\mathscr{B}$ is an $A_{\infty}$ algebra of type A$^n_a$.
\end{proposition}
\begin{proof}
$Z$ has an isolated singularity at the origin, by Lemma \ref{lemma:wnacrit}.
It follows that $\widetilde{\mathscr{B}}_{\C}$ split-generates the triangulated category of singularities (see \cite[\S 12]{Seidel2008a} and \cite{Dyckerhoff2009}), and hence the category of matrix factorizations.

By \cite[Remark 7.6]{Sheridan2015}, there is a fully faithful embedding of $\Z/2\Z$-graded DG categories
\begin{equation} \bm{\sigma}_* \bm{p}^* MF^{\bm{G}}(S,\tilde{Z}) \hookrightarrow MF^{\Gamma^*}(S,\tilde{Z}),\end{equation}
where $\Gamma^*$ acts on $S$ by multiplying the coordinate functions $u_j$ by $a$th roots of unity (and trivially on $R$).
There is then an obvious embedding
\begin{equation} MF^{\Gamma^*}(S,\tilde{Z}) \otimes_R \C \hookrightarrow MF^{\Gamma^*}(S_\C,Z),\end{equation}
by applying $-\otimes_R \C$.

Under this chain of embeddings and equivalences, the lifts of $(K,\delta_K)$ map to the objects corresponding to equivariant twists of the skyscraper sheaf of the origin under Orlov's equivalence (see \cite[\S 3]{Dyckerhoff2009}); the result now follows by Proposition \ref{proposition:bana}, for all $a \ge 3$.

When $a=2$, the proof is slightly different: we can not apply the homological perturbation lemma as in Proposition \ref{proposition:bana}.
Nevertheless, the superpotential $Z^n_2$ has a non-degenerate critical point at the origin, so we have $H^*(\mathscr{B} \otimes_R \C) \cong \Cl_n$ by \cite[\S 4.4]{Dyckerhoff2009}, which is isomorphic to $H^*(\scrA \otimes_R \C)$ for an $A_\infty$ algebra of type A$^n_2$ by Corollary \ref{corollary:cohalg}.
It follows from intrinsic formality of the Clifford algebra (Corollary \ref{corollary:intform}) that there is an $A_\infty$ quasi-isomorphism $\mathscr{B} \otimes_R \C \cong \scrA \otimes_R \C$.
The rest of the proof is as in the case $a \ge 3$.
 \end{proof}

\begin{proof}\textit {of Theorem \ref{theorem:big}.}
It follows from Corollary \ref{corollary:anaemb}, Proposition \ref{proposition:btild} and Theorem \ref{theorem:typean} that the subcategory of $\cF(X^n_a)_{\bm{w}}$ generated by lifts of $L$ is quasi-isomorphic to the subcategory of $MF^{\Gamma^*}(W - \bm{w})$ generated by equivariant twists of the skyscraper sheaf at the origin (note that $W-\bm{w} = Z$).
We recall here Remark \ref{remark:dgen}: namely, we take the DG enhancement of $D^bSing^{\Gamma^*}$ given by $MF^{\Gamma^*}$, by convention.
It follows from Corollary \ref{corollary:asplitgens} and Proposition \ref{proposition:btild} that these subcategories split-generate, so the proof is complete.
 \end{proof}

\begin{proof}\textit{of Theorem \ref{theorem:small}.}
Let $v \in \C^n$ be a small critical point of $W^n_a$ with critical value $w$.
On the category of singularities side, we have the $\Gamma^*$-equivariant object $\mathcal{O}_{crit}$, which is the direct sum of the skyscraper sheaves $\mathcal{O}_v$ at the small critical points $v$ in $W^{-1}(w)$.
Because $\Gamma^*$ acts freely and transitively on the critical points by Lemma \ref{lemma:wnacrit}, it is clear that the endomorphism algebra of $\mathcal{O}_{crit} $ is isomorphic to the endomorphism algebra of a single $\mathcal{O}_v$.
This endomorphism algebra is isomorphic to $\Cl_n$ by \cite[\S 4.4]{Dyckerhoff2009}, because the Hessian of $Z^n_a$ at $v$ is non-degenerate by Lemma \ref{lemma:wnacrit}.
Furthermore, because each $v$ is an isolated critical point of $W^n_a$, $\mathcal{O}_{crit}$ split-generates $MF^{\Gamma^*}(W-w)$ (see \cite[\S 12]{Seidel2008a} and \cite{Dyckerhoff2009}).
It follows by Corollary \ref{corollary:cliffcat} that there is an $A_\infty$ quasi-isomorphism
\begin{equation}
D^\pi MF^{\Gamma^*} (W-w) \cong \left\{ \begin{array}{ll}
									D^b(\C) & \mbox{ if $n$ is even} \\
									D^\pi(\Cl_1) & \mbox{ if $n$ is odd}
								\end{array} \right.
\end{equation}
Lemma \ref{lemma:wnacrit} shows that $w$ corresponds to a small eigenvalue of $c_1 \star$, so the result now follows by Corollary \ref{corollary:smallcat}.
 \end{proof}

\appendix

\section{$A_\infty$ bimodules}
\label{sec:ainf}

\subsection{$A_\infty$ categories}
\label{subsec:ainfcat}

We recall some basic facts about $A_{\infty}$ categories and modules.
Recall that a \emph{grading datum} $\bm{G}$ consists of an abelian group $Y$, with homomorphisms
\begin{equation}\Z \overset{f}{\To} Y \overset{\sigma}{\To} \Z/2\Z\end{equation}
whose composition $\sigma \circ f$ is the standard map $\Z \To \Z/2\Z$.
All of our categories will be $\bm{G}$-graded, which means that they are $Y$-graded, and when we say an operation has degree $j \in \Z$, we really mean its degree is $f(j) \in Y$; and all signs in our formulae will be determined via the map $\sigma$.

A $\bm{G}$-graded $A_{\infty}$ category $\cA$ has a set of objects $L$, with an action of $Y$ on the objects by `shifts'.
For each pair of objects there is a hom-space, which is a $\bm{G}$-graded free $R$-module, which we will denote $hom_\cA^*(L_0,L_1)$.
We introduce the convenient notation
\begin{equation} \cA(L_s, \ldots, L_0) := hom_\cA^*(L_{s-1},L_s)[1] \otimes \ldots \otimes hom_\cA^*(L_0,L_1)[1].\end{equation}

We define the \emph{Hochschild cochain complex} of $\cA$,
\begin{equation} CC^*(\cA) := \prod_{L_0,\ldots,L_s} \mathrm{Hom}(\cA(L_s,\ldots,L_0),\cA(L_s,L_0)[-1]).\end{equation}
It has the Gerstenhaber product
\begin{multline}
\phi \circ \psi (a_s,\ldots,a_1) := \\ \sum_{i+j+k = s} (-1)^{\sigma'(\psi) \cdot \maltese^i_1} \phi^{i+k+1}(a_{i+j+k},\ldots,\psi^j(a_{i+j},\ldots),a_i,\ldots,a_1),
\end{multline}
where we establish running notation that, for any sign $\sigma$, $\sigma':=\sigma+1$ denotes the opposite sign, and
\begin{equation} \maltese^i_j := \sum_{k=j}^i \sigma'(a_k).\end{equation}
It also has the Gerstenhaber bracket,
\begin{equation} [\phi,\psi] := \phi \circ \psi - (-1)^{\sigma'(\phi)\cdot \sigma'(\psi)} \psi \circ \phi,\end{equation}
which is a $\bm{G}$-graded Lie bracket.

A curved $A_{\infty}$ structure on $\cA$ is an element $\mu^* \in CC^2(\cA)$ satisfying $\mu^* \circ \mu^* = 0$.
If the length-zero component vanishes, $\mu^0 = 0$, then $\mu^*$ is called an $A_\infty$ structure.
In this appendix, we will always assume that $\mu^0 = 0$, with the exception of \S \ref{subsec:curved}.

We will always assume our $A_\infty$ categories to be cohomologically unital ($c$-unital) in the sense of \cite[\S 2a]{Seidel2008}, i.e., that the cohomology category $H^*(\cA)$ has identity morphisms.
We will denote the morphism spaces in the cohomology category by
\begin{equation} \mathrm{Hom}^*(K,L) := H^*(hom^*(K,L),\mu^1).\end{equation}
We recall that the composition of morphisms is defined by
\begin{equation}
\label{eqn:assoc}
 [a_2] \cdot [a_1] := (-1)^{\sigma(a_1)} [\mu^2(a_2,a_1)],
\end{equation}
and is associative.

It follows from the $A_\infty$ equation $\mu^* \circ \mu^* = 0$ that $[\mu^*,\mu^*] = 0$, and hence (by the Jacobi relation) that the Hochschild differential $[\mu^*,-]$ squares to zero.
We define the \emph{Hochschild cohomology} of $\cA$ to be
\begin{equation} HH^*(\cA) := H^*(CC^*(\cA),[\mu^*,-]).\end{equation}

\subsection{$A_\infty$ bimodules}
\label{subsec:ainfbi}

We now recall basic notions about the DG category of $A_{\infty}$ $\fmod{\cA}{\cA}$ bimodules from \cite{Seidel2008c} (our conventions are identical, except that the order of inputs is reversed in all operations; this is in line with the convention of \cite{Seidel2008}).
We denote this DG category by $\fmodf{\cA}{\cA}$.

The objects of $\fmodf{\cA}{\cA}$ are $A_{\infty}$ bimodules $\cM$ over the $A_{\infty}$ category $\cA$.
An \emph{$A_\infty$ bimodule} $\cM$ associates to each pair $K,L$ of objects of $\cA$ a $\bm{G}$-graded vector space $\cM(K,L)$, together with maps
\begin{equation} \mu^{k|1|l}:   \cA(K_k,\ldots, K_0) \otimes \cM(L_0,K_0) \otimes \cA(L_0, \ldots, L_l) \To \cM(L_l,K_k)\end{equation}
satisfying the $A_{\infty}$ associativity equation
\begin{align}
&&\sum_{i \le k,j \le l} (-1)^{\maltese^{|j+1}_{|l}} \mu^{k-i|1|l-j}(a_k,\ldots,\mu^{i|1|j}(a_i,\ldots,a_1,\bm{m},a_{|1},\ldots,a_{|j}),\ldots,a_{|l}) \\
 &&+ \sum_{i+j \le l} (-1)^{\maltese^{|i+j+1}_{|l}}\mu^{k|1|l-j}(a_k,\ldots,\bm{m},\ldots,\mu^j(b_{i+1},\ldots,b_{i+j}),\ldots,b_l) \\
 &&+ \sum_{i+j \le k} (-1)^{\maltese^j_{|l}} \mu^{k-j|1|l}(a_k,\ldots,\mu^i(a_{i+j},\ldots,a_{j+1}), \ldots, \bm{m}, \ldots,a_{|l}) = 0
\end{align}
where
\begin{align}
 \maltese^{|j}_{|l} &:=  \sum_{i=j}^l \sigma'(a_{|i}) \\
 \maltese^{j}_{|l} &:=  \sigma(\bm{m}) + \sum_{i=1}^j \sigma'(a_{|i}) + \sum_{i=1}^l \sigma'(a_i).
\end{align}
We require our bimodules to be cohomologically unital.

The morphism space $hom^*_{\fmodf{\cA}{\cA}}(\cM,\mathcal{N})$ is the space of $A_{\infty}$ bimodule pre-ho\-mo\-mor\-phisms.
An $A_\infty$ \emph{bimodule pre-homomorphism} from $\cM$ to $\mathcal{N}$ is a collection of maps
\begin{equation} F^{k|1|l}: \cA(K_k, \ldots, K_0) \otimes \cM(L_0,K_0) \otimes \cA(L_0, \ldots, L_l) \To \mathcal{N}(L_l,K_k).\end{equation}
There is a differential on the space of pre-homomorphisms (we refer to \cite[Equation (2.8)]{Seidel2008c} for the formula), and composition of morphisms is defined by
\begin{multline}
\label{eqn:bimodcomp}
 (F\circ G)^{k|1|l}(a_k,\ldots, a_1,\bm{m},a_{|1},\ldots,a_{|l}) := \\
\sum_{i \le k,j \le l} (-1)^{\dagger}F(a_k,\ldots,a_{i+1},G(a_i,\ldots,a_1,\bm{m},a_{|1},\ldots,a_{|j}),a_{|j+1},\ldots,a_{|l})
\end{multline}
where
\begin{equation}\dagger := \sigma(G)\cdot \maltese^{|j+1}_{|l}.\end{equation}
These differential and composition maps make $\fmodf{\cA}{\cA}$ into a DG category.
We denote the cohomology category of this DG category by $H^*(\fmodf{\cA}{\cA})$.

An $A_\infty$ pre-homomorphism which is closed is called an $A_\infty$ \emph{bimodule homomorphism}.
If $F^{k|1|l}$ is an $A_\infty$ bimodule homorphism from $\cM$ to $\mathcal{N}$, then
\begin{equation} F^{0|1|0}: (\cM(K,L),\mu^{0|1|0}) \To (\mathcal{N}(K,L),\mu^{0|1|0})\end{equation}
is a chain map; if it is a quasi-isomorphism, then we say that $F$ is a \emph{quasi-isomorphism}.
If $R$ is a field, then for any quasi-isomorphism $F$, the corresponding cohomology-level morphism $[F]$ is an isomorphism in $H^*(\fmodf{\cA}{\cA})$.

The simplest quasi-isomorphisms in $\fmodf{\cA}{\cA}$ are the identity maps.
$Id \in hom^*_{\fmodf{\cA}{\cA}}(\cM,\cM)$ is defined as follows:
\begin{equation} Id^{0|1|0}: \cM(K,L) \To \cM(K,L)\end{equation}
is the identity map, for all $K,L$, and all other $Id^{k|1|l}$ vanish.
These are identity morphisms in the DG category $\fmodf{\cA}{\cA}$.

One obvious $A_{\infty}$ bimodule is the \emph{diagonal bimodule}, $\cA_{\Delta}$:
\begin{align}
\cA_{\Delta}(K,L) &:= hom^*_\cA(K,L),\\
\mu^{k|1|l} &:= (-1)^{\maltese^{|1}_{|l}+1}\mu^{k+1+l}.
\end{align}

For any bimodule $\cM$ and $y \in Y$, we can define the shifted bimodule $\cM[y]$, with
\begin{align}
 \cM[y](K,L) &:= \cM(K,L)[y] \\
 \mu^{k|1|l}_{\cM[y]} &:= (-1)^{\sigma(y) \cdot (\maltese^{|1}_{|l} + 1)} \mu^{k|1|l}_{\cM}.
\end{align}

\subsection{Hochschild invariants}
\label{subsec:hhi}

For any $\fmod{\cA}{\cA}$ bimodule $\cM$, we define the \emph{Hochschild cochain complex}
\begin{equation} CC^*(\cA,\cM) := \prod_{L_0,\ldots,L_s} \mathrm{Hom}(\cA(L_s,\ldots,L_0),\cM(L_0,L_s)),\end{equation}
with the differential given in \cite[Equation (1.13)]{Seidel2011}.
Its cohomology is the \emph{Hochschild cohomology} of $\cA$ with coefficients in $\cM$, denoted $HH^*(\cA,\cM)$.
This accords with our previous definition of the Hochschild cohomology of $\cA$, in that there is an isomorphism
\begin{equation} HH^*(\cA) \cong HH^*(\cA,\cA_\Delta)\end{equation}
coming from an isomorphism of the underlying cochain complexes.

Following \cite{Ganatra2012}, for any $\fmod{\cA}{\cA}$ bimodule $\cM$, we also define the \emph{two-pointed Hochschild cochain complex}
\begin{equation} _2 CC^*(\cA,\cM) := hom^*_{\fmodf{\cA}{\cA}}(\cA_\Delta,\cM).\end{equation}
There is a chain map
\begin{equation}
\label{eqn:2cccc} CC^*(\cA,\cM) \To {} _2CC^*(\cA,\cM),\end{equation}
with the formula given in \cite[Equation (2.200)]{Ganatra2012}, where it is also proven that this map is a quasi-isomorphism if $R$ is a field.

We define the \emph{Hochschild chain complex}
\begin{equation} CC_*(\cA,\cM) := \bigoplus_{L_0,\ldots,L_s} \cM(L_s,L_0) \otimes \cA(L_s,\ldots,L_0),\end{equation}
with the differential given in \cite[Equation (5.15)]{Abouzaid2010a}.
Its cohomology is the Hochschild homology of $\cA$ with coefficients in $\cM$, denoted $HH_*(\cA,\cM)$.
We define the \emph{Hochschild homology} of $\cA$ to be
\begin{equation} HH_*(\cA) := HH_*(\cA,\cA_\Delta).\end{equation}

Now we recall that, given $\fmod{\cA}{\cA}$ bimodules $\cM$ and $\mathcal{N}$, we can define the chain complex
\begin{equation} \cM\otimes_{\fmod{\cA}{\cA}} \mathcal{N} := \bigoplus_{\substack{L_0,\ldots,L_l, \\ K_0,\ldots,K_k}}\cM(K_k,L_l) \otimes \cA(K_k,\ldots,K_0) \otimes \mathcal{N}(L_0,K_0) \otimes \cA(L_0,\ldots,L_l),\end{equation}
with differential given in \cite[Equation (5.1)]{Seidel2008c}.
This is functorial in $\cM$ and $\mathcal{N}$, in the following sense: for any $A_\infty$ bimodule homomorphism $F^{k|1|l}$ from $\mathcal{N}$ to $\mathcal{N}'$, there is an induced homomorphism of chain complexes
\begin{equation} F_\#: \cM \otimes_{\fmod{\cA}{\cA}} \mathcal{N} \To \cM \otimes_{\fmod{\cA}{\cA}} \mathcal{N}',\end{equation}
given explicitly by
\begin{multline}
\label{eqn:tensoract} F_\#(\bm{m} \otimes a_k \otimes \ldots \otimes a_1 \otimes \bm{n} \otimes a_{|1} \otimes \ldots \otimes a_{|l}) := \\
\sum_{i \le k, j \le l} (-1)^* \bm{m} \otimes a_k \otimes \ldots \otimes F^{i|1|j}(a_i,\ldots,a_1,\bm{n},a_{|1},\ldots,a_{|j})\otimes \ldots \otimes a_{|l}, \end{multline}
where
\begin{equation} * = \sigma(F) \cdot \maltese^{|j+1}_{|l}.\end{equation}
It is not difficult to see that $(F \circ G)_\# = F_\# \circ G_\#$.
Functoriality in $\cM$ is similar.

Following \cite{Ganatra2012}, we define the \emph{two-pointed Hochschild chain complex}
\begin{equation} _2 CC_*(\cA,\cM) := \cA_\Delta \otimes_{\fmod{\cA}{\cA}} \cM;\end{equation}
there is a chain map
\begin{equation} \label{eqn:2cccchom}_2CC_*(\cA,\cM) \To CC_*(\cA,\cM),\end{equation}
with the formula given in \cite[Equation (2.196)]{Ganatra2012}, where it is also proven that this map is a quasi-isomorphism if $R$ is a field.

We will call $CC^*(\cA,\cM)$ and $CC_*(\cA,\cM)$ the `one-pointed Hochschild (co)chain complexes' when we want to distinguish them from the two-pointed versions.

\subsection{Algebra and module structures}
\label{subsec:hhalgmod}

Because $HH^*(\cA)$ is the endomorphism algebra of the object $\cA_\Delta$ in $H^*(\fmodf{\cA}{\cA})$, it is naturally equipped with the structure of a unital associative $\bm{G}$-graded algebra.
The product on this algebra is called the \emph{Yoneda product}, and we denote it by $\cup$.
Equation \eqref{eqn:bimodcomp}  gives an explicit formula for $F \cup G$, written on the cochain level in $_2CC^*(\cA)$.

Here is a formula for the Yoneda product, on the cochain level in the one-pointed complex $CC^*(\cA)$:
\begin{equation}
\label{eqn:yon1pt}
\varphi \cup \psi(a_s,\ldots,a_1) := \sum (-1)^{\ddag} \mu^*(a_s,\ldots,\varphi(a_{l},\ldots),a_{k},\ldots,\psi(a_{j},\ldots),a_i,\ldots,a_1),
\end{equation}

where the sum is over all $s \ge l \ge k \ge j \ge i \ge 0$, and
\begin{equation} \ddag = \sigma'(\varphi) \cdot \maltese_1^k + \sigma'( \psi) \cdot \maltese_1^j \end{equation}
(compare \cite[Equation (2.183)]{Ganatra2012}).
This product, together with the Hochschild differential, can be extended to an $A_\infty$ structure on $CC^*(\cA)$ (see \cite{Getzler1993}); in particular, the product $(-1)^{\sigma(\psi)} \varphi \cup \psi$ is associative on $HH^*(\cA)$.
The chain map of \eqref{eqn:2cccc} respects the Yoneda product, up to an explicit homotopy.

The formula \eqref{eqn:yon1pt} has the advantage that, for any object $L$, it is clear that the map
\begin{equation}
\label{eqn:hhproj}
HH^*(\cA) \To \mathrm{Hom}^*(L,L),
\end{equation}
induced by projection to the $0$th graded piece with respect to the length filtration, is a homomorphism of algebras for any $L$.
If $\cA$ can be equipped with homotopy units, then this homomorphism is also unital.

It follows immediately from the interpretation of Hochschild cohomology as morphism spaces in $H^*(\fmodf{\cA}{\cA})$ that $HH^*(\cA,\cM)$ is an $HH^*(\cA)$-module for any bimodule $\cM$, and that Hochschild cohomology defines a unital functor
\begin{equation} HH^*(\cA,-): H^*(\fmodf{\cA}{\cA}) \To HH^*(\cA)\fmod{}{\mbox{mod}}.\end{equation}
If $R$ is a field, then because quasi-isomorphisms are invertible in $H^*(\fmodf{\cA}{\cA})$, this functor takes quasi-isomorphisms to isomorphisms.

From the fact that $\cA_\Delta \otimes_{\fmod{\cA}{\cA}} \cM$ is functorial in the first variable, we see that $HH_*(\cA,\cM)$ is also an $HH^*(\cA)$-module.
We denote the action of $F \in HH^*(\cA)$ on $\varphi \in HH_*(\cA,\cM)$ by $F \cap \varphi$, and call it the \emph{cap product}.
Equation \eqref{eqn:tensoract} gives an explicit formula for $F \cap \varphi$ on the cochain level in the two-pointed Hochschild complexes.

Here is a formula for the cap product, on the cochain level in the one-pointed complex $CC_*(\cA,\cM)$:
\begin{multline}
\label{eqn:cap1pt}
\alpha \cap (\bm{m} \otimes a_s \otimes \ldots \otimes a_1) := \\
\sum (-1)^{\diamond} \mu^{*|1|*}(a_i, \ldots, a_1, \bm{m}, a_s \ldots, \alpha(a_l,\ldots),a_k, \ldots) \otimes a_j \otimes \ldots \otimes a_{i+1}
\end{multline}
where the sum is over all $s \ge l \ge k \ge j \ge i \ge 0$, and
\begin{equation} \diamond = \maltese^i_1 \cdot (|\bm{m}| + \maltese^s_{i+1}) + \sigma'(\alpha) \cdot \maltese^k_{i+1} + \maltese^j_{i+1}\end{equation}
(compare \cite[Equation (2.187)]{Ganatra2012}).
This cap product, together with the Hochschild differential, can be extended to give $CC_*(\cA,\cM)$ the structure of a right $A_\infty$ module over $CC^*(\cA)$, with the $A_\infty$ structure referenced above (see \cite[Theorem 1.9]{Getzler1993}, although the sign convention for the Hochschild chain complex is different from ours).
In particular, the operation $(-1)^{\sigma(\alpha)}\alpha \cap -$ defines a structure of right $HH^*(\cA)$-module on $HH_*(\cA,\cM)$.
The chain maps \eqref{eqn:2cccc} and \eqref{eqn:2cccchom} between the one-pointed and two-pointed Hochschild complexes take one formula for the cap product to the other, up to an explicit homotopy.

The formula \eqref{eqn:cap1pt} has the advantage that it is clear that the obvious map
\begin{equation}
\label{eqn:hhinj}
\mathrm{Hom}^*(L,L) \To HH_*(\cA),
\end{equation}
induced by inclusion on the cochain level in the one-pointed complex, induces a homomorphism of $HH^*(\cA)$-algebras for any $L$.
Here, the $HH^*(\cA)$-module structure on $\mathrm{Hom}^*(L,L)$ factors through the $\mathrm{Hom}^*(L,L)$-module structure given by multiplication of endomorphisms, via the projection map \eqref{eqn:hhproj}.

From the fact that tensor product of bimodules is functorial in the second variable, we see that Hochschild homology defines a unital functor
\begin{equation} HH_*(\cA,-): H^*(\fmodf{\cA}{\cA}) \To HH^*(\cA)\fmod{}{\mbox{mod}},\end{equation}
which takes quasi-isomorphisms to isomorphisms, if $R$ is a field.

\subsection{$\infty$-inner products}
\label{subsec:infinprod}

In this section, we will assume that the coefficient ring $R$ is a field.

For any bimodule $\cM$, we can define the \emph{linear dual bimodule} $\cM^{\vee}$, where
\begin{equation} \cM^{\vee}(K,L) := \cM(L,K)^{\vee},\end{equation}
with structure maps $\check{\mu}^{k|1|l}$ defined by
\begin{equation} \check{\mu}^{l|1|k}(a_{|1}, \ldots, a_{|l}, \bm{\alpha}, a_{k}, \ldots, a_{1})(\bm{m}) := (-1)^{\sigma'(\bm{m})} \bm{\alpha}\left(\mu^{k|1|l} (a_{k}, \ldots, a_{1}, \bm{m}, a_{|1}, \ldots, a_{|l}) \right)\end{equation}
(compare \cite[Equation (2.7)]{Seidel2014}).

We observe that there is then an isomorphism of chain complexes
\begin{equation}
\label{eqn:lindual}
 _2 CC^*(\cA,\cM^{\vee}) \cong {_2 CC}_*(\cA,\cM)^{\vee}.
\end{equation}
Because these are complexes of $\C$-vector spaces, we have
\begin{equation} HH^*(\cA,\cM^{\vee}) \cong HH_*(\cA,\cM)^{\vee}.\end{equation}
One easily checks that this isomorphism respects the $HH^*(\cA)$-module structures.

We recall (from \cite[Definition 5.3]{Tradler2008}, see also \cite{Cho2008}) that an \emph{$\infty$-inner-product} on the $A_{\infty}$ category $\cA$ is an $A_{\infty}$ bimodule homomorphism
\begin{equation} \phi: \cA_{\Delta} \To \cA_{\Delta}^{\vee}.\end{equation}
We say that $\phi$ is $n$-dimensional if it is of degree $n$.
By definition, an $n$-dimensional $\infty$-inner product $\phi$ is a closed element of $_2CC^n(\cA,\cA^{\vee})$, and hence defines a class
\begin{equation} [\phi] \in HH^{n}(\cA,\cA^{\vee}) \cong HH_n(\cA)^\vee.\end{equation}
We say two $\infty$-inner products are equivalent if they have the same class in $HH_n(\cA)^\vee$; so the choice of an $\infty$-inner product $\phi$ up to equivalence is equivalent to the choice of a class $[\phi] \in HH_n(\cA)^\vee$.

For such a class $[\phi]$, we obtain a map
\begin{equation} \int: \mathrm{Hom}^n(L,L) \To \C\end{equation}
for any object $L$, as the composition of the map \eqref{eqn:hhinj} with $[\phi]$.

\begin{definition}
\label{definition:homnond}
We say that the class $[\phi] \in HH_n(\cA)^\vee$ is \emph{homologically non-degenerate} (compare \cite[Theorem 4.1]{Cho2008}) if the composition
\begin{align}
 \mathrm{Hom}^*(K,L) \otimes \mathrm{Hom}^{n-*}(L,K) & \To  \C \\
a \otimes b & \mapsto  \int  \mu^2(a , b)
\end{align}
is a perfect pairing of degree $n$, for any objects $K,L$.
\end{definition}

\begin{lemma}
\label{lemma:homnond}
The class $[\phi]$ is homologically non-degenerate if and only if the corresponding $\infty$-inner product $\phi$ is a quasi-isomorphism of $A_\infty$ bimodules.
\end{lemma}
\begin{proof}
$\phi$ is a quasi-isomorphism if and only if $F([\phi])$ is an isomorphism for any objects $K,L$, where $F$ denotes the composition
\begin{equation} HH_*(\cA)^\vee \To \mathrm{Hom}^*_{\fmodf{\cA}{\cA}}(\cA,\cA^\vee) \To \mathrm{Hom}^*(\mathrm{Hom}^*(K,L),\mathrm{Hom}^*(L,K)^\vee).\end{equation}
$F([\phi])$ is an isomorphism if and only if the composition
\begin{equation} \mathrm{Hom}^*(K,L) \otimes \mathrm{Hom}^*(L,K) \overset{F^\vee}{\To} HH_*(\cA) \overset{[\phi]}{\To} \C\end{equation}
is a perfect pairing, where we have identified
\begin{equation} \mathrm{Hom}^*(K,L) \otimes \mathrm{Hom}^*(L,K) \cong \mathrm{Hom}^*(\mathrm{Hom}^*(K,L),\mathrm{Hom}^*(L,K)^\vee)^\vee.\end{equation}
$F^\vee$ is given as a composition of maps on the cochain level
\begin{equation} hom^*(K,L) \otimes hom^*(L,K) \To {}_2CC^*(\cA,\cA^\vee)^\vee \To {} _2CC_*(\cA) \To CC_*(\cA),\end{equation}
for each of which we have an explicit formula; their composition is the map
\begin{equation} hom^*(K,L) \otimes hom^*(L,K) \overset{\mu^2}{\To} hom^*(L,L) \overset{\eqref{eqn:hhinj}}{\To} CC_*(\cA).\end{equation}
This completes the proof.
 \end{proof}

\begin{definition}
\label{definition:weakcy}
An \emph{$n$-dimensional weak proper Calabi--Yau structure} on an $A_\infty$ category $\cA$ is a class $[\phi] \in HH_n(\cA)^\vee$ which is homologically non-degenerate.
\end{definition}

There are various alternative notions of Calabi--Yau structures on $A_\infty$ categories: see \cite[\S 6.1]{Ganatra2015} for a summary.

\begin{remark}
According to \cite[\S 10.2]{Kontsevich2006a}, an \emph{$n$-dimensional proper Calabi--Yau structure} is an element $[\phi] \in HC_n(\cA)^\vee$ which is homologically non-degenerate, where $HC$ denotes the (positive) cyclic homology \cite[\S 7]{Kontsevich2006a}.
There is a natural map
\begin{equation} HH_*(\cA)  \To HC_*(\cA)\end{equation}
coming from an inclusion of chain complexes, whose dual takes $n$-dimensional proper Calabi--Yau structures to $n$-dimensional weak proper Calabi--Yau structures. 
See \cite[\S 6.1]{Ganatra2015} for a summary
\end{remark}

\begin{lemma}
\label{lemma:hhdual}
If $[\phi]$ is an $n$-dimensional weak proper Calabi--Yau structure on $\cA$, then the map
\begin{align}
HH^*(\cA) &\To HH_*(\cA)^\vee[-n] \\
\alpha & \mapsto  \alpha \cap [\phi]
\end{align}
is an isomorphism.
Here, `$\cap$' denotes the $HH^*(\cA)$-module structure dual to the cap product on $HH_*(\cA)$.
\end{lemma}
\begin{proof}
We recall the isomorphisms of $HH^*(\cA)$-modules:
\begin{equation} HH_*(\cA)^\vee \cong HH^*(\cA,\cA^\vee) \cong \mathrm{Hom}^*_{\fmodf{\cA}{\cA}}(\cA,\cA^\vee).\end{equation}
The module action of $HH^*(\cA)$ on the right-most of these is given by composition of morphisms in the category $H^*(\fmodf{\cA}{\cA})$.
Since $[\phi]$ is homologically non-degenerate, $\phi$ is a quasi-isomorphism by Lemma \ref{lemma:homnond}, i.e., it is an isomorphism in $H^*(\fmodf{\cA}{\cA})$.
The result follows, since post-composition with an isomorphism defines an isomorphism of morphism spaces in the category $H^*(\fmodf{\cA}{\cA})$.
 \end{proof}

\subsection{Split-generation}

Now we recall that, for any $A_{\infty}$ category $\cA$, we can form the $A_{\infty}$ category of twisted complexes, $D^b \cA$, which is triangulated in the $A_\infty$ sense, and its split-closure, $D^\pi \cA$, which is triangulated and split-closed (see \cite[Chapters 3 and 4]{Seidel2008}, where they are denoted $Tw \cA$ and $\Pi(Tw \cA)$ respectively).
For any full subcategory $\cG \subset \cA$, we can consider the smallest full subcategory of $D^\pi \cA$ which contains $\cG$, is closed under quasi-isomorphism, and is triangulated and split-closed.
We say that the objects of this category are split-generated by $\cG$.
If this category is all of $D^\pi \cA$, we say that $\cG$ split-generates $\cA$.

We recall that an $A_{\infty}$ left $\cA$-module $\cM$ associates to each object $L$ of $\cA$ a graded vector space $\cM^*(L)$, together with maps
\begin{equation} \mu^{k|1}: \cA(L_k, \ldots, L_0) \otimes \cM(L_0) \To \cM(L_k)\end{equation}
satisfying the $A_{\infty}$ relation, given in \cite[Equation 1.19]{Seidel2008}.
$A_{\infty}$ left $\cA$-modules are the objects of a DG category, which we denote by $\fmod{\cA}{\mbox{mod}}$.
Given an object $K$ of $\cA$, we define a left module $\mathcal{Y}^l_K$, with
\begin{align}
\mathcal{Y}^l_K(L) &:=  \cA(K,L) \\
\mu^{k|1} &:=  \mu^{k+1}.
\end{align}
This extends to a cohomologically full and faithful $A_{\infty}$ embedding
\begin{equation} \mathcal{Y}^l: \cA \To \fmod{\cA}{\mbox{mod}},\end{equation}
which is the $A_{\infty}$ version of the Yoneda embedding (see \cite[\S 2g]{Seidel2008}).

Similarly, one can define the DG category of $A_{\infty}$ right $\cA$-modules, and we denote it by $\fmod{\mbox{mod}}{\cA}$.
For each object $K$ of $\cA$, there is a right-module $\mathcal{Y}^r_K$, with $\mathcal{Y}^r_K(L) := \cA(L,K)$ and structure maps given by $\mu^*$.

We can then form the $\fmod{\cA}{\cA}$ bimodule $\mathcal{Y}^l_K \otimes \mathcal{Y}^r_K$, where
\begin{align}
\mathcal{Y}^l_K \otimes \mathcal{Y}^r_K(L_0,L_1) &:=  \cA(K,L_0) \otimes \cA(L_1,K) \\
\mu^{0|1|0}(p \otimes q) &:= (-1)^{\sigma'(q)}\mu^{1}(p)\otimes q - p\otimes\mu^{1}(q) \\
\mu^{k|1|0}(a_k,\ldots,a_1,p\otimes q)  & :=  (-1)^{\sigma'(q)}\mu^{k+1}(a_k,\ldots,a_1,p) \otimes q \\
\mu^{0|1|l}(p \otimes q,a_{|1},\ldots,a_{|l}) &:=  (-1)^{\maltese^{|1}_{|l}+1} p \otimes  \mu^{l+1}(q,a_{|1},\ldots,a_{|l}) \\
\mu^{k|1|l} &= 0 \mbox{ if both $k$ and $l$ are non-zero.}
\end{align}

There is a map $H^*(\mu)$:
\begin{equation} H^*(\mu): HH_*(\cA,\mathcal{Y}^l_K \otimes \mathcal{Y}^r_K) \To \mathrm{Hom}^*(K,K),\end{equation}
defined on the cochain level by contracting the chain of morphisms with the $A_{\infty}$ structure maps:
\begin{equation} \mu((p \otimes q) \otimes a_s \otimes \ldots \otimes a_1) := (-1)^{\sigma'(p) \cdot (\sigma'(q)+\maltese^s_1)} \mu^{s+2}(q,a_s, \ldots, a_1, p).\end{equation}

Now let $\cG \subset \cA$ be a full subcategory.
Denote by $\mathcal{YG}^l_K$ the restriction of $\mathcal{Y}^l_K$ to an object of $\cG$-mod, and similarly define $\mathcal{YG}^r_K$.

\begin{lemma}
\label{lemma:husplgen} (\cite[Lemma 1.4]{Abouzaid2010a})
If the identity of $\mathrm{Hom}^*(K,K)$ lies in the image of the map
\begin{equation} H^*(\mu): HH_*(\cG,\mathcal{YG}^l_K \otimes \mathcal{YG}^r_K) \To \mathrm{Hom}^*(K,K),\end{equation}
then $\cG$ split-generates $K$.
\end{lemma}

\subsection{Curvature and units}
\label{subsec:curved}

We will also be concerned with \emph{curved} $A_\infty$ categories $\cA$.
The definition of the DG category of $A_\infty$ bimodules $\fmodf{\cA}{\cA}$ still makes perfect sense: the formulae are as in \S \ref{subsec:ainfbi}, modified to allow for $\mu^0$.
The diagonal bimodule is a well-defined object of $\fmodf{\cA}{\cA}$.
The one- and two-pointed Hochschild cochain complexes $CC^*(\cA,\cM)$ and $_2CC^*(\cA,\cM)$ also still make sense, with differentials defined by the same formulae as before.
The Yoneda products on $CC^*(\cA)$ and $_2CC^*(\cA)$ are defined on the cochain level by the same formulae as before, and again define associative products on the level of cohomology.
As before, there is a chain map
\begin{equation}
\label{eqn:cc2cc}
 CC^*(\cA) \To {} _2CC^*(\cA)
\end{equation}
defined by \cite[Equation (2.200)]{Ganatra2012}, which induces a homomorphism on the level of cohomology.
In fact it is an algebra homomorphism, as one can construct an explicit homotopy between the two Yoneda products.

However, in contrast to the non-curved case, this homomorphism need not be a quasi-isomorphism: the length filtration used in the proof of \cite[Proposition 2.5]{Ganatra2012} is no longer a filtration if the $A_\infty$ category is curved.
Thus we have two Hochschild cohomology algebras, together with an algebra homomorphism
\begin{equation}
\label{eqn:hh2hh}
HH^*(\cA) \To {}_2 HH^*(\cA).
\end{equation}

Similarly, the one- and two-pointed Hochschild chain complexes $CC_*(\cA,\cM)$ and $_2CC_*(\cA,\cM)$ are well-defined, via the same formulae as before, and there is a chain map
\begin{equation}
\label{eqn:cc2cc2}
_2CC_*(\cA,\cM) \To CC_*(\cA,\cM)
\end{equation}
defined by \cite[Equation (2.196)]{Ganatra2012}.
Thus we have a homomorphism
\begin{equation}
\label{eqn:hh2hh2}
_2 HH_*(\cA,\cM) \To HH_*(\cA,\cM),
\end{equation}
but it need not be a quasi-isomorphism.
In fact this is a homomorphism of $HH^*(\cA)$-modules, as one can check by constructing an explicit homotopy between the two cap products.

Now suppose that $\cA$ has strict units $e_L \in hom^*(L,L)$, and that the curvature $\mu^0_L$ is a multiple of $e_L$, for each object $L$.
In this case, for each $w \in R$ (where $R$ is the base ring), we can define an honest (non-curved) $A_\infty$ category $\cA_w$: its objects are those objects $L$ of $\cA$ such that $\mu^0_L = w \cdot e_L$, its morphism spaces are inherited from $\cA$, and its $A_\infty$ structure maps $\mu^{\ge 1}$ are inherited from $\cA$, with $\mu^0$ set equal to $0$.
It follows from strict unitality and the condition on $\mu^0$ that these structure maps satisfy the $A_\infty$ relations.

We also consider strictly unital bimodules $\cM$, in the sense of \cite[Equation (2.6)]{Seidel2008c} (with the appropriate adjustment to signs).
In particular, if $\cA$ is strictly unital, then the diagonal bimodule is strictly unital in this sense.
If $\cM$ is a strictly unital $\fmod{\cA}{\cA}$ bimodule, then for each $w \in R$ we obtain an $\fmod{\cA_w}{\cA_w}$ bimodule $\cM_w$; it satisfies the $A_\infty$ relations by the strict unitality condition and the condition on $\mu^0$.

We would like to compare the Hochschild invariants of the curved $A_\infty$ category $\cA$ with those of the non-curved categories $\cA_w$.
To this end, for any strictly unital bimodule $\cM$ we introduce the \emph{normalized Hochschild cochain complex} \cite[1.5.7]{Loday1998},
\begin{equation} \overline{CC}^*(\cA,\cM) \subset CC^*(\cA,\cM),\end{equation}
the subcomplex of \emph{normalized} Hochschild cochains, namely, those $\alpha$ such that
\begin{equation} \alpha(\ldots,e_L,\ldots) = 0.\end{equation}
We similarly introduce the \emph{normalized Hochschild chain complex} \cite[1.1.14]{Loday1998}, the quotient of the Hochschild chain complex by the \emph{degenerate} subcomplex
\begin{equation}  \overline{CC}_*(\cA,\cM) := CC_*(\cA,\cM)/D_*,\end{equation}
where the subcomplex $D_*$ is the span of all Hochschild chains $\bm{m} \otimes a_s \otimes \ldots e_L \otimes \ldots \otimes a_1$.

Our assumptions on $\cA$ and $\cM$ ensure that these are subcomplexes, so we can define the normalized Hochschild cohomology and homology, which come with maps
\begin{equation}
\label{eqn:hhnorm1} \overline{HH}^*(\cA,\cM) \To HH^*(\cA,\cM)
\end{equation}
and
\begin{equation}
\label{eqn:hhnorm2} HH_*(\cA,\cM) \To \overline{HH}_*(\cA,\cM).
\end{equation}
Furthermore, our assumptions on $\cA$ ensure that the Yoneda product defines an associative algebra structure on $\overline{HH}^*(\cA)$, so that the map \eqref{eqn:hhnorm1} is an algebra homomorphism when $\cM = \cA_\Delta$, and that the cap product makes $\overline{HH}_*(\cA)$ into an $\overline{HH}^*(\cA)$-module, so that \eqref{eqn:hhnorm2} is a homomorphism of $\overline{HH}^*(\cA)$-modules.

Now, our assumption that $\mu^0$ is proportional to $e_L$ means that the restriction morphism
\begin{equation} \overline{CC}^*(\cA,\cM) \To \overline{CC}^*(\cA_w,\cM_w) \end{equation}
and the inclusion morphism
\begin{equation} \overline{CC}_*(\cA_w,\cM_w) \To \overline{CC}_*(\cA,\cM) \end{equation}
are morphism of chain complexes.

Furthermore, the natural inclusion
\begin{equation} \overline{CC}^*(\cA_w,\cM_w) \hookrightarrow CC^*(\cA_w,\cM_w)\end{equation}
and projection
\begin{equation} CC_*(\cA_w,\cM_w) \twoheadrightarrow \overline{CC}_*(\cA_w,\cM_w) \end{equation}
are quasi-isomorphisms of cochain complexes, by the argument of \cite[1.6.5]{Loday1998} (we remark that the corresponding maps for the curved $A_\infty$ categories need \emph{not} be quasi-isomorphisms: the presence of curvature destroys the length filtration used in the argument).
Thus we have

\begin{lemma}
\label{lemma:curved}
There are algebra homomorphisms
\begin{equation} HH^*(\cA,\cM) \leftarrow \overline{HH}^*(\cA,\cM) \To \overline{HH}^*(\cA_w,\cM_w) \cong  HH^*(\cA_w,\cM_w)\end{equation}
and $\overline{HH}^*(\cA) $-module homomorphisms
\begin{equation} HH_*(\cA,\cM) \To \overline{HH}_*(\cA\cM) \leftarrow \overline{HH}_*(\cA_w,\cM_w) \cong HH_*(\cA_w,\cM_w).\end{equation}
\end{lemma}

\begin{remark}
Lemma \ref{lemma:curved} has a straightforward analogue for the two-pointed Hochschild complexes, compatible with the morphisms of \eqref{eqn:hh2hh} and \eqref{eqn:hh2hh2}.
Furthermore, the homomorphisms
\begin{equation} HH^*(\cA_w,\cM_w) \To {}_2HH^*(\cA_w,\cM_w)\end{equation}
of \eqref{eqn:hh2hh}, and
\begin{equation} _2HH_*(\cA_w,\cM_w) \To HH_*(\cA_w,\cM_w)\end{equation}
of \eqref{eqn:hh2hh2}, \emph{are} quasi-isomorphisms by the arguments of \cite[\S 2.11]{Ganatra2012}, because $\cA_w$ is not curved.
\end{remark}

\section{The cubic surface}
\label{sec:cubsurf}

For the purposes of this section, let $X = X^4_3$ be the cubic hypersurface in $\CP{3}$.

\subsection{The $27$ lines and an open Gromov--Witten invariant}

Following a computation from \cite[`Control example 3' after Corollary 10.9]{Givental1996}, we can compare the Gromov--Witten invariant counting the number of lines on $X$ to the open Gromov--Witten invariant $w(L)$ which counts the number of Maslov index $2$ discs with boundary on one of our Lagrangians $L$.

By Proposition \ref{proposition:qhfromhh}, the class $P$ Poincar\'{e} dual to a hyperplane satisfies the relation
\begin{equation} (P - w)^{\star 3} = 3^3 (P - w)^{\star 2}\end{equation}
in $QH^*(X)$.
Using the axioms of the Gromov--Witten invariants, we compute the number of rigid rational curves in $X$ (such a curve necessarily has degree $1$): it is equal to
\begin{align}
\langle P^2, P \rangle &= \langle P^3,1 \rangle \\
&= \langle 3w P^2 - 3 w^2 P + w^3 + 27P^2 - 54 w P + 27 w^2, 1 \rangle \\
&= 3(w+9) \langle P,P \rangle \\
&= 9(w+9).
\end{align}
Thus, the number of lines is $27$, if and only if $w(L) = -6$.

\subsection{The quantum cohomology of the cubic surface}

Recall that the cubic surface can be expressed as $\CP{2}$ blown up at six points.
Therefore, $QH^*(X)$ has a basis $\{1,h,p,e_1,\ldots, e_6\}$, where $1,h,p$ are the standard basis for $H^*(\CP{2})$ and $e_j$ are the classes of the exceptional divisors.
We introduce convenient auxiliary classes
\begin{equation} M := e_1 + \ldots + e_6,\end{equation}
and
\begin{equation} A := 3h-M + 6.\end{equation}
We remark that the first Chern class is $c_1 = A-6$.

\begin{proposition}
\label{proposition:craud} (From \cite{Crauder1995}).
We have a complete description of $QH^*(X)$:
\begin{align}
p\star p &= 84A + 36\\
h\star p &= 42A - 6h  \\
e_i \star  p &= 14A - 6e_i \\
h\star h &= p + 25A-12h-30\\
h\star e_i &= 9A - 2h -6e_i -12 \\
e_i \star  e_i &= -p+5A-4e_i-10 \\
e_i \star e_j &= 3A - 2(e_i+e_j) -4.
\end{align}
\end{proposition}

This allows us to check Corollaries \ref{corollary:c1eval} and \ref{corollary:evects} explicitly: the eigenvalues of $c_1 \star$ are $-6$ (the big eigenvalue) and $21$ (the small eigenvalue), and the ranks of their generalized eigenspaces are $8$ and $1$, respectively.
The $-6$-generalized eigenspace is spanned by $A-27$, $A\star(A-27)$, and $3e_i - A - 6$ for $i = 1, \ldots, 6$.

\subsection{The small eigenvalue}
\label{subsec:othereig}

Although we have proven homological mirror symmetry for the small components of the Fukaya category (Theorem \ref{theorem:small}), the only objects of the small components that we actually constructed were weak bounding cochains on the Lagrangian spheres which make up the big component.
It would be interesting to construct a monotone Lagrangian submanifold that lives in the small component of the Fukaya category, without any weak bounding cochains being required.
Here we speculate on a possible construction of such a Lagrangian, in the Fano index one case $a=n-1$, where there is only one small component of the Fukaya category.

The degree-$a$ hypersurface in $\CP{n-1}$ can be degenerated to the union of coordinate hyperplanes $\{z_1 \ldots z_a = 0 \}$, giving a tropical manifold in the sense of Gross and Siebert.
The resulting tropical manifold in the case of the cubic surface is illustrated in \cite[Figure 24]{Gross2011}.
There is a monotone Lagrangian torus $L$ in the cubic surface, which is the torus fibre over the centroid of the triangle which is the only compact cell of the tropical manifold.
This construction generalizes immediately to a construction of a monotone Lagrangian torus $L_n \subset X^n_a$ for all $a=n-1$, which is the torus fibre over the centroid of the simplex which is the only compact cell of the corresponding tropical manifold.

\begin{conjecture}
\label{conjecture:othereig}
Let $a=n-1$ and $L_n \subset X^n_a$ be the monotone Lagrangian torus just constructed.
Then $w(L_n) = a^a - a!$ (i.e., $L_n$ lies in the unique small component of the monotone Fukaya category), and $HF^*(L_n,L_n) \cong \Cl_{n-2}$.
\end{conjecture}

\begin{remark}
In particular, note that Conjecture \ref{conjecture:othereig} implies that $L_n$ split-generates the small component of the Fukaya category by Corollary \ref{corollary:semisimpgen}, and would give an alternative proof of Theorem \ref{theorem:small} with `$\cF^{wbc,u}(X)_w$' replaced by `$\cF(X)_w$'.
\end{remark}

We can give a non-rigorous explanation of why this conjecture ought to be true in the case of the cubic surface.
Firstly, there are exactly $21$ tropical discs of Maslov index 2 with boundary on this Lagrangian torus fibre, so one may hope that the honest holomorphic disc count is also $21$ (the tropical disc count is similar to the tropical curve count explained in \cite[\S 1]{Gross2011}).
Secondly, one can check that $HF^*(L,L) \neq 0$ when $L$ is equipped with the trivial $\C^*$-local system: that is because the obvious $\Z/3\Z$ symmetry implies that the disc potential has a critical point at the origin, so the Floer cohomology does not vanish.
If the origin were furthermore a \emph{non-degenerate} critical point, then $HF^*(L,L)$ would be a Clifford algebra as in Conjecture \ref{conjecture:othereig}.

\subsection{Non-semi-simplicity of $HF^*(L,L)$ and $w$}

This section explains a computation made by Seidel (private communication).
We consider the endomorphism algebra $HF^*(L,L)$ of a single lift of our Lagrangian sphere $L$ to $X$.
We know that
\begin{equation}CF^*(L,L) \cong \C\langle e, \theta \rangle\end{equation}
as a vector space.
The differential vanishes, so $HF^*(L,L)$ is a unital $\C$-algebra of rank $2$.
It must be of the form $\C[\theta]/p(\theta)$, where $p(\theta)$ is a quadratic polynomial.
There are two possibilities up to isomorphism: either the algebra is semisimple ($p$ has distinct roots) or not ($p$ has a double root).
It turns out not to be semi-simple, and the proof relies crucially on the value of $w(L) = -6$.

\begin{proposition}
\label{proposition:nonss} (Seidel)
The endomorphism algebra $HF^*(L,L)$ is not semi-simple.
\end{proposition}
\begin{proof}
We recall the algebra homomorphism
\begin{equation} \CO^0: QH^*(X) \To HF^*(L,L),\end{equation}
and that we have
\begin{equation} \CO^0 (c_1) = w \cdot e\end{equation}
(Lemma \ref{lemma:c1u0}).

We derive the following relation in $QH^*(X)$ from Proposition \ref{proposition:craud}:
\begin{equation} c_1^{\star 2} = 3p + 9c_1 + 108.\end{equation}
Applying $\CO^0$, it follows that
\begin{equation} \CO^0(p) = \frac{w^2-9w-108}{3}.\end{equation}
We also have relations from Proposition \ref{proposition:craud}:
\begin{equation} (h-6)^{\star 2} = p + 25 c_1 + 156\end{equation}
and
\begin{equation} (e_i + 2)^{\star 2} = -p + 5c_1 + 20.\end{equation}
Applying $\CO^0$, we obtain
\begin{equation} (\CO^0(h) - 6)^2 = \frac{1}{3}(w+6)(w+60)\end{equation}
and
\begin{equation} (\CO^0(e_i) + 2)^2 = - \frac{1}{3}(w+6)(w-30).\end{equation}
In particular, because $w = -6$, the right-hand sides vanish.

Now if either $\CO^0(h)$ or $\CO^0(e_i)$ had a non-trivial $\theta$ term, these equations would immediately imply that $HF^*(L,L)$ is non-semisimple.
The $\theta$ term of $\CO^0(\alpha)$, for $\alpha \in H^2(X)$, corresponds to the image of $\alpha$ under the restriction map
\begin{equation} H^*(X) \To H^*(L),\end{equation}
because in this case the disc count defining $\CO^0$ reduces to a count of constant holomorphic discs (see, for example, \cite[\S 5a]{Seidel2008b} for the argument in the exact case, and compare Lemma \ref{lemma:ocLf}).
In particular, because the classes $h$ and $e_i$ span $H^2(X)$, it suffices to show that the homology class of $L$ is non-trivial.

We argue by contradiction: if $L$ were homologically trivial, then all the other lifts of the Lagrangian sphere to $X$ would be homologically trivial, by symmetry.
However we proved, in Lemma \ref{lemma:primspheres}, that the homology classes of the lifts of the Lagrangian sphere span the primitive homology, and we know the primitive homology to be non-trivial; therefore, the homology class of $L$ is non-trivial.
This completes the proof.
 \end{proof}


\bibliographystyle{spmpsci}
\bibliography{./library}

\end{document}